\documentclass[11pt,english,a4paper,reqno]{smfart}
\usepackage[english,frenchb]{babel}
\usepackage{enumerate}
\usepackage{amsmath,amssymb,bbold}

\usepackage{color}
\usepackage[all]{xy}
\entrymodifiers={+!!<0pt,\fontdimen22\textfont2>}

\def\fl{
\xymatrix@C15pt{\ar[r]&}
}
\usepackage{ mathrsfs }

\newtheorem{theorem}{Theorem}[section]
\newtheorem{lemma}[theorem]{Lemma}

\newtheorem{corollary}[theorem]{Corollary}
\newtheorem{proposition}[theorem]{Proposition}

\newtheorem*{coro*}{Corollary}

{\theoremstyle{definition}}

{\theoremstyle{definition}\newtheorem{example}[theorem]{Example}}

\theoremstyle{definition}
\newtheorem{definition}[theorem]{Definition}
\newtheorem{question}[theorem]{Question}
\newtheorem{fact}[theorem]{Fact}
\newtheorem{claim}[theorem]{Claim}

{\theoremstyle{definition}\newtheorem{remark}[theorem]{Remark}}

\def\K{\ensuremath{\mathbb K}}
\def\T{\ensuremath{\mathbb T}}
\def\R{\ensuremath{\mathbb R}}
\def\Z{\ensuremath{\mathbb Z}}

\def\C{\ensuremath{\mathbb C}}
\def\Q{\ensuremath{\mathbb Q}}
\def\N{\ensuremath{\mathbb N}}
\def\P{\ensuremath{\mathbb P}}

\newcommand{\hy}{hypercyclic}
\newcommand{\fhy}{frequently hypercyclic}
\newcommand{\ops}{operators}
\newcommand{\op}{operator}

\newcommand{\nt}{non-trivial}

\newcommand{\inv}{invariant}

\newcommand{\erg}{ergodic}
\newcommand{\mea}{measure}

\newcommand{\eva}{eigenvalue}
\newcommand{\ga}{Gaussian}

\newcommand{\pss}[2]{\ensuremath{{\langle #1,#2\rangle}}}

\newcommand{\gn}{n\ge 1}
\newcommand{\gk}{k\ge 1}

\newcommand{\ds}{\displaystyle}
\newcommand{\ba}[1]{\overline{#1}}

\newcommand{\ti}[1]{\widetilde{#1}}

\newcommand{\h}{\mathcal H}
\newcommand{\cct}{C-type}
\newcommand{\ccut}{C$_{2}$-type}
\newcommand{\cpt}{C$_{+}$-type}
\newcommand{\cput}{C$_{+,1}$-type}
\newcommand{\cpdt}{C$_{+,2}$-type}

\newcommand{\tvw}{T_{v,\,w,\,\varphi,\, b\,}}

\newcommand{\bh}{\mathfrak{B}(\h)}

\newcommand{\La}{\pmb{\Lambda }}
\newcommand{\la}{\pmb{\lambda }}
\newcommand{\Om}{\pmb{\Omega }}
\newcommand{\om}{\pmb{\omega }}
\newcommand{\te}{\pmb{\theta }}

\newcommand{\gd}{G_{\delta }}

\newcommand{\ffb}{\mathcal{F}_{B}}
\newcommand{\cmh}{\textrm{CH}_{M}(\h)}
\newcommand{\bmh}{\mathfrak{B}_{M}(\h)}
\newcommand{\hch}{\textrm{HC}(\h)}
\newcommand{\hcmh}{\textrm{HC}_{M}(\h)}
\newcommand{\sot}{\texttt{SOT}}
\newcommand{\sote}{\texttt{SOT}^{*}}

\newcommand{\invh}{\textrm{INV}(\h)}
\newcommand{\invoh}{\textrm{INV}_{f}(\h)}
\newcommand{\ginvoh}{\textrm{G-INV}_{f}(\h)}
\newcommand{\ergh}{\textrm{ERG}(\h)}
\newcommand{\gergh}{\textrm{G-ERG}(\h)}
\newcommand{\smxh}{\textrm{MIX}(\h)}
\newcommand{\gsmxh}{\textrm{G-MIX}(\h)}

\newcommand{\fhch}{\textrm{FHC}(\h)}
\newcommand{\ufhch}{\textrm{UFHC}(\h)}

\newcommand{\nevh}{\textrm{NEV}(\h)}
\newcommand{\nevmh}{\textrm{NEV}_{M}(\h)}

\newcommand{\cevh}{\textrm{CEV}(\h)}
\newcommand{\sph}{\textrm{SPAN}(\h)}
\newcommand{\psph}{\textrm{PSPAN}(\h)}
\newcommand{\cch}{\textrm{CH}(\h)}
\newcommand{\tmxh}{\textrm{TMIX}(\h)}
\newcommand{\dch}{\textrm{DCH}(\h)}
\newcommand{\ddch}{\textrm{DDCH}(\h)}
\newcommand{\ct}{c(T)}
\newcommand{\tth}{\mathfrak{T}(\h)}
\newcommand{\toh}{\mathfrak{T}_{0}(\h)}
\newcommand{\tindh}{\mathfrak{T}_{\textrm{ind}}(\h)}
\newcommand{\ttmh}{\mathfrak{T}_{M}(\h)}
\newcommand{\tomh}{\mathfrak{T}_{0,M}(\h)}
\newcommand{\tindmh}{ \mathfrak{T}_{\textrm{ind},M}(\h)}
\newcommand{\cmaxh}{c\textrm{MAX}(\h)}

\newcommand{\hct}{\textrm{HC}(T)}

\newcommand{\psmh}{\textrm{PSPAN}_{M}(H)}

\author{S. Grivaux}
\address{CNRS, Laboratoire Ami\'enois de Math\'ematique Fondamentale et Appliqu\'ee, UMR 7352\\
Universit\'{e} de Picardie Jules Verne\\
33, rue Saint Leu\\
80039 Amiens Cedex 1\\
France}
\email{sophie.grivaux@u-picardie.fr}
\author{\'{E}. Matheron}
\address{Laboratoire de Math\'{e}matiques de Lens\\
Universit\'{e} d'Artois\\
Rue Jean Souvraz SP 18\\
62307 Lens\\
France}
\email{etienne.matheron@univ-artois.fr}
\author{Q. Menet}
\address{Laboratoire de Math\'{e}matiques de Lens\\
Universit\'{e} d'Artois\\
Rue Jean Souvraz SP 18\\
62307 Lens\\
France}
\email{quentin.menet@univ-artois.fr}
\title[Linear dynamical systems on Hilbert spaces]{Linear dynamical 
systems 
on Hilbert spaces: typical properties and explicit examples}

\date{\today}

\thanks{The work of the first-named author is supported in part by EU IRSES grant AOS (PIRSES-GA-2012-318910).}

\begin{document}

\begin{abstract} We solve a number of questions pertaining to the dynamics of linear operators on Hilbert spaces, sometimes 
by using Baire category arguments and sometimes by constructing explicit examples. In particular, we prove the following results.
\begin{enumerate}
\item[\rm (i)] A typical hypercyclic operator  is not topologically mixing, has no eigenvalues and admits no non-trivial invariant measure, but is densely
distributionally chaotic.
\item[\rm (ii)] A typical {upper-triangular} operator is ergodic in the Gaussian sense, whereas a typical operator of the form ``diagonal plus backward unilateral weighted shift" 
is ergodic but has only countably many unimodular eigenvalues; in particular, it is ergodic but {not} ergodic in the Gaussian sense.
\item[\rm (iii)] There exist Hilbert space operators which are chaotic and $\mathcal U$-frequently hypercyclic but not frequently hypercyclic, 
Hilbert space operators which are {chaotic and} frequently hypercyclic but not ergodic, and Hilbert space operators which are chaotic and topologically mixing but not $\mathcal U$-frequently hypercyclic.
\end{enumerate}

\noindent
We complement our results by investigating the descriptive complexity of some natural classes of operators defined by dynamical properties.
\end{abstract}

\maketitle

\tableofcontents

\section{Introduction}\label{Section 1-a}

\subsection{General overview}\label{Subsection 1.a.1} This paper is a contribution to the study of linear dynamical systems 
on Hilbert spaces. In other words, we are interested in the behavior of orbits of bounded linear operators defined on a 
Hilbert space. The symbol $\h$ will always designate a complex separable infinite-dimensional Hilbert 
space, and we denote by $\bh$ the algebra of all bounded linear operators on $\h$. 
\par\smallskip
We refer to the books \cite{BM} or \cite{GEP} for background on linear dynamics, and to the papers \cite{GW} or \cite{Gr} for a glimpse at the richness of the class of linear dynamical systems 
and its potential usefulness in general ergodic theory. A quick 
review of a number of definitions will be given in the next subsection. Let us just recall here that an operator $T\in\bh$ is said to be \emph{hypercyclic} if it 
admits at least one dense orbit (and hence a dense $G_\delta$ set of such orbits). The class of all hypercyclic operators 
on $\h$ will be denoted by $\hch$. Recall also that an operator $T\in\bh$ is said to be \emph{chaotic} if it is hypercyclic and admits a dense set 
of periodic vectors, and \emph{frequently hypercyclic} (resp. $\mathcal{U}$-\emph{frequently hypercyclic}) if there exists at least one vector $x\in\h$ whose orbit visits frequently every non-empty open set $V\subseteq\h$, 
in the sense that 
the set of integers $i\in\N$ such that $T^ix$ belongs to $ V$ has positive lower (resp. upper) density.
\par\smallskip
Obviously, the dynamical properties of linear operators can be studied (and {have} been studied extensively) in a very general setting -- arbitrary  
Banach spaces, or even arbitrary topological vector spaces. However, there are good reasons for focusing on Hilbert spaces. The first  that may come to mind 
is perhaps 
that ``this is the natural setting for doing operator theory" -- which is of course a highly questionable statement. 
Less subjectively,  the richness of the Hilbert space structure allows for the construction of many interesting examples. Also, in some parts of linear dynamics, especially all that concerns \emph{ergodic-theoretic} 
properties of linear operators, the general picture is neater on Hilbert spaces 
than on arbitrary Banach spaces. Finally, some natural questions in the area have been solved recently for operators  
on general Banach space but \emph{not} in the Hilbertian case. This should not seem paradoxical, if one compares for example with the 
current state of affairs regarding the famous Invariant  Subspace Problem.
\par\smallskip
The questions we are considering in this paper are quite basic. In very general terms, they are of the following type: given two interesting dynamical 
properties, do there exist linear operators on $\h$ satisfying one of them but not the other? 
\par\smallskip
Perhaps the most famous question of this type is Herrero's 
``$T\oplus T$ problem", stated in \cite{Her} and asking whether there exist hypercyclic operators which are not \emph{topologically weakly mixing}. As should be expected 
since this is trivially true for dynamical systems on compact spaces, the answer is ``Yes": this  was shown by De La Rosa and Read in \cite{Manuel},
and then in \cite{BM3} for operators on $\ell_p$ spaces, in particular for Hilbert space operators. However, the proof is surprisingly non-trivial. To give just one more example, 
the third named author was recently able to solve  another 
well-known problem in the area by constructing in \cite{Me} operators on $\ell_p$ spaces which are {chaotic}  but 
not {frequently hypercyclic} -- actually, not even $\mathcal{U}$-{frequently hypercyclic}. 
In the present paper, we will be especially interested in the following questions:
\par\smallskip
\begin{itemize}
\item[-] are there operators which are frequently hypercyclic but not \emph{ergodic}, \mbox{\it i.e.} do not admit an ergodic measure with full support? 
\item[-] are there operators which are {$\mathcal U$-frequently hypercyclic} but not frequently hypercyclic? 
\item[-] are there operators 
which are ergodic but do not admit 
a \emph{Gaussian} ergodic measure with full support?
\end{itemize}
\par\smallskip
Note that the answer to the first two questions is known to be ``Yes" for operators on the Banach space $c_0$ (see \cite{BR} and \cite{GM}, as well as 
\cite{BGE} for further developments); 
but the Hilbertian case
had not been settled. As for the third question, we are not aware of any previous answer, on any Banach space.
\par\smallskip
One can attack questions of this type 
from two complementary points of view (of course, there are other strategies as well).
\par\smallskip
\begin{itemize}
\item[-] ``Collective" point of view: among the properties one is interested in, one may try to determine which ones 
are \emph{generic} and which ones are not, in a {Baire category} sense. Once this is done, one may be able to distinguish 
two properties because one of them is generic and the other one is not. 
 \par\smallskip
\item[-] ``Individual" point of view: when it is unclear how a Baire category approach could work, one may still try 
to construct explicit examples of operators satisfying (or not) such or such a property. 
\end{itemize}
\par\smallskip
Note that it is quite natural to present the two viewpoints in this order, because the indirect, Baire category approach is in some sense simpler.  Indeed, it is usually not too difficult to guess when it should work, and 
in this case the technical details are likely to be rather soft. On the other hand, one may reasonably expect to face technical difficulties when attempting 
a direct construction (and perhaps quite serious ones if Baire category is unefficient).
\par\smallskip
To give meaning to the ``collective" viewpoint, one has to fix an appropriate topological setting. We will in fact consider \emph{two} 
natural 
topologies on $\bh$: the \emph{Strong Operator Topology} (denoted by $\texttt{SOT}$), and the 
\emph{Strong$^{\,*}$ Operator Topology} (denoted by $\texttt{SOT}^*$).
Restricted to any closed ball $\bmh:=\{ T\in\bh;\; \Vert T\Vert\leq M\}$, these topologies are \emph{Polish} (separable and completely metrizable). 
Moreover, $\textrm{HC}(\h)$ is easily seen to be $\texttt{SOT}$-$G_\delta$ in $\bh$, so that for any $M>1$, the family 
$\textrm{HC}_M(\h):=\{ T\in\hch;\; \Vert T\Vert\leq M\}$ is itself a Polish space  in its own right with respect to both $\texttt{SOT}$ and $\texttt{SOT}^*$. 
This opens the way to Baire category arguments in $\bmh$ or $\hcmh$. Of course, there is nothing new in this observation: Baire category methods have already proved as useful in operator theory as 
anywhere else; see in particular \cite{EM}.
\par\smallskip
 Recall that a subset of a Polish space $X$ is said to be \emph{meager} in $X$ if it can be covered by countably many 
closed sets with empty interior, and \emph{comeager} if its complement is meager (equivalently, 
if it contains a dense $G_\delta$ subset of $X$). Following a well-established terminology, we will say that a property of elements of $X$ is \emph{typical} if the set of all $x\in X$ satisfying it 
is comeager in $X$, and that a property is \emph{atypical} if its negation is typical.
 In this paper, we obtain among other things the following results, for any $M>1$.
 \par\smallskip
 \begin{itemize}
 \item[-] An $\texttt{SOT}^*$-typical $T\in\bmh$ is (topologically) weakly mixing but not mixing. 
  \item[-] An $\texttt{SOT}^*$-typical $T\in\bmh$ is densely distributionally chaotic.
 \item[-] An $\texttt{SOT}^*$-typical $T\in\hcmh$ has no  eigenvalues and admits no non-trivial invariant measure. 
 In particular, chaoticity and $\mathcal U$-frequent hypercyclicity are atypical properties of hypercyclic operators (in the $\sote$ sense).
 \end{itemize}
\par\smallskip
Admittedly, there results are not {that} surprising (except, perhaps, the one concerning $\mathcal U$-frequent hypercyclicity). Still, 
they do sketch the landscape, and they explain in some sense why some results in the area turn out to be (or are likely to be) harder to prove than 
some others. For example, the fact that topological weak mixing is a 
typical property seems to prevent the use of ``soft" arguments for establishing the existence of hypercyclic operators which are not
weakly mixing; and indeed, the proofs in \cite{Manuel} and \cite{BM3} rely on rather technical arguments. 
\par\smallskip
On the other hand, something more unexpected happens if one restricts oneself to the class of operators of the form $T=D+B$, where 
$D$ is a diagonal operator and $B$ is a weighted unilateral backward shift (with respect to some fixed orthonormal basis of $\h$): 
\par\smallskip
\begin{itemize}
\item[-] 
 {in the $\texttt{SOT}$ sense, most ``diagonal plus shift" hypercyclic operators are ergodic and yet have only countably many unimodular eigenvalues. In particular, there exist operators on $\h$ which are ergodic but admit no {Gaussian} ergodic measure with full support.} 
 \end{itemize}

\smallskip\noindent
 This solves a rather intriguing question, which seems to go back to Flytzanis' paper~\cite{Fl} and has been very much in the air in the last few years.
\par\medskip
As for the ``individual" viewpoint, we will elaborate on the kind of operators constructed by the third named author in \cite{Me}. Recall that 
the main result of \cite{Me} is the existence of chaotic operators on $\ell_{p}$ spaces which are not {$\mathcal{U}$-}frequently hypercyclic. In the present paper, we describe 
a general scheme allowing to produce, among other things, 
\par\smallskip
\begin{itemize}
\item[-] operators on $\mathcal H$ which {are {chaotic and} frequently hypercyclic but not ergodic};
\item[-] operators on $\h$ which are {{chaotic and} $\mathcal U$-frequently hypercyclic but not frequently hypercyclic}.
\end{itemize}  
\par\smallskip
These are the first examples of such kinds of operators living on a Hilbert space, even if one dispenses with the chaoticity assumption: no examples of frequently hypercyclic non-ergodic operators, nor of $\mathcal{U}$-frequently hypercyclic non frequently hypercyclic operators living on a Hilbert space were known before. As already mentioned, up to now the only available examples were 
operators acting on $c_0$.
\par\smallskip
Moreover, our constructions will also enable us to improve the main result of  \cite{Me} by showing that there exist

\smallskip
\begin{itemize}
\item[-] operators on $\h$ which are {chaotic and \emph{topologically mixing} but not $\mathcal U$-frequently hypercyclic}.
\end{itemize}
\par\smallskip
One point is especially worth mentioning regarding these results: it is not at all accidental that all the operators we construct turn out to be chaotic. On the contrary, our proofs rely heavily on new criteria for frequent 
hypercyclicity and $\mathcal U$-frequent hypercyclicity (stronger than the usual ones), in which the periodic points play a central role.
\par\smallskip
In this part of the paper, our constructions and arguments are rather technical, and 
it is not at all clear that they could be by-passed by suitable Baire category arguments. In fact, as already suggested above, 
there are  reasons to believe 
that technicalities are unavoidable here: since $\mathcal U$-frequent hypercyclicity, frequent hypercyclicity and 
ergodicity are atypical properties, it seems difficult to distinguish them using simply the Baire category theorem; unless one
restricts oneself to some cleverly chosen special class of operators, which has  yet to be found.

\medskip The basic questions we address in this paper can also be considered from a third point of view,  which is that of \emph{descriptive set theory}. Indeed, once it is known (by any argument) 
that two properties of linear operators are not the same, \textit{i.e.} 
that two classes of operators are distinct, it is natural to wonder if a stronger conclusion might hold true, namely that these classes do not have the same \emph{complexity} in the sense of 
descriptive set theory. More generally, it can be a quite interesting problem to determine the exact descriptive complexity of a given class of operators. In this paper, we do so for chaotic operators and for 
topologically mixing operators. We also obtain a partial result for $\mathcal U$-frequently hypercyclic operators, whose proof relies on our general scheme for constructing operators with special properties.

\par\medskip
We finish this overview by pointing out that the difficulty of the existence results presented in this paper is specifically connected with the linear setting, and that the corresponding results are essentially trivial if one moves over to the broader setting of \emph{Polish dynamical systems} (\mbox{\it i.e.} dynamical systems of the form $(X,T)$, where $T$ is a continuous self-map of a Polish space $X$; see for instance the book \cite{Gl}, as well as \cite{GlDa} for more on linear systems as special cases of Polish systems). 
Indeed, extending the definitions of frequent and $\mathcal{U}$-frequent hypercyclicity to the Polish setting in the obvious way, it is not difficult to see that frequently hypercyclic non-ergodic Polish dynamical systems do exist, as well as $\mathcal{U}$-frequently hypercyclic Polish dynamical systems which are not frequently hypercyclic. Here are two simple examples.

First, consider an irrational rotation $R$ of the unit circle $\T$, and denote by $m$ the normalized Lebesgue \mea\ on $\T$. Then $(\T,\mathcal{B}, m; R)$ is an \erg\  dynamical system for which all points are \fhy. Let $C$ be a compact subset of $\T$ which has empty interior and is such that $m(C)>0$, and consider the set $X:=\T\setminus\bigcup_{n\in\N}R^{-n}(C)$. Then $X$ is a dense $G_{\delta }$ subset of $\T$ which is $R$-\inv, and $(X,R)$ is thus a \fhy\ Polish dynamical system. But, as $m(X)=0$ (by ergodicity) and $R$ is uniquely \erg, $(X,R)$ admits no \inv\ \mea\ at all. This shows in particular the existence of frequently hypercyclic Polish dynamical systems which are not ergodic.
This example is due to B. Weiss (private communication).

Now, consider a \fhy\ operator $T$ on $\h$. Then, the set $\textrm{UFHC}(T)$ of all $\mathcal{U}$-frequently hypercyclic vectors for $T$ is {comeager} in $\h$ (\cite{BR}, see also \cite{BGE}), whereas the set $\textrm{FHC}(T)$ of frequently hypercyclic vectors for $T$ is {meager} (\cite{Moo},  \cite{BR}). Hence $\textrm{UFHC}(T)\setminus \textrm{FHC}(T)$ is comeager in $\h$. Let $G$ be a dense $G_{\delta}$ subset of $\h$ contained in $\textrm{UFHC}(T)\setminus \textrm{FHC}(T)$, and set $X:=\bigcap_{n\ge 0}T^{-n}(G)$. Then $X$ is a dense $G_{\delta}$ subset of $\h$ which is $T$-invariant, so that $(X,T)$ is a Polish dynamical system. Since $X$ is still contained in $\textrm{UFHC}(T)\setminus \textrm{FHC}(T)$, all points of $X$ are $\mathcal{U}$-frequently hypercyclic for $T$ but none of them is frequently hypercyclic. In particular, the Polish dynamical system $(X,T)$ is $\mathcal{U}$-frequently hypercyclic but not frequently hypercyclic.
\par\smallskip

\subsection{Background and notations} 
In this subsection, we recall some well-known definitions, referring to \cite{BM} or \cite{GEP} for more details. We also fix some 
notations that will be used throughout the paper.
\par\smallskip

\subsubsection{Transitivity, mixing and weak mixing} If $T\in\bh$, we set, for any subsets $A,B$ of $\h$, 
\[\mathcal N_T(A,B):=\{ i\in\N;\; T^i(A)\cap B\neq\emptyset\}.\]
\par\smallskip
It is well-known that $T$ is hypercyclic if and only if it is \emph{topologically transitive}, \mbox{\it i.e.} $\mathcal N_T(U,V)\neq\emptyset$ 
for all open sets $U,V\neq\emptyset$. A stronger property is \emph{topological mixing}: $T$ is topologically mixing if all sets 
$\mathcal N_T(U,V)$ are \emph{cofinite} rather than just non-empty; that is, for any pair $(U,V)$ of non-empty 
open sets in $\h$, one has $T^i(U)\cap V\neq\emptyset$ for all but finitely many $i\in\N$. In between transitivity and mixing is 
\emph{topological weak mixing}: an operator $T\in\bh$ is {topologically weakly mixing} if $T\times T$ is hypercyclic on $\h\times\h$; 
in other words, if $\mathcal N_T(U_1,V_1)\cap\mathcal N_T(U_2,V_2)\neq\emptyset$ for all non-empty open sets 
$U_1,V_1,U_2,V_2\subseteq\h$. We set 
\begin{align*}
\textrm{TWMIX}(\h)&:=\{ \hbox{topologically weakly mixing operators on $\h$}\},\\
\tmxh&:=\{ \hbox{topologically mixing operators on $\h$}\};
\end{align*}
so that
\[\tmxh\subseteq \textrm{TWMIX}(\h)\subseteq\hch.\]
\par\smallskip
It is easy to see (for example by considering weighted backward shifts) that the inclusion $\tmxh\subseteq \textrm{TWMIX}(\h)$ is proper. 
As already mentioned, the inclusion $\textrm{TWMIX} \subseteq \textrm{HC}$ is also proper: this was first proved by De La Rosa and Read in \cite{Manuel} 
for operators living on a suitably manufactured Banach space, and then in \cite{BM3} for Hilbert space operators. Recall also that, according to a nice result of B\`es and Peris \cite{BP}, 
the topologically weakly mixing operators are exactly those satisfying the so-called \emph{Hypercyclicity Criterion}.
\par\smallskip

\subsubsection{Chaos} An operator $T\in\bh$ is said to be \emph{chaotic} in the sense of Devaney if it is hypercyclic 
and its periodic points are dense in $\h$ (in the linear setting, hypercyclicity automatically implies sensitive dependence on the initial conditions).
We set
\[\cch:=\{\hbox{chaotic operators on $\h$}\}.\]
\par\smallskip
It is not completely obvious, but nonetheless true, that chaotic operators are topologically weakly mixing (see \mbox{e.g.} \cite[Ch. 4]{BM}). In other words, 
\[\cch\subseteq \textrm{TWMIX}(\h).\]
\par\smallskip

\subsubsection{Ergodic-theoretic properties} In this paper, the word ``measure" will always mean 
``Borel probability measure". A measure $\mu$ on $\h$ is \emph{invariant} for some operator $T\in\bh$ if $\mu\circ T^{-1}=\mu$; 
and an invariant measure $\mu$ for $T$ is \emph{non-trivial} if $\mu\neq\delta_0$ (note that the point mass $\delta_0$ is invariant for 
any $T\in\bh$ since $T(0)=0$). Also, a measure $\mu$ on $\h$ is said to have \emph{full support} if $\mu(V)>0$ for every non-empty open set $V$. This
means exactly that the topological support of $\mu$ is equal to the whole space $\h$. 
\par\smallskip
An operator $T\in\bh$ is said to be \emph{ergodic} if it admits an invariant measure $\mu$ with full support with respect to which it is ergodic, \textit{i.e.} such that for every Borel subset $B$ of $\h$ satisfying $T^{-1}B=B$, we have $\mu(B)=0$ or $1$. 
If the measure $\mu$ can be taken to be Gaussian, we say that $T$ is \emph{ergodic in the Gaussian sense}, or \emph{$G$-ergodic}. The operator $T$ is said to be
 \emph{mixing} if it admits an invariant measure $\mu$ with full support with respect to which it is strongly mixing, \textit{i.e.} such that $\mu(A\cap T^{-n}(B))\to \mu(A)\mu(B)$ for any Borel sets $A,B\subseteq\h$. 
 Mixing \emph{in the Gaussian sense} is defined as expected. The corresponding notations are the following:
\begin{align*}
\ergh&:=\bigl\{ \hbox{ergodic operators on $\h$}\bigr\};\\
\gergh&:=\bigl\{ \hbox{operators on $\h$ which are ergodic in the Gaussian sense}\bigr\};\\
\smxh&:=\bigl\{ \hbox{mixing operators on $\h$}\bigr\};\\
\gsmxh&:=\bigl\{\hbox{operators on $\h$ which are mixing in the Gaussian sense}\bigr\}.
\end{align*}
Since the measures involved are required to have full support, it is obvious by definition that ergodic operators are hypercyclic and that mixing 
operators are topologically mixing. That is, 
\[\ergh\subseteq\hch\qquad{\rm and}\qquad \smxh\subseteq\tmxh.\]
\par\smallskip
These inclusions are proper: for example, if $B$ is any compact weighted backward shift on the Hilbert space $\ell_{2}(\N)$, the operator $T=I+B$ is topologically 
mixing (see \cite{BM} or \cite{GEP}) but not ergodic: its spectrum is reduced to the point $\{ 1\}$, so that it is by \cite{S2} not even $\mathcal{U}$-frequently hypercyclic.
We also set 
\begin{align*}
\invh&:=\bigl\{ T\in\bh\ \hbox{admitting a non-trivial invariant measure}\bigr\};\\
\invoh&:=\bigl\{T\in\bh\ \hbox{admitting an invariant measure with full support}\bigl\};\\
\ginvoh&:=\bigl\{T\in\bh\ \hbox{admitting a Gaussian invariant measure with full support}\bigr\}.
\end{align*}
Thus, for example, any operator admitting a non-zero periodic point belongs to $\invh$, and any chaotic operator lies in $\invoh$. So we have
\[\cch\subseteq\invoh\cap \hch.\]
This inclusion cannot be reversed: there exist even {$G$-ergodic} operators on $\h$ which are not chaotic (see \cite{BayBer} or \cite[Section 6.5]{BM}).
\par\smallskip 
At this point, we would like to stress that in the present paper, we will mostly focus 
on invariant measures which 
are \emph{not required to be Gaussian}. This is a true change of 
point of view compared with earlier works on the ergodic theory of linear dynamical 
systems (see for instance, among other works, \cite{Fl}, \cite{BG2}, \cite{BG3}, \cite{BM} \dots), where 
Gaussian measures play a central role. But in retrospect, this is in fact quite natural: as we shall see in Section~\ref{UPPERTRIANG}, it turns out that within 
a certain natural class of upper-triangular operators on $\h$, the ergodic operators which are not ergodic in the Gaussian sense
are typical. Note also that non-Gaussian measures 
 already played an essential role in such works as \cite{MP}, 
\cite{GW} or \cite{Gr} for instance. 
\par\smallskip

\subsubsection{{Frequent hypercyclicity} and {$\mathcal U$-frequent hypercyclicity}} Let $T\in\bh$. 
For any $x\in \h$ and $B\subseteq\h$, set
\[\mathcal N_T(x,B):=\{ i\in\N;\; T^ix\in B\}.\]
The operator $T$ is said to be \emph{frequently hypercyclic} if there exists some $x\in \h$ such that, for any open set 
$V\neq\emptyset$, the set $\mathcal N_T(x,V)$ has positive lower density; and $T$ is 
\emph{$\mathcal U$-frequently hypercyclic} if there is some $x\in\h$ such that all sets $\mathcal N_T(x,V)$ has 
positive upper density.
\par\smallskip
Frequent hypercyclicity was introduced in \cite{BG2} and rather extensively studied since then 
(for instance in \cite{BR}, \cite{BoGR}, \cite{G}, \cite{GM}, \cite{Me}, \cite{S2} \dots). The study of upper-frequent hypercyclicity is more recent. This notion was introduced by Shkarin in \cite{S2}. Until the last 
few years, it has perhaps been unfairly considered as somehow ``less interesting", probably because of a lack of examples or results exhibiting truly different behaviors between frequently and 
$\mathcal{U}$-frequently hypercyclic operators. 
However, despite a formal similarity in the definitions, there are some important differences between frequent 
and $\mathcal U$-frequent hypercyclicity. For example, the set of frequently hypercyclic vectors for a frequently hypercyclic operator $T$ on $\h$ is 
always \emph{meager} in $\h$ (\cite{Moo},  \cite{BR}) whereas if $T$ is a $\mathcal{U}$-frequently hypercyclic operator
on $\h$,  the set of all $\mathcal U$-frequently hypercyclic 
vectors for $T$ is \emph{comeager} in $\h$ (\cite{BR}; see also \cite{BGE} for more along these lines). This may lead to believe that 
$\mathcal U$-frequent hypercyclicity is a typical property while frequent hypercyclicity is not; which is in fact a wrong intuition: as we shall see, 
both properties are atypical. Yet, these properties are not equivalent: Bayart and Rusza exhibited in \cite{BR} examples of $\mathcal{U}$-frequently hypercyclic operators on $c_{0}$ which are not frequently hypercyclic. The notion of $\mathcal{U}$-frequent hypercyclicity was further studied by Bonilla and Grosse-Erdmann in \cite{BGE}, and it plays an important role in the present paper.
\par\smallskip
One can go one step further and study \emph{$\mathcal F$-hypercyclic} operators, for a given family 
$\mathcal F$ of subsets of $\N$ (see in particular \cite{BMPP} and \cite{BGE}); but we will not follow this quite interesting route here. 
As for the notations, we set 
\begin{align*}
\fhch&:=\{\hbox{frequently hypercyclic operators on $\h$}\},\\
\ufhch&:=\{\hbox{$\mathcal U$-frequently hypercyclic operators on $\h$}\}.
\end{align*}
\par\smallskip
Although the definitions of frequent hypercyclicity and $\mathcal U$-frequent hypercyclicity make no explicit reference 
to measures, they have a partly ``metrical" flavour; and indeed, invariant measures are quite relevant here. First, it follows 
easily from the pointwise ergodic theorem that ergodic operators are frequently hypercyclic. Moreover, it is shown in
 \cite{GM} that $\mathcal U$-frequently hypercyclic operators on $\h$ admit invariant measures with full support 
 (this is in fact true for operators living on any {reflexive} Banach space). So we have
\[\ergh\subseteq\fhch\subseteq\ufhch\subseteq\hch\cap\invoh.\]
\par\smallskip
The inclusion $\ufhch\subseteq\hch\cap\invoh$ is proper: as already mentioned, it is proved in \cite{Me} that there exist even {chaotic} 
operators on $\h$ which are not $\mathcal U$-frequently hypercyclic. As we shall see in Section \ref{SPECIAL} (Theorems \ref{Theorem 54} and 
\ref{Theorem 58}), 
the other two inclusions are also proper.
\par\smallskip

\subsubsection{Properties related to {eigenvalues}} If $T\in\bh$, we denote by $\mathcal E(T)$ 
the set of all \emph{unimodular eigenvectors} of $T$, \mbox{\it i.e.} eigenvectors 
$x$ of $T$ whose associated eigenvalue $\lambda(x)$ has modulus $1$. We say that $T$ has a \emph{spanning set of unimodular eigenvectors} 
if $\overline{\rm span}\;\mathcal E(T)=\h$; and that $T$ has a \emph{perfectly  spanning set of unimodular eigenvectors} if for 
any countable set $D\subseteq\mathbb T$, where $\T$ denotes the unit circle in $\C$, we have $\overline{\rm span}\,\{x\in\mathcal E(T);\;\lambda(x)\not\in D\}=\h$. 
The notations are as follows: 
\begin{align*} 
\sph&:=\bigl\{T\in\bh\ \hbox{with spanning unimodular 
eigenvectors}\bigr\};\\
\psph&:=\bigl\{T\in\bh\ \hbox{with perfectly spanning unimodular 
eigenvectors}\bigr\}.
\end{align*}
We also define
\begin{align*}
 \nevh&:=\bigl\{T\in\bh\ \textrm{with no eigenvalues}\bigr\};\\
\cevh&:=\bigl\{T\in\bh\ \textrm{with only countably many  eigenvalues}\bigr\}.
\end{align*}
Properties of unimodular eigenvectors of an operator $T\in\bh$ are closely related to those of 
invariant \emph{Gaussian} measures for $T$ (see for instance 
\cite{BG2}, \cite{BG3}, 
\cite{BM}): indeed, it turns out that in fact 
\[\sph=\ginvoh\qquad{\rm and}\qquad \psph=\gergh;\]
so that there was in fact no need for introducing new notations. However, things are not that neat on arbitrary 
Banach spaces: see \cite{BG3}, \cite{BM} or \cite{BM2}. This is one of the advantages of working on Hilbert 
spaces, or at least on ``nice" Banach spaces.
\par\smallskip

\subsubsection{Distributional chaos} Distributional chaos was first defined by Schweitzer and Sm\'ital in \cite{SS}, under the name \emph{strong chaos}, 
for self-maps of a compact interval, and 
several definitions have been proposed afterwards in the context of general metric spaces. 
There is no need to recall these definitions here, because things simplify greatly in the linear setting: as shown in \cite{BBMP}, 
an operator $T\in\bh$ is distributionally chaotic if and only if it admits a \emph{distributionally irregular vector}, 
\mbox{\it i.e.} a vector $x\in\h$ for 
which 
there 
exist two sets $A$ and $B$ of integers, both of upper density $1$ in $\N$, such 
that
\[
\lim 
_{\genfrac{}{}{0pt}{1}{i\to\infty }{i\in A}}\|T^{i}
x\|=0\quad\textrm { and } \quad
\lim_{\genfrac{}{}{0pt}{1}{i\to\infty }{i\in 
B}}\|T^{i}x\|=\infty.
\]
The corresponding notation is the following one:
\[
\dch:=\bigl\{ \hbox{distributionally chaotic operators on $\h$}\bigr\}.
\]

A related class is that of \emph{densely distributionally chaotic operators}, \mbox{\it i.e.} of operators which admit a dense set of distributionally irregular vectors. We use the following notation for this set:
\[
\ddch:=\bigl\{ \hbox{densely distributionally chaotic operators on $\h$}\bigr\}.
\]
Obviously, $\ddch\subseteq\dch$.
\par\smallskip

\subsubsection{The parameter $c(T)$}
Ergodicity and distributional chaos are closely related to a natural parameter introduced in 
\cite{GM}. This quantity $\ct\in[0,1]$ is associated to any {hypercyclic} operator 
$T$, and essentially represents the maximal frequency with which 
the 
orbit of a hypercyclic vector for $T$ can visit a ball in $\h$ centered 
at $0$. The precise definition is as follows: for any $\alpha >0$, we have
\[ c(T)=\sup_{x\in HC(T)}\,\overline{\rm dens}\,\mathcal N_T\bigl(x,B(0,\alpha)\bigr).\]

It is shown in \cite{GM} that in fact $\overline{\rm dens}\,\mathcal N_T\bigl(x,B(0,\alpha)\bigr)=c(T)$ for a comeager set of vectors $x\in HC(T)$; so we have in particular (for any $\alpha>0$) 
\[ c(T)=\sup\,\left\{ c\geq 0;\; \overline{\rm dens}\Bigl(\mathcal N_T\bigl(x, B(0,\alpha)\bigr)\Bigr)\geq c\quad 
\hbox{for a comeager set of $x\in HC(T)$}\right\}.
\]

Note also that $c(T)>0$ if the operator $T$ is $\mathcal U$-frequently hypercyclic.
\par\smallskip
The last class of operators which we introduce is that of operators 
$T\in\hch$ such that $c(T)$ is maximal:
\[
\cmaxh:=\bigl\{T\in\hch\,;\,\ct=1 \bigr\}.
\]

Since, as proved in \cite{BBMP}, the set of all distributionally irregular vectors for a given operator $T\in\ddch$ is comeager in $\h$, it 
is clear that $\ddch\cap\textrm{HC}(\h)\subseteq \cmaxh$. Moreover, it is shown in \cite{GM} (more accurately, half in \cite{BR} and half in \cite{GM}) that ergodic operators are densely
distributionally chaotic: $\ergh\subseteq\ddch$. So we have
\[\ergh\subseteq \ddch\cap\textrm{HC}(\h)\subseteq\cmaxh.
\]

It is left as an open question in \cite{GM} to determine whether $\hch\cap\invoh\subseteq\cmaxh$, or whether at least 
$\fhch\subseteq\cmaxh$. It follows from the proof of the main result of \cite{Me} that the first inclusion does 
{not} hold true: the chaotic operator $T$ constructed there belongs to $\hch\cap\invoh$ (as does any chaotic operator on $\h$) 
but satisfies $c(T)=0$. However this operator is not frequently hypercyclic, and so the examples of \cite{Me} do not disprove the second inclusion. 
We will nonetheless show in the present paper that this inclusion does not hold true either (Theorem \ref{Theorem 54}).
\par\smallskip

\subsubsection{A last notation} If $\Gamma(\h)$ is a class of operators on $\h$, then, 
for any $M>0$, we set
\[\Gamma_M(\h):=\{ T\in\Gamma(\h);\; \Vert T\Vert\leq M\}.\]
\par\smallskip

\subsection{{Organization of the paper}} In Section \ref{TYPICAL}, we recall a few basic facts concerning the topologies $\texttt{SOT}$ and $\texttt{SOT}^*$, and then  prove
some ``typicality" results in the spaces $\bmh$ and $\hcmh$, $M>1$ with respect to the Strong$^*$ Operator Topology $\sote$. These results can be summarized as follows: for any $M>1$, an $\texttt{SOT}^*$-typical $T\in \hcmh$ 
is topologically weakly mixing but not mixing (Propositions \ref{WM-dense} and \ref{Proposition 13}), 
has no eigenvalues (Corollary \ref{Prop5coro}), admits no non-trivial invariant measure (Theorem \ref{Theorem 9}), but is densely distributionally chaotic (Proposition \ref{Theorem 15}). 
\par\smallskip
Section \ref{COMPLEXITY} is a digression, in which we discuss the {descriptive complexity} of some of the families of operators introduced above. We show that for any $M>1$, 
 $\textrm{TMIX}_M(\h)$ and $\cmh$ are Borel subsets of $(\bmh, \texttt{SOT}^*)$ of class exactly 
$\mathbf\Pi_3^0$, aka $F_{\sigma\delta}$ (Propositions~\ref{Prop en plus} and~\ref{Proposition 8}), 
whereas $\textrm{UFHC}_M(\h)$ is Borel of class at most $\mathbf\Pi_4^0$, and neither $\mathbf\Pi_3^0$ nor $\mathbf\Sigma_3^0$ (Corollary \ref{Proposition 9.4}). We also show that some rather natural classes of operators defined by dynamical properties  are \emph{non-Borel} 
in $\bmh$; for example, the family of all operators $T\in\bmh$ admitting a hypercyclic restriction to an invariant subspace is non-Borel (Proposition \ref{restrictions}), as well as the family of distributionally chaotic operators $T\in\bmh$ (Proposition \ref{Proposition 3.23}). In contrast, the class of {densely} distributionally chaotic operators $T\in\bmh$ is $G_{\delta}$
(Proposition \ref{Theorem 15}).
\par\smallskip
In Section \ref{UPPERTRIANG}, we consider ergodicity properties of \emph{upper triangular} operators, this time with respect to the Strong Operator Topology $\texttt{SOT}$.  We show first that for any $M>1$, an $\texttt{SOT}$-typical upper triangular 
operator $T\in\bmh$ is ergodic in the Gaussian sense (Proposition \ref{Proposition 26}). On the other hand, we essentially show in Theorem \ref{Theorem 28} that an $\texttt{SOT}$-typical operator $T\in\bmh$ 
of the form ``diagonal + backward shift" is ergodic but admits only countably many eigenvalues (and hence is ergodic but not ergodic in the Gaussian sense). 
\par\smallskip
 In Section \ref{CRITERIA}, we prove several criteria for an operator $T$ to be $\mathcal U$-frequently hypercyclic or frequently hypercyclic. 
 These criteria are rather different in spirit from the by now classical Frequent Hypercyclicity Criterion (\cite{BoGR}) or from the more recent criteria which
 can be found in \cite{BMPP} or \cite{BGE}, since they rely explicitly on the existence of many \emph{periodic} (or almost periodic) vectors for the operator $T$. 
 However, our criterion for frequent hypercyclicity turns out to be stronger than the classical one: indeed, any operator satisfying the so-called 
 \emph{Operator Specification Property} also satisfies the assumptions of our criterion, whereas it is known \cite{BMPe2} that operators satisfying the assumptions 
 of the classical Frequent Hypercyclicity Criterion have the Operator Specification Property (but not conversely, see \cite{BMPe2}).
In the present paper, these criteria for $\mathcal U$-frequent hypercyclicity and frequent hypercyclicity are instrumental: we  use them in order to simplify the proofs of several later
results (more precisely, Corollary \ref{Proposition 36}, Theorem \ref{Theorem 39} and Theorem \ref{Theorem 41} will be used in the proofs of the main results of Section \ref{SPECIAL}).
However, we believe that these criteria might be useful in other situations as well; and for this reason we have stated them in the 
setting of general Banach spaces.
\par\smallskip
In Section \ref{SPECIAL}, which is by far the the most technical part of the paper, we develop a general machinery for producing hypercyclic 
operators with special properties. This machinery is very much inspired from the construction of \cite{Me}, but things are done here in greater
generality. Again, we hope that this approach could be useful to solve other questions as well. The operators we construct depend on a number of 
parameters, and we are able to determine in a rather precise way how the parameters influence on $\mathcal U$-frequent or frequent hypercyclicity, ergodicity or topological 
mixing. This
allows us to produce the examples we are looking for.  
The main results we obtain are  the existence of chaotic and frequently hypercyclic operators which are not in $c\textrm{MAX}(\h)$ and hence
are not ergodic (Theorem \ref{Theorem 54}), the existence 
of chaotic and $\mathcal U$-frequently hypercyclic operators which are not frequently hypercyclic (Theorem \ref{Theorem 58}), and the existence of chaotic and topologically 
mixing operators which are not $\mathcal U$-frequently hypercyclic (Theorem \ref{ex mix}).
\par\smallskip
We conclude the paper (Section \ref{QUESTIONS}) with a short list of possibly interesting questions.
\par\medskip
\textbf{Acknowledgements:} We are grateful to Alfred Peris for interesting discussions, and also for kindly allowing us to reproduce here his proof of Theorem \ref{Theorem 42}.
\par\smallskip

\section{Typical properties of hypercyclic operators}\label{TYPICAL}

\subsection{The strong and strong$^{\,*}$ topologies} The \emph{Strong Operator Topology} ($\texttt{SOT}$) on 
$\bh$ is defined as follows: any $T_{0}\in\bh$ 
has a neighborhood basis consisting of sets of the form
\[
\mathfrak U_{T_0; x_{1},\dots,x_{s},\varepsilon 
}:=\bigl\{T\in\bh\,;\, 
\|(T-T_{0})x_{i}\|<\varepsilon\quad\hbox{for $i=1,\dots ,s$} \bigr\} 
\]
where $x_{1},\dots,x_{s}\in \h$ and $\varepsilon>0$. Thus, a net $(T_{i})\subseteq\bh$ 
tends to $T_{0}\in\bh$ with respect to $\texttt{SOT}$ if and only if 
$T_ix\to T_0x$ in norm for every $x\in H$. 
\par\smallskip
The 
second (and perhaps a little less well-known) topology we will use is the \emph{Strong$^{\,*}$ Operator Topology} ($\texttt{SOT}^*$), which is the ``self-adjoint" 
version of $\sot$.  
A basis of $\texttt{SOT}^*$-neighborhoods of $T_0\in\bh$ is provided by the sets of the form
\[
\mathfrak V_{T_0; x_{1},\dots,x_{s},\varepsilon 
}:=\bigl\{T\in\bh\,;\,
\|(T-T_{0})x_{i}\|<\varepsilon\ \textrm{and}\
\|(T-T_{0})^{*}x_{i}\|<\varepsilon\quad\hbox{for $i=1,\dots ,s$}
\bigr\} 
\]
where $x_{1},\dots,x_{s}\in \h$ and $\varepsilon>0$. In other words, a net 
$(T_{i})$ tends to $T_{0}\in\bh$ with respect to $\texttt{SOT}^*$ if and only if $T_i\xrightarrow{\texttt{SOT}}T_0$ and $T_i^*\xrightarrow{\texttt{SOT}}T_0^*$. 
 Obviously, $\sot$ is coarser than $\sote$, which 
is in turn coarser than the norm topology. 
The topologies 
induced by $\sot$ and $\sote$ on any closed ball $\bmh$ are easily seen to be {Polish}, \mbox{\it i.e.} separable and completely metrizable (see \mbox{e.g.} \cite[Section 4.6.2]{Ped}). 
This will be of primary importance for us.
\par\smallskip
The following simple fact will be used repeatedly, sometimes without explicit mention.

\begin{fact}\label{power} Let $M>0$. If $\mathfrak B(\h)$ is endowed with either $\sot$ or $\sote$, then the map $(T,S)\mapsto TS$ is continuous from $\bmh\times\bmh$ into $\mathfrak B(\h)$. 
Consequently, for any fixed integer $n\ge 1$, the map $T\mapsto T^n$ is $(\texttt{SOT},\texttt{SOT})$-continuous and 
$(\texttt{SOT}^*,\texttt{SOT}^*)$-continuous from $\bmh$ into $\bh$.
\end{fact}

\begin{proof} It is enough to check the statement involving the 
 $\sot$ topology; so we have to show that for any fixed vector $x\in\h$, the map $(T,S)\mapsto TSx$ is continuous from $\bmh\times\bmh$ into $\h$. Since the map $(T,S)\mapsto (T,Sx)$ 
is continuous from $\bmh\times\bmh$ into $\bmh\times\h$, it suffices to check that the map $(T,u)\mapsto Tu$ is continuous from $\bmh\times\h$ into $\h$. Now, for any operators $T,T_0\in\bmh$ and any vectors $u,u_0\in\h$, we have 
\[\Vert Tu-T_0u_0\Vert=\Vert T(u-u_0)+(T-T_0)u_0\Vert\leq M\Vert u-u_0\Vert +\Vert (T-T_0)u_0\Vert.
\]
The result follows immediately.
\end{proof}

It is easy to deduce from this fact that $\hcmh$ is $\texttt{SOT}$-$G_\delta$ in $\bmh$ for any $M>1$. Indeed, it follows from 
Fact \ref{power} that if $A$ and $V$ are two non-empty subsets of $\h$ with $V$ open, then, for any fixed $n\ge 1$, the set 
\[\mathfrak O_{n; A,V}:=\{ T\in\bmh;\; T^{n}(A)\cap V\neq\emptyset\} \]
is $\texttt{SOT}$-open in $\bmh$. 
Now, fix a countable basis of non-empty open sets $(V_p)_{p\geq 1}$ of $\h$, and observe that
\[
\hcmh=\bigcap_{p,\,q\ge 1}\bigcup_{n\ge 1}\big\{T\in\bmh\,;\, T^{n}(V_{p})\cap V_{q}\neq\emptyset\bigr\}
=\bigcap_{p,\,q\ge 1}\bigcup_{n\ge 1} \mathfrak O_{n; V_p,V_q}.
\]
This proves that $\hcmh$ is $G_\delta$ in $\bmh$ with respect to $\texttt{SOT}$. 
Since the topology $\texttt{SOT}^*$ is finer than the topology $\texttt{SOT}$, $\hcmh$ is also $G_\delta$ in $\bmh$ with respect to $\texttt{SOT}^*$. 
Thus, we may state:

\begin{fact}\label{Fact1}
 For any $M>1$, $\hcmh$ is $\texttt{SOT}$-$G_\delta$ and $\texttt{SOT}^*$-$G_\delta$ in $\bmh$. 
 Hence $(\hcmh,\sot)$ and $(\hcmh,\sote)$ are Polish 
spaces.
\end{fact}

\par\smallskip
We now state a less immediate fact, which will be proved in Corollary \ref{Proposition 4} below. The corresponding (weaker) $\texttt{SOT}$ statement 
can be found in \cite{C?}, and its $\texttt{SOT}^*$ analogue undoubtedly would have been proved there if there had been any need to do so.

\begin{fact}\label{comeager} For any $M>1$, $\hcmh$ is dense in $(\bmh,\sote)$. Hence, $\hcmh$ is comeager 
in $\bmh$, both for the $\texttt{SOT}$ and the $\texttt{SOT}^*$ topology.
\end{fact}

Note that this result is indeed not completely trivial. For example, since the map $T\mapsto T^*$ is a homeomorphism of $(\bmh,\texttt{SOT}^*)$, it immediately implies 
that a typical operator $T\in\bmh$ is {hypercyclic and has a hypercyclic adjoint}. In particular, this argument proves that \emph{there exist} hypercyclic operators whose adjoint is also hypercyclic; 
which is a classical result of Salas obtained via an explicit construction in \cite{Sal}.
\par\smallskip
We deduce immediately  from Fact \ref{comeager} that the word ``typical" has the same meaning in the whole of $\bmh$ or in the 
subclass $\hcmh$ of hypercyclic operators in $\bmh$. This will be used repeatedly below, sometimes without explicit mention. 

\begin{fact}\label{comeagerbis} Let $\Gamma(\h)$ be a class of operators on $\h$, and let $M>1$. If $\Gamma_M(\h)$
 is a dense $G_\delta$ subset of $(\bmh,\tau)$, where 
$\tau$ is either $\texttt{SOT}$ or $\texttt{SOT}^*$, then 
$\Gamma_M(\h)\cap\hch$ is a dense $G_\delta$ subset of  $(\hcmh,\tau)$). Conversely, if $\Gamma(\h)\subseteq\hch$ 
and if $\Gamma_M(\h)$ is dense $G_\delta$ (resp. comeager)  in $(\hcmh,\tau)$, then $\Gamma_M(\h)$ is dense $G_\delta$ (resp. comeager) in $(\bmh,\tau)$. 
\end{fact}

\begin{proof} This is obvious: if $\Gamma_M(\h)$ is dense $G_\delta$ in $(\bmh,\tau)$ then, by Fact~\ref{Fact1}, 
Fact~\ref{comeager} and the Baire category theorem, $\Gamma_M(\h)\cap\hch$ is dense $G_\delta$ in $\bmh$, 
and hence in $\hcmh$; and likewise for the converse. The ``comeager" case follows from the ``$G_\delta$" case.
\end{proof}

Here is a last fact concerning the topology $\texttt{SOT}^*$ that will be quite useful for us. Note that the corresponding statement
for $\texttt{SOT}$ is \emph{false}. This is an important difference between the two topologies, which explains in particular why we will encounter some  subsets of $\bmh$ which
are $G_\delta$ with respect to $\texttt{SOT}^*$ and not with respect to $\texttt{SOT}$.

\begin{fact}\label{jointcontinuity} 
Let us denote by $w$ the weak topology of $\h$. If $B$ is a bounded subset of $\h$, 
then the map $(T,x)\mapsto Tx$ is continuous from $(\bmh,\texttt{SOT}^*)\times (B,w)$ into $(\h,w)$.
\end{fact}

\begin{proof} We have to show that for any fixed vector $e\in\h$, the map $(T,x)\mapsto \pss{Tx}{e}$ is continuous on $(\bmh,\texttt{SOT}^*)\times (B,w)$.
The key point is that one can separate $T$ from $x$ by writing $\pss{Tx}{e}=\pss{x}{T^*e}$ (this trick would be useless, of course, 
if we were working with the $\texttt{SOT}$ topology).
\par\smallskip
Let $(T_i,x_i)$ be a net in $\bh\times B$ converging to some element $(T,x)\in\bh\times B$; that is, 
$T_i\xrightarrow{\texttt{SOT}^*}T$ and $x_i\xrightarrow{w}x$. Then
\begin{align*}
\bigl\vert \pss{T_ix_i}{e}-\pss{Tx}{e}\bigr\vert&= \bigl\vert \pss{x_i}{T_i^*e}-\pss{x}{T^*e} \bigr\vert\\
&\leq \bigl\vert \pss{x_i}{T_i^*e-T^*e}\bigr\vert+\bigl\vert\pss{x_i-x}{T^*e} \bigr\vert.
\end{align*}
Since the net $(x_i)$ is bounded in norm and $x_i\xrightarrow{w}x$, this shows that $\pss{T_ix_i}{e}\to\pss{Tx}{e}$.
\end{proof}
\par\smallskip

\subsubsection{Why \emph{$\texttt{SOT}$} and \emph{$\texttt{SOT}^*$}?} 
There are other natural topologies on $\bh$, most notably the operator norm topology, of course,  
and the \emph{Weak Operator Topology} ($\texttt{WOT}$). The norm topology is not very well-suited for 
Baire category arguments, mainly because it is much too strong; in particular, the lack of separability 
seems 
unacceptable. The topology $\texttt{WOT}$ is better behaved
 in this respect, being Polish on any closed ball $\bmh$. However, since we are interested in typical properties of 
 \emph{hypercyclic} operators,  it seems better to consider topologies with respect to which $\hcmh$ is comeager 
 in $\bmh$ for any $M>1$. This is definitely not the case for $\texttt{WOT}$. Indeed, it is proved in \cite{E} that a typical element of
$(\mathfrak{B}_{1}(\h),\texttt{WOT})$ is \emph{unitary}. It follows that a $\texttt{WOT}$-typical $T\in\bmh$ is a multiple of a unitary operator 
and hence not hypercyclic. 
Incidentally, the operator norm topology has the 
same ``drawback", in an even stronger way: the hypercyclic operators (actually, even the {cyclic} operators) are in fact 
\emph{nowhere dense} 
in $\bh$; see \cite[Section 2.5]{BM}. 
\par\smallskip
So we will consider neither the operator norm topology nor the Weak Operator Topology in this paper. 
Actually, when working in the whole of $\bmh$ or $\hcmh$, $M>1$, we will always use $\texttt{SOT}^*$ rather than $\texttt{SOT}$. 
Indeed, with respect to $\texttt{SOT}$, a result from \cite{EM} gives a rather complete picture as far as typical properties are concerned: \emph{a typical 
element of $(\mathfrak{B}_{1}(\h),\emph{\sot})$ is unitarily equivalent to the 
operator $B^{(\infty )}$\!, the countable direct $\ell_2$-sum of the 
unilateral backward shift $B$ on $\ell_2(\N)$.} It follows that for every $M>1$, the class of operators $T\in\hcmh$ 
which are unitarily equivalent to $MB^{(\infty )}$ is comeager in $(\hcmh,\texttt{SOT})$. Now, the dynamical properties of 
the operators $MB^{(\infty )}$, $M>1$ 
are quite strong and very well understood: these operators are mixing in the Gaussian sense (and hence ergodic and topologically mixing), densely
distributionally chaotic, and they have nearly any other strong dynamical 
property one might think of. This explains why, when trying to determine which properties are typical within the
 class of \emph{all} hypercyclic operators, we will use 
$\texttt{SOT}^*$ rather than $\texttt{SOT}$. The situation in this setting is more involved, and thus leads to more interesting results.
\par\smallskip
On the other hand, we will see that within specific subclasses of $\hch$ consisting of upper triangular operators, the topology 
$\texttt{SOT}$ becomes much more useful. This is not really surprising, since triangularity is not exactly a self-adjoint property.
\par\smallskip

\subsection{How to prove density results} 
For future reference, we state here a simple criterion for a class of operators 
to be dense in $(\bmh, \texttt{SOT})$ or $(\bmh,\sote)$, $M>0$. We will use it repeatedly in the 
rest 
of the paper. 

\begin{lemma}\label{Lemma 4 bis}
 Let $\Gamma (\h)$ be a class of operators on $\h$, and let $M>0$. Let also $(e_{k})_{\gk}$ be an orthonormal basis of $\h$, 
 and for each $r\ge 1$, denote 
 by $\h_{r}$ the finite-dimensional subspace 
\emph{$\textrm{span}\,[e_{k}\,;\,1\le k\le r]$} of $\h$. Suppose that the following 
property 
holds true:
\par\smallskip 
\noindent{for every $r\ge 1$, every operator $A\in\mathfrak{B}(\h_{r})$ 
with $
\|A\|<M$ and every $\varepsilon >0$, there exists an operator 
$T\in\Gamma_M (\h) $ such that}
\begin{equation}\label{Equation une etoile}
 \|(T-A)e_{k}\|<\varepsilon \qquad\hbox{for $k=1,\dots ,r$.}\tag{$*$}
\end{equation}
Then $\Gamma_M (\h)$ is dense in \emph{$(\mathfrak{B}_{M}(\h),\sot)$}. If \emph{(\ref{Equation une etoile})} is replaced by its self-adjoint version
\begin{equation}\label{Equation deux etoiles}
 \|(T-A)e_{k}\|<\varepsilon\quad{\rm and}\quad 
 \|(T-A)^{*}e_{k}\|<\varepsilon \qquad\hbox{for $k=1,\dots ,r$},\tag{$**$}
\end{equation}
then $\Gamma_{M} (\h)$ is dense in \emph{$(\mathfrak{B}_{M}(\h),\sote)$}.
 \end{lemma}
 
 Note that there is a slight abuse of notation in the statement of Lemma \ref{Lemma 4 bis}: we consider the operator 
$A\in\mathfrak B(\h_r)$ as an operator on $\h$ by 
identifying it  with $P_r^*AP_r$, where $P_r:\h\to\h_r$ is the orthogonal projection of $\h$ onto $\h_r$.

\begin{proof} 
 We will prove the assertion concerning the $\sote$-topology, the 
proof of the $\texttt{SOT}$ statement being exactly the same. Fix $T_{0}\in\bmh$, 
$\varepsilon >0$, and $x_{1},\dots,x_{s}\in\h$. Without loss of 
generality, we can suppose that $\|T_{0}\|< M$. We 
are looking for an operator $T\in\Gamma_M (\h) $ such that 
\begin{equation}\label{Equation 3}
 \max_{1\le j\le s}\max\bigl(\|(T-T_{0})x_{j}\|,
 \|(T-T_{0})^{*}x_{j}\|\bigr)<\varepsilon.
\end{equation}
\par\smallskip
Since $\|T\|\le M$ for every $T\in\Gamma_M (\h)$, and since every vector $x_j$ can be approximated by a finite linear combination of the basis vectors 
$e_k$, there exists an integer $r_{0}\ge 1$ sufficiently large such that (\ref{Equation 3}) above holds 
true as soon as
\[
\max_{1\le k\le 
r_{0}}\max\,\bigl(\|(T-T_{0})e_{k}\|,
 \|(T-T_{0})^{*}e_{k}\|\bigr)<\dfrac{\varepsilon 
}{2}\cdot 
\]
For each $r\ge r_{0}$, consider the 
operator $A_{r}=P_{r}T_{0}P_{r}$
Observe that 
$\|A_{r}\|\le\|T_{0}\|
<M$. By our assumption, there exists $T\in\Gamma (\h) $ such that
\[
\max_{1\le k\le r}\max\,\bigl(\|(T-A_{r})e_{k}\|,
 \|(T-A_{r})^{*}e_{k}\|\bigr)<\dfrac{\varepsilon 
}{2}\cdot 
\]
Now 
\[
\|(T-T_{0})e_{k}\|\le
\|(T-A_{r})e_{k}\|+\|
(P_{r}T_{0}P_{r}-T_{0})e_{k}\|<\dfrac{\varepsilon 
}{2}+\|(P_{r}-I)T_{0}e_{k}
\|
\]
and
\[
\|(T-T_{0})^{*}e_{k}\|<\dfrac{\varepsilon }{2}+
\|(P_{r}-I)T_{0}^{*}e_{k}\|
\]
for every $1\le k\le r$. Since $\lim_{r\to\infty }P_{r}=I$ for the $\sot$ 
topology, one can choose $r$ so large 
that 
\[
\max_{1\le k\le 
r_{0}}\max\,\bigl(\|(P_{r}-I)T_{0}e_{k}\|,
\|(P_{r}-I)T_{0}^{*}e_{k}\|\bigr)<\varepsilon, 
\]
from which the conclusion of Lemma \ref{Lemma 4 bis} follows.
\end{proof}

\begin{remark}\label{remarque en plus}
It is sometimes more convenient to endow $\h$ with an orthonormal basis $(f_{k})_{k\in\Z}$ 
indexed by $\Z$ rather than by $\N$. In this case, the corresponding version of Lemma \ref{Lemma 4 bis} reads as follows.
\par\smallskip\noindent
\it{For each $r\ge 0$, denote 
 by $\h_{r}$ the finite-dimensional subspace 
$\textrm{span}\,[f_{k}\,;\,-r\le k\le r]$ of $\h$. Suppose that the following 
property 
holds true:
\par\smallskip 
\noindent{for every $r\ge 0$, every operator $A\in\mathfrak{B}(\h_{r})$ 
with $
\|A\|<M$ and every $\varepsilon >0$, there exists an operator 
$T\in\Gamma_M (\h) $ such that}
\begin{equation}\label{Equation une etoile bis}
 \|(T-A)f_{k}\|<\varepsilon \qquad\hbox{for $k=-r,\dots ,r$.}\tag{$*$}
\end{equation}
Then $\Gamma_M (\h)$ is dense in \emph{$(\mathfrak{B}_{M}(\h),\sot)$}. If \emph{(\ref{Equation une etoile})} is replaced by its self-adjoint version
\begin{equation}\label{Equation deux etoiles bis}
 \|(T-A)f_{k}\|<\varepsilon\quad{\rm and}\quad 
 \|(T-A)^{*}f_{k}\|<\varepsilon \qquad\hbox{for $k=-r,\dots ,r$},\tag{$**$}
\end{equation}
then $\Gamma_{M} (\h)$ is dense in \emph{$(\mathfrak{B}_{M}(\h),\sote)$}.}
\end{remark}
\par\smallskip

\subsection{Construction of mixing operators, and density of ${\gsmxh}$}\label{Subsection 1.b.1}
In order to show that a property is typical, we first need to prove the density of the set of operators satisfying it. In this subsection, we show 
that the class $\textrm{G-MIX}_{M}(\h)$ is dense in $(\hcmh,\sote)$ for any $M>1$. This will be achieved by considering some perturbations of weighted unilateral or 
bilateral weighted shifts with respect to some orthonormal basis of $\h$. We 
show that these operators admit spanning eigenvector fields which are analytic in a 
neighborhood of the unit circle, and hence are mixing in the Gaussian sense. The precise statement we will use is the following.

\begin{fact}\label{holo} Let $T\in\bh$. Assume that there exists a connected open set $\Omega\subseteq\C$ with $\Omega\cap \T\neq\emptyset$ and a family $(E_i)_{i\in I}$ of holomorphic 
maps, $E_i:\Omega\to\h$, such that $TE_i(\lambda)=\lambda E_i(\lambda)$ for every $i\in I$ and every $\lambda\in\Omega$, and 
$\overline{\rm span}\, \{ E_i(\lambda);\; i\in I\,,\; \lambda\in\Omega\}=\h$. Then $T$ is mixing in the Gaussian sense. 
\end{fact}

\begin{proof} Recall that we denote by $\mathcal E(T)$ the set of all unimodular eigenvectors of $T$, and by $\lambda(x)$ the 
eigenvalue associated to $x\in\mathcal E(T)$. By \cite[Th. 3.29]{BG2}, it is enough to show that for any Borel set $D\subseteq\T$ of Lebesgue measure $0$, we have 
$\overline{\rm span}\, \{x\in\mathcal E(T);\; \lambda(x)\not\in D\}=\h$.
Let $y\in\h$ be orthogonal to the set
$\{ x\in\mathcal E(T);\; \lambda(x)\not\in D\}$. Then $\pss{y}{E_i(\lambda)}=0$ for every $i\in I$ and every $\lambda\in (\Omega\cap\T)\setminus D$. Since the 
functions $\pss{y}{E_i(\,\cdot\,)}$ are holomorphic on $\Omega$, and since $(\Omega\cap\T)\setminus D$ certainly has an accumulation point in $\Omega$ (because 
$D$ has Lebesgue measure $0$), it follows that $y$ is orthogonal to all vectors $E_i(\lambda)$, $i\in I$, $\lambda\in \Omega$, and hence that $y=0$. This concludes the proof.
\end{proof}

\begin{remark}\label{chaos en plus} The assumptions of Fact \ref{holo} imply that the operator $T$ is also {chaotic}: 
since the roots of unity contained in $\Omega$ have an accumulation point in $\Omega$, this follows as above from the identity principle for holomorphic functions.
\end{remark}

Let us first consider perturbations of unilateral weighted shifts. Let $(e_{k})_{\gk}$ be an 
orthonormal basis of the Hilbert space $\h$, and let $\pmb{\omega} =(\omega 
_{k})_{\gk}$ be a \emph{unilateral weight sequence}, \mbox{\it i.e.} a bounded sequence of positive real numbers. Let $r\ge 1$ be an 
integer, and let $A$ be an operator acting on the finite-dimensional space
$\h_{r}=\textrm{span}[e_{k}\,;\,1\le k\le r]$. We define a bounded operator 
 $B_{A,\pmb{\omega} }$ on $\h$ by setting
 \[
B_{A,\pmb{\omega} }e_{k}=
\begin{cases}
 Ae_{k}&\textrm{for every}\ 1\le k\le r\\
 \omega _{k-r}e_{k-r}&\textrm{for every}\ k>r.
\end{cases}
\]
\begin{proposition}\label{Proposition 1}
 Let $\pmb{\omega} $ be a unilateral weight sequence, $r\ge 1$, and $A\in\mathfrak{B}
 (\h_{r})$. Suppose that for every  $1\le l\le r$,
\[
R_{l}:=\liminf_{p\to\infty }\bigl(\omega _{pr+l}\cdots\omega_{r+l}\omega _{l} 
\bigr)^{1/p}>\max(1,\Vert A\Vert).
\]
Then the operator 
$B_{A,\pmb{\omega} }$ is mixing in the Gaussian sense. Besides, $B_{A,\om}$ is also chaotic.
\end{proposition}

\begin{proof}
 Solving formally the equation $B_{A,\pmb{\omega} }x=\lambda x$, where $\lambda \in\C$ 
and $x=(x_k)_{k\geq 1}\in\C^\N$, one gets the following identities:
\[ (\lambda-A)P_rx=\sum_{l=1}^r \omega_l x_{l+r}e_l\qquad{\rm and}\qquad \omega_kx_{k+r}=\lambda x_k\quad\hbox{for every $k>r$}.
\]
From this (setting $y:=P_rx$) we infer that the eigenvectors of $B_{A,\pmb{\omega} }$ associated to the eigenvalue $\lambda$ must be given 
by the  formula
\begin{equation}
  E_{y}(\lambda )=y+\sum_{l=1}^{r}\dfrac{1}{\omega _{l}}
 \left\langle(\lambda -A)y,e_{l}\right\rangle\,\biggl(e_{r+l}+\sum_{p\ge 2}\dfrac{\lambda ^{p-1}}{\omega 
_{(p-1)r+l}\dots\omega _{r+l} }\,e_{pr+l}\biggr)\!
\end{equation}
where $y$ is a non-zero vector of $\h_{r}$. Conversely, if $y$ belongs to $\h_r\setminus\{ 0\}$ is such that the above formula makes sense, then 
$E_y(\lambda)$ is an eigenvector of $B_{A,\om}$ with associated eigenvalue $\lambda$.
It follows that the complex number $\lambda $ is an eigenvalue of 
$B_{A,\pmb{\omega} }$ as soon as 
\[
\sum_{p\ge 2}\left|\dfrac{\lambda ^{p-1}}{\omega _{(p-1)r+l}\dots\omega 
_{l} }\right|^{2}<\infty \qquad\hbox{for all $1\leq l\leq r$},
\]
which holds true whenever $|\lambda |<R:=\min_{1\le l\le r}R_{l}$. In this case, the eigenvector field $E_{y}$ is 
well-defined and holomorphic on the open disk $D(0,R)$ for every $y\in\h_r$. Note that our  
assumption 
implies that $R>1$, so that the disk $D(0,R)$ contains $\T$.
\par\smallskip
By Fact \ref{holo} and Remark \ref{chaos en plus}, in order to show that $B_{A,\pmb{\omega} }$ belongs to $\gsmxh$ and is chaotic, it suffices to 
check that the eigenvectors 
$E_{y}(\lambda )$, $y\in \h_{r}$, $|\lambda|<R$, span a dense 
subspace of $\h$. So let $u\in \h$ be such that $\pss{E_{y}(\lambda 
)}{u}=0$ 
for every $y\in \h_{r}$ and every $|\lambda |<R$. Writing $u$ as 
$u=\sum\limits_{\gk}u_{k}e_{k}$, this means that 
\[
\pss{y}{u}+\sum_{l=1}^{r}\dfrac{1}{\omega 
_{l}}\pss{y}{(\lambda 
-A)^{*}e_{l}}\,\biggl(\ba{u}_{r+l}+\sum_{p\ge 
2}\ba{u}_{pr+l}\dfrac{\lambda ^{p-1}}{\omega 
_{(p-1)r+l}\dots\omega _{r+l} }\biggr)=0
\]
for every $y\in \h_{r}$. It follows that each vector
\[
u+\sum_{l=1}^{r}\dfrac{1}{\omega _{l}}\biggl({u}_{r+l}+\sum_{p\ge 
2}{u}_{pr+l}\dfrac{\ba\lambda ^{\, p-1}}{\omega 
_{(p-1)r+l}\dots\omega _{r+l} }\biggr)\,(\lambda -A)^{*}e_{l},\qquad 
|\lambda |<R
\]
is orthogonal to $\h_{r}$, \mbox{\it i.e.} belongs to the 
closed linear span of the vectors $e_{k}$, $k>r$. In other words,
\[
P_{r}u+\sum_{l=1}^{r}\dfrac{1}{\omega _{l}}\biggl({u}_{r+l}+\sum_{p\ge 
2}{u}_{pr+l}\dfrac{\ba\lambda^{\, p-1}}{\omega 
_{(p-1)r+l}\dots\omega _{r+l} }\biggr)\,(\lambda -A)^{*}e_{l}=0.
\]
Consider now the open subset $\Omega =D(0,R)\setminus \sigma (A)$ of $\C$, where 
$\sigma (A)$ denotes the spectrum of $A$. Applying the operator $(\lambda 
-A)^{*}{}^{-1}$ to the 
previous equation, we obtain that for every $\lambda\in\Omega$:
\begin{equation}\label{equaspan}
(\lambda 
-A)^{*}{}^{-1}P_{r}u=-\sum_{l=1}^{r}\dfrac{1}{\omega_{l}}\biggl({u}_{r+l
}+\sum_{ p\ge 
2}{u}_{pr+l}\dfrac{\ba\lambda ^{\, p-1}}{\omega 
_{(p-1)r+l}\dots\omega _{r+l} }\biggr) \,e_l.
\end{equation}
Since the expression on the right hand side of (\ref{equaspan}) defines an antiholomorphic map on $D(0,R)$ and since the disk 
$D(0,R)$ contains $\sigma(A)$ (recall that $R>\Vert A\Vert$ by assumption), 
it follows that the map $\lambda \mapsto(\lambda 
-A)^{*}{}^{-1} P_ru$ extends antiholomorphically to the whole complex plane. But $(\lambda 
-A)^{*}{}^{-1} P_ru\to 0$ as $\vert\lambda\vert\to\infty$, so the function $\lambda\mapsto(\lambda 
-A)^{*}{}^{-1} P_ru$ must vanish identically on $\Omega$ by Liouville's Theorem, which is possible only if $P_ru=0$.
\par\smallskip
Going back to (\ref{equaspan}), the fact that $P_{r}u=0$ yields that $u_{pr+l}=0$ for every $1\le l\le r$ and every  
$p\ge 2$, and that $u_{r+l}=0$ for every 
$1\le l\le r$. Thus $u=0$, and this concludes the proof of Proposition 
\ref{Proposition 
1}.
\end{proof}

\begin{remark} The definition of the operator $B_{A,\om}$ shows that every eigenvalue of $A$ is also an eigenvalue of $B_{A,\om}$. This explains why the conditions $R_{l}>1$ for every $l=1,\ldots, r$ are not sufficient to ensure that $B_{A,\om}$ be mixing in the Gaussian sense. Indeed, if $\lambda\in\sigma(A)$ is such that $|\lambda|>\max_{1\le l\le r}R_{l}$, then $\lambda$ is an isolated \eva\ of $B_{A,\om}$ with $|\lambda|>1$, and this prevents $B_{A,\om}$ from being \hy.
This is to be compared with 
Remark \ref{Remark 3} below. 
\end{remark}

From Proposition \ref{Proposition 1} we easily deduce a basic density result, which gives in particular the promised proof 
of Fact \ref{comeager} above.

\begin{corollary}\label{Proposition 4}
 For every $M>1$, the class \emph{$\textrm{G-MIX}_{M}(\h)\cap \textrm{CH}_M(\h)$} is dense in \emph{$(\hcmh,\sote)$}. 
  \end{corollary}
  
\begin{proof} We are going to apply Lemma \ref{Lemma 4 bis}. So, let $(e_{k})_{\gk}$ be an orthonormal basis of $\h$, and let us fix $r\ge 1$, 
an operator $A\in
 \mathfrak{B}(\h_{r})$ with $\Vert A\Vert<M$ and $\varepsilon >0$. Let also $\delta >0$ be a small 
positive number to be specified below. 
We define a weight sequence $\pmb{\omega} =(\omega _{k})_{\gk}$ as follows:
\[
\omega _{k}=
\begin{cases}
 \delta&\textrm{for every}\ 1\le k\le r\\
 M&\textrm{for every}\ k>r;
\end{cases}
\]
and we consider the associated
operator $B_{A,\pmb{\omega} }$ acting on $\h$. Identifying $A$ with 
$P_{r}^*AP_{r}\in\bh$, we have for every $1\le k\le r$:
\[
B_{A,\pmb{\omega}} e_{k}=Ae_k\qquad{\rm and}\qquad B_{A,\pmb{\omega} }^*e_k=
A^{*}e_{k}+\delta e_{k+r}\,,
\]
so that 
$\|(B_{A,\pmb{\omega} }-A 
)^{*}e_{k}\|=\delta$. It follows that if $\delta <\varepsilon $, then
\[
\max_{1\le k\le r}\max\,(\|\bigl(B_{A,\pmb{\omega} }-A 
)e_{k}\|,\|(B_{A,\pmb{\omega} }-A 
)^{*}e_{k}\|\bigr)<\varepsilon. 
\]
The assumption of Proposition \ref{Proposition 1} is clearly satisfied, 
so that $B_{A,\pmb{\omega} }$ belongs to the class $\gsmxh \cap \textrm{CH}(\h)$. 
\par\smallskip
To estimate the 
norm of $B_{A,\pmb{\omega} }$, note that for every $x=\sum\limits_{j\ge 1}x_{j}e_{j}\in H$, we have
\[
B_{A,\pmb{\omega} }x=AP_{r}x+\sum_{k=1}^{r}x_{k+r}\,\delta 
e_{k}+\sum_{k>r}x_{k+r}\,Me_{k}
\]
so that 
\[
\|\,B_{A,\pmb{\omega} }x\,|	|^{2}=
\biggl|\biggl|AP_{r}x+\delta \sum_{k=1}^{r}x_{k+r}e_{k}\biggr|\biggr|^{2}+
M^{2}\sum_{k>r}|x_{k+r}|^{2}.
\]
Since $\|A\|<M$, it follows that 
$\|B_{A,\pmb{\omega} }\|=M$ if $\delta >0$ is sufficiently small. So 
$B_{A,\pmb{\omega}}$ belongs to $\textrm{G-MIX}_{M}(\h)\cap \textrm{CH}_M(\h)$, and Lemma 
\ref{Lemma 4 bis} 
now allows us to conclude the proof of Corollary \ref{Proposition 4}.
\end{proof}

\begin{remark} Corollary \ref{Proposition 4} does {not} state that $\textrm{G-MIX}_{M}(\h)\cap \textrm{CH}_M(\h)$ is comeager in $(\hcmh,\texttt{SOT}^*)$. Indeed, we will prove 
below that 
$\textrm{G-MIX}_{M}(\h)$ and $\textrm{CH}_M(\h)$ are actually {meager} in {$(\hcmh,\sote)$}.
\end{remark}

Let us now turn to a bilateral analogue of Proposition \ref{Proposition 1}, which we 
state as Proposition~\ref{Proposition 2} below. Let $(f_{k})_{k\in\Z}$ be 
an orthonormal basis of the space $\h$, and let $\pmb{\omega} =(\omega _{k})_{k\in\Z}$ be a 
\emph{bilateral weight sequence}, \mbox{\it i.e.} a 
bounded sequence of positive real numbers indexed by $\Z$. For any integer 
$r\ge 0$, we write $\h_{r}=\textrm{span}[e_{k}\,;\,|k|\le r]$, and let 
$A
\in\mathfrak{B}(\h_{r})$ be a bounded operator on $\h_{r}$. We define a 
bounded operator $S_{A,\pmb{\omega} }$ on $\h$ by 
setting 
\[
S_{A,\pmb{\omega} }f_{k}=
\begin{cases}
 Af_{k}+\omega _{k-(2r+1)}f_{k-(2r+1)}&\textrm{for 
every}\ |k|\le r\\
\omega _{k-(2r+1)}f_{k-(2r+1)}&\textrm{for 
every}\ |k|> r.
\end{cases}
\]

\begin{proposition}\label{Proposition 2}
Let $\om$ be a  bilateral weight sequence, $r\geq 1$ and $A\in\mathfrak{B}(\h_{r})$. Suppose that  
for every $-r\le l\le r$,
\[  R_{l}:=\liminf_{p\to\infty}\bigl(\omega _{p(2r+1)+l}\dots\omega _{(2r+1)+l} 
\bigr)^{1/p}>1
 \]
 and
 \[
r_{l}:=\limsup_{p\to\infty}\bigl(\omega 
_{l-p(2r+1)}\dots\omega _{l-(2r+1)} 
\bigr)^{1/p}<1.\]
Then $S_{A,\pmb{\omega} }$ is mixing in the Gaussian sense and chaotic.
\end{proposition}

\begin{proof}
 The proof is so similar to that of 
Proposition \ref{Proposition 1} that we will not give it in detail.
A complex number $\lambda $ is 
an eigenvalue of $S_{A,\pmb{\omega} }$ as soon as the series 
\[
\sum_{p\ge 2}\left|\dfrac{\lambda ^{p-1}}{\omega 
_{(p-1)(2r+1)+l}\dots\omega _{(2r+1)+l}}\right|^{2}\quad\textrm{and}\quad
\sum_{p\ge 1}\left|\dfrac{\omega _{l-p(2r+1)}\dots\omega 
_{l-(2r+1)}}{\lambda ^{p}}\right|^{2}
\]
are convergent for all $-r\le l\le r$. If we define $R:=\min_{-r\le l\le r}R_{l}$ and 
$r:=\max_{-r\le l\le r}r_{l}$, our assumption implies that $r<1<R$. Any 
complex 
number $\lambda $ belonging to the annulus $\{ r<|\lambda |<R\}$ is 
an eigenvalue of $S_{A,\pmb{\omega} }$, and the eigenvectors of $S_{A,\pmb{\omega} }$ associated to 
$\lambda $ have the form 
\begin{align}\label{quelquechose}
\notag E_{y}(\lambda )=\sum_{l=-r}^{r}\dfrac{1}{\omega _{l}}\,&\pss{(\lambda 
-A)y}{e_{l}} \biggl(e_{(2r+1)+l}+
\sum_{p\ge 2}\dfrac{\lambda ^{p-1}}{\omega 
_{(p-1)(2r+1)+l}\dots\omega _{(2r+1)+l}}e_{p(2r+1)+l}\\
&+\sum_{p\ge 1}\dfrac{\omega _{l-p(2r+1)}\dots\omega 
_{l-(2r+1)}}{\lambda ^{p}}e_{-p(2r+1)+l}\biggr)
\end{align}
where $y\in\h_{r}$.
Since $r<1<R$, the annulus $\{  r<\vert\lambda\vert <R\}$ contains $\T$, and an argument similar to the one given in the 
proof of 
Proposition \ref{Proposition 1} show that 
$S_{A,\pmb{\omega} }$ belongs to $\textrm{G{-}MIX}(H)\cap\textrm{CH}_{M}(\h)$. Indeed, if $u\in\h$ is such that $\pss{E_{y}(\lambda)}{u}=0$ for every $y\in \h_{r}$ and every $r<|\lambda|<R$, then
\begin{align*}
 \sum_{l=-r}^{r} \dfrac{1}{\omega _{l}}\, \biggl(u_{(2r+1)+l}&+
\sum_{p\ge 2}\dfrac{\overline{\lambda }^{p-1}}{\omega 
_{(p-1)(2r+1)+l}\dots\omega _{(2r+1)+l}}u_{p(2r+1)+l}\\
&+\sum_{p\ge 1}\dfrac{\omega _{l-p(2r+1)}\dots\omega 
_{l-(2r+1)}}{\overline{\lambda }^{p}}u_{-p(2r+1)+l}\biggr)
{(\lambda 
-A)}^{*}{e_{l}}=0
\end{align*}
for every $r<|\lambda|<R$ and every $l=-r,\ldots,r$, from which it follows that
\begin{align*}
 \sum_{l=-r}^{r} \dfrac{1}{\omega _{l}}\, \biggl(u_{(2r+1)+l}&+
\sum_{p\ge 2}\dfrac{\overline{\lambda }^{p-1}}{\omega 
_{(p-1)(2r+1)+l}\dots\omega _{(2r+1)+l}}u_{p(2r+1)+l}\\
&+\sum_{p\ge 1}\dfrac{\omega _{l-p(2r+1)}\dots\omega 
_{l-(2r+1)}}{\overline{\lambda }^{p}}u_{-p(2r+1)+l}\biggr)=0
\end{align*}
for every $r<|\lambda|<R$ and every $l=-r,\ldots,r$.
Hence $u=0$, and Proposition \ref{Proposition 2} is proved.
\end{proof}

\begin{remark}\label{Remark 3} 
The description of the eigenvectors of $S_{A,\pmb{\omega} }$ given in (\ref{quelquechose}) above shows 
that if the two series 
 \[
\sum_{p\ge 1}\bigl(\omega _{p(2r+1)+l}\dots\omega _{(2r+1)+l}
\bigr)^{-2}\quad\textrm{and}\quad \sum_{p\ge 1}\bigl(\omega 
_{l-p(2r+1)}\dots
\omega _{l-(2r+1)}
\bigr)^{-2}
\]
are divergent for every $|l|\le r$, then the operator $S_{A,\pmb{\omega} }$ has no 
eigenvalue, whatever\ the choice of $A\in\mathfrak{B}(\h_{r})$. This 
observation will be useful for the proof of Proposition 
\ref{Proposition 5} below. It is also interesting to observe that the assumptions on $R_{l}$ and $r_{l}$ in Proposition \ref{Proposition 2} do not involve the \op\ $A$, contrary to what happens in Proposition \ref{Proposition 1}. This is coherent with the fact that the \eva s of $A$ do not necessarily appear as \eva s of $S_{A,\pmb{\omega}}$, while they do appear as \eva s of $B_{A,\pmb{\omega}}$.
\end{remark}
\par\smallskip

\subsection{Topological weak mixing and topological mixing}\label{WMSM}
 In this subsection, we show that topological 
weak mixing is a typical property, whereas topological mixing is atypical. In view of the corresponding well-known 
analogues in ergodic theory due to Halmos and Rohlin (see \mbox{e.g.} \cite[pp. 77-80]{Hal}, and \cite{ES} 
for a more general result), this should not be surprising at all.

\begin{proposition}\label{WM-dense} For every $M>1$, the class \emph{$\textrm{TWMIX}_{M}(\h)$} is a dense $G_\delta$ 
(and hence comeager) subset of 
\emph{$(\hcmh,\sote)$}.
\end{proposition}

\begin{proof} That $\textrm{TWMIX}_{M}(\h)$ is $G_\delta$ 
follows from the fact that $\textrm{HC}_M(\h\times\h)$ is $G_\delta$ in 
$\mathfrak B_M(\h\times \h)$, and the $\texttt{SOT}^*$-continuity of the map 
$T\mapsto T\times T$.  Since operators in $\textrm{G-MIX}(\h)$ are topologically mixing, density follows 
from Corollary \ref{Proposition 4}.
\end{proof}

\begin{proposition}\label{Proposition 13}
 For every $M>1$, the class \emph{$\textrm{TMIX}_{M}(\h)$} is meager in the space \emph{$(\hcmh,\sote)$}.
\end{proposition}

The proof of Proposition \ref{Proposition 13} relies on Lemmas \ref{Lemma 10} and \ref{Lemma 11}, which belong to the forthcoming Section \ref{Subsection 1.b.3}. We nonetheless prefer to present things in this order, because the purposes of  Lemmas \ref{Lemma 10} and \ref{Lemma 11} will appear more clearly in our proof of the typicality of \ops\ without \nt\ \inv\ \mea s.

\begin{proof} By Fact \ref{comeagerbis}, it is enough to show that the class $\textrm{TMIX}_M(\h)$ is meager in $(\bmh,\sote)$.
 Let $B$ be a non-trivial closed ball in $\h$. We certainly have
\[ \textrm{TMIX}_M(\h)\subseteq \bigcup_{N\geq 1}  \mathfrak F_N,
\]
where
\[ \mathfrak F_N:=
\bigcap_{n\geq N} \bigl\{T\in\bmh\,;\,T^{n}(B)\cap B\neq \emptyset\bigr\} \quad \textrm{ for every } N\ge 1.
\]
Each set $\mathfrak F_N$ is closed in $(\bmh,\sote)$. Indeed, we may write
\[T\in\mathfrak F_N\iff\forall n\geq N\; \exists x\in B\;:\; T^nx\in B.\]
Since $B$ is weakly closed in $\h$, the condition  ``$T^nx\in B$" defines a closed subset of $(\bmh,\texttt{SOT}^*)\times (B,w)$ by Facts \ref{power} 
and \ref{jointcontinuity}; and since $B$ is weakly compact, this shows that
$\mathfrak F_N$ is closed in $(\bmh,\sote)$.
To conclude the proof, it is enough to show that for some suitable choice of  the ball $B$, the closed sets $\mathfrak F_N$ have empty interior in $(\bmh,\sote)$; or, equivalently, 
that the open sets
\[\mathfrak O_N:=\{ T\in\bmh;\; \exists n\geq N\;:\; T^n(B)\cap B=\emptyset\}
\]
are dense in $(\bmh,\sote)$.
We choose for $B$ the ball $\overline{B}(e,1/2)$, where $e\in\h$ satisfies $\Vert e\Vert=1$. Then $\mathfrak O_N$ contains the set $\mathfrak O_{N+1,B}$ 
of Lemma \ref{Lemma 10} below, so it is dense in $(\bmh,\sote)$ by Lemma \ref{Lemma 11}.
\end{proof}

\begin{remark} The same proofs would show that topological weak mixing is typical and topological mixing is atypical for operators on 
$\ell_p$ spaces. It would be interesting to know if this is still true on every Banach space with separable dual. In this respect, it 
is worth mentioning that there exist on \emph{any} separable Banach space hypercyclic operators which are not topologically mixing 
(\cite{GSh}).
\end{remark}
\par\smallskip

\subsubsection{Some illustrations} In this subsection, we present some consequences of Propositions \ref{WM-dense} and \ref{Proposition 13}. 

\smallskip First, we have the following amusing fact: 
{a typical \op\ {$T\in\hcmh$, $M>1$}, satisfies the Hypercyclicity Criterion but not Kitai's criterion}.

\smallskip
In the same spirit, it follows from Proposition \ref{WM-dense} and the Baire Category Theorem 
that a typical $T\in\bmh$ is such that $T$ and $T^*$ are both topologically weakly mixing; but no operator with this property can be topologically mixing 
since otherwise $T\times T^*$ would be hypercyclic on $\h\times\h$,
which can never happen.

\smallskip
Here is now a less immediate consequence of the comeagerness of $\textrm{TWMIX}_M(\h)$, which is  a partial strengthening of the main theorem of \cite{W}. 
This result could also be easily deduced from \cite[Th. 4.1]{prescribed}, the proof of which is, however, quite different.

\begin{proposition}\label{ZW} Let $Z$ be a linear subspace of $\h$ with countable algebraic dimension. For any $M>1$, the set of all $T\in\mathfrak B_M(\h)$ such that every vector 
$z\in Z\setminus\{ 0\}$ is {hypercyclic} for $T$ is comeager in \emph{$(\bmh,\sote)$}. 
\end{proposition}

\begin{proof} Let us first recall that any topologically weakly mixing operator $T$ on $\h$ is in fact \emph{totally hypercyclic}, which means that for any $N\geq 1$, 
the $N$-fold product operator $T_{N}:=T\times\cdots \times T$ is hypercyclic on $\h^{N}:=\h\times\cdots \times \h$. This is a classical result, which has nothing to do 
with linearity; see \mbox{e.g} \cite{Gl}, or \cite[Th. 4.6]{BM}.
\par\smallskip
For any finite sequence $\mathbf f=(f_1,\dots, f_N)$ of vectors of $\h$, we will denote by ${\rm Gram}(\mathbf f)$ the associated \emph{Gram matrix},
\[ {\rm Gram}(\mathbf f):=\bigl(\pss{f_i}{f_j}\bigr)_{1\leq i,j\leq N}\,\]
and we define
$ \h_{\bf f}:=\bigl\{ \mathbf x=(x_1,\dots ,x_N)\in \h^{N};\; {\rm Gram}(\mathbf x)={\rm Gram}(\mathbf f)\bigr\}.$
Note that $\h_{\bf f}$ is a closed subset of $\h^{N}$, and hence a Polish space.
The key point in the proof of Proposition \ref{ZW} is the following fact.

\begin{fact}\label{W1} Let $\mathbf f=(f_1,\dots ,f_N)$ be a  finite sequence of linearly independent vectors in $\h$. For any operator $T\in \textrm{TWMIX}(\h)$, the 
set $\h_{\bf f}\cap \textrm{HC}(T_{N})$ is dense in $\h_{\bf f}$.
\end{fact}

\begin{proof}[Proof of Fact \ref{W1}] Let $(u_1,\dots ,u_N)\in\h_{\bf f}$ be arbitrary. We are looking for some $N$-tuple $(x_1,\dots ,x_N)\in\h_{\bf f}$ such that 
$(x_1,\dots ,x_N)$ lies in $ \textrm{HC}(T_{N})$ and $(x_1,\dots ,x_N)$ is very close to $(u_1,\dots ,u_N)$. 
Since $\textrm{HC}(T_{N})$ is dense in $\h^{N}$, one can first choose $(z_1,\dots ,z_N)\in \textrm{HC}(T_{N})$ very close to 
$(u_1,\dots ,u_N)$. 
\par\smallskip
Note that the $N$-tuple $\mathbf u=(u_1,\dots ,u_N)$ consists of linearly independent vectors since the Gram matrix 
${\rm Gram}(\mathbf u)={\rm Gram}(\mathbf f)$ is invertible; and $(z_1,\dots ,z_N)$ is linearly independent as well since $(z_1,\dots ,z_N)$ belongs to $\textrm{HC}(T_{N})$. 
Let us denote by $(\widetilde u_1,\dots ,\widetilde u_N)$ and $(\widetilde z_1,\dots ,\widetilde z_N)$ the sequences obtained by applying the Gram-Schmidt 
orthonormalization process to $(u_1,\dots ,u_N)$ and $(z_1,\dots ,z_N)$ respectively. Then $(\widetilde u_1,\dots ,\widetilde u_N)$ and $(\widetilde z_1,\dots ,\widetilde z_N)$ are 
very close to each other, provided $(u_1,\dots ,u_N)$ and $(z_1,\dots ,z_N)$ are sufficiently close. 
Now, define $(x_1,\dots ,x_N)\in\h^{N}$ as follows:
\[x_i:=\sum_{l=1}^N \pss{u_i}{\widetilde u_l}\, \widetilde z_l\qquad\hbox{for $i=1,\dots ,N$.}
\] 
We have by definition
\[\pss{x_i}{x_j}=\sum_{l=1}^N \pss{u_i}{\widetilde u_l} \,\overline{\pss{u_j}{\widetilde u_l}}=\pss{u_i}{u_j} =\pss{f_i}{f_j}\quad \textrm{ for } i,j=1,\dots ,N,
\]
so that $(x_1,\dots ,x_N)$ belongs to $\h_{\bf f}$. Moreover, each vector $x_i$, $1\le i\le N$, is very close to $\sum_{l=1}^N \pss{u_i}{\widetilde u_l}\, \widetilde u_l=u_i$, 
so that $(x_1,\dots ,x_N)$ is very close to $(u_1,\dots ,u_N)$. It remains to show that $(x_1,\dots ,x_N)$ belongs to $ \textrm{HC}(T_{N})$.
\par\smallskip
Since the vectors $x_i$ are linearly independent (because $(x_1,\dots ,x_N)$ belongs to $ \h_{\bf f}$) and belong to 
${\rm span}\,[\widetilde z_1,\dots ,\widetilde z_N]={\rm span}\,[z_1,\dots ,z_N]$, they form a basis of 
${\rm span}\,[z_1,\dots ,z_N]$. So we may write each vector $x_{i}$ as
\[x_i=\sum_{j=1}^N c_{i,j}\, z_j,
\]
where the matrix $(c_{i,j})_{1\le i,j\le N}$ is invertible.
Now, let $V_1,\dots ,V_N$ be non-empty open sets in $\h$, and define
\[\mathcal V:=\Bigl\{ (y_1,\dots ,y_N)\in\h^{N};\; \sum_{j=1}^N c_{i,j}\, y_j\in V_i\quad\hbox{for $i=1,\dots ,N$}\Bigr\}.
\]
Since the matrix $(c_{ij})_{1\le i,j\le N}$ is invertible, this is a \emph{non-empty} open subset of $\h^{N}$. As $(z_1,\dots ,z_N)\in \textrm{HC}(T_{N})$, 
one can find an integer $n$ such that $(T^nz_1,\dots ,T^nz_N)\in\mathcal V$, which means that $(T^nx_1,\dots ,T^nx_N)\in V_1\times\cdots\times V_N$ 
by the definition of $\mathcal V$. This shows that $(x_1,\dots ,x_N)$ belongs to $ \textrm{HC}(T_{N})$.
\end{proof}

From Fact \ref{W1}, we now deduce:

\begin{fact}\label{W2} Let $M>1$. For any finite family  $\mathbf f=(f_1,\dots ,f_N)$ of linearly independent vectors in $\h$, the set of all $T\in\mathfrak B_M(\h)$ such that 
$(f_1,\dots ,f_N)$ belongs to $ \textrm{HC}(T_{N})$ is comeager in {$(\bmh,\sote)$}. 
\end{fact}

\begin{proof}[Proof of Fact \ref{W2}] Let us consider the set 
\[\mathcal G:=\Bigl\{ (x_1,\dots ,x_N, T)\in \h_{\bf f}\times\bmh;\; (x_1,\dots ,x_N)\in \textrm{HC}(T_{N})\Bigr\}.
\]
This is a $G_\delta$ subset of $\h_{\bf f}\times (\bmh,\sote) $. Moreover, Fact \ref{W1} asserts that for any $T\in \textrm{TWMIX}_M(\h)$, the $T$-section 
of $\mathcal G$ is dense in $\h_{\bf f}$, and hence  comeager in $\h_{\bf f}$ since this is a $G_\delta$ set. Since $\textrm{TWMIX}_M(\h)$ is comeager 
in {$(\bmh,\sote)$}, it follows, by the Kuratowski-Ulam Theorem (see for instance \cite[Section 8.K]{Ke}), that there exists at least one (in fact, comeager many) $\mathbf x=(x_1,\dots ,x_N)\in\h_{\bf f}$ 
such that the set 
\[ \mathfrak G_{\bf x}:=\{ T\in \bmh;\; (x_1,\dots ,x_N)\in \textrm{HC}(T_{N})\}
\] 
is comeager in {$(\bmh,\sote)$}. 
\par\smallskip
Now, since $(x_1,\dots ,x_N)$ belongs to $\h_{\bf f}$, \mbox{\it i.e.} $\pss{x_i}{x_j}=\pss{f_i}{f_j}$ for every $i,j=1\dots ,N$, 
one can find a unitary operator $U:\h\to\h$ such that $Uf_i=x_i$ for every $i=1,\dots ,N$. Then the map $T\mapsto U^{-1}TU$ 
maps $\bmh$ bijectively onto itself because $U$ is unitary, and is a homeomorphism with respect to the topology 
$\texttt{SOT}^*$. Therefore, the set $\mathfrak G_{\bf f}:=U^{-1}\mathfrak G_{\bf x}U$ is comeager in {$(\bmh,\sote)$}. Since by definition
$(f_1,\dots ,f_N)$ belongs to $ \textrm{HC}(T_{N})$ if $T$ belongs to $ \mathfrak G_{\mathbf f}$,  this concludes the proof of Fact \ref{W2}.
\end{proof}

We are now ready to conclude the proof of Proposition \ref{ZW}. Let $(f_i)_{i\geq 1}$ be an algebraic basis of $Z$. By Fact \ref{W2}, the set 
\[ \mathfrak G:=\{ T\in\mathfrak B_M(\h);\; \forall N\geq 1\;:\; (f_1,\dots ,f_N)\in \textrm{HC}(T_{N})\}
\]
is comeager in {$(\bmh,\sote)$}. So it is enough to show that that if $T$ belongs to $\mathfrak G$, every non-zero vector $z\in Z$ is 
a hypercyclic vector for $T$.
Let $V$ be a non-empty open set in $\h$. Write $z$ as $z=\sum_{i=1}^N z_i f_i$, and consider the set 
\[\mathcal V_z:=\Bigl\{ (x_1,\dots ,x_N)\in\h^{N};\; \sum_{i=1}^N z_i x_i\in V\Bigr\}.
\]
Since $z\neq 0$, this is a {non-empty} open set in $\h^{N}$. As $T$ belongs to $\mathfrak G$, it follows that there exists $n\ge 1$ such that $(T^nf_1,\dots ,T^nf_N)$ belongs to $\mathcal V_z$, which means exactly that 
$T^nz$ belongs to $ V$. Hence $z$ is indeed a hypercyclic vector for $T$.
\end{proof}
\par\smallskip

\subsection{Hypercyclic operators without eigenvalues}\label{Subsection 1.b.2} 
The following result shows that operators without eigenvalues are typical.
\begin{proposition}\label{Proposition 5}
 For any $M>1$, the class \emph{$\nevmh$} is a dense $\gd $ subset of 
 \emph{$(\bmh,\sote)$}. 
\end{proposition}

From this and Fact \ref{comeagerbis}, we obtain
\begin{corollary}\label{Prop5coro} For any $M>1$, a typical \op\ $T\in \emph{$(\hcmh,\sote)$}$ 
has 
no eigenvalues.
\end{corollary} 

As a matter of fact, Proposition \ref{Proposition 5} is 
already proved  in \cite{EM}, where typical 
properties of {contraction operators} are studied for various 
topologies (see also \cite{E}).  
However, since the proof is not that complicated, 
and in order to keep the paper as self-contained as possible, we outline 
it below. 

\begin{proof}[Proof of Proposition \ref{Proposition 5}] We divide the 
proof into two steps. In what follows, we fix $M>0$.

\begin{fact}\label{Lemma 5 bis} The set
$\nevmh$ is a $\gd$ subset 
of 
$(\bmh,\sote)$.
\end{fact}

\begin{proof}[Proof of Fact \ref{Lemma 5 bis}]
To any closed ball $B\subseteq\h$, we associate the following
subset of $\bmh$:
\[
\mathfrak M_B=\bigl\{T\in\bmh\,;\,\exists\,\lambda \in\C,\ \exists\,x\in 
B\ \textrm{with}\ Tx=\lambda x\bigr\}.
\]
Let us show that this set $\mathfrak M_B$ is $F_{\sigma }$ in $(\bmh,\sote)$. To 
this aim, we endow $B$ with the weak topology, and introduce the set
\[
\ffb=\bigl\{(T,\lambda ,x)\in\bmh\times \C\times B\,;\,Tx=\lambda x\bigr\}.
\]
Then $\mathfrak M_B$ is the projection of $\ffb$ on the first coordinate. Moreover, the set 
$\mathcal F_B$ is closed in $(\bmh,\sote)\times\C\times (B,w)$ by Fact \ref{jointcontinuity}. Since the space 
$\C\times (B,w)$ is $K_\sigma$ (because $(B,w)$ is compact), it follows that $\mathfrak M_B$ is $F_\sigma$.
\par\smallskip
Let now $(B_{q})_{q\ge 1}$ be a sequence of closed balls of $\h$ not containing the point $0$, 
such that  
$\bigcup_{q\ge 1}B_{q}=H\setminus\{0\}$. Then 
$\bmh\setminus \nevmh=\bigcup_{q\ge 1}\mathfrak{M}_{B_{q}}$ is an $F_{\sigma }$ 
set, so that $\nevmh$ is a $\gd$ set in $(\bmh,\sote)$. 
\end{proof}

\begin{fact}\label{Lemma 5 ter} The set
$\nevmh$ is dense in $(\bmh,\sote)$.
\end{fact}

\begin{proof}[Proof of Fact \ref{Lemma 5 ter}]
Let $(f_{k})_{k\in\Z}$ be an 
orthonormal basis of $\h$. For each $r\ge 
0$, we set $\h_{r}:=\textrm{span}
[f_{k},\,;\,|k|\le r]$ and we denote by $P_{r} $ the orthogonal projection of $\h$ onto $\h_{r}$. By 
Lemma \ref{Lemma 4 bis} and Remark \ref{remarque en plus}, it suffices to show that for every 
$r\ge 0$, every $A\in\mathfrak{B}(\h_{r})$, with $\|A\|<M$, and every 
$\varepsilon >0$, there exists an operator $T\in\nevmh$ such that 
\[
\|(T-A)f_{k}\|<\varepsilon\quad{\rm and}\quad \|(T-A)^{*}f_{k}\| 
<\varepsilon\qquad\hbox{for every $k=-r,\dots ,r$.}
\]
Let $\pmb{\omega} =(\omega_{k})_{k\in\Z}$ be a bilateral weight sequence with $0<\omega _{k}\le M$ for every 
 $k\in\Z$ and $\omega _{k}=\delta $ for every index $k$ with
 $|k+r|\le 2r+1$. As in the proof of Corollary \ref{Proposition 4}, one 
easily checks that if $\delta >0$ is sufficiently small, the bilateral weighted shift  operator 
$S_{A,\pmb{\omega} }$ (defined with respect to the basis 
$(f_{k})_{k\in\Z}$) satisfies 
\[
\|(S_{A,\pmb{\omega}}-A)f_{k}\|<\varepsilon\quad{\rm and}\quad 
\|(S_{A,\pmb{\omega} }-A)^{*}f_{k}\| <\varepsilon
\qquad\hbox{for every $k=-r,\dots ,r$.}
\]
Moreover, if the weight sequence
$\pmb{\omega} $ is chosen in such a way that the series
\[
\sum_{p\ge 1}\,\bigl(\omega _{p(2r+1)+l}\dots\omega _{(2r+1)+l}
\bigr)^{-2}\quad \textrm{and}\quad
\sum_{p\ge 1}\,\bigl(\omega _{l-p(2r+1)}\dots\omega _{l-(2r+1)}
\bigr)^{-2}
\]
are divergent for every $|l|\le r$, then $S_{A,\pmb{\omega} }$ has no eigenvalue by 
Remark \ref{Remark 3}. So  $T:=S_{A,\pmb{\omega} }$ satisfies the required assumptions for a suitable choice of the weight $\pmb\omega$.
\end{proof}

The two facts above complete the proof of
Proposition \ref{Proposition 5}.
\end{proof}

\begin{remark}\label{Remark 6}
 The fact that we are using the topology $\sote$ is crucial in the proof of Fact  
\ref{Lemma 
5 bis} in order to obtain that the sets $\ffb$ above are closed in 
$(\bmh,\sote)\times\C\times (B,w)$. The situation turns out to be completely different if one considers the topology 
$\texttt{SOT}$ instead of $\texttt{SOT}^*$. Indeed, it is proved in 
\cite{EM} that an $\texttt{SOT}$-typical 
$T\in\mathfrak{B}_{1}(\h)$ has the property that \emph{every $\lambda 
\in\C$ with $|\lambda |<1$ is an eigenvalue of $T$}. More precisely, a 
typical $T\in\mathfrak{B}_{1}(\h)$ is unitarily equivalent to the 
infinite-dimensional backward unilateral shift operator. So a typical $T\in(\bmh,\sot)$ has the whole disk $D(0,M)$ 
within the set of its eigenvalues.
\end{remark}

\begin{remark}\label{Remark 6 bis}
 The proof of Corollary
\ref{Prop5coro} is a good example of the usefulness of Fact \ref{comeagerbis} for simplifying arguments of 
this kind: although it is easy to construct the weight sequence $\pmb{\omega} $ above in such a 
way that the operator $S_{A,\pmb{\omega} }$ is $\texttt{SOT}^*$-close to $A$ and has no eigenvalue, it is technically much less obvious to 
ensure that $S_{A,\pmb{\omega} }$ is additionally hypercyclic. In other words, it would be less easy to prove Corollary 
\ref{Prop5coro} directly, without using Fact \ref{comeagerbis}.
\end{remark}

One of the main consequences of Corollary \ref{Prop5coro} is the 
following result concerning 
chaotic operators.

\begin{corollary}\label{Corollary 7}
 For every $M>1$, \emph{$\cmh$} is meager in \emph{$(\hcmh,\sote)$}. In other words, a typical hypercyclic operator on 
 $\h$ is not chaotic.
\end{corollary}

This is indeed obvious since chaotic operators have plenty of eigenvalues. Note that we are using here the fact that $\h$ is a \emph{complex} 
Hilbert space, so that periodic points are linear combinations of eigenvectors whose associated eigenvalues are roots of unity. Nonetheless, Corollary~\ref{Corollary 7}
holds true on real Hilbert spaces as well; see Remark~\ref{periodicmeager} below.
\par\smallskip

\subsection{Hypercyclic operators without invariant measures}\label{Subsection 1.b.3} 
It follows easily from Proposition \ref{Proposition 5} that for any $M>1$, the operators in $\hcmh$ admitting a non-trivial invariant measure with a 
second-order moment form a meager class in $(\hcmh,\sote)$. Indeed, if 
$T\in\hcmh$ admits an invariant measure $\mu\neq\delta_0$ such that 
$\int_{H}\|x\|^{2}\,
d\mu (x)<\infty $, then the Gaussian measure $m$ whose covariance operator is 
given by the formula
\[
\pss{Rx}{y}=\int_{H}\pss{x}{z}\,\ba{\pss{y}{z}}\,d\mu (z),\quad
x,y\in H
\]
is $T$-invariant. Its support is the closed linear span of the support 
of 
$\mu $, and hence is non-trivial. This closed subspace is spanned by 
unimodular eigenvectors of $T$ (see \cite{BG2} or \cite{BM} for 
details), 
from which it follows that $T$ does not belong to $\nevmh$. 
\par\smallskip
The main result of this subsection is that the $\texttt{SOT}^*$-typical 
operator $T\in\hcmh$ actually admits no non-trivial invariant measure at all.

\begin{theorem}\label{Theorem 9}
 For every $M>1$, the set \emph{$\hcmh\setminus\invh$} is comeager in the 
space \emph{$(\hcmh,\sote)$}.
\end{theorem}

The proof of Theorem \ref{Theorem 9} relies on the next two lemmas.

\begin{lemma}\label{Lemma 10} Let $B$ be a closed ball of $\h$ not containing the point $0$. 
For any integer $n\geq 1$, the set
\begin{align*}
 {\mathfrak O} _{n,B}=\bigl\{T\in\bmh\,;\,\textrm{there}\ &\textrm{exist}\ n\ 
\textrm{distinct integers}\ p_{1},\dots,p_{n}\ \textrm{such that}\\
&T^{p_{i}}(B)\cap T^{p_{j}}(B)=\emptyset\ \textrm{for every}\ 
i\neq j,\ 1\le i,j\le n\bigr\}
\end{align*}
is open in \emph{$(\bmh,\sote)$}. Consequently, the set 
\begin{align*}
  \mathfrak G_{B}=\bigl\{T\in\bmh\,;\,\textrm{for every}&\ \gn,\ \textrm{there 
exist}\ n\ \textrm{iterates of}\ B\\
&\textrm{under the action of}\ T\ \textrm{which are pairwise 
disjoint}\bigr\}
\end{align*}
is $\gd$ in \emph{$(\bmh,\sote)$}.
\end{lemma}

\begin{proof}[Proof of Lemma \ref{Lemma 10}] The second part of the lemma follows immediately from the first, since 
$ \mathfrak G_{B}= \bigcap_{\gn}{\mathfrak O} _{n,B}$.
To derive the first part, it is enough to show that if we fix $p,q\ge 1$, then the set 
\[{\mathfrak O}:=\{ T\in\bmh;\; T^p(B)\cap T^q(B)=\emptyset\}
\]
is $\texttt{SOT}^*$-open in $\bmh$. 
If $T\in\bmh$, we may write
\[
T\in\mathfrak O\iff \forall x,y\in B\; :\;T^p x\neq T^qy.
\]
Since the map $(T,u)\mapsto{T^nu}$ is continuous on $(\bmh, \texttt{SOT}^*)\times (B,w)$ for any $n\ge 1$ by 
Facts \ref{power} and \ref{jointcontinuity}, 
the condition ``$T^px\neq T^qy$" defines an open subset of $(\bmh,\texttt{SOT}^*)\times (B,w)\times (B,w)$. 
Since $B$ is weakly compact, it follows that 
$\mathfrak O$ is indeed 
$\texttt{SOT}^*$-open in $\bmh$, its complement being the projection of a closed 
subset of $\bmh\times B\times B$ along the compact factor $B\times B$.
\end{proof}

\par\smallskip
\begin{lemma}\label{Lemma 11} Let $e\in\h$ with $\Vert e\Vert=1$, and let $0<\rho <1$. Denote by $B$ the closed ball
 $\ba{B}(e,\rho )$. Let also $M>1$. For any $n\geq 1$, the open set ${\mathfrak O} _{n,B}$ is dense in \emph{$(\bmh,\sote)$}.
\end{lemma}

\begin{proof}[Proof of Lemma \ref{Lemma 11}] Let us fix an orthonormal basis 
 $(e_{k})_{\gk}$ of $\h$ with $e_1=e$. 
Our aim being to apply Lemma \ref{Lemma 4 bis}, we fix $r\ge 
1$, an operator
$A\in\mathfrak{B}(\h_{r})$ such that $\|A\|<M$, and $\varepsilon >0$. We are looking 
for an operator $T\in {\mathfrak O} _{n,B}$ such that 
\[
\|(T-A)e_{k}\|<\varepsilon\quad{\rm and}\quad \|(T-A)^{*}e_{k}\|<\varepsilon\qquad\hbox{for $k=1,\dots ,r.$} 
\]
\par\smallskip
We will define a sequence $(C_{N})_{N>2r}$ of \ops, with $C_{N}\in\mathfrak{B}(\h_{N})$ for every $N>2r$, and show that if $N$ is sufficiently large, the \op\  $P_{N}^*C_{N}P_{N}\in\bh$ belongs 
to  ${\mathfrak O}_{n,B} $ and satisfies the above estimates. Here $P_{N}$ denotes as usual the canonical projection of $\h$ onto $\h_{N}$.
\par\smallskip
Let $\delta>0$ and $\gamma >1$, to be fixed later on in the proof. For each $N> 2r$, consider the operator 
$C_N\in\mathfrak{B}(\h_{N})$ defined in the following way:
\[C_{N}e_{k}=
\begin{cases}
Ae_{k}+\delta e_{k+r}&\textrm{for every}\ 1\le k\le r\\
\gamma e_{k+r}&\textrm{for every}\ r<k\le N-r\\
0&\textrm{for every}\ N-r<k\le N.
\end{cases}
\]
Thus, in matrix form, 
\par\smallskip
\[C_N=
\left(
\begin{array}{ccc}
\begin{array}{|c|}
\hline \\
\begin{array}{ccc}
& {\hbox{ \fontsize{24.88}{24.88}\selectfont A}}& \\
& &
\end{array}\\
\hline
\begin{array}{ccc}
\delta& & \\
 &\ddots&\\
 & &\delta
 \end{array}\\
\hline
\end{array}

&

\begin{array}{c}
{\hbox{ \fontsize{24.88}{24.88}\selectfont 0}}

\end{array}

&

\begin{array}{c}
{\hbox{ \fontsize{24.88}{24.88}\selectfont 0}}
\end{array}

\\

\begin{array}{c}
{\hbox{ \fontsize{24.88}{24.88}\selectfont 0}}
\end{array}

&

\begin{array}{|cccc|}
\hline
\gamma& & & \\
 &\ddots&&\\
 & &\ddots&\\
 & & &\gamma\\
 \hline
\end{array} 

&

\begin{array}{c}
\phantom{bla}{\hbox{ \fontsize{24.88}{24.88}\selectfont 0}}
\phantom{bla}
\end{array}

\end{array}
\right)
\]
\par\medskip
We note that if $\|A\|<\gamma <M$ and if $\delta$ is sufficiently small, then $\|C_{N}\|=\gamma <M$. Moreover, if $\delta $ is 
sufficiently small, then
\begin{equation}\label{Equation 1}
 \|(C_{N}-A)e_{k}\|<\varepsilon\quad{\rm and}\quad \|(C_{N}-A)^{*}e_{k}\| <\varepsilon \qquad\hbox{for every $k=1,\dots ,r$}
\end{equation}
whatever the choice of the integer $N> 2r$. The key of the 
proof lies in the following simple computation.

\begin{fact}\label{Fact 12}
 For every $N>2r$, every $p\ge 1$ such that $pr+1\le N$, and every 
 $x\in\h$,
 \[
\pss{C_{N}^{p}x}{e_{pr+1}}=\gamma ^{p-1}\delta \pss{x}{e_{1}}.
\]
\end{fact}

\begin{proof}[Proof of Fact \ref{Fact 12}]
Clearly $\pss{C_{N}x}{e_{r+1}}=
\delta \pss{x}{e_{1}}$; so we may write $C_{N}x$ as
\[\smash[b]{C_{N}x=\delta \pss{x}{e_{1}}e_{r+1}+\ds\sum
_{\genfrac{}{}{0pt}{}{k=1}{k\neq r+1}}^{N}\pss{C_{N}x}{e_{k}}e_{k}}.
\]
Hence
\[
C_{N}^{p-1}x=\delta \gamma ^{p-1}\pss{x}{e_{1}}e_{pr+1}+
\sum
_{\genfrac{}{}{0pt}{1}{k=1}{k\neq r+1}}^{N}\pss{C_{N}x}{e_{k}}C_{N}^{p-1}e
_{k}.
\]
Now $C_{N}^{p-1}e_{k}$ belongs to the closed linear span of the vectors 
$e_{j}$, $1\le j\le N$, $j\neq pr+1$, for every $1\le k\le N$ with $k\neq 
r+1$. Indeed, 
\begin{enumerate}
 \item [-] if $1\le k\le r$, a straightforward induction shows that 
$C_{N}^{p-1}e_{k}\in\textrm{span}[e_{j}\,;\,1\le j\le pr]$;
\item[-] if $r+1<k\le N-(p-1)r$, then $C_{N}^{p-1}e_{k}=\gamma ^{p-1}
e_{k+(p-1)r}$;
\item[-] if $N-(p-1)r<k\le N$, we have $C_{N}^{p-1}e_{k}=0$.
\end{enumerate}
Thus $\pss{C_{N}^{p-1}x}{e_{pr+1}}=\delta \gamma ^{p-1}\pss{x}{e_{1}}$, 
which is the claim of Fact \ref{Fact 12}.
\end{proof}

From Fact \ref{Fact 12}, it is not hard to deduce

\begin{fact}\label{factdisjoint}
Let $1\le p,q\le N$ be such that $pr+1\le N$. If $\gamma ^{p-q-1}\delta >\frac{1+\rho }{1-\rho }$, then 
$C_N^p(P_NB)\cap C_N^q(P_NB)=\emptyset$.
\end{fact}

\begin{proof}[Proof of Fact \ref{factdisjoint}]
If $x,y\in P_{N}B$, then by Fact \ref{Fact 12}
\[
\|C_{N}^{p}y-C_{N}^{q}x\|\ge \bigl|\pss{C_{N}^{p}x-C_{N}^{q}y}
{e_{pr+1}}\bigr|\ge \gamma ^{p-1}\delta|\pss{x}{e_{1}}|-
\|C_{N}^{q}\|\, \|y\|.
\]
Moreover, since $x$ and $y$ belong to $P_{N}B\subseteq B$, we have $|\pss{x}{e_{1}}|\ge 
1-\rho $
and $\|y\|\le 1+\rho $. Hence 
\[
\|C_{N}^{p}y-C_{N}^{q}x\|\ge \gamma ^{p-1}\delta (1-\rho )-\gamma 
^{q}(1+\rho ).
\]
Thus $C_{N}^{q}(P_{N}B)\cap C_{N}^{p}(P_{N}(B))$ is empty as soon as 
$\gamma ^{p-1}\delta (1-\rho )-\gamma ^{q}(1+\rho )>0$, which proves our claim.
\end{proof}

We now choose the various parameters in this construction in the following order: first 
we choose $\gamma $ such that $\max(1,\|A\|)<\gamma <M$. Then we choose 
$\delta>0 $ so small that (\ref{Equation 1}) holds true for every $N\ge 
1$. 
Lastly, we choose $N>2r$ so large that there exist $n$ distinct integers 
$1\le p_{1}<p_{2}<\dots<p_{n}\le N$ with $p_{n}r+1\le N$, such that the 
gaps between two consecutive integers $p_{j}$ are so large that 
$\gamma ^{p_{j}-p_{i}-1}\delta >(1-\rho )/(1+\rho )$ for every $1\le 
i<j\le n$. 
The operator $C_{N}$ then satisfies $C_{N}^{p_{j}}(P_{N}B)\cap 
C_{N}^{p_{i}}(P_{N}B)=\emptyset$ for every $1\le i<j\le n$. So the operator 
$T:=P_{N}^*C_{N}P_{N}\in\bmh$ is such that
$T^{p_{j}}(B)\cap T^{p_{i}}(B)=\emptyset$
for every $1\le i<j\le n$, that is,
$T$ belongs to $ {\mathfrak O}_{n,B}$; and by (\ref{Equation 1}) we also have 
$\|(T-A)e_{k}\|<\varepsilon$ and $\|(T-A)^{*}e_{k}\|<\varepsilon$ {for every $k=1,\dots ,r.$} 
\end{proof}

\begin{proof}[Proof of Theorem \ref{Theorem 9}]
Combining Lemmas \ref{Lemma 10} and \ref{Lemma 11}, we obtain that $\mathfrak G_{B}$ 
is a dense $\gd$ subset of $(\bmh,\sote)$ for every ball 
$B=\ba{B}(e,\rho )$, where $\Vert e\Vert=1$ and 
$0<\rho <1$. 
By Fact \ref{comeagerbis}, it follows that $\mathfrak G_{B}\cap \hcmh$ is a dense $\gd$ subset of 
$(\hcmh,\sote)$ for each such ball $B$.
Let $(B_{q})_{q\ge 1}$ be a 
countable family of such balls with the property that 
\[\bigcup_{q\ge 1}B_{q}=B(0,2)\setminus\{0\}.
\]
Then 
$ \mathfrak G:=\bigl(\bigcap_{q\ge 1} \mathfrak G_{B_{q}} \bigr)\cap \hcmh$ is a dense $\gd$ 
subset of $(\hcmh,\sote)$. 
\par\smallskip
Every element $T$ of $ \mathfrak G$ enjoys the property that for any $q\ge 1$ and 
any 
 $\gn$, there exist $n$ distinct iterates $T^{j}(B_{q})$ of $B_{q}$ which  
are pairwise distinct. It follows that if $m$ is any invariant (probability) measure for $T$, 
then $m(B_{q})=0$ for every $q\ge 1$ and hence 
$m(B(0,2)\setminus\{0\})=0$. 
\par\smallskip
Suppose now that $T\in  \mathfrak G$ admits an invariant measure $m\neq \delta _{0}$. Then one 
can find a closed ball $B'$ not containing $0$ such that $m(B')>0$. 
Consider for any $R>0$ the measure $m_{R}$ on $\h$ defined by $m_{R}(C)=
m(RC)$ for any Borel subset $C$ of $\h$. Each such measure $m_{R}$ is an 
invariant probability measure on $\h$. Moreover, if $R$ is sufficiently large then 
$R^{-1}B'\subseteq B(0,2)\setminus \{0\}$, so that $m_{R}(B(0,2)\setminus 
\{0\})>0$, which is a contradiction. We have thus proved that any operator $T\in \mathfrak G$ 
admits no invariant measure except $\delta _{0}$, and hence that 
$\hcmh\setminus \invh$ is comeager in $(\hcmh,\sote)$.
\end{proof}

\begin{remark}\label{periodicmeager} 
The above proof does not use the fact that $\h$ is a 
\emph{complex} Hilbert space; so Theorem \ref{Theorem 9} holds true as well
for  real Hilbert spaces. This shows in particular that in the real setting also, 
the chaotic operators form a meager subset of $\bmh$; more precisely, 
a typical operator $T\in\bmh$ has no periodic point except $0$. Indeed, 
any periodic point $x\neq 0$ for an operator $T$ gives rise to a ``canonical" 
invariant measure supported on the orbit of $x$, namely $m:=\frac{1}N\sum_{i=0}^{N-1} \delta_{T^ix}$, 
where $N \ge 1$ is a period of $x$.
\end{remark}

\begin{remark} Theorem \ref{Theorem 9} implies a weak form of Proposition \ref{Proposition 5}, namely that operators without any eigenvalue
\emph{of modulus $1$} are typical. Indeed, if $T$ admits a unimodular eigenvalue and if $x$ is an associated eigenvector, 
then there is a canonical invariant measure $m\neq \delta_0$ supported on $\T\cdot x$, 
namely the image of the Lebesgue measure on $\T$ under the map $\lambda\mapsto\lambda x$. 
\par\smallskip
Regarding the eigenvalues, one may also note that if $m$ 
is any non-trivial invariant measure for an operator $T\in\bh$, then $m\bigl(\ker(T-\lambda)\bigr)=0$ 
for every complex number $\lambda$ such that 
$\vert\lambda\vert\neq 1$. Indeed, if $m\neq\delta_0$ is a measure such that 
$m\bigl(\ker(T-\lambda)\bigr)>0$,  one can find a ball $B$ not containing $0$ such that
$m(B_\lambda)\neq 0$, where $B_\lambda=B\cap \ker(T-\lambda)$. Since $\vert\lambda\vert\neq 1$, 
it is easily checked that $B_\lambda$ has infinitely many pairwise
disjoint iterates under $T$, which is not possible since the measure $m$ is $T$-invariant.
\end{remark}

Since $\mathcal{U}$-frequently hypercyclic operators on a Hilbert space 
always admit an invariant measure with full support by \cite{GM}, Theorem \ref{Theorem 9} 
immediately implies:

\begin{corollary}\label{Corollary 9 bis}
 For every $M>1$, $\emph{\textrm{UFHC}}_{M}(\h)$ is meager in 
\emph{$(\hcmh,\sote)$}. In other words, a typical hypercyclic operator on $\h$ is not 
 $\mathcal U$-frequently hypercyclic.
\end{corollary}
\par\smallskip

\subsection{Densely distributionally chaotic operators}\label{Subsection 1.b.4}
In this subsection, our aim is to show that, for any $M>1$, the class $\textrm{DDCH}_{M}(\h)$ of densely
distributionally chaotic operators in $\bmh$ is a dense $G_{\delta}$ subset of 
$(\bmh,\sote)$, from which it follows that the class $\textrm{DCH}_{M}(\h)$ of 
distributionally chaotic operators in $\bmh$ is a comeager subset of 
$(\bmh,\sote)$.

\begin{proposition}\label{Theorem 15}
 For any $M>1$, the set \emph{$\textrm{DDCH}_{M}(\h)$} of densely distributionally chaotic operators in $\bmh$ is a dense $G_{\delta }$ subset of \emph{$(\bmh,\sote)$}.
\end{proposition}

\begin{proof}
 Recall that an operator $T\in \bh$ is densely distributionally chaotic if and only 
if it admits a dense set of distributionally irregular vectors, \mbox{\it i.e.} of vectors $x\in \h$ for 
which there exist two sets of integers $A,B\subseteq \N$ with upper density $1$ such that 
$ T^{i}x \to 0$ as $i\to\infty$ along $A$ and 
$\Vert T^{i}x\Vert \to \infty$ as $i\to\infty$ along $B$.
 Since the set of distributionally irregular vectors for $T$ can be written as
\[
G_{T}=\bigcap_{\varepsilon \in\Q^{+*}}\bigcap_{N\ge 1} G_{T,\varepsilon ,N}
\]
where
\begin{align*}
 G_{T,\varepsilon ,N}:=\Bigl \{ x\in \h\,;\,\exists\,m,n\ge N&\ :\ 
 \#\{1\le i\le m\,;\,\Vert T^{i}x\Vert <\varepsilon \}\ge m(1-\varepsilon )\\
  &\textrm{and}\ \#\{1\le i\le n\,;\, \Vert T^{i}x\Vert >1/\varepsilon\} \ge n(1-\varepsilon )\Bigr\}, 
\end{align*}
it follows that $G_{T}$ is a dense $G_{\delta }$ subset of $\h$ whenever $T$ is densely 
distributionally chaotic. Denoting by $(V_{p})_{p\ge 1}$ a countable basis of non-empty open 
subsets of $\h$, we infer from this observation that an operator $T\in \bmh$ belongs to 
$\textrm{DDCH}_{M}(\h)$ if and only if 
\begin{align*}
 \forall\,\varepsilon \in\Q^{+*}\; \forall\,N\ge 1\; \forall\,p\ge 1\;\; \exists x\in V_p\;\;\exists n,m\geq N\
 &\#\{1\le i\le m\,;\,\Vert T^{i}x\Vert <\varepsilon \}\ge m(1-\varepsilon )\\
  \textrm{and}\ &\#\{1\le i\le n\,;\, \Vert T^{i}x\Vert >1/\varepsilon\} \ge n(1-\varepsilon ).
\end{align*}

Using this, we can now prove

\begin{fact}\label{FactUn}
 The set $\textrm{DDCH}_{M}(\h)$ is a $G_{\delta }$ subset of $(\bmh,\sot)$.
\end{fact}

\begin{proof}
It suffices to show that for every $\varepsilon \in\Q^{+*}$, $N\ge 1$, $p\ge 1$,  $m,n\ge 
N$ and $x\in\h$, the set 
\begin{align*}
 \Bigl\{ 
 T\in \bmh\,;\ &\#\{1\le i\le m\,;\,\Vert T^{i}x\Vert <\varepsilon \}\ge m(1-\varepsilon )\\
  \textrm{and}\ &\#\{1\le i\le n\,;\, \Vert T^{i}x\Vert >1/\varepsilon\} \ge n(1-\varepsilon )
 \Bigr \}
\end{align*}
is $\textrm{SOT}$-open. So let $T_{0}\in \bmh$ belong to this set, and let $r\ge 
m(1-\varepsilon)$ be an integer such that there exist $r$ indices $1\le i_{1}<\dots<i_{r}\le 
m$ such that $\Vert T^{i_{j}}x\Vert <\varepsilon $ for every $1\le j\le r$. If $T\in \bmh$ is sufficiently 
close to $T_{0}$ for the $\textrm{SOT}$-topology, we still have $\Vert T^{i_{j}}x\Vert <\varepsilon $ for every 
$1\le j\le r$. Hence $\#\{1\le i\le m\,;\,\Vert T^{i}x\Vert <\varepsilon \}\ge m(1-\varepsilon )$. In 
the same way, the set of operators $T$ such that $\#\{1\le i\le n\,;\, 
\Vert T^{i}x\Vert >1/\varepsilon\} \ge n(1-\varepsilon )$ is $\textrm{SOT}$-open in $\bmh$. This proves 
Fact \ref{FactUn}.
\end{proof}

The last step of the proof of Proposition \ref{Theorem 15} is the following fact.

\begin{fact}\label{FactDeux}
 The set $\textrm{DDCH}_{M}(\h)$ is dense in $(\bmh,\sote)$.
\end{fact}

\begin{proof}
By a result of \cite{GM}, every ergodic operator on $\h$ is densely distributionally chaotic. 
Since $\textrm{G-MIX}(\h)$ is dense in $(\bmh,\sote)$ by Corollary \ref{Proposition 4}, it immediately 
follows that $\textrm{DDCH}_{M}(\h)$ is dense in $(\bmh,\sote)$.
\end{proof}

By Facts \ref{FactUn} and \ref{FactDeux}, the proof of Proposition \ref{Theorem 15} is now complete.
\end{proof}

Since $\textrm{DDCH}(\h)\subseteq \textrm{DCH}(\h)$, and since any operator $T\in \textrm{DDCH}(\h) \cap \textrm{HC}(\h)$ satisfies $c(T)=1$, 
Proposition \ref{Theorem 15} has the following immediate consequences:

\begin{corollary}\label{Corollary 16}
 For any $M>1$, the set \emph{$\textrm{DCH}_{M}(\h)$} of distributionally chaotic operators in \emph{$\bmh$} is 
comeager in \emph{$(\bmh,\sote)$}.
\end{corollary}

\begin{corollary}\label{Corollary 17}
 For any $M>1$, the set \emph{$\textrm{cMAX}_{M}(\h)$} of operators \emph{$T\in {\hcmh}$} with $c(T)=1$ is 
comeager in \emph{$(\hcmh,\sote)$}.
\end{corollary}

\smallskip
There is an alternative approach for proving the comeagerness of $\textrm{DDCH}_M(\h)$, which relies 
on the next proposition, combined with the Kuratowski-Ulam Theorem.

\begin{proposition}\label{Proposition 14}
 Let $M>1$. For every vector $x\in\h\setminus\{ 0\}$, the set 
 \[
\mathfrak G^{x}=\bigl\{T\in \hbox{\emph{$\bmh$}}\,;\,x\ \textrm{is a distributionally 
irregular vector for}\ T\bigr\}
\]
is comeager in \emph{$(\hcmh,\sote)$}.
\end{proposition}

\smallskip
By the Kuratowski-Ulam Theorem, Proposition \ref{Proposition 14} implies that the 
set
\[\mathfrak G:=\bigl\{(T,x)\in \hbox{\emph{$\bmh$}}\times\h\,;\,x\ \textrm{is a distributionally 
irregular vector for}\ T\bigr\}\] is comeager in $(\bmh, \sote)\times\h$; and this, in turn, implies that the set
\[
\bigl\{T\in\bmh\,;\,\forall^{*}x\in\h\ \textrm{is distributionally 
irregular for}\ T\bigr\}
\]
is comeager in $(\bmh,\sote)$. Here, ``$\forall^* x\in\h$" means ``for quasi-all $x\in\h$ in the Baire category sense".  This shows that 
$\textrm{DDCH}_M(\h)$ is comeager 
in $\bmh$. 
\par\medskip
For the proof of Proposition \ref{Proposition 14}, we will need the following lemma.
\begin{lemma}\label{Lemma 15}
 Let $T$ be an ergodic operator on $\h$, and let $m$ be a 
$T$-invariant  
ergodic measure with full support for $T$. For every $\varepsilon $, 
$R>0$,  one can find two other ergodic measures $\mu $ and $\nu $ for $T$, both with full support, 
 such that $\mu (\ba{B}(0,R))>1-\varepsilon$ and 
 $\nu (\ba{B}(0,R))<\varepsilon$ respectively.
\end{lemma}

\begin{proof}
The proof relies on a dilation argument already used several times in
\cite{GM}. This argument has been essentially given at the end of the proof of 
Theorem \ref{Theorem 9}, but we repeat it anyway.
For any $\rho >0$, let $m_{\rho }$ be the measure on $\h$
defined by $m_{\rho }(C)=m(\rho C)$ for any Borel subset $C$ of $\h$.
We have $m_{\rho }(\ba{B}(0,R))=m(\ba{B}(0,\rho R))$. All these measures 
$m_{\rho }$ are ergodic for $T$ and have full support. Moreover, $m(\ba{B}(0,\rho R))\to m(\h)=1$ as $\rho\to\infty$ 
and $m(\ba{B}(0,\rho R))\to m(\{0\})=0$ as $\rho\to 0$ (that $m(\{0\})$ is necessarily equal to $0$ follows from the 
ergodicity of $m$ with respect to $T$ and the fact that $m\neq\delta_{0}$). So there exist $\rho _{1}$, 
$\rho
_{2}>0$ such that $\mu =m_{\rho _{1}}$ and $\nu =m_{\rho _{2}}$ satisfy 
 $\mu (\ba{B}(0,R))>1-\varepsilon $ and $\nu (\ba{B}
(0,R))<\varepsilon $. 
\end{proof}

\begin{proof}[Proof of Proposition \ref{Proposition 14}] Let us consider the following two subsets of $\bmh$:
\[\displaylines{
\mathfrak G^x_\infty:= \bigl\{T\in\bmh\,;\,\forall\,R>0\ :
\underline{\textrm{dens}}\;\mathcal{N}_{T}(x,\ba{B}(0,R))=0 \bigr\},\quad{\rm and}\cr 
\mathfrak G_0^x:=\bigl\{T\in\bmh\,;\,\forall\,r>0\ :\;
\underline{\textrm{dens}}\;\mathcal{N}_{T}(x,\h\setminus\ba{B}(0,r))=0 
\bigr\}.}\] 
It is not difficult to see that 
\[ \mathfrak G^x_\infty\cap \mathfrak G^x_0\subseteq \mathfrak G^x.\]
Indeed, if $T\in \mathfrak G_\infty^x$, then one can find a set $A\subseteq\N$ with $\overline{\rm dens}(A)=1$ such that 
$\Vert T^ix\Vert\to \infty$ as $i\to\infty$ along $A$; whereas if $T\in \mathfrak G_0^x$, 
one can find a set $B\subseteq\N$ with $\overline{\rm dens}(B)=1$ such that 
$\Vert T^ix\Vert\to 0$ as $i\to\infty$ along $B$.
So it is enough to show that $\mathfrak G^x_\infty$ and $ \mathfrak G^x_0$ are both comeager in $(\bmh, \texttt{SOT}^*)$.  
We will actually concentrate on $\mathfrak G^x_\infty$ only, the proof for $\mathfrak G^x_0$ being completely similar.
\par\smallskip
 For any $\varepsilon $, $R>0$, we introduce the set
 \[
 \mathfrak H^{x}_{\varepsilon ,R}=\bigcap_{\gk}\ \bigcup_{n\ge k}\
 \Bigl\{T\in\bmh\,;\,\#\{1\le i\le n\,;\,\|T^{i}x\|< R\}<
 n\varepsilon \Bigr\}.
\]
Reasoning as in the proof of Fact \ref{FactUn}, we observe that
each set  $ \mathfrak H^{x}_{\varepsilon ,R}$ is  a $\gd$ subset of $(\bmh,\sot)$. We also have:

\begin{fact}\label{Hx2} 
Each set $ \mathfrak H^{x}_{\varepsilon ,R}$ is dense in 
$(\bmh,\sote)$.
\end{fact}

\begin{proof}[Proof of Fact \ref{Hx2}]
We will use the density of 
$\textrm{G-MIX}_{M}(\h)$ in $(\bmh,\sote)$, proved in Corollary 
\ref{Proposition 4} 
above. 

Let 
$\mathfrak{U}$ be a non-empty open subset of $(\bmh,\sote)$. By Corollary 
\ref{Proposition 4}, we know that $\textrm{G-MIX}_{M'}(\h)$ is dense in $(\mathfrak B_{M'}(\h),\sote)$ for every 
$1<M'<M$. Since $\bigcup_{M'<M} \mathfrak B_{M'}(\h)$ is obviously dense in $\bmh$, it follows that 
$ \mathfrak U$ contains an operator $T$ which is mixing in the Gaussian sense and satisfies $\|T\|<M$. 
The operator $T$ is in particular ergodic. By Lemma \ref{Lemma 15}, $T$ 
admits an ergodic measure with full support $\nu $ such that 
$\nu (\ba{B}(0,2R))<\varepsilon$. Birkhoff's ergodic theorem then 
implies that 
the set 
\[
\mathcal{E}:=\Bigl\{y\in\h\,;\,\limsup_{n\to \infty 
}\,\dfrac{1}{n}\#\{1\le i\le 
n\,;\,\|T^{i}y\|\le 2R\}<\varepsilon
\Bigr\}
\]
is dense in $\h$. 
\par\smallskip
Let now $\delta >0$ be a small positive number, to be 
fixed later on in the proof. Since $x\neq 0$, the density of $\mathcal{E}$ in 
$\h$ implies the existence of an isomorphism $L$ of $\h$ with the 
following 
properties: $Lx\in\mathcal{E},\ 
\|I-L\|<\delta,\ \textrm{and}\  \|I-L^{-1}\|<\delta  $. Consider now the 
operator $S=
L^{-1}TL$. Since $\|T\|<M$, we have $\|S\|<M$ if 
$\delta $ is sufficiently small. Also, $S$ belongs to $ \mathfrak U$ as 
soon as $\delta $ is sufficiently small. We thus fix $\delta >0$ 
such that these two conditions are satisfied. It now remains to 
prove that $S$ belongs to $ \mathfrak H^{x}_{\varepsilon ,R}$, which will 
conclude our proof that 
$ \mathfrak H^{x}_{\varepsilon ,R}$ is dense in $(\bmh,\sote)$.
\par\smallskip 
For every $\gn$, we have $LS^{n}x=T^{n}Lx$. Since $Lx$ belongs to $\mathcal{E}$, it follows that
\[
\limsup_{n\to\infty }\,\dfrac{1}{n}\#\bigl\{1\le i\le n\,;\,\|LS^{i}
x\|\le 2R\bigr\}<\varepsilon.
\]
Observe now that if $\|S^{i}x\|\le R$, then $\|LS^{i}x\|\le \|L\|\,\|S^{i}x\|
\le 2R$ (as soon as $\delta \le 2$, of course). It follows that 
\[
\#\bigl\{1\le i\le n\,;\,\|S^{i}x\|\le R\bigr\}\le
\#\bigl\{1\le i\le n\,;\,\|LS^{i}x\|\le 2R\bigr\}
<n\varepsilon 
\]
for all sufficiently large $n$. Hence $S$ belongs to 
$ \mathfrak H^{x}_{\varepsilon ,R}$, which concludes the proof of Fact~\ref{Hx2}.
\end{proof}

The two facts above imply that all the sets $ \mathfrak H^x_{\varepsilon, R}$ are comeager (in fact, dense $G_\delta$) in $(\bmh,\sote)$. Since
\[
 \mathfrak G^{x}_\infty=\bigcap\limits_{\varepsilon \in\Q_{+}^{*}} 
\bigcap\limits_{R\in\Q_{+}^{*}} \mathfrak H^{x}_{\varepsilon ,R},
\]
it follows that $ \mathfrak G^x_\infty$
is comeager as well. The case of $ \mathfrak G^x_0$ being exactly similar, the proof of 
Proposition \ref{Proposition 14} is now complete.
\end{proof}
\par\smallskip

\subsection{Summary}
Let us summarize the results obtained so far: for any $M>1$, an $\sote$-typical 
hypercyclic operator $T$
\begin{enumerate}
 \item[-] is topologically weakly mixing but not topologically mixing;
 \item[-] has empty point spectrum, and hence is not chaotic;
 \item[-] has no non-trivial invariant measure, hence is not 
 $\mathcal{U}$-frequently hypercyclic, and \textit{a fortiori} not ergodic;
 \item[-] but is densely distributionally chaotic and hence satisfies $c(T)=1$.
\end{enumerate}
\par
We shall see in Section \ref{UPPERTRIANG} below that the picture changes 
drastically when we consider $\sot$-typical elements of some natural classes of 
upper-triangular operators with respect to a given orthonormal basis of $\h$.
\par\smallskip

\section{Descriptive set-theoretic issues}\label{COMPLEXITY} 

In what follows, we fix $M>1$. We have seen in the 
previous section that $\cmh$, $\textrm{TMIX}_M(\h)$ and $\textrm{UFHC}_M(\h)$ 
are meager in $(\bmh, \texttt{SOT}^*)$. In this section, we are going to show that these classes of operators are also \emph{Borel} 
in $\bmh$ with respect to $\texttt{SOT}$ and $\texttt{SOT}^*$,  and we will discuss their exact descriptive complexity in some details.
Moreover, we will show that some natural classes 
of operators defined by dynamical properties are \emph{non-Borel} in $\bmh$. 

\smallskip
Recall the standard notations for Borel classes: 
$\mathbf\Sigma_1^0=\hbox{open}$, $\mathbf\Pi_1^0=\hbox{closed}$, $\mathbf\Sigma_2^0=F_\sigma$, 
$\mathbf\Pi_2^0=G_\delta$ and so on. We refer the reader to \cite{Ke}
for more information on the Borel hierarchy.
\par\smallskip

\subsection{Complexity of the families $\textrm{TMIX}_M(\h)$, $\cmh$,  $\textrm{UFHC}_M(\h)\cap\textrm{CH}_M(\h)$ and $\textrm{UFHC}_M(\h)$} 
The following fact will allow us to concentrate mainly on the topology $\texttt{SOT}^*$. 

\begin{lemma}\label{Baire1} The identity map \emph{$id:(\bmh,\texttt{SOT})\to (\bmh, \texttt{SOT}^*)$} is \emph{Baire $1$}; in other words, any \emph{$\texttt{SOT}^*$}-open subset of $\bmh$ 
is \emph{$\texttt{SOT}$}-$\mathbf\Sigma_2^0$. Therefore, for every countable ordinal $\xi$, any \emph{$\texttt{SOT}^*$}-$\mathbf\Sigma_\xi^0$ subset of $\bmh$ is \emph{$\texttt{SOT}$}-$\mathbf\Sigma_{\xi+1}^0$ and any 
\emph{$\texttt{SOT}^*$}-$\mathbf\Pi_\xi^0$ set is \emph{$\texttt{SOT}$}-$\mathbf\Pi_{\xi+1}^0$.
\end{lemma}

\begin{proof} Since $(\bmh, \texttt{SOT}^*)$ is second-countable, it is enough to show that any basic $\texttt{SOT}^*$-open set is $\texttt{SOT}$-$\mathbf\Sigma_2^0$. Therefore, we just have to check that if $x, a\in\h$ and 
$\varepsilon >0$, then the set $\mathfrak U:=\{ T\in\bmh;\; \Vert T^{*}x-a\Vert<\varepsilon\}$ is $\texttt{SOT}$-$F_\sigma$. But this is clear since 
\[T\in\mathfrak U\iff \exists k\in\N \left( \forall h\in\h,\ \|h\|\le 1\;:\; \vert\pss{x}{Th}-\pss{a}{h}\vert\leq \varepsilon-\frac1k\right)
\]
and the condition under brackets is $\texttt{SOT}$-closed.
\end{proof}
\par\smallskip

\subsubsection{Complexity of \emph{$\textrm{TMIX}_M(\h)$}} The complexity of $\textrm{TMIX}_M(\h)$ is given by the following proposition:

\begin{proposition}\label{Prop en plus} The set
\emph{$\textrm{TMIX}_M(\h)$} is a $\mathbf \Pi_4^0$ subset of \emph{$(\bmh, \texttt{SOT})$}, and a 
``true" $\mathbf\Pi_3^0$ subset of \emph{$(\bmh, \texttt{SOT}^*)$}, \mbox{\it i.e.} a $\mathbf\Pi_3^0$ set 
which is not $\mathbf\Sigma_3^0$.
\end{proposition}

\begin{proof}  Let us first show that $\textrm{TMIX}_M(\h)$ is $\mathbf\Pi_3^0$ with respect to $\texttt{SOT}^*$, and hence (by Lemma \ref{Baire1}) $\mathbf\Pi_4^0$ with respect to $\texttt{SOT}$. Let 
 $(B_p)_{p\geq 1}$ be a countable family of closed balls of $\h$ whose interiors form a basis of open sets for $\h$.  Then an operator $T\in\bmh$ 
is topologically mixing if and only if
\[
\forall p,q\geq 1\;\exists N\in\N\;\forall n\geq N\;:\; T^n(B_p)\cap B_q\neq\emptyset.
\]
For each fixed data $(p,q,n)$, the condition ``$T^n(B_p)\cap B_q\neq\emptyset$" is 
$\texttt{SOT}^*$-closed, by weak compactness of $B_p$, weak closedness of $B_q$ 
and continuity of the map $(T,x)\mapsto T^nx$ from $(\bmh,\texttt{SOT}^*)\times (B_p,w)$ into $(\h,w)$. This shows that $\textrm{TMIX}_M(\h)$
 is $\mathbf\Pi_3^0$ in $(\bmh,\texttt{SOT}^*)$.
\par\smallskip
 In order to show that $\textrm{TMIX}_M(\h)$ is a true $\mathbf\Pi_3^0$ subset of
 $(\bmh,\sote)$, we assume that $\h=\ell^2(\N)$ and we use weighted backward shifts on $\h$. 
 It is well-known that a weighted backward shift $B_{\om}$ on $\h$ is topologically mixing if and only if the weight sequence $\om=(\omega_k)_{k\geq 1}$ satisfies
 \[\lim_{n\to\infty} \,\vert \omega_1\cdots \omega_n\vert=\infty.\]
Let us denote by $\mathcal S$ the set of all sequences of positive integers $s=(s_k)_{k\geq 1}$ such that $s_{k+1}\leq Ms_k$ for all $k\geq 1$. This is a closed 
subset of the Polish space $\N^\N$, and hence a Polish space as well. We need the following fact.

\begin{fact}\label{c_0}
The set $\mathcal S_\infty:=\{ s\in\mathcal S;\; s_k\to\infty \;{\rm as}\;k\to\infty\}$ is a true $\mathbf\Pi_3^0$ set in $\mathcal S$.
\end{fact}

\begin{proof}[Proof of Fact \ref{c_0}] It is known (see \mbox{e.g.} \cite[Section 23.A]{Ke}) that 
$\mathcal N_\infty:=\{ \alpha\in\N^\N;\;\alpha_k\to\infty\}$ is a true $\mathbf\Pi_3^0$ set in $\N^\N$. So we just need to find 
a continuous map $\Phi:\N^\N\to \mathcal S$ such that $\Phi^{-1}(\mathcal S_\infty)=\mathcal N_\infty$. 
In other words, our goal is to associate to each $\alpha\in\N^\N$ another sequence $s\in\N^\N$ in such
 a way that $s_k\to\infty$ exactly when $\alpha_k\to\infty$ and, additionally,  $s_{k+1}\leq Ms_k$ for all $k\geq 1$; and this needs to be done in a continuous fashion. For any $\alpha\in\N^{\N}$,
 we define $s=s(\alpha)$ as follows: $s_1=\alpha_1$ and, for every $k\ge 1$, 
 \[s_{k+1}=\left\{ \begin{matrix}
 \alpha_{k+1}&\hbox{if $\alpha_{k+1}\leq Ms_{k}$}\\
 Ms_{k}&\hbox{if $\alpha_{k+1}> Ms_{k}$}.
 \end{matrix}
 \right.
 \]
 It is obvious that $s_{k+1}\leq Ms_k$ for all $k\ge 1$ and that the map $\alpha\mapsto s(\alpha)$ is continuous. Moreover, since $M>1$, it is straightforward to check 
 that $\alpha_k\to\infty$ if and only if $s_k\to\infty$.
 \end{proof}
 
Going back to the proof of Proposition \ref{Prop en plus},
we associate to each $s\in\mathcal S$ a weight sequence $\om(s)$ defined as follows:
 \[\omega_1(s)=1\qquad{\rm and}\qquad \omega_{k+1}(s)=\frac{s_{k+1}}{s_k}\quad \textrm{ for every } k\ge 1.
 \]
 Since $s\in\mathcal S$, we have $0<\omega_k(s)\leq M$ for all $k\geq 1$, so that the weighted shift $B_{\om(s)}$ on $\h$ satisfies $\Vert B_{\om(s)}\Vert\leq M$. 
 Moreover,  the map $s\mapsto \om(s)$ is clearly continuous from $\mathcal S$ into $\R^\N$, and hence the map $s\mapsto B_{\om(s)}$ is continuous from $\mathcal S$ 
 into $(\bmh, \texttt{SOT}^*)$. Finally, since $\omega_1(s)\cdots \omega_k(s)=s_{k}/s_1$ for all $k\ge 1$, the shift $B_{\om(s)}$ is topologically mixing if and 
 only if $s_k\to\infty$. We have thus constructed a continuous map $\Phi:\mathcal S\to \bmh$ such that $\Phi^{-1}(\textrm{TMIX}_M(\h))=\mathcal S_{\infty}$, which proves that 
 $\textrm{TMIX}_M(\h)$ is a true $\mathbf\Pi_3^0$ subset of $(\bmh,\sote)$ by Fact \ref{c_0}.
  \end{proof}
\par\smallskip

\subsubsection{Complexity of \emph{$\cmh$}} We now consider the class of chaotic \ops\ on $\h$.
  
\begin{proposition}\label{Proposition 8} The set
\emph{$\cmh$} is a $\mathbf\Pi_4^0$ subset of \emph{$(\bmh,\sot)$}, and a true $\mathbf\Pi_3^0$ subset of \emph{$(\bmh,\texttt{SOT}^*)$}.
\end{proposition} 

\begin{proof}  Let us first show that $\cmh$ is $\mathbf\Pi_3^0$ with respect to $\texttt{SOT}^*$, and hence $\mathbf\Pi_4^0$ with respect to $\texttt{SOT}$. 
Let $(B_p)_{p\geq 1}$ be a countable family of closed balls whose interiors form a basis of open sets for $\h$. 
By definition, an operator $T\in\bmh$ is chaotic if and only if it is hypercyclic and each ball $B_p$ contains a periodic point for $T$. In other words:
\[
T\in\cmh\iff T\in\hcmh\quad{\rm and}\quad \forall p\in\N\;\exists N\in\N\; : \;\Bigl(\exists x\in B_p\;:\; T^Nx=x\Bigr).
\]
For each fixed pair $(p,N)$, the condition under brackets is $\texttt{SOT}^*$-closed by continuity of the map $(T,x)\mapsto T^Nx$ and weak compactness 
of the ball $B_p$. Therefore, the second half of the condition on the right hand side of the above display defines a $\mathbf\Pi_3^0$ set; and since $\hcmh$ is 
$G_\delta$, it follows that $\cmh$ is $\mathbf\Pi_3^0$ in $(\bmh,\sote)$.
\par\smallskip 
The proof that $\cmh$ is a true $\mathbf\Pi_3^0$ set is a little bit more involved. We will use the so-called 
\emph{Kalisch operators}, introduced by Kalisch in \cite{Ka}, which display interesting dynamical properties (see for instance \cite[Section 5.5.3]{BM}). They are defined as follows. Let $T:L^2(0,2\pi)\to L^2(0,2\pi)$ be the operator defined for every $f\in L^{2}(0,2\pi)$ by 
\[Tf(\theta)=e^{i\theta}f(\theta)-\int_0^\theta ie^{it}f(t)\, dt, \quad \theta\in(0,2\pi).
\]
A simple computation shows that for any $\lambda=e^{i\alpha}\in \T\setminus\{ 1\}$, the function $f_\lambda:=\mathbf 1_{(\alpha,2\pi)}$ is an eigenvector of $T$ 
associated to the eigenvalue $\lambda$, and that $\ker(T-\lambda)={\rm span}\,[f_\lambda]$.
Note also that the map $\lambda\mapsto f_{\lambda}$ is continuous from $\T\setminus\{1\}$ into $L^{2}(0,2\pi)$.
In particular, if $\Lambda\neq\emptyset$ is any compact subset of $\T\setminus\{1\}$, 
the closed subspace $\mathcal H_\Lambda$ of $L^2(0,2\pi)$ spanned by the 
functions $f_{\lambda}, \lambda\in\Lambda$, is $T$-invariant. The Kalisch operator associated to $\Lambda$ is the operator 
$T_\Lambda :\h_\Lambda\to\h_\Lambda$ induced by $T$ on $\h_\Lambda$. These operators $T_\Lambda$ have the following properties.

\begin{itemize}
\item[$\bullet$] The spectrum of $T_\Lambda$ is exactly equal to $\Lambda$.
\item[$\bullet$] The \op\ $T_\Lambda$ is hypercyclic if and only  if $\Lambda$ is a {perfect} set, in which case $T_\Lambda$ is actually ergodic 
in the Gaussian sense.
\end{itemize}
Recall that a \emph{perfect set} is a non-empty 
compact set without isolated points.
\par\smallskip
It is easy to deduce from these properties the following characterization of compact subsets $\Lambda$ of $\T\setminus\{1\}$ such that $T_{\Lambda}$ is chaotic. We denote by $\Omega$ the subset of $\T$ consisting of all roots of unity.

\begin{fact}\label{Kalisch} Let $\Lambda$ be a perfect subset of $\T\setminus\{ 1\}$. Then the operator $T_\Lambda$ is chaotic if and only if $\Omega\cap\Lambda$ 
is dense in $\Lambda$.
\end{fact}

\begin{proof}[Proof of Fact \ref{Kalisch}] Since the map $\lambda\mapsto f_\lambda$ is continuous and $f_\lambda$ is a periodic point 
of $T_\Lambda$ if $\lambda\in\Omega\cap \Lambda$, it is clear that ${\rm Per}(T_\Lambda)$ is dense in $\h_\Lambda$ if 
$\Omega\cap\Lambda$ is dense in $\Lambda$.
Since $\Lambda$ is assumed to be perfect, $T_{\Lambda}$ is \hy, and hence chaotic.
 Conversely, assume that $\Omega\cap \Lambda$ is not dense in $\Lambda$ and choose 
$\lambda\in\Lambda\setminus \Lambda_0$, 
where $\Lambda_0=\overline{\Omega\cap\Lambda}$. Then $f_\lambda$ is not an eigenvector of $T_{\Lambda_0}$ since $\sigma(T_{\Lambda_0})=\Lambda_0$. 
Since however $f_\lambda$ is an eigenvector of $T$, this means that $f_\lambda$ does not belong to $\h_{\Lambda_0}=\overline{\rm span}\, [f_\xi;\; \xi\in \Lambda_0]$, \mbox{\it i.e.} that
$f_\lambda$ does not belong to $\overline{\rm span}\, [f_\xi;\; \xi\in \Omega\cap \Lambda]$. But since $\sigma(T_\Lambda)=\Lambda$ and $\ker(T-\xi)={\rm span}\,[f_\xi]$ 
for every $\xi\in\T\setminus\{ 1\}$, we have 
\[
{\rm span}\, [f_\xi;\; \xi\in \Omega\cap \Lambda]={\rm span}\,\Bigl[\bigcup_{\xi\in\Omega\cap\Lambda} \ker(T_\Lambda-\xi)\Bigr]={\rm Per}(T_\Lambda).
\]
So 
${\rm Per}(T_\Lambda)$ is not dense in $\h_\Lambda$ and $T_\Lambda$ is not chaotic.
\end{proof}

Here is now a purely descriptive set-theoretic fact, which is certainly well-known but for which we were unable to locate a reference. 
Here and afterwards, given a compact metric space $E$, we denote by $\mathcal K(E)$ the space 
of all non-empty compact subsets of $E$ endowed with its usual topology, and by $\mathcal K_{\rm perf}(E)$ the set of all perfect subsets of $E$. Recall that $\mathcal K(E)$ is compact metrizable, 
and that $\mathcal K_{\rm perf}(E)$ is a $G_\delta$ subset of $\mathcal K(E)$. So $\mathcal K_{\rm perf}(E)$ is a Polish space. See \cite{Ke} or \cite{Tod} for more details.

\begin{fact}\label{denseinlambda} Let $E\subseteq\T$ be a perfect set such that $\Omega\cap E$ is dense in $E$. Then the set 
$\mathcal W:=\{ \Lambda\in \mathcal K_{\rm perf}(E);\; \hbox{$\Omega\cap\Lambda$ is dense in $\Lambda$}\}$ is a true $\mathbf\Pi_3^0$ set in $\mathcal K_{\rm perf}(E)$.
\end{fact}

\begin{proof}[Proof of Fact \ref{denseinlambda}] Since $\mathcal K_{\rm perf}(E)$ is $G_\delta$ in $\mathcal K(E)$, it is in fact enough to show that $\mathcal W$ is a true $\mathbf\Pi_3^0$ set in $\mathcal K(E)$.
That $\mathcal W$ is $\mathbf\Pi_3^0$ is easy to check. In order to show that it is a true $\mathbf\Pi_3^0$ set, we proceed as follows.
\par\smallskip
Let us denote by $\mathbf C$ the Cantor space $\{ 0,1\}^\N$, and by $\mathbf Q$ the set of all ``rationals" of $\mathbf C$, \mbox{\it i.e.} 
$\mathbf Q=\{ \alpha=(\alpha(i))_{i\geq 1}\in\mathbf C;\; \hbox{$\alpha(i)$ is eventually $0$}\}$. It is known (see \cite[Section 23.A]{Ke}) that the set 
\[ W:=\{ \bar\alpha=(\alpha_n)_{n\geq 1}\in\mathbf C^\N;\; \forall n\ge 1\;:\;\alpha_n\in\mathbf Q\}
\]
is a true $\mathbf\Pi_3^0$ set in $\mathbf C^\N$. So it is enough to find a continuous map $\Phi:\mathbf C^\N\to \mathcal K(E)$ such that 
$\Phi^{-1}(\mathcal W)=W$.
\par\smallskip
We first note that if 
$M\subseteq E$ is a perfect set such that $\Omega\cap M$ is dense in $M$, then $\mathcal W\cap \mathcal K(M)$ is not $G_\delta$ 
in $\mathcal K(M)$. Indeed, on the one hand $\mathcal W\cap\mathcal K(M)$ is easily seen to be dense in $\mathcal K(M)$ because 
$M$ is perfect and $\Omega\cap M$ is dense in $M$; and on the other hand $\mathcal W\cap\mathcal K(M)$ is meager in $\mathcal K(M)$ because it is disjoint from the $G_\delta$ set 
$\mathcal K(M\setminus\Omega)$, which is dense in $\mathcal K(M)$ because $M$ is perfect and $\Omega$ is countable. Hence $\mathcal W\cap\mathcal K(M)$ cannot 
be $G_\delta$ in $\mathcal K(M)$ by the Baire Category Theorem. 
By Wadge's Lemma (see \cite[Th. 21.14]{Ke}) applied to the two Borel sets
$\mathcal{W}\cap\mathcal K(M)$ and $\mathbf C\setminus \mathbf Q$, it follows that for any perfect set $M$ as above, $\mathcal W\cap \mathcal K(M)$ 
is ``$\mathbf\Sigma_2^0$-hard", \mbox{\it i.e.} one can find a continuous map 
$\varphi :\mathbf C\to \mathcal K(M)$ such that $\varphi^{-1}\bigl(\mathcal W\cap\mathcal K(M)\bigr)=\mathbf Q$.
\par\smallskip
Now let us choose a sequence $(M_n)_{n\geq 1}$ of pairwise disjoint perfect subsets of $E$ such that $\Omega\cap M_n$ is dense in $M_n$ for every $n\ge 1$ and the sets $M_n$ 
accumulate to some point $a\in E$, which means that every neighborhood of $a$ contains all but finitely many of the sets $M_n$. This is possible because $E$ is perfect and $\Omega\cap E$ 
is dense in $E$. For each $n\geq 1$, let $\varphi_n:\mathbf C\to \mathcal K(M_n)$ be a continuous map such that $\varphi_n^{-1}\bigl(\mathcal W\cap\mathcal K(M_n)\bigr)=\mathbf Q$. One can then define a map $\Phi:\mathbf C^\N\to\mathcal K(E)$ by setting
\[\Phi(\bar\alpha):=\{ a\}\cup \bigcup_{n=1}^\infty \varphi_n(\alpha_n)\qquad\hbox{for every $\bar\alpha=(\alpha_n)\in\mathbf C^\N$}.
\]
Since the sets $M_n$ accumulate to $a$, it is clear that each $\Phi(\bar\alpha)$ is indeed a compact subset of $\T$ and that the map $\Phi$ is continuous. Moreover, it is equally clear 
that $\Phi(\bar\alpha)$ is perfect if and only if all the sets $\varphi_n(\alpha_n)$ are perfect, and that $\Omega\cap\Phi(\bar\alpha)$ is dense in $\Phi(\bar\alpha)$ if and only if $\Omega\cap \varphi_n(\alpha_n)$ is dense in 
$\varphi_n(\alpha_n)$ for every $n\geq 1$. 
Hence $\Phi(\bar\alpha)$ belongs to $\mathcal W$ if and only if $\alpha_n$ belongs to $\varphi^{-1}_n\bigl(\mathcal W\cap\mathcal K(M_n)\bigr)=\mathbf Q$ for all $n\ge 1$. We have thus proved that
\begin{align*}\Phi^{-1}(\mathcal W)
=&\bigl\{ \bar\alpha;\; \forall n\ge 1\;:\; \alpha_n\in\mathbf Q\bigr\}=W,
\end{align*}
which concludes the proof of Fact \ref{denseinlambda}.
\end{proof}

We need yet one more fact, which is a simple and certainly well-known consequence of the Michael Selection Theorem. The version of this theorem which we use here runs as follows (see \cite[Section III.19]{Tod}).

\smallskip
\emph{Let $X$ be a zero-dimensional compact space, and let $Y$ be a complete metric space. Let $\Phi:X\rightarrow\mathcal{F}(Y)$ be a lower semi-continuous map, where $\mathcal{F}(Y)$ denotes the set of all non-empty closed subsets of $Y$. Then $\Phi$ admits a \emph{continuous selection}, \mbox{\it i.e.} there exists a continuous map $f: X\rightarrow Y$ such that $f(x)$ belongs to $ \Phi(x)$ for every $x\in X$.}

\begin{fact}\label{selection} Let $E$ be a zero-dimensional compact metric space.
There exists a sequence $(\xi_{n})_{n\ge 1}$ of continuous maps from $\mathcal K_{\rm perf}(E)$ into $ E$ such that for each $\Lambda\in\mathcal K_{\rm perf}(E)$, the points $\xi_n(\Lambda)$, $n\ge 1$, belong to $\Lambda$, are pairwise distinct,  
 and the countable set $\{ \xi_n(\Lambda);\; n\geq 1\}$ is dense in $\Lambda$. 
\end{fact}
 
\begin{proof}[Proof of Fact \ref{selection}] Since $E$ is zero-dimensional, one can choose a countable basis $(V_i)_{i\ge 1}$ 
for the topology of $E$ consisting of (non-empty) clopen sets. For each $\Lambda\in\mathcal K_{\rm perf} (E)$, let us denote 
by $i_1(\Lambda)<i_2(\Lambda)<\dots $ the integers $i$ such that $\Lambda\cap V_i\neq\emptyset$ (since $\Lambda$ is perfect, there are infinitely many such integers $i$). 
As the $V_i$ are clopen sets, it is not hard to see that the functions $i_n$, $n\ge 1$, are 
locally constant. In particular, the map $\Phi :\Lambda\mapsto \Lambda\cap V_{i_1(\Lambda)}$  from $\mathcal K_{\rm perf} (E)$ into $\mathcal K(E)$ is continuous. Since $E$ is zero-dimensional, $\mathcal K_{\rm perf}(E)$ is zero-dimensional as well, and 
one can apply the version of Michael's Selection Theorem quoted above to the map $\Phi$. This yields a continuous map $\Lambda\mapsto \xi_1(\Lambda)$ from $\mathcal K_{\rm perf}(E)$ into $E$ such that 
$\xi_1(\Lambda)$ belongs to $ \Lambda\cap V_{i_1(\Lambda)}$ for every $\Lambda\in\mathcal K_{\rm perf}(E)$. Now,  any {perfect} set $\Lambda\in \mathcal K_{\rm perf} (E)$ 
satisfies $\Lambda\cap(V_{i_2(\Lambda)}\setminus\{ \xi_1(\Lambda)\})\neq\emptyset$; so 
there exists an integer $j$ such that $\xi_1(\Lambda)\in V_{j}$ and $V_{i_2(\Lambda)}\setminus V_{j}\neq\emptyset$. Moreover, if we denote by $j_2(\Lambda)$ the smallest such integer $j$ with these two properties, the map 
$\Lambda\mapsto j_2(\Lambda)$ is locally constant, and hence the map $\Lambda\mapsto \Lambda\cap \bigl(V_{i_2(\Lambda)}\setminus V_{j_2(\Lambda)})$ is continuous from $\mathcal K_{\rm perf}(E)$ 
into $\mathcal K(E)$. Applying Michael's Selection Theorem a second time, we obtain a continuous map $\xi_2:\mathcal K_{\rm perf}(E)\to E$ such that 
$\xi_2(\Lambda)\in \Lambda\cap \bigl(V_{i_2(\Lambda)}\setminus V_{j_2(\Lambda)})$ for every $\Lambda\in \mathcal K_{\rm perf}(E)$; in particular, 
$\xi_2(\Lambda)\in \Lambda\cap V_{i_2(\Lambda)}$ and $\xi_2(\Lambda)\neq\xi_1(\Lambda)$. Continuing in this fashion, we obtain the required sequence $(\xi_n)_{n\ge 1}$.
\end{proof}

We are now in position to prove that $\cmh$ is a true $\mathbf\Pi_3^0$ set in $(\bmh,\sote)$. 
Let us first choose a compact set $E\subseteq \T\setminus\{ 1\}$ with the following properties: $E$ is perfect and zero-dimensional, $\Omega\cap E$ is dense in $E$ and, moreover, $\Vert T_E\Vert\leq M$. 
Such a compact set does exist. Indeed, the definition of the operator $T$ shows that the norm of the restriction of $T$ to the subspace 
$L^2(u,2\pi)\subseteq L^2(0,2\pi)$ tends to $1$ as $u\to 2\pi$. Therefore, if $\Lambda_0\subseteq\T\setminus\{ 1\}$ is a non-trivial closed arc 
sufficiently close to $1$ then $\Vert T_{\Lambda_0}\Vert\leq M$; so it is enough to take as $E$ any perfect set with empty interior contained in $\Lambda_0$
 and such that $\Omega\cap E$ is dense in $E$.
\par\smallskip
Having fixed $E$ in this way, let $(\xi_n)_{n\ge 1}$ be a
sequence of continuous maps given by Fact~\ref{selection}, selecting a dense sequence of pairwise distinct points in each perfect set $\Lambda\subseteq E$. 
Note that if $\Lambda\in\mathcal K_{\rm perf}(E)$, 
the points $\xi_n(\Lambda)$ are pairwise distinct and form a dense subset of $\Lambda$, so that
the functions $f_{\xi_n(\Lambda)}$ are linearly independent and span a dense subspace of $\h_{\Lambda}$. Applying 
the Gram-Schmidt orthonormalization process to the sequence $(f_{\xi_n(\Lambda)})_{n\ge 1}$, we obtain an orthonormal basis $(e_n(\Lambda))_{n\geq 1}$ of $\h_\Lambda$ 
which depends continuously on $\Lambda$, that is, each map $\Lambda\mapsto e_n(\Lambda)$
is continuous from $\mathcal K_{\rm perf}(E)$ into $L^2(0,2\pi)$.
\par\smallskip
Let us now fix an orthonormal basis $(e_n)_{n\geq 1}$ of $\h$. For each $\Lambda\in\mathcal K_{\rm perf}(E)$, 
denote by $U_\Lambda:\h_\Lambda\to\h$ the unitary operator defined by setting $U_{\Lambda}e_n(\Lambda)=e_n$ for every $n\geq 1$. Since $\Vert T_E\Vert\leq M$, one can define a map 
$\Phi:\mathcal K_{\rm perf}(E)\to\bmh$ by setting 
\[\Phi(\Lambda):= U_\Lambda T_\Lambda U_{\lambda}^{-1}\qquad\hbox{for every $\Lambda\in\mathcal K_{\rm perf}(E)$}.
\]
Since $\Phi(\Lambda)$ and $T_\Lambda$ are unitarily isomorphic,  $\Phi(\Lambda)$ is chaotic if and only if $T_\Lambda$ is, which holds true 
exactly when $\Omega\cap\Lambda$ is dense in $\Lambda$. Thus, we have 
\[\Phi^{-1}(\cmh)=\mathcal W.
\]
Fact \ref{denseinlambda} will allow us to conclude the proof, provided that we are able to show that the map $\Phi$  is continuous from $\mathcal K_{\rm perf}(E)$ into $(\bmh,\texttt{SOT}^*)$. 
This relies on the following observation. For every $\Lambda\in\mathcal K_{\rm perf}(E)$, let us denote by $P_\Lambda:L^2(0,2\pi)\to\h_\Lambda$ the orthogonal 
projection of $L^2(0,2\pi)$ onto $\h_\Lambda$, and by $J_\Lambda:\h_\Lambda\to L^2(0,2\pi)$ the canonical embedding of $\h_{\Lambda}$ into $L^2(0,2\pi)$. In other words, $J_\Lambda=P_\Lambda^*$.

\begin{claim}\label{undetrop} For any $f\in L^2(0,2\pi)$, the map $\Lambda\mapsto U_\Lambda P_\Lambda f$ is continuous from $\mathcal K_{\rm perf}(E)$ into $\h$; and for every 
$x\in\h$, the map $\Lambda\mapsto J_\Lambda U_\Lambda^{-1}x$ is continuous  from $\mathcal K_{\rm perf}(E)$ into $L^2(0,2\pi)$.
\end{claim}

\begin{proof}[Proof of Claim \ref{undetrop}] Fix $f\in L^2(0,2\pi)$.
Since $P_\Lambda f=\sum\limits_{n=1}^\infty \pss{f}{e_n(\Lambda)}\, e_n(\Lambda)$, we have 
\[U_\Lambda P_\Lambda f=\sum_{n=1}^\infty \pss{f}{e_n(\Lambda)}\, e_n.
\] The maps $\Lambda\mapsto e_n(\Lambda)$, $n\ge 1$, being continuous, it suffices to show that the convergence of the above series is uniform (with respect to $\Lambda$) on compact 
subsets of $\mathcal K_{\rm perf}(E)$ in order to derive the continuity of the map $\Lambda\mapsto U_{\Lambda}P_{\Lambda}f$. Now, we have for any $N\geq 1$
\[ R_N(\Lambda):=\Bigl\Vert \sum_{n>N} \pss{f}{e_n(\Lambda)}\, e_n\Bigr\Vert^2=\sum_{n>N} \vert \pss{f}{e_n(\Lambda)}\vert^2=\Bigr\Vert f-\sum_{n=1}^N \pss{f}{e_n(\Lambda)}\, e_n(\Lambda)\Bigr\Vert^2.
\]
In particular, the maps $\Lambda\mapsto R_N(\Lambda)$ are continuous. Since the sequence $(R_N)_{N\ge 1}$ is non-increasing, it converges uniformly to $0$ on 
compact subsets of $\mathcal K_{\rm perf}(E)$ by Dini's Theorem. This proves the first part of Claim \ref{undetrop}. The proof of the second part is exactly similar.
\end{proof}

Let us now prove that $\Phi$ is continuous from $\mathcal K_{\rm perf}(E)$ into $(\bmh,\texttt{SOT}^*)$, which amounts to showing that for any $x\in \h$, the maps $\Lambda\mapsto \Phi(\Lambda)x$ and 
$\Lambda\mapsto \Phi(\Lambda)^*x$ are continuous. Since $T_\Lambda=P_\Lambda TJ_\Lambda$, we have 
\[\Phi(\Lambda)x=U_\Lambda P_\Lambda\, T \,J_\Lambda U_\Lambda^{-1}x\qquad {\rm and} \qquad \Phi(\Lambda)^*x= U_\Lambda P_\Lambda\, T^*\,J_\Lambda U_\Lambda^{-1}x.
\]
By Claim \ref{undetrop}, the map $\Lambda\mapsto f_\Lambda:=TJ_\Lambda U_\Lambda^{-1}x$ is continuous from $\mathcal K_{\rm perf}(E)$ into $L^2(0,2\pi)$. By Claim \ref{undetrop} again, and 
since  the map $\Lambda\mapsto U_\Lambda P_\Lambda$ takes values in a bounded subset of $\mathfrak B(L^2(0,2\pi),\h)$,  the map 
$\Lambda\mapsto\Phi(\Lambda)x=U_\Lambda P_\Lambda f_\Lambda$ 
is continuous. Likewise, the map $\Lambda\mapsto\Phi(\Lambda)^*x$ is continuous. The equality $\Phi^{-1}(\cmh)=\mathcal W$, the continuity of $\Phi$ and Fact \ref{denseinlambda} now imply that $\cmh)$ is a true $\mathbf\Pi_3^0$ set in  {$(\bmh,\texttt{SOT}^*)$}.
\end{proof}
\par\smallskip

\subsubsection{Complexity of \emph{$\textrm{UFHC}_M(\h)$} and \emph{$\textrm{UFHC}_{M}(\h)\cap
\textrm{CH}_{M}(\h)$}}\label{Ufhc}
In this subsection, we study the descriptive complexity of the sets $\textrm{UFHC}_{M}(\h)$ and $\textrm{UFHC}_{M}(\h)\cap
\textrm{CH}_{M}(\h)$ in $(\bmh,\sote)$, for any $M>1$. Although we have been unable to determine 
the exact complexity of these sets, we obtain some simple upper bounds and some rather non-trivial lower bounds. Our first result reads as follows.

\begin{proposition}\label{Proposition 1004}
 For any $M>1$, the set \emph{$\textrm{UFHC}_{M}(\h)$} is $\mathbf\Pi _{4}^{0}$ in \emph{$(\bmh,\sot)$}, hence also in 
\emph{$(\bmh,\sote)$}.
\end{proposition}

\begin{proof}
 The proof of Proposition \ref{Proposition 1004} relies on the following
observation, first made in \cite{BR} and then
developed in \cite{BGE}. 

\begin{fact}\label{Fact 9 ter} Let $(V_{q})_{q\ge 1}$ be a countable basis of (non-empty) open subsets of $\h$, and let $T$ be a bounded \op\ on $\h$.
\begin{enumerate}
\item[(1)] Assume that $T$ belongs to $\ufhch$, and let $x_{0}\in H$ be a $\mathcal{U}$-frequently 
hypercyclic vector for $T$. For each $q\ge 1$, denote by $2\delta _{q}$ the upper density of the set 
$\mathcal{N}_{T}(x_{0}, V_q)$. For any integers $q,N\geq 1$, define
 \[G_{q,N}:=\Bigl\{ x\in\h;\; \exists n\geq N\;:\;\#\{1\le i\le n\,;\,T^{i}x\in V_q\}\geq n\delta _{q}\Bigr\}.\]
Then 
\[G:=\bigcap_{q\geq 1}\bigcap_{N\geq 1} G_{q,N}\]
is a dense $\gd$ subset of $\h$ which consists entirely of $\mathcal{U}$-frequently 
hypercyclic vectors for $T$.

\item[(2)] Conversely, assume that there exists a sequence of positive numbers $(\delta_q)_{q\geq 1}$ 
such that for any $q,N\geq 1$, the open set $G_{q,N}$ defined by the formula above is dense in $\h$. 
Then $T$ is $\mathcal U$-frequently hypercyclic and admits a dense $G_{\delta}$ set of $\mathcal{U}$-frequently \hy\ vectors $x$ which satisfy 
\[\overline{\textrm{dens}}\;\mathcal{N}_{T}(x, V_q)\ge \delta_{q} \quad\textrm{ for every } q\ge 1.
 \]
\end{enumerate}
\end{fact}

 Let 
us denote by $\Q_+^*$ the set of all positive rational numbers. 
It follows immediately from Fact \ref{Fact 9 ter} that an operator $T\in\bmh$ belongs to $\textrm{UFHC}_M(\h)$ if and only if
\[
\forall q\geq 1\;\exists\delta\in\Q_+^*\;\forall N\geq 1\;\forall p\geq 1\;\exists x\in V_p\;\exists n\geq N\;:\;  
\#\{1\le i\le n\,;\,T^{i}x\in V_q\}\geq n\delta.
\]
In order to prove that {$\textrm{UFHC}_{M}(\h)$} is $\mathbf\Pi ^{0}_{4}$ in {$(\bmh,\sot)$},
it is enough to check that each set 
\[
\Bigl\{T\in\bmh\,;\,\#\{1\le i\le n\,;\,T^{i}x\in 
V_q\}\geq n\delta \Bigr\}
\]
is $\sot$-open in $\bmh$. This is easy: if $T_{0}$ belong to this set, 
one can find $k\geq n\delta$ distinct integers $1\le i_1,\dots i_k\leq n$ such that $T_{0}^{i_{s}}x$ belongs to $ V_q$ for every $s=1,\dots ,k$. If 
$T\in\bmh$ is sufficiently close to $T_{0}$ for the $\sot$-topology, it 
still satisfies $T^{i_{s}}x\in V_q$ for every $1\le s\le k$, and thus 
belongs to our set. Hence {$\textrm{UFHC}_{M}(\h)$} is $\mathbf\Pi ^{0}_{4}$ in {$(\bmh,\sot)$}, and thus in {$(\bmh,\sote)$} as well.
\end{proof}
\par\smallskip
Our next result deals with the complexity of the class $\textrm{UFHC}_{M}(\h)\cap
\textrm{CH}_{M}(\h)$.

\begin{proposition}\label{Proposition 1005}
 The set \emph{$\textrm{UFHC}_{M}(\h)\cap
\textrm{CH}_{M}(\h)$} is a difference of $\mathbf\Sigma _{3}^{0}$ sets in 
\emph{$(\bmh,\sote)$}, \mbox{\it i.e.} it can be written as $A\setminus B$, where $A$ and $B$ are both \emph{$\texttt{SOT}^{*}$-$\mathbf\Sigma_3^0$} 
sets. Besides, \emph{$\textrm{UFHC}_{M}(\h)\cap
\textrm{CH}_{M}(\h)$} is neither \emph{$\texttt{SOT}^{*}$-$\mathbf\Sigma_3^0$}, 
 nor \emph{$\texttt{SOT}^{*}$-$\mathbf\Pi_3^0$} in $\bmh$. More precisely, there is no \emph{$\texttt{SOT}^{*}$-$\mathbf\Sigma_3^0$} set $A$ such that 
 \emph{$\textrm{G-MIX}_{M}(\h)\cap
\textrm{CH}_{M}(\h)\subseteq A\subseteq \textrm{UFHC}_{M}(\h)\cup
\textrm{CH}_{M}(\h)$}, and no \emph{$\texttt{SOT}^{*}$-$\mathbf\Pi_3^0$} set $B$ such that \emph{$\textrm{FHC}_{M}(\h)\cap
\textrm{CH}_{M}(\h)\subseteq B\subseteq \textrm{UFHC}_M(\h)$}.
\end{proposition}

As a straightforward consequence of Proposition \ref{Proposition 1005}, we obtain:

\begin{corollary}\label{Proposition 9.4}
The set \emph{$\textrm{UFHC}_{M}(\h)$} is neither \emph{$\mathbf\Sigma_3^0$}, 
 nor \emph{$\mathbf\Pi_3^0$} in \emph{$(\bmh,\sote)$}.
\end{corollary}
\begin{proof} By Proposition \ref{Proposition 1005}, $\textrm{UFHC}_{M}(\h)$ is not $\mathbf\Sigma_3^0$, and it is not $\mathbf\Pi_3^0$ because $\textrm{CH}_M(\h)$ is $\mathbf\Pi_3^0$ and 
$\textrm{UFHC}_{M}(\h)\cap
\textrm{CH}_{M}(\h)$ is not.
\end{proof}

Another immediate consequence is
\begin{corollary} The set \emph{$\textrm{TMIX}_M(\h)\cap\textrm{CH}_M(\h)$} is a true $\mathbf\Pi_3^0$ set in \emph{$(\bmh,\sote)$}.
\end{corollary}
\begin{proof} By Propositions \ref{Prop en plus} and \ref{Proposition 8}, $\textrm{TMIX}_M(\h)\cap\textrm{CH}_M(\h)$ is $\mathbf\Pi_3^0$; and by Proposition~\ref{Proposition 1005}, it is not $\mathbf\Sigma_3^0$.
\end{proof}

\begin{proof}[Proof of Proposition \ref{Proposition 1005}]
 The proof of Proposition \ref{Proposition 1005} relies on the following fact, which 
characterizes in a rather surprising way the $\mathcal{U}$-frequently hypercyclic operators 
within the class of chaotic operators. (We will come back to this at the end of Section \ref{CRITERIA}).

\begin{fact}\label{Theorem 1000}
Let $T$ be a chaotic operator on $\h$. Then, $T$ is $\mathcal U$-frequently hypercyclic if and only if $c(T)>0$.
 \end{fact}
\begin{proof}[Proof of Fact \ref{Theorem 1000}] This follows from Theorem \ref{10000}, to be proved below.
\end{proof}

We now start the proof of Proposition \ref{Proposition 1005}, which we divide into 
three parts.

\medskip\noindent
\textbf{Part 1.} \emph{\emph{$\textrm{UFCH}_{M}(\h)\cap \textrm{CH}_{M}(\h)$} is a difference of \emph{$\texttt{SOT}^{*}$-$\mathbf\Sigma_3^0$} 
subsets of $\bmh$.}

 \smallskip
 By Fact \ref{Theorem 1000}, we have 
$\textrm{UFCH}_{M}(\h)\cap \textrm{CH}_{M}(\h)=\textrm{c}^+_{M}(\h)\cap \textrm{CH}_{M}(\h)$, where 
\[
 \textrm{c}^+(\h)=\{T\in\hch\,;\,c(T)>0\}.
\]
Hence, since $\textrm{CH}_{M}(\h)$ is $\texttt{SOT}^{*}$-$\mathbf\Pi_3^0$  by Proposition \ref{Proposition 8}, it suffices to show that $\textrm{c}^+_{M}(\h)$ is 
$\texttt{SOT}^{*}$-$\mathbf\Sigma_3^0$. Let 
$(V_{q})_{q\ge 1}$ be a countable basis of non-empty open subsets of $\h$, and observe that, by the definition of $c(T)$, an operator $T\in
\hcmh$ belongs to $\textrm{c}^+(\h)$ if and only if 
\[
\exists\,\varepsilon \in\Q^{+*}\ \forall q\ge 1\ \forall N\ge 1\
\exists x\in V_{q}\ \exists n\geq N :\ 
\#\{1\le i\le n\,;\,\Vert T^{i}x\Vert<1\}\ge n\varepsilon .\]
This condition is easily seen to define an $\texttt{SOT}^{*}$-$\mathbf\Sigma_3^0$ subset of $\bmh$; and since $\hcmh$ is $G_\delta$ in $\bmh$, it follows that $\textrm{c}^+_{M}(\h)$ is 
$\texttt{SOT}^{*}$-$\mathbf\Sigma_3^0$.

\medskip\noindent
\textbf{Part 2.} \emph{There is no \emph{$\texttt{SOT}^{*}$-$\mathbf\Sigma_3^0$} subset $A$ of $\bmh$ such that 
 \emph{$\textrm{G-MIX}_{M}(\h)\cap
\textrm{CH}_{M}(\h)\subseteq A\subseteq \textrm{UFHC}_{M}(\h)\cup
\textrm{CH}_{M}(\h)$}.}

\smallskip
In order to prove this, we need some preliminary facts. Let us choose a sequence 
$(\h_n)_{n\geq 1}$ of infinite-dimensional closed subspaces of $\h$ such that
$\h$ can be decomposed as the orthogonal direct sum of the spaces $\h_{n}$, \mbox{\it i.e.} $\h=\bigoplus_{n\geq 1}\h_{n}$.

\begin{fact}\label{esperons} Let $T\in\bh$ have the form $T=\bigoplus_{n\geq 1} T_n$ with respect to the decomposition $\h=\bigoplus_{n\geq 1}\h_{n}$, where $T_n\in\mathfrak B(\h_n)$ for each $n\ge 1$.

\begin{itemize}
\item[\rm (i)] If all   operators $T_n$ are chaotic and mixing in the Gaussian sense, then so is $T$.
\item[\rm (ii)] If $T$ is $\mathcal U$-frequently hypercyclic or chaotic, then so are all  operators $T_n$.
\end{itemize}
\end{fact}

\begin{proof}[Proof of Fact \ref{esperons}] (i) Assume that all  operators $T_n$ are chaotic and mixing in the Gaussian sense, and fix for each $n\ge 1$ a mixing Gaussian 
measure with full support $\mu_n$ for $T_n$. Each second-order moment 
$\int_{\h_n}\Vert x_n\Vert^2 d\mu((x_{n}))$ is finite; and by rescaling (\mbox{\it i.e.} replacing each measure $\mu_n$ 
by a \mea\ $\tilde
{\mu}_{n}$ defined by setting $\tilde
{\mu}_{n}(B)=\mu_{n}(\varepsilon_{n}^{-1}B)$ for every Borel subset $B$ of $\h_{n}$, for a  suitably small $\varepsilon_n>0$), we may assume without loss of generality that  the series
$\sum_{n\geq 1} \int_{\h_n} \Vert x_n \Vert^2d\mu_n(x_n)$ is convergent. 
\par\smallskip
Let us denote by $\pmb\mu$ the infinite product measure $\bigotimes_{n\geq 1} \mu_n$ on the product space $\prod_{n\geq 1} \h_n$. Then 
\[ \int_{\prod_{n\geq 1} \h_n} \Bigl(\sum_{n=1}^\infty \Vert x_n\Vert^2\Bigr)\, d\pmb\mu((x_{n}))=\sum_{n=1}^\infty \int_{\h_n} \Vert x_n\Vert^2d\mu_n(x_n)<\infty,
\]
from which it follows that $\pmb\mu$ is in fact concentrated on the set 
\[\mathbf H:=\Bigl\{ (x_n)_{n\ge 1}\in\prod_{n\geq 1} \h_n;\; \sum_{n=1}^\infty \Vert x_n\Vert^2<\infty\Bigr\}.
\] 
We may therefore define a probability measure $\mu$ on $\h$ as follows: for any Borel subset $C$ of $ \h$,
\[
\mu(C)=\pmb\mu\Bigl(\Bigl\{ (x_n)_{n\ge 1}\in\mathbf H;\; \sum_{n=1}^\infty x_n\in C\Bigr\}\Bigr).
\]
This measure $\mu$ is obviously Gaussian, and it is $T$-invariant because each \mea\ $\mu_n$ is $T_n$-invariant. Indeed, for any  bounded Borel function $f:\h\to\R$, we have
\[\int_\h (f\circ T)\, d\mu=\int_{\mathbf H} f\Bigl(\sum_{n=1}^\infty T_nx_n\Bigr) \,d\pmb\mu((x_{n}))=\int_{\mathbf H} f\Bigl(\sum_{n=1}^\infty  x_n\Bigr) \,d\pmb\mu((x_{n}))=\int_\h f\, d\mu.
\]
The measure $\mu$ has full support. Let indeed $V$ be a non-empty open set in $\h$. Choose a point $a\in V$, and $\varepsilon >0$ such that $B(a,2\varepsilon)\subseteq V$. 
We may assume that $a$ belongs to $ \h_1\oplus\cdots\oplus \h_N$ for some $N\ge 1$, where $N$ is so large that
\begin{equation}\label{eqenplus}
\int_{\mathbf H^N} \sum_{n>N} \Vert x_n\Vert^2\, d\pmb\mu^N((x_{n}))<\varepsilon^{2},
\end{equation}
where $\mathbf H^N:=\Bigl\{ (x_n)_{n\ge N}\in\prod_{n\geq N} \h_n;\; \sum_{n=1}^\infty \Vert x_n\Vert^2<\infty\Bigr\}$ and $\pmb\mu^N:=\bigotimes_{n>N}\mu_n$.
Then the set $\{ (x_n)_{n\ge 1}\in\mathbf H;\; \sum_{n\geq 1} x_n\in V\}$ contains 
\[
\mathbf A:=\Bigl\{ (x_n)_{n\ge 1}\in\mathbf H;\; x_1+\cdots +x_N\in B(a,\varepsilon)\quad{\rm and}\quad \sum_{n>N} \Vert x_n\Vert^2<\varepsilon^2\Bigr\}.
\]
Moreover, if we set  $\mathbf A^N:=\bigl\{ (x_n)_{n>N};\; \sum_{n>N} \Vert x_n\Vert<\varepsilon^2\}$, then $\pmb\mu^N(\mathbf A^N)>0$ by (\ref{eqenplus}).
Also, setting $\pmb\mu_N:=\bigotimes_{n\leq N}\mu_n$ and $\mathbf A_N:=\{ (x_1,\dots ,x_N);\; x_1+\cdots +x_N\in B(a,\varepsilon)\}$, the fact that
 the measures 
$\mu_1,\dots ,\mu_N$ have full support implies that $\pmb\mu_N(\mathbf A_N)>0$. Hence
$\mu(V)\geq \mu(A)=\pmb\mu_N(\mathbf A_N)\,  \pmb\mu^N(\mathbf A^N)>0.
$
\par\smallskip
Finally, $T$ is mixing with respect to $\mu$ because each \op\ $T_n$, $n\ge 1$, is mixing with respect to $\mu_n$: this can be shown easily by checking that 
$\ \mu(A\cap T^{-k}(B))\to \mu(A)\mu(B)
$ as $k\to\infty$ 
for any Borel sets $A,B\subseteq\h$  whose definition depends on finitely many coordinates only with respect to the decomposition $\h=\bigoplus_{n\geq 1}\h_n$.

\smallskip Thus, we have shown that $T$ is mixing in the Gaussian sense. Finally, since all  operators $T_n$ are chaotic, it is easily checked that the periodic points of $T$ are dense in $\h$; hence $T$ is chaotic.

\par\smallskip
(ii) Fix $n\geq 1$. If we denote by $\pi_n:\h\to\h_n$ the canonical projection of $\h$ onto $\h_{n}$, then $T_n\pi_n=\pi_n T$. Since $\pi_n$ is continuous and onto (that it has dense range would suffice for our argument), (ii) 
follows easily: if $x\in\h$ is a $\mathcal U$-frequently hypercyclic vector for $T$,  then $x_n:=\pi_nx$ is a $\mathcal U$-frequently hypercyclic vector for $T_n$; and if $T$ is chaotic, then $\pi_n(\textrm{HC}(T))\subseteq \textrm{HC}(T_n)$ and $\pi_n({\rm Per}(T))$ is a dense subset of $\h_n$ consisting of periodic vectors for $T_n$.
\end{proof}

\par\smallskip 
Recall that we denote by $\mathbf Q$ the set of all ``rationals" of the Cantor space $\mathbf C$,
\[\mathbf Q=\{ \alpha=(\alpha(i))_{i\geq 1}\in\mathbf{C};\; \alpha(i)=0\;\hbox{for all but finitely many $i$}\}.
\]
\begin{fact}\label{I+B} 
There exists a continuous map $\alpha\mapsto T_\alpha$ from $\mathbf C$ into $(\bmh,\sote)$ such that the following holds true: 

\begin{itemize}
\item[-] if $\alpha\in\mathbf Q$, then $T_\alpha$ 
is chaotic and mixing in the Gaussian sense, and hence chaotic and $\mathcal U$-frequently hypercyclic; 
\item[-] if $\alpha\not\in\mathbf Q$, then the spectrum of $T_\alpha$ is reduced to the point $\{ 1\}$, and hence $T_\alpha$ is not chaotic and not $\mathcal U$-frequently hypercyclic.
\end{itemize}
\end{fact}

\begin{proof}[Proof of Fact \ref{I+B}] The very last part of the second assertion follows from the fact that the spectrum of a chaotic or $\mathcal U$-frequently hypercyclic operator has no isolated points (\cite{S2}).
\par\smallskip
We assume that $\h=\ell^{\, 2}(\N)$, endowed with its canonical basis. For any unilateral weight sequence $\om=(\omega_j)_{j\geq 1}$, let us denote as usual by $B_{\om}$ 
the associated weighted shift acting on $\h$.
If $\inf_{j\ge 1} \vert\omega_j\vert>0$, 
the operator $T_{\om}=I+B_{\om}$ is chaotic and mixing in the Gaussian sense 
since it admits a spanning holomorphic eigenvector field defined in a neighborhood of the point $1$  (see Fact \ref{holo}). On the other hand, if 
$\omega_j\to 0$ as $j\to\infty$, then $\sigma(B_{\om})=\{ 0\}$, so that
$\sigma(T_{\om})=\{ 1\}$.
So, setting $c:= M-1>0$, it is enough to find a continuous map $\alpha\mapsto \om(\alpha)$ from $\mathbf C$ into $(0,c]^\N$ such that 
$\inf_{j\ge 1} \omega_j(\alpha)>0$ if $\alpha\in\mathbf Q$ while $\omega_j(\alpha)\to 0$ as $j\to\infty$ if $\alpha\not\in\mathbf Q$. Once this is done, the map $\alpha\mapsto T_\alpha:=T_{\om(\alpha)}$ will enjoy the two properties stated in Fact \ref{I+B}. Now, it is quite 
easy to define such a map $\alpha\mapsto\om(\alpha)$: just set, for every $j\ge 1$,
\[\omega_j(\alpha):=\frac{c}{1+n_j(\alpha)},\quad \hbox{where}\quad n_j(\alpha)=\#\{ i\leq j;\; \alpha(i)=1\}. \]
Fact \ref{I+B} is thus proved.
\end{proof}

It is now easy to deduce from Facts \ref{esperons} and \ref{I+B} that no subset $A$ of $\bmh$ 
with the property that $\hbox{G-MIX}_M(\h)\cap\textrm{CH}_{M}(\h)\subseteq A\subseteq \textrm{UFHC}_M(\h)\cup\textrm{CH}_M(\h)$ can be {$\texttt{SOT}^{*}$-$\mathbf\Sigma_3^0$}. Let us fix such a set $A$. 
Decompose $\h$ as $\mathcal H=\bigoplus_{n\ge 1}\mathcal H_n$, where the spaces $\h_n$, $n\ge 1$, are infinite-dimensional. 
Let us define a map $\Phi:\mathbf C^\N\to\bmh$ as follows: for every $\bar\alpha=(\alpha_n)_{n\geq 1}\in \mathbf C^\N$, 
\[ \Phi(\bar\alpha):=\bigoplus_{n\geq 1} T_{\alpha_n},
\]
where the \ops\ $T_{\alpha_n}\in\mathfrak B_M(\h_n)$, $n\ge 1$, are given by Fact \ref{I+B}. Then $\Phi$ is continuous, and Facts \ref{esperons} and \ref{I+B} imply that
\[\Phi^{-1}(A)=W=\{ \bar\alpha=(\alpha_n)_{n\geq 1}\in\mathbf C^\N;\; \forall n\ge 1\;:\;\alpha_n\in\mathbf Q\}.
\]
This concludes the proof since $W$ is a true $\mathbf\Pi_3^0$ subset of $\mathbf{C}^{\N}$, hence not a $\mathbf\Sigma_3^0$ set.

\medskip\noindent
\textbf{Part 3.} \emph{There is no \emph{$\texttt{SOT}^{*}$-$\mathbf\Pi_3^0$} set $B$ such that \emph{$\textrm{FHC}_{M}(\h)\cap
\textrm{CH}_{M}(\h)\subseteq B\subseteq \textrm{UFHC}_{M}(\h)$}.}

\smallskip
In order to prove this, we will make use of the machinery developed in Section \ref{SPECIAL} below. The reader may therefore prefer to skip the proof and return to 
it after reading Section \ref{SPECIAL}.

 In what follows, 
we assume that $\h=\ell_2(\N)$, endowed with its canonical basis.
Let  $M_{0}$ be such that $1<M_{0}<M$, and choose an even integer $C\geq 3$ large enough to have 
\begin{equation}\label{stupide1} M_{0}+ \sum_{k=1}^\infty 2^{k-1} M_{0}^{-\frac12\, C^k}\leq M
\quad\textrm{ and} \quad
\sum_{k=1}^\infty 2^k M_{0}^{-\frac1{12}C^k} C^k\leq 1.
\end{equation}

Now, let us denote by $\mathcal D$ the set of all infinite sequence of integers $\delta=(\delta^{(k)})_{k\geq 0}$ with $\delta^{(0)}=1$ such that $C\leq \delta^{(k)}\leq C^{4k}$ and $\delta^{(k)}\geq C\, \delta^{(k-1)}$ for 
every $k\geq 1$. This is a closed subset of $\N^\N$ and hence a Polish space. 
For any element $\delta$ of $\mathcal D$, we denote by $T_{\delta}$ the operator of \cpt\ on $\h$ defined as follows: for any $k\geq 1$ 
\[\Delta^{(k)}=C^{4k},\quad v^{(k)}=M_{0}^{-\frac12\delta^{(k)}}\quad{\rm and}\quad 
w_i^{(k)}=\left\{ \begin{matrix}
M_{0}&{\rm if}&1\leq i\leq \delta^{(k)},\\
1&{\rm if}& \delta^{(k)}<i<\Delta^{(k)}.
\end{matrix}\right.
\]

Note that in the terminology of Section \ref{SPECIAL}, $T_\delta$ looks exactly like an \emph{operator of \cput\ } with $\tau^{(k)}=\frac12\delta^{(k)}$, except that in the definition of the weights $v^{(k)}$ 
and $w_i^{(k)}$ the constant $2$ has been replaced by $M_{0}$.
(See Subsection \ref{Subsection 4.2} for the definition of \cpt\ operators, and Subsection \ref{C+1type} for the definition of \cput\ operators.)
Observe that if $\delta=(\delta^{(k)})_{k\ge 0}$ belongs to $\mathcal D$, then $\delta^{(k)}\geq C^k$ for all $k\geq 1$.  By (\ref{stupide1}), it follows that $\Vert T_{\delta}\Vert\leq M$ for every $\delta\in\mathcal D$. Moreover, it is 
clear that the map $\delta\mapsto T_{\delta}$ is continuous from $\mathcal D$ into $(\bmh, \texttt{SOT}^*)$.
\par\smallskip
The key point of the proof that $\textrm{UFHC}_{M}(\h)\cap\textrm{CH}_{M}(\h)$ is not
$\texttt{SOT}^{*}$-$\mathbf\Pi_3^0$ is the following fact, which is actually nothing but a reformulation of Theorem \ref{Theorem 53} in a slightly different setting.

\begin{fact}\label{arnaque} Let
$\delta=(\delta^{(k)})_{k\ge 0}$ belong to $\mathcal D$.
Then the operator $T_{\delta}$ is always chaotic, and $T_\delta$ is $\mathcal U$-frequently hypercyclic if and only if $\delta^{(k)}/C^{4k}$ does not tend to $0$ as $k\to\infty$; in which case 
$T_\delta$ is in fact frequently hypercyclic.
\end{fact}

\begin{proof}[Proof of Fact \ref{arnaque}] We apply a modified version 
of Theorem \ref{Theorem 53}, with $p=2$. 
Accordingly, we set for every $k\geq 1$
\[\gamma_k:=M_{0}^{\delta^{(k-1)}-\frac12\delta^{(k)}}\sqrt{\Delta^{(k)}}=M_{0}^{\delta^{(k-1)}-
\frac12\delta^{(k)}} C^{2k}.
\]
Since $\delta^{(k)}\leq C^{-1}\delta^{(k+1)}\leq\frac13 \delta^{(k+1)}$, we have $\delta^{(k)}-\frac12 \delta^{(k+1)}\leq-\frac12\delta^{(k)}\le \delta^{(k-1)}-\frac12 \delta^{(k)}$, from which it follows easily that 
the sequence $(\gamma_k)_{k\ge 1}$ is non-increasing. Moreover, we have $\gamma_k\leq M_{0}^{-\frac16 C^k}C^{2k}$ for all 
$k\geq 1$. Hence condition (\ref{stupide1}) above yields that
\[ \sum_{k=1}^\infty 2^k \gamma_k^{1/2}\leq 1.
\]
The analogues of the ``additional assumptions" in Theorem \ref{Theorem 53} are thus satisfied, and it follows that $T_{\delta}$  is $\mathcal U$-frequently hypercyclic on $\h$ if and only if $\delta^{(k)}/C^{4k}$ does not tend to $0$ as $k\to\infty$.
\end{proof}

Now, let $B\subseteq\bmh$ have the property that $\textrm{FHC}_{M}(\h)\cap
\textrm{CH}_{M}(\h)\subseteq B\subseteq \textrm{UFHC}_{M}(\h)$. Consider the map $\Phi:\mathcal D\to \bmh$  defined by setting $\Phi(\delta):=T_\delta$ for every $\delta\in\mathcal{D}$. As already mentioned, the map $\Phi$ is continuous from $\mathcal{D}$ into $(\bmh,\sote)$. By Fact \ref{arnaque} we have
\[\Phi^{-1}(B)=\mathcal D\setminus\mathcal D_0,\quad\hbox{where}\quad \mathcal D_0=\{ \delta\in\mathcal D;\; C^{-4k}\delta^{(k)}\to 0\;{\rm as}\;k\to\infty\}.
\]
Now, a proof quite similar to that of Fact \ref{c_0} shows that $\mathcal D_0$ is a true $\mathbf\Pi_3^0$ set in $\mathcal D$, so that $\mathcal{D}\setminus\mathcal{D}_{0}$ is a true $\mathbf\Sigma_3^0$ set in $\mathcal D$. Hence $B$ cannot be $\mathbf\Pi_3^0$ in $(\bmh,\sote)$.
\end{proof}

\begin{remark}
Propositions \ref{Proposition 8} and \ref{Proposition 1005} together formally yield the statement that, 
since  $\textrm{UFHC}_M(\h)\cap\textrm{CH}_{M}(\h)$ is not $\mathbf\Pi_3^0$  while $\textrm{CH}_{M}(\h)$ is $\mathbf\Pi_3^0$ in $(\bmh,\sote)$, the class $\textrm{CH}_{M}(\h)\setminus \textrm{UFHC}_M(\h)$ is not void, \mbox{\it i.e.} there exist chaotic \ops\ in $\bmh$ which are not $\mathcal{U}$-frequently hypercyclic. However, there is nothing magic here: the proof of Proposition~\ref{Proposition 1005} relies heavily on a construction, carried out in Section \ref{SPECIAL}, of explicit chaotic operators which are not $\mathcal{U}$-frequently hypercyclic.
\end{remark}

\begin{remark} The proof of Proposition \ref{Proposition 1005} has established that the set $\textrm{c}^+_M(\h)=\{ T\in\hcmh;\; c(T)>0\}$ is a true {$\texttt{SOT}^{*}$-$\mathbf\Sigma_3^0$} subset of $\bmh$.
\end{remark}

\begin{remark}
An ``alternative" proof of the Borelness of $\textrm{UFHC}_M(\h)$ could run as follows. According to Fact 
\ref{Fact 9 ter}, an operator $T$ belongs to $\textrm{UFHC}(\h)$ if and only 
if 
quasi-all vectors $x\in H$ (in the Baire category sense) are $\mathcal{U}$-frequently hypercyclic for 
$T$. So, for $T\in\bmh$, we may write
\[ T\in \textrm{UFHC}_M(\h)\iff\forall^*x\in\h\;\, \bigl(\hbox{$x$ is $\mathcal U$-frequently hypercyclic for $T$}  \bigr).
\]
Since the relation $B(T,x)\longleftrightarrow \hbox{($x\,$ is $\,\mathcal U$-frequently hypercyclic for $T$) }$
is Borel in the pro\-duct space $(\bmh,\texttt{SOT})\times\h$, and since the category quantifier $\forall^*$ {preserves Borelness}  
(see for instance \cite[Section 16.A]{Ke}), it follows that $\textrm{UFHC}_M(\h)$ is Borel in $(\bmh,\sot)$.
\end{remark}

\begin{remark} It is likely that $\textrm{TMIX}_M(\h)$ and $\cmh$ are both true $\mathbf\Pi_4^0$ subsets of  
$(\bmh, \texttt{SOT})$, \mbox{\it i.e.} $\mathbf\Pi_4^0$ but not $\mathbf\Sigma_4^0$, that $\textrm{UFHC}_M(\h)$ is a {true} $\mathbf\Pi_4^0$ set 
in $(\bmh, \sote)$, and that $\textrm{UFHC}_M(\h)\cap\textrm{CH}_M(\h)$ is a true difference of $\mathbf\Sigma_3^0$ sets in $(\bmh,\sote)$. However, we have been unable to prove any of these facts.
\end{remark}

\begin{remark}
In view of Proposition \ref{Proposition 1004}, it is natural to wonder 
whether the class of {frequently} hypercyclic operators $\textrm{FHC}_{M}(\h)$  is also Borel in $(\hcmh,\sot)$. Let $(V_q)_{q\geq 1}$ be a countable basis 
of open sets for $\h$. An operator $T\in\bmh$ is frequently hypercyclic if and only if
\[\exists x\in\h\;\Bigl( \forall q\geq1\ \exists\delta \in\Q_{+}^*\;\exists N\ge 1\;\forall n\geq N\;:\; 
\# \{1\le i\le n\,;\,T^{i}x\in V_q\}>n\delta\Bigr);
\]
and since the relation $R(T,x)$ defined by the expression between brackets is Borel in $\bmh\times \h$, we deduce that $\textrm{FHC}_{M}(\h)$ 
is a $\mathbf\Sigma_1^1$ (aka {analytic}) subset of 
$(\bmh,\sot)$. It is quite tempting to conjecture that it 
is in fact {non}-Borel in $(\bmh,\sot)$.
\end{remark}
\par\smallskip

\subsection{Some non-Borel sets in $\bmh$} We conclude this section by showing that some natural classes 
of operators defined by dynamical properties are non-Borel in the space $(\bmh,\sote)$, for any $M>1$. 
\par\smallskip 
If $\Gamma$ is any class of Hilbert space operators, we will say that an operator $T$ \emph{has a restriction in $\Gamma$} if there is a closed $T$-invariant 
(infinite-dimensional) subspace $\mathcal E$ of $ \h$ such that the \op\ $T_{|\mathcal E}$ induced by $T$ on $\mathcal E$ belongs to $\Gamma(\mathcal E)$. We denote by $\widehat\Gamma(\h)$ the class of all operators $T\in\bh$ admitting a restriction in $\Gamma$. 

\begin{proposition}\label{restrictions} Let $\Gamma(\h)$ be a class of operators on $\h$ such that \emph{G-ERG}$(\h)\,\subseteq\Gamma(\h)\subseteq\,$\emph{HC}$(\h)$. For any $M>1$, the set 
$\widehat\Gamma_M(\h)$ is non-Borel in \emph{$(\bmh,\texttt{SOT}^*)$}.
\end{proposition}

\begin{proof} Recall that for every compact subset $\Lambda$ of $\T\setminus\{1\}$, $T_\Lambda:\h_\Lambda\to\h_\Lambda$ denotes the Kalisch operator associated to 
 $\Lambda$ (this notation has been introduced in the proof of Proposition \ref{Proposition 8}).
For any compact subset $E$ of $\T\setminus\{ 1\}$, we denote by $\mathcal K_\infty(E)$ the set of all infinite compact subsets of $E$.

\begin{fact}\label{onbBorel} Let $E$ be a compact subset of $\T\setminus\{ 1\}$. There exists a sequence $(e_{n})_{n\ge 1}$ of Borel maps from
$\mathcal K_{\infty}(E)$ into $L^2(0,2\pi)$ such that $(e_n(\Lambda))_{n\geq 1}$ is an orthonormal basis of $\h_\Lambda$ for every $\Lambda\in\mathcal K_\infty(E)$. 
\end{fact}

\begin{proof}[Proof of Fact \ref{onbBorel}] 
By the Kuratowski-Ryll-Nardzewski Selection Theorem (see \cite[Section 12.C]{Ke}), there exists a sequence $(\xi_{n})_{n\ge 1}$
 of Borel maps $\xi_n:\mathcal K(E)\to E$ such that $\{ \xi_n(\Lambda);\; n\geq 1\}$ 
is a dense subset of $\Lambda$ for every $\Lambda\in\mathcal K(E)$. 
If $\Lambda$ belongs to $\mathcal K_{\infty}(E)$, there exists a sequence  $(\zeta_n(\Lambda))_{n\geq 1}$ of distinct elements of $\Lambda $ such that $\{ \zeta_n(\Lambda);\; n\geq 1\}$ is dense in $\Lambda$. Indeed, it suffices to set
 $\zeta_1(\Lambda)=\xi_1(\Lambda)$ and, for every $n\ge 1$,
$\zeta_{n}(\Lambda)=\xi_{k_{n}}(\Lambda)$ where $k_{n}:=\min\{j\ge n\;;\; \textrm{ for all } 1\le i<j,\; \xi_{i}(\Lambda)\not =\xi_{j}(\Lambda)\}$. Then the maps $
\Lambda\mapsto \zeta_n(\Lambda)$, $n\ge 1$, are Borel on $\mathcal K_{\infty}(E)$, and by construction the elements $\zeta_n(\Lambda)$, $n\ge 1$, are pairwise distinct and form a dense subset of $\Lambda$. Applying the Gram-Schmidt orthonormalization process 
to the sequence $(f_{\zeta_n(\Lambda)})_{n\geq 1}$, which is linearly independent, we get the required orthonormal basis $(e_n(\Lambda))_{n\ge 1}$ of $\h_{\Lambda}$.
\end{proof}

The key step in the proof of Proposition \ref{restrictions} is the following result of Waterman \cite{Wat}.

\begin{fact}\label{Waterman} If $\Lambda$ is a compact subset of $\T\setminus\{ 1\}$, the operator 
$T_\Lambda$ \emph{admits spectral synthesis}, which means that every closed $T_\Lambda$-invariant subspace of  $\h_{\Lambda}$
is spanned by the eigenvectors belonging to it. In other words, any invariant subspace $\mathcal E\subseteq\h_\Lambda$ for $T_\Lambda$ 
has the form $\mathcal E=\h_K$ for some compact set $K\subseteq\Lambda$. 
\end{fact}

We are now ready to prove Proposition \ref{restrictions}. Let us fix a non-trivial arc $E\subseteq\T\setminus\{ 1\}$ such that $\Vert T_E\Vert\leq M$. 
As in the proof of Proposition \ref{Proposition 8}, let $(e_n)_{n\geq 1}$ be an orthonormal basis of $\h$, and for any $\Lambda\in\mathcal K_\infty(E)$, 
denote by $U_\Lambda:\h_\Lambda\to\h$
the unitary operator sending $e_n(\Lambda)$ to $e_n$ for all $n\geq 1$, where $e_n(\Lambda)$ is given by Fact \ref{onbBorel}. Then define a map $\Phi:\mathcal K_\infty(E)\to\bmh$ 
by setting
\[ \Phi(\Lambda)=U_\Lambda T_\Lambda U_{\Lambda}^{-1}\quad\textrm{ for every }\Lambda\in \mathcal K_\infty(E).
\]
Proceeding as in the proof of Proposition \ref{Proposition 8}, and using the fact that the maps $\Lambda\mapsto e_{n}(\Lambda)$, $n\ge 1$, are Borel on $\mathcal K_\infty(E)$, we deduce that
the map $\Phi$ is Borel from $\mathcal K_\infty(E)$ into $(\bmh,\sote)$.
\par\smallskip
If $\Lambda\in\mathcal K_\infty(E)$ is uncountable, it admits a perfect subset $K$, so the restriction of $T_\Lambda$ to the invariant subspace $\h_K$ (which is nothing else but $T_K$) 
is ergodic in the Gaussian sense. Thus $\Phi(\Lambda)$ belongs to $ \widehat{\hbox{$G$-}ERG}(\h)\subseteq\widehat\Gamma(\h)$. On the other hand, if $\Lambda$ is countable then, by Fact \ref{Waterman}, $T_\Lambda$ 
has no hypercyclic restriction. Hence $\Phi(\Lambda)$ belongs to $\bh\setminus \widehat{HC}(\h)\subseteq \bh\setminus\widehat\Gamma(\h)$. So we have proved that
\[\Phi^{-1}\Bigl(\widehat\Gamma_M(\h)\Bigr)=\mathcal K_{\infty}(E)\setminus\mathcal K_\omega(E).
\]
Since $\mathcal K_{\omega}(E)$ is notoriously non-Borel in $\mathcal K(E)$ (see \cite[Section 27.B]{Ke}) and hence also in $\mathcal K_\infty(E)$ because the latter is Borel in $\mathcal K(E)$, this concludes the proof.
\end{proof}

\begin{remark}\label{blabla} The same proof shows that $\widehat{\hbox{G-MIX}}_M(\h)$, 
the family of all operators $T\in\bmh$ admitting a restriction which is mixing in the Gaussian sense, is non-Borel in $(\bmh,\sote)$. Instead of working with the class $\mathcal K_\omega(E)$ of countable subsets of $E$, 
one has to consider the class $U_{0}\cap\mathcal{K}(E)$ of compact subsets of $E$ which are \emph{sets of extended uniqueness}, also called \emph{$U_0$-sets}. By definition, a compact set $\Lambda\subseteq\T$ is a $U_0$-set 
if it is negligible for every probability measure $\sigma$ on $\T$ whose Fourier coefficients vanish at infinity. It follows from the main result of \cite{BM2} that 
if $\Lambda\subseteq\T\setminus\{ 1\}$ is not a $U_0$-set, then the Kalisch operator $T_\Lambda$ has a restriction which is mixing in the Gaussian sense, whereas if $\Lambda$ is a $U_0$-set,  
$T_\Lambda$ is not mixing in the Gaussian sense. Moreover, it is well-known that the class $U_0\cap \mathcal K(E)$ is non-Borel in $\mathcal K(E)$ for any non-trivial arc $E\subseteq \T$ (see \cite[Th.~VIII.2.6]{KL}). 
Using these facts, one can prove the aforementioned result in exactly the same way as above: the key identity
\[\Phi^{-1}\Bigl(\widehat\Gamma_M(\h)\Bigr)=\mathcal K_{\infty}(E)\setminus\mathcal K_\omega
\] at the end of the proof of Proposition \ref{restrictions} is replaced by
\[\Phi^{-1}\Bigl(\widehat\Gamma_M(\h)\Bigr)=\mathcal K_{\infty}(E)\setminus U_{0}.
\]
\end{remark}

\begin{remark}\label{Remark 3.21}
 Let $\Gamma_{0}(\h)$ be a class of operators of $\h$ with the following two properties:
 \begin{enumerate}
  \item[\rm (i)] $\textrm{G-ERG}(\h)\subseteq \Gamma _{0}(\h)$;
  \item[\rm (ii)] if $\Lambda $ is any countable compact subset of $\T\setminus\{1\}$, the 
Kalisch operator $T_{\Lambda }$ does not belong to $\Gamma _{0}(\h)$.
 \end{enumerate}
 
\noindent
Exactly the same proof as that of Proposition \ref{restrictions} shows that the following 
statement holds true: \emph{if $\Gamma(\h)$ is a class of operators on $\h$ such that \emph{$\textrm{G-ERG}(H)\subseteq \Gamma(\h)\subseteq \Gamma_{0}(\h)$} then, for any $M>1$, 
the set $\widehat{\Gamma}_{M}(\h)$ is non-Borel
 in 
\emph{$(\bmh,\sote)$}.}
\end{remark}

Our aim is now to prove the following result:
\begin{proposition}\label{Proposition 3.23}
 For any $M>1$, the class \emph{$\textrm{DCH}_{M}(\h)$} of distributionally chaotic operators in 
 $\bmh$ is non-Borel in 
\emph{$(\bmh,\sote)$}.
\end{proposition}

Note that this is in strong contrast with Proposition \ref{Theorem 15}, according to which 
the class of \emph{densely} distrobutionally chaotic operators in $\bmh$ is $G_\delta$.

\smallskip
The proof of Proposition \ref{Proposition 3.23} relies on the following observation:
\begin{fact}\label{Fact 3.24}
 An operator $T\in\bh$ is distributionally chaotic if and only if it has a 
restriction which is densely distributionally chaotic.
\end{fact}

\begin{proof}[Proof of Fact \ref{Fact 3.24}]
If there exists a closed subspace $\mathcal E$ of $\h$ such that $T|_{\mathcal E}$ is densely distributionally chaotic, 
$T$ admits in particular a distributionally irregular vector belonging to $\mathcal E$, and hence $T$ is distributionally chaotic. Conversely, assume that $T$ is distributionally chaotic, and 
let $x\in\h$ be a distributionally irregular vector. Let $\mathcal E$ be the closed $T$-invariant 
subspace spanned by the orbit of $x$, and denote by $S$ the operator induced by $T$ on 
$\mathcal E$. As the vector $x\in \mathcal E$ has a distributionally unbounded orbit under the action of 
$S$ (\mbox{\it i.e. } $\Vert S^{n}x\Vert \to\infty$ as 
$n\to\infty$ along a subset $B$ of $\N$, with 
$\overline{\textrm{dens}}(B)=1$), $S$  admits  by \cite[Prop. 8]{BBMP} a comeager set of vectors with 
distributionally unbounded orbit. Also, the orbit of $x$ 
under the action of $S$ is distributionally near to $0$ (\mbox{\it i.e.}
$\Vert S^{n}x\Vert \to 0$ as 
$n\to\infty$ along a subset $A$ of $\N$, with 
$\overline{\textrm{dens}}(A)=1$). It follows that 
$\Vert S^{n}y\Vert \to0$ as 
$n\to\infty$ along $A$ for every vector $y$ belonging 
to the linear span of the vectors $S^{p}x$, $p\ge 0$. Hence $S$ admits a dense set 
of vectors whose orbit is distributionally near to $0$. By \cite[Prop. 9]{BBMP}, it follows that the 
set of vectors whose orbit under the action of $S$ is distributionally near to $0$ 
is comeager in $\mathcal E$. By the Baire Category Theorem, we conclude that $S$ admits a comeager set of 
distributionally irregular vectors, \mbox{\it i.e.} $S$ is densely distributionally chaotic.
\end{proof}

We will also need the following characterizations of (densely) distributionally 
chaotic Kalisch operators:
\begin{fact}\label{Proposition 3.25}
 Let $\Lambda $ be a compact subset of $\T\setminus\{1\}$, and let $T_{\Lambda }$ be 
the associated Kalisch operator on $H_{\Lambda }$.
\begin{enumerate}
 \item[\rm (1)] The operator $T_{\Lambda } $ is densely distributionally chaotic 
if and 
only if $\Lambda $ is a perfect set.
\item[\rm (2)] The operator $T_{\Lambda }$ is distributionally chaotic if and only 
if 
$\Lambda $ is uncountable.
\end{enumerate}
\end{fact}

\begin{proof}[Proof of Fact \ref{Proposition 3.25}]
Let us start by proving assertion $(1)$. If $\Lambda $ is a perfect set then $T_{\Lambda }$ is an 
ergodic operator, so that it is in particular densely distributionally chaotic by \cite{GM}. 
Conversely, suppose that $\Lambda $ is not a perfect set, and let $\lambda _{0}$ be 
an isolated point of $\Lambda $. Then the eigenfunction $f_{\lambda _{0}}$ does not belong to 
$\h_{\Lambda\setminus\{\lambda _{0}\} }$, and $\h_{\Lambda}$ can be written as a topological 
direct sum $\h_{\Lambda }=\textrm{span}[f_{\lambda _{0}}]\oplus 
\h_{\Lambda\setminus\{\lambda _{0}\} }$. Hence, there exists a non-zero vector $g_{0}$ 
of $\h_{\Lambda }$ with $\pss{g_{0}}{f_{\lambda _{0}}}=1$ and $g\perp \h_{\Lambda\setminus\{\lambda_0\}}$, and a bounded 
projection $Q$ of $\h_{\Lambda }$ onto $\h_{\Lambda\setminus\{\lambda _{0}\} }$, such 
that $f=\pss{g_{0}}{f} f_{\lambda _{0}}+Qf$ for every $f\in\h_{\Lambda}$.
It follows that  $T^{n}_{\Lambda }f=\pss{g_{0}}{f}
\lambda _{0}^{n}f_{\lambda _{0}}+T^{n}_{\Lambda }Qf$ for every $n\ge 1$, so that $\pss{g_{0}}{T^{n}_{\Lambda }f}=\pss{g_{0}}{f}\lambda 
_{0}^{n}$. Hence $\Vert T^{n}_{\Lambda }f\Vert\ge 
\vert \pss{g_{0}}{f}\vert/\Vert g_{0}\Vert$ for every $n\ge 1$. Thus, we see that the orbit under 
 $T_{\Lambda }$ of any vector $f\in \h_{\Lambda }$ with $\pss{g_{0}}{f}\neq 0$ is bounded away 
from $0$, and hence that no such vector can be distributionally irregular. So 
$T_{\Lambda}$ is not densely distributionally chaotic.
\par\smallskip
The proof of assertion $(2)$ relies on Facts \ref{Waterman} and 
\ref{Fact 3.24} above. By Fact \ref{Fact 3.24}, $T_{\Lambda }$ is distributionally chaotic 
if and only if it admits a restriction which is densely distributionally chaotic.  Since, by Fact \ref{Waterman},
any such restriction is of the form $T_{K}\in\mathfrak{B}(\h_{K})$, where $K$ is a compact 
subset of $\Lambda $, it follows from assertion $(1)$ that $T_{\Lambda }$ is distributionally chaotic if and 
only if $\Lambda $ contains a perfect set, which happens exactly when 
$\Lambda $ is uncountable. This proves $(2)$.
\end{proof}

We are now ready for the proof of Proposition \ref{Proposition 3.23}.

\begin{proof}[Proof of Proposition \ref{Proposition 3.23}] Consider the class 
$\Gamma (\h)=\Gamma _{0}(\h):=\ddch$. This class $\Gamma _{0}(\h)$ satisfies the assumptions 
of Remark \ref{Remark 3.21}: $\textrm{G-ERG}(\h)$ is contained in $\Gamma _{0}(\h)$ by a 
result of \cite{GM}, and it follows from Fact \ref{Proposition 3.25} that whenever 
$\Lambda $ is a countable subset of $\T\setminus\{1\}$, the Kalisch operator $T_{\Lambda }$ is not 
(densely) distributionally chaotic. Since $\textrm{G-ERG}(\h)\subseteq \Gamma (\h)
\subseteq\Gamma _{0}(\h)$, we conclude that for any $M>1$, 
$\widehat{\Gamma }_{M}(\h)$ is non-Borel in $(\bmh,\sote)$. But $\widehat{\Gamma }_{M}(\h)=\textrm{DCH}_{M}(\h)$ by Fact \ref{Fact 
3.24}, and hence $\textrm{DCH}_{M}(\h)$ is non-Borel in 
$(\bmh,\sote)$.
\end{proof}

\begin{corollary}\label{Corollary 3.26}
 For any $M>1$, the set \emph{$\textrm{DCH}_{M}(\h)\setminus \textrm{DDCH}_{M}(\h)$} of operators in $\bmh$ which are distributionally chaotic but not densely distributionally chaotic is 
non-Borel in \emph{$(\bmh,\sote)$}.
\end{corollary}

\begin{proof}
Since $\textrm{DDCH}_{M}(H)$ is a $G_{\delta }$ subset of $(\hcmh,\sote)$ by Proposition 
\ref{Theorem 15}, it is in particular Borel. So one 
immediately deduces from Proposition \ref{Proposition 3.23} that 
$\textrm{DCH}_{M}(\h)\setminus \textrm{DDCH}_{M}(\h)$ is non-Borel in $(\bmh,\sote)$.
\end{proof}

\begin{remark}
 Corollary \ref{Corollary 3.26} formally yields the existence of distributionally chaotic 
operators which are not densely distributionally chaotic. But the existence of such 
operators can of course be deduced directly from Proposition \ref{Proposition 3.25}: a Kalisch operator $T_{\Lambda }$ 
is distributionally chaotic but not densely distributionally chaotic if and only if 
$\Lambda $ is uncountable but not perfect.
\end{remark}

\par\smallskip

\section{Ergodicity for upper-triangular operators}\label{UPPERTRIANG}

\subsection{Definitions and setting}\label{Subsection 3.1} Let us fix an 
orthonormal basis 
$(e_{k})_{\gk}$ of the complex separable infinite-dimensional Hilbert space $\h$. We will denote by $\h_{00}$ the 
linear subspace of $\h$ consisting of finitely supported vectors with 
respect to the basis $(e_{k})_{k\ge 1}$.
\par\smallskip
In this subsection, we concentrate ourselves on the 
study of the typical properties of operators on $\h$ which are 
\emph{upper-triangular} with respect to the basis $(e_{k})_{\gk}$, with diagonal coefficients 
\emph{of modulus $1$} and \emph{pairwise distinct}.  This class of operators will be denoted by $\toh$:
\[
 \toh:=\bigl\{T\in\bh\,;\,\forall\,k\ge 1\ :\ Te_{k}\in\textrm{span}
 [e_{1},\dots,e_{k}]\ \textrm{and}\ |\pss{Te_{k}}{e_{k}}|=1\bigr\}.
 \]
{For each $T\in\toh$ and $\gk$, we write $\lambda _{k}(T):=
 \pss{Te_{k}}{e_{k}}$. We also define}
 \[\tth:=\bigl\{T\in\toh\,;\,\forall\,j,k\ge 1\ \textrm{with}\ j\neq k,\ \lambda 
_{j}(T)\neq\lambda _{k}(T)\bigr\}
\]
and
\[\tindh:=\bigl\{T\in\tth\,;\,\textrm{the diagonal coefficients}\ \lambda 
_{k}(T)\ \textrm{are rationally independent} \bigr\}.
\]
Recall that if $(\lambda_{k})_{k\in I}$ is a finite or infinite family of elements of $\T$ with $\lambda_{k}=e^{2i\pi\theta_{k}}$ for every $k\in I$, it is said to be \emph{rationally independent} if for any finite subset $F$ of $I$, the family $(\theta_k)_{k\in F}$ consists of $\Q$-independent numbers.
\par\smallskip
For any $M>0$, we denote by 
$\tomh$, $\ttmh$, and $\tindmh$ the set of operators in $\bmh$ which 
belong to $\toh$, $\tth$, and $\tindh$ respectively. It is not difficult 
to check that these sets are $\gd$ in $(\bmh,\sot)$. So we may state:

\begin{fact}\label{Fact 17}
For any $M>0$, the sets $\tomh$, $\ttmh$, and $\tindmh$ are Polish 
spaces when endowed with the topology  $\sot$.
\end{fact}

\par\smallskip
Note that (by the Gram-Schmidt orthonormalization process) any bounded operator on $\h$ with spanning unimodular 
eigenvectors associated to distinct eigenvalues is upper triangular with 
respect to 
some orthonormal basis $(e_{k})_{\gk}$, and belongs to the associated 
class $\tth$. Hence, given the importance in linear dynamics of unimodular 
eigenvectors, it is quite natural to investigate typical properties of elements 
of the spaces $(\ttmh,\sot)$, $M>1$. Some interesting and unexpected 
phenomena occur in this setting, which we will describe shortly. 
Meanwhile, we state for future reference two elementary facts concerning the 
$\sot$-continuity of certain maps naturally associated to 
operators in $\tth$.
\par\medskip 
When $T$ belongs to $\tth$, there exists for each $\gk$ a unique vector 
$u\in\textrm{span}[e_{1},\dots,e_{k}]$ with $\pss{u}{e_{k}}=1$ such that 
$Tu=\lambda _{k}(T)u$. We denote this vector by $u_{k}(T)$, so that 
$Tu_{k}(T)=\lambda _{k}(T)u_{k}(T)$. We also denote by $\ti{u}_{k}(T)$ the normalized eigenvector $\ti{u}_{k}(T)=u_{k}(T)/\Vert u_{k}(T)\Vert $ of $T$.
Here is a useful fact concerning the 
functions $\lambda _{k}$ and $u_{k}$:

\begin{fact}\label{Fact 18}
 For each $\gk$, the functions $T\mapsto\lambda_k(T)$ and $T\mapsto u_k(T)$  are continuous from $(\tth,\sot)$ into $\T$ and 
$\h$ respectively.
\end{fact}

\begin{proof}
The continuity of the function $T\mapsto\lambda_k(T)$  is obvious from the definition of 
$\lambda _{k}(T)$. As to the function $T\mapsto u_k(T)$, 
 the fact that $\pss{u_{k}(T)}{e_{k}}=1$ implies that 
 \[
\pss{u_{k}(T)}{e_{i}}=\dfrac{1}{\lambda _{k}(T)-\lambda _{i}(T)}\sum
_{j=i+1}^{k}\pss{Te_{j}}{e_{i}}\,\pss{u_{k}(T)}{e_{j}}
\]
for every $1\le i\le k-1$. A straightforward induction then shows that the 
scalar functions $T\mapsto\pss{u_{k}(T)}{e_{i}}$,  $1\le i\le k-1$, are $\sot$-continuous on 
$\tth$, from which Fact \ref{Fact 18} follows.
\end{proof}

\begin{fact}\label{Fact 19}
 Let $u\in \h_{00}$, and choose $r\geq 1$ such that $u\in{\rm span}\,[e_1,\dots ,e_r]$. 
 For any $T\in\tth$,  $u$ 
can be written as $u=\sum_{k=1}^{r}a_{k}(T)u_{k}(T)$, where the functions $T\mapsto a_k(T)$,  
$1\le k\le r$, are continuous from $(\tth,\sot)$ into 
$\C$.    
\end{fact}

\begin{proof} This follows directly from Cramer's formulas and the continuity of the maps $T\mapsto u_k(T)$, $k\ge 1$.
\end{proof}
\par\smallskip

\subsection{Perfect spanning is typical} The following result shows that within any 
class $\ttmh$, $M>1$, the perfect spanning property (aka ergodicity in the Gaussian sense) is a 
typical property with respect to the topology $\texttt{SOT}$.

\begin{proposition}\label{Proposition 26}
 For every $M>1$, the set \emph{$\psmh=\textrm{G-ERG}_M(\h)$} is comeager in \emph{$(\ttmh,\sot)$}. 
Consequently, \emph{$\textrm{ERG}_{M}(H)$}, 
\emph{$\textrm{FHC}_{M}(H)$} and \emph{$\textrm{UFHC}_{M}(H)$} are all comeager in \emph{$(\ttmh,\sot)$}. 
\end{proposition}

\begin{proof} We first recall the following key fact, which goes back to \cite[Th. 4.1]{G} (see also \cite[Prop. 6.1]{GM}).
 
\begin{fact}\label{PSPfact} An operator $T\in\tth$ has 
a perfectly spanning set of unimodular eigenvectors as soon as the following 
property is satisfied: for every $\varepsilon >0$ and every $\gk$, there 
exists $l\ge 1$ with $l\neq k$ and $\alpha\in\C$ such that $\|{u}_{k}(T)-\alpha {u}_{l}(T)\|<\varepsilon$. 
\end{fact}

From Fact \ref{PSPfact}, we deduce that the set
\[
 \mathfrak{G}=\bigcap_{i\ge 1}\ \bigcap_{\gk}\ \bigcup_{l\neq k}\ 
\bigcup_{\alpha\in\C }\ \bigl\{T\in\ttmh\,;\,\|{u}_{k}(T)-\alpha 
{u}_{l}(T)\|<2^{-i}\bigr\}
\]
is contained in $\psmh$. Moreover, it follows from Fact \ref{Fact 18} that $ \mathfrak G$ 
is also a $G_\delta$ subset of $(\ttmh,\sot)$. So we just have to prove that $ \mathfrak G$ is dense in $\ttmh$; 
and for this, it is enough to show that each open set 
\[
 \mathfrak{O}_{i,k}=\bigcup_{l\neq k}\ \bigcup_{\alpha \in\C}\ 
\bigl\{T\in\ttmh\,;\,\|{u}_{k}(T)-\alpha 
{u}_{l}(T)\|<2^{-i}\bigr\},\quad i,\,k\ge 1
\]
is dense in $(\ttmh,\sot)$. 
\par\smallskip
Let us fix $i,k\ge 1$ and $A\in\ttmh$. 
We have to show that for any $r\geq 1$, there exists an operator 
$T\in\mathfrak O_{i,k}$ such that $Te_j$ is arbitrarily close to $Ae_j$ for every $j=1\dots ,r$. Without 
loss of generality, we assume that $r\geq k$. Without loss of generality, we can suppose also that $\Vert A\Vert <M$. Indeed, if $\Vert A\Vert =M$, consider the function $f:\,t\mapsto \Vert D+t(A-D)\Vert $ defined on the compact set $[0,1]$, where $D$ is the diagonal \op\ given by the diagonal coefficients of $A$. The function $f$ is  convex on $[0,1]$, so that its maximum is attained at either $0$ or $1$. Since $f(0)=1<M=f(1)$, $f(t)<M$ for every $t\in [0,1)$. If $t\in [0,1)$ is sufficiently close to $1$, the \op\ $A':=D+t(A-D)$ is upper-triangular with respect to $(e_{k})_{k\ge 1}$, has the same diagonal coefficients as $A$, satisfies $\Vert A'\Vert <M$, and is as close as we wish to $A$ for the $\sot$ topology. We thus suppose to begin with that $\Vert A\Vert <M$.
\par\smallskip
 We set, for every $r\ge 1$,
$\h_r:={\rm span}\,[e_1,\dots ,e_r]$, and denote by
$P_r$ the orthogonal projection of $\h$ onto $\h_r$.
\par\smallskip
For 
every $1\le j\le r$, we denote by $\lambda _{j}$ the $j$-th diagonal coefficient of $A$:
$\lambda_j=\pss{Ae_j}{e_j}$. 
We recall that $\vert\lambda_j\vert=1$.
 Let us denote by $\La_r$ the set of all sequences $\la=(\lambda_j)_{j>r}$ with $\vert\lambda_j\vert=1$ such that all the elements $\lambda_j$, $j>r$, 
 are pairwise distinct and distinct from $1, \lambda_1,\dots,\lambda_r$, and $\lambda_j\to 1$ as $j\to\infty$.
 For any element $\la$ of $\La_r$ and any unilateral weight sequence $\om=(\omega_j)_{j\geq 1}$, we consider the operator $A_{\la,\om}$ on $\h$ defined by
\[
A_{\la,\om}e_{j}=
\begin{cases}
 Ae_{j}&\textrm{for every}\ 1\le j\le r\\
 \omega_{j-r} e_{j-r}+\lambda _{j}e_{j}&\textrm{for every}\ j>r.
\end{cases}
\]
Note that since $\Vert A\Vert <M$, $\Vert A_{\la,\om}\Vert\leq M$ for every $\la\in\La_r$ provided
$\Vert\om\Vert_\infty:=\sup_{j\ge 1}|\omega_{j}|$ is small enough. Hence in this case
$A_{\la,\om}$ belongs to $\ttmh$. Moreover, we have $A_{\la,\om}e_j=Ae_j$ for every $j=1,\dots ,r$.
Let us now fix a weight sequence $\om$ such that $\omega_j\to0$ as $j\to\infty$ and $\Vert\om\Vert_\infty$ is small enough. 
We are going to show that $A_{\la,\om}$ belongs to $\mathfrak{O}_{i,k}$ provided the sequence $\la$ is well-chosen. 
This will prove the density of $\mathfrak{O}_{i,k}$ in $(\ttmh,\sot)$.
\par\smallskip
We first note the following fact concerning the \eva s of the \op\ $A_{\la,\om}$.

\begin{fact}\label{onycroit} For any $\la\in\La_r$, the only eigenvalues of 
$A_{\la,\om}$ are $\lambda_1,\dots,\lambda_r$, the terms of the sequence $\la$, and possibly the point $1$. Moreover, for any $i\geq 1$, the 
eigenspace associated to $\lambda_i$ is spanned by the vector ${u}_i( A_{\la,\om})$.
\end{fact}

\begin{proof}[Proof of Fact \ref{onycroit}] If $\lambda$ is any complex number, solving formally the equation $A_{\la,\om}x= \lambda x$ with $x=\sum_{j\geq 1} x_j e_j$, 
yields that
\[ (A-\lambda)P_rx=\sum_{j=1}^r\omega_j x_{j+r}e_j\qquad{\rm and}\qquad x_{j+r}=\frac{\lambda-\lambda_j}{\omega_j}\, x_j\quad\hbox{for every $j>r$}.
\]
Since $\omega_j\to 0$ and $\lambda_j\to 1$ as $j\to\infty$, we infer from these equations that there is no solution   $x\in\h\setminus\{ 0\}$ to these equations if 
$\lambda$ does not belong to the set $\{ \lambda_j;\; j\geq 1\}\cup\{ 1\}$. This proves the first part of Fact \ref{onycroit}.
\par\smallskip
For the same reasons combined with the fact that the $\lambda_j$ are pairwise distinct and distinct from $1$, we also observe that 
if a vector $x\in\h$ satisfies $A_{\la,\om}x=\lambda_i x$ 
for some $i\geq 1$, then  $x_j=0$ for all $j>i$. 
Hence $\ker(A_{\la,\om}-\lambda_i)$ is contained in $\h_i$ for every $i\geq 1$, which proves the second part of 
Fact \ref{onycroit} since $A_{\la,\om}$ belongs to $\tth$ and the \eva s $\lambda_{j}$ of $A_{\la,\om}$ are all distinct.
\end{proof}

The key point is now
the following observation. Recall that $\lambda_1,\dots ,\lambda_r$ are fixed, that $1\le k\le r$, and that we are especially interested in $\lambda_k$. Recall also that 
$\widetilde u_{k+r}(A_{\pmb{\lambda}, \pmb{\omega}})=u_{k+r}(A_{\pmb{\lambda}, \pmb{\omega}})/\Vert u_{k+r}(A_{\pmb{\lambda}, \pmb{\omega}})\Vert$.

\begin{fact}\label{Claim 27}
 There exists a positive constant $\gamma $ such that for any $\pmb{\lambda }=(\lambda _{j})_{j>r}\in\pmb{\Lambda }_{r}$,
\[
|\lambda _{k+r}-\lambda _{k}|\ge \gamma\, \textrm{dist}\bigl( \ti{u}_{k+r}(A_{\pmb{\lambda },\pmb{\omega }}),\textrm{span}[{u}_{k}(A_{\pmb{\lambda },\pmb{\omega }})]\bigr).
\]
\end{fact}

\begin{proof} For any 
 $\pmb{\lambda }=(\lambda _{j})_{j>r}\in\pmb{\Lambda }_{r}$, the restriction of the operator 
$A_{\pmb{\lambda },\pmb{\omega }}$ to $\h_{r}$ is equal to $P_{r}AP_{r}$, which does not depend on 
$\pmb{\lambda }$. Also, $\ker(A_{\pmb{\lambda },\pmb{\omega }}-\lambda _{k})=\textrm{span}[{u}_{k}(A_{\pmb{\lambda },\pmb{\omega }})]\subseteq \h_{r}$, and 
the $r$ eigenvalues $\lambda _{1},\dots,\lambda _{r}$ of $P_{r}AP_{r}$ are distinct. These observations imply that there exist two positive constants $\delta $ and $C$ such that the following two properties hold true:
\begin{equation}\forall\pmb{\lambda }\in\pmb{\Lambda }_{r}\;\, \forall x\in \h_{r}\;:\; 
\Vert (A_{\pmb{\lambda },\pmb{\omega }}-\lambda _{k})x\Vert\ge\delta \,\textrm{dist}(x,\textrm{span}
[{u}_{k}(A_{\pmb{\lambda },\pmb{\omega }})])\label{Eeequation 1}
\end{equation}
and 
\begin{equation}\forall\lambda \in\C\setminus\{\lambda _{1},\dots,\lambda _{r}\}\;\, \forall x\in\h_{r}\;:\; 
\Vert (A_{\pmb{\lambda },\pmb{\omega }}-\lambda )^{-1}x\Vert\ge \dfrac{C}{\textrm{dist}(
\lambda,\{\lambda _{1},\dots,\lambda _{r}\})}\,\Vert x\Vert.\label{Eeequation 2}
\end{equation}

Now, a simple computation yields that 
\[
 \ti{u}_{k+r}(A_{\pmb{\lambda },\pmb{\omega }})=\dfrac{v_{k}+e_{k+r}}{\bigl( \Vert v_{k}\Vert^{2}+1\bigr)^{1/2}}
\]
 where $v_{k}\in\h_{r}$ satisfies the equation $(A_{\pmb{\lambda },\pmb{\omega }}-\lambda _{k+r})v_{k}+\omega _{k}e_{k}=0$. 
Since $\lambda _{k+r}$ does not belong to the set $\{\lambda _{1},\dots,\lambda _{r}\}$, we may write $v_{k}=-\omega _{k}
(A_{\pmb{\lambda },\pmb{\omega }}-\lambda _{k+r})^{-1}e_{k}$, and $(\ref{Eeequation 2})$ implies that
\begin{equation}\label{Eeequation 3}
 \Vert v_{k}\Vert \ge C\,\dfrac{|\omega _{k}|}{\textrm{dist}\bigl( \lambda _{k+r},\{\lambda _{1},\dots,\lambda _{r}\}\bigr)}\ge C\,\dfrac{|\omega _{k}|}{|\lambda _{k+r}-\lambda _{k}|}\cdot
\end{equation}
On the other hand, $(A_{\pmb{\lambda },\pmb{\omega }}-\lambda _{k})v_{k}=(\lambda _{k+r}-\lambda _{k})v_{k}-\omega _{k}e_{k}$, so that by (\ref{Eeequation 1})
\[
|\lambda _{k+r}-\lambda _{k}|\, \,\Vert v_{k}\Vert+\vert \omega _{k}\vert\ge\delta \,\textrm{dist}\bigl( v_{k},
\textrm{span}[{u}_{k}(A_{\pmb{\lambda },\pmb{\omega }})]\bigr).
\]
Hence 
\[
\textrm{dist}\bigl(v_{k}+e_{k+r},\textrm{span}[{u}_{k}(A_{\pmb{\lambda },\pmb{\omega }})]\bigr)\le\dfrac{1}{\delta }\bigl( |\lambda _{k+r}-\lambda _{k}|\,\Vert v_{k}\Vert +|\omega _{k}|\bigr)+1
\]
and
\begin{align*}
 \textrm{dist}\bigl( \ti{u}_{k+r}(A_{\pmb{\lambda },\pmb{\omega }})&,\textrm{span}[{u}_{k}(A_{\pmb{\lambda },\pmb{\omega }})]\bigr)\\&\le
\dfrac{1}{\delta }\biggl( |\lambda _{k+r}-\lambda _{k}|\dfrac{\Vert v_{k}\Vert }{\bigl( \Vert v_{k}\Vert^{2}+1\bigr)^{1/2}}+\dfrac{|\omega _{k}|}{\bigl( \Vert v_{k}\Vert^{2}+1\bigr)^{1/2}}\biggr)+
\dfrac{1}{\bigl( \Vert v_{k}\Vert^{2}+1\bigr)^{1/2}}\\
&\le\dfrac{1}{\delta }\biggl( |\lambda _{k+r}-\lambda _{k}|+\dfrac{|\omega _{k}|}{\Vert v_{k}\Vert }\biggr)+\dfrac{1}{\Vert v_{k}\Vert }\\
&\le|\lambda _{k+r}-\lambda _{k}|\biggl( \dfrac{1}{\delta }+\dfrac{1}{C}+\dfrac{1}{C|\omega _{k}|}\biggr)\qquad\qquad\hbox{by (\ref{Eeequation 3})}.
\end{align*}
 Setting
${1}/{\gamma }:={1}/{\delta }+{1}/{C}+{1}/{(C\min\limits_{1\le j\le r}|\omega _{j}|)}$ yields the desired inequality.
\end{proof}

We now come back to our main proof: recall that our aim is to show that $A_{\la,\om}$ belongs to $\mathfrak{O}_{i,k}$ for some suitable choice of $\lambda\in\La_{r}$.
By Fact \ref{Claim 27}, we have
\begin{equation}\label{onnesaitpas}
{\rm dist}\,({\widetilde u}_{k+r}(A_{\la,\om}),
\textrm{span}[{ u}_{k}(A_{\la,\om})])\le \dfrac{|\lambda _{k+r}-\lambda _{k}|}{
\gamma}\quad\hbox{for every $\la\in\La_r$}.
\end{equation}
We now choose $\la\in\La_r$ such that $\lambda _{k+r}$ is so close to $\lambda_{k}$ that the quantity on the right hand side of the inequality (\ref{onnesaitpas}) is less that $2^{-i}\eta_k$, where 
$\eta_k=\min\bigl(1,(2\Vert u_k(A_{\la,\om})\Vert)^{-1}\bigr)$. Note that 
$\eta_k$ does not depend on $\la$ since the restriction of $A_{\la,\om}$ to $\h_r$ does not.
With this choice of $\la$, there exists a scalar $\beta \in\C$ such that 
$\|{\widetilde u}_{k+r}(A_{\la,\om})-\beta {u}_{k}(A_{\la,\om})\|<2^{-i}\eta_k$. Since $\Vert\widetilde u_{k+r}(A_{\la,\om})\Vert=1$ and $2^{-i}\eta_k\leq 1/2$, we have 
$\Vert\beta {u}_{k}(A_{\la,\om})\Vert\geq 1/2$, so that $\vert\beta\vert\geq \eta_k$. It follows that $\Vert \alpha {u}_{k+r}(A_{\la,\om})- {u}_{k}(A_{\la,\om})\Vert <2^{-i}$, where 
$\alpha:=1/(\beta\Vert u_{k+r}(A_{\la,\om})\Vert)$, and this shows that 
$A_{\la,\om}$ 
belongs to $ \mathfrak{O}_{i,k}$. 
\par\smallskip
Thus we have proved that 
$\mathfrak{O}_{i,k}$ is indeed dense in $(\ttmh,\sot)$ for every $i,k\ge 1$, which concludes the proof of 
Proposition \ref{Proposition 26}. 
\end{proof}

\begin{remark} It is natural to wonder whether the class of \emph{chaotic} operators is also comeager in $\ttmh$. This does not look quite clear from 
the above proof.
\end{remark}
\par\smallskip

\subsection{Ergodicity vs ergodicity in the Gaussian sense}\label{Subsection 2.3}
Let $(e_k)_{k\geq 1}$ be as usual a fixed orthonormal basis of $\h$. In this subsection, 
we focus on special operators of the associated class $\tth$ of upper-triangular \ops\ with respect to $(e_{k})_{k\ge 1}$, 
which are the sum of a diagonal operator with respect to the basis 
$(e_{k})_{k\ge 1}$ and a backward weighted shift operator with 
respect to this same basis. More precisely, we introduce the following 
notation: we denote by $\La $ the set of sequences $\la=(\lambda 
_{k})_{\gk}$ of pairwise distinct complex numbers of modulus $1$ such that 
$(\lambda _{k})_{\gk}$ tends to $1$ as $k$ tends to infinity, and for 
every $M>0$, by $\Om_{M}$ the set of all weight sequences $\om=(\omega 
_{k})
_{\gk}$ such that $0<\omega _{k}\le M$ for every $\gk$. We also set
$\Om:=\bigcup_{M>0}\Om_{M}$.
\par\smallskip
We endow the set $\pmb{\Lambda}$ with the topology induced by $\ell^\infty(\N)$. Since $\pmb{\Lambda}$ is contained in 
the {separable} closed subspace $c(\N)$ of $\ell^\infty(\N)$ consisting of all convergent sequences, and since $\pmb{\Lambda}$ is easily seen to be 
a $G_\delta$ subset of $\ell^\infty(\N)$, it follows that $\pmb{\Lambda}$ is a Polish space.
As for the spaces $\Om_{M}$, $M>0$, we 
endow them with the product topology. Each 
$\Om_M$ being a $\gd$ subset of $\R^\N$, it is thus a Polish space as well.
\par\smallskip
To each pair $(\la,\om)\in\La\times\Om$, we associate the operator 
$T_{\la,\,\om}$ defined by $T_{\la,\,\om}=D_{\la}+B_{\om}$, where 
$D_{\la}$ is the diagonal operator with diagonal coefficients $\lambda 
_{k}$, $\gk$, associated to the basis $(e_{k})_{\gk}$, and $B_{\om}$ is 
the 
backward shift operator with respect to $(e_{k})_{\gk}$ with weights 
$\omega _{k}$,  $\gk$. Each operator $T_{\la,\,\om}$ belongs to $\tth$.
\par\smallskip
In this subsection, our aim is to investigate the properties of the operator 
$T_{\la\,\om}$ for a typical choice of parameters 
$(\la,\om)\in\La\times\Om$. For every $M>0$, we consider the following two 
sets of parameters:
\[
 \pmb{\mathcal{E}}_{M}:=\bigl\{(\la,\om)\in\La\times\Om_M\,;\,
 T_{\la,\,\om}\ \textrm{is ergodic}\bigr\}
\]
and 
\[\pmb{\mathcal{D}}_{M}:=\bigl\{(\la,\om)\in\La\times\Om_M\,; \ \sigma_p( T_{\la,\,\om} )\subseteq \{\lambda _{k},\,k\ge 1\}\,\cup\,\{1\}\bigr\},
\]
where $\sigma_p(T)$ denotes the point spectrum (\mbox{\it i.e.} the set of eigenvalues) of an operator $T\in\bh$.
The operators belonging to $\pmb{\mathcal{D}}_{M}$ are those which have the 
smallest possible set of eigenvalues among the operators $T_{\la,\,\om}$, 
$(\la,\om)\in\La\times\Om_M$. In particular, they have countable unimodular 
point spectrum.
\par\smallskip 
Our main result can now be stated as follows:

\begin{theorem}\label{Theorem 28}
 For any $M>2$, the two sets $\pmb{\mathcal{E}}_{M}$ and 
$\pmb{\mathcal{D}}_{M}$ are comeager in $\La\times\Om_{M}$.
\end{theorem}

As an immediate consequence, we obtain

\begin{corollary}\label{Corollary 29}
 There exist ergodic operators on $\h$, of the form $T_{\la,\,\om}$, 
 $(\la,\om)\in\La\times\Om$, which have countable unimodular point spectrum. In particular, these operators are ergodic but 
not ergodic in the Gaussian sense.
\end{corollary}

\begin{remark} Corollary \ref{Corollary 29} provides examples of hypercyclic 
operators on a Hilbert space with a spanning set of unimodular eigenvectors but 
only countably many unimodular eigenvalues. The question of the 
existence of such operators 
was first raised by Flytzanis in \cite{Fl}, and answered recently 
in \cite{Me}. Indeed, the chaotic non-frequently hypercyclic operators constructed 
there have only countably many eigenvalues, which are all roots of unity, 
and the associated eigenvectors span the space. Corollary \ref{Corollary 
29} strengthens this result by showing that \emph{ergodic} counterexamples to 
Flytzanis' conjecture exist, and that such counterexamples are in some sense 
much less exotic than suggested by the rather technical construction of \cite{Me}.
\end{remark}

\begin{remark} Since all the operators we are considering here have infinitely many eigenvalues, 
Corollary \ref{Corollary 29} leaves open the question of the existence 
of ergodic operators on $\h$ {without any eigenvalue at all}. Such operators are known to exist 
on the Banach space $\mathcal C_{0}([0,2\pi])$ of continuous functions on $[0,2\pi]$ vanishing at the point $0$. Indeed, the Kalisch operator $T$ defined by 
\[Tf(\theta)=e^{i\theta}f(\theta)-\int_0^\theta ie^{it}f(t)\, dt, \]
when considered as acting on $\mathcal C_{0}([0,2\pi])$, is \erg\ in the \ga\ sense but does not admit any unimodular \eva; see \cite{BG3} or \cite[Section 5.5.4]{BM} for details. 
On the other hand, an \op\ on a Hilbert space which is \erg\ in the \ga\ sense definitely has a lot of unimodular \eva s (recall that $\gergh=\psph$); but a general \erg\ \op\ might possibly have no \eva\ at all.
\end{remark}

We now turn to the proof of Theorem \ref{Theorem 28}. We first state 
a simple fact, in which a complete description of the unimodular 
eigenvectors of the operators $T_{\la,\,\om}$ is given.

\begin{fact}\label{Fact 30}
Fix $(\la,\om)\in\La\times\Om$, and $\lambda \in\C$. Then $\lambda$ is an eigenvalue of $T_{\la,\,\om}$ 
 if and only if the vector
 \[
E_{\la,\,\om}(\lambda ):=e_{1}+\sum_{n\ge 2}\,\Bigl(\,\prod_{j=1}^{n-1}
\dfrac{\lambda -\lambda _{j}}{\omega _{j}}\Bigr)e_{n}
\]
is a well-defined vector of $\h$. In this case, 
$\ker(T_{\la,\,\om}-\lambda )$ =$\textrm{span}\,[E_{\la,\,\om}(\lambda )]$.
\end{fact}

\begin{proof} 
It suffices to solve formally the equation $T_{\la,\,\om}x=\lambda x$ in $\C^\N$.
\end{proof}
For every $\gk$,
 \[
E_{\la,\,\om}(\lambda_{k} )=e_{1}+\sum_{n= 
2}^{k}\,\Bigl(\,\prod_{j=1}^{n-1}
\dfrac{\lambda_{k} -\lambda _{j}}{\omega _{j}}\Bigr)e_{n}
\]
is thus an eigenvector of $T_{\la,\,\om}$ associated to the eigenvalue 
$\lambda_{k} $, and hence is proportional to $u_{k}(T_{\la,\,\om})$.
\par\smallskip
The next lemma provides necessary conditions for the point spectrum of 
$T_{\la,\,\om}$ to be either ``maximal'' or ``minimal''. For any element 
$\om$ of $\Om$, we denote by $R_{\om}$ the radius of convergence of the 
series
\[\sum_{j\ge 1}\frac{z^{j}}{\omega _{1}\cdots\omega _{j}}\cdot
\]

\begin{lemma}\label{Lemma 31}
 Let $(\la,\om)$ be an element of $\La\times\Om$.
 \begin{enumerate}
  \item [\emph{(1)}] If $R_{\om}>2$, the map $\lambda\mapsto E_{\la,\,\om}(\lambda )$
is well-defined and analytic on a neighborhood of the unit circle $\T$. 
Consequently,
$\overline{\vphantom{t}\emph{span}}\,[E_{\la,\,\om}(\lambda _{k})\;;\;k>
k_{0}]=\h$ 
for every $k_{0}\ge 1$.
  \item[\emph{(2)}] If $R_{\om}=0$, the eigenvalues of $T_{\la,\,\om}$ are 
contained in the set $\{\lambda _{k}\;;\;\gk\}\cup\{1\}$.
 \end{enumerate}
\end{lemma}

\begin{proof}
 The first part of assertion (1) is a direct consequence of Fact \ref{Fact 30}: since 
$R_{\om}>2$, the series defining $E_{\la,\,\om}(\lambda )$ is convergent 
in 
$\h$ for 
every $\lambda$ belonging to the open disk $D(0, R_{\om}-1)$, and the map $\lambda\mapsto E_{\la,\,\om}(\lambda )$ is analytic there. In particular, this map is analytic in a neighborhood of $\T$. Since the sequence 
$(\lambda _{k})_{k> k_{0}}$ has an accumulation point in 
$D(0,R_{\om}-1)$ for every $k_{0}\ge 1$, it follows from the analyticity of $E_{\la,\om}$ that the closed linear span of the vectors 
$E_{\la,\,\om}(\lambda _{k})$, $k> k_{0}$, coincides with the closed linear span 
of all the eigenvectors $E_{\la,\,\om}(\lambda )$, $\lambda \in\T$. Since 
 $E_{\la,\,\om}(\lambda _{k})$ is proportional to $u_k(T_{\la,\om})$ and since the 
vectors $u_{k}(T_{\la,\,\om})$, $\gk$, 
span a dense subspace of $\h$, this proves the second part of assertion (1).
\par\smallskip
As to assertion (2), suppose that $R_{\om}=0$. Let us show that the 
series defining $E_{\la,\,\om}(\lambda )$ does not converge when $\lambda \in\C$
does not belong to $\{\lambda _{k}\,;\;\gk\}\cup\{1\}$. So let fix such a complex number $\lambda$. 
 Since $\lambda 
_{k}$ tends to $1$ as $k$ tends to infinity, there exists a number $\delta 
>0$ such that $|\lambda -\lambda _{k}|\ge \delta $ for every $\gk$. Hence 
\[
\biggl|\, \prod_{j=1}^{n-1}\dfrac{\lambda -\lambda _{j}}{\omega 
_{j}}\biggr|^{2}\ge\dfrac{\delta ^{\,2(n-1)}}{(\omega _{1}\dots\omega 
_{n-1})^{2}}\qquad \textrm{for every}\ n\ge 2.
\]
Since $R_{\om}=0$, it follows that 
\[
\sum_{n\ge 2}\ \biggl|\, \prod_{j=1}^{n-1}\dfrac{\lambda -\lambda 
_{j}}{\omega 
_{j}}\biggr|^{2}=\infty,
\]
and hence $E_{\la,\,\om}(\lambda )$ is not defined as a vector of $\h$.
\end{proof}

After these preliminary facts, we now state a proposition which will 
be the key step to prove that ergodic elements are typical in 
$\La\times\Om_{M}$ for any $M>2$. We postpone the proof of Proposition
\ref{Proposition 32} to the end of this section, and explain first how 
Theorem \ref{Theorem 28} may be deduced from it.
In what follows, we fix $M>2$. For any $\om\in\Om_M$, any open set $U\neq\emptyset$ in $\h$ and any $\alpha>0$, we set 
\begin{align*}
 \pmb{ \mathcal{O}}^{\,\om}_{U,\,\alpha }:=\Bigl\{\la\in\La\ ;\ 
 &\exists\,r\ge 1\;\, \exists\,(a_{1},\dots,a_{r})\in\C^{r}\ :\\[-1ex]
 &\sum_{k=1}^{r}a_{k}u_{k}\bigl( T_{\la,\,\om}\bigr)\in 
U\quad\textrm{and}\quad\sum_{k=1}^{r}
|a_{k}|^{2}\|u_{k}(T_{\la,\,\om})\|^{2}<\alpha \Bigr\}.
\end{align*}
\par\smallskip\noindent
Note that $\pmb{ \mathcal{O}}^{\,\om}_{U,\,\alpha }$ is an \emph{open} subset of $\pmb\Lambda$: this 
follows from the continuity of the map $(\la,\,\om)\mapsto T_{\la,\,\om}$ from $\La\times\Om_{M}$ into $(\mathfrak{T}_{M+1}(\h),\sot)$, combined with 
Fact \ref{Fact 18}.

\begin{proposition}\label{Proposition 32} Let $U$ 
be a non-empty open set in $\h$, and let $\alpha>0$. If $\om\in\Om_{M}$ is such that $R_{\om}>2$, then the open set  $\pmb{ \mathcal{O}}^{\,\om}_{U,\,\alpha }$
 is dense in $\La$.
\end{proposition}

Proposition \ref{Proposition 32} will allow us to apply the following ergodicity criterion proved 
in \cite[Cor.\,5.5]{GM}.

\begin{lemma}\label{ERGOCRIT} Let $X$ be a separable Banach space, and let $T\in\mathfrak B(X)$. 
Assume that for any open set  $\Omega\neq\emptyset $ of $X$ with $T^{-1}(\Omega) \subseteq\Omega $, 
any neighborhood $W$ of  $0$, and any $\varepsilon >0$, there exists a $T$-invariant probability
measure $\mu $ on $X$ with compact support such that $\mu (\Omega )=1$ and $\mu 
(W)>1-\varepsilon $. Then $T$ is ergodic.
\end{lemma}

The link with Proposition \ref{Proposition 32} is perhaps not quite obvious at first sight. It is provided by the following fact.

\begin{fact}\label{link} Let $T\in \tth$, let $U$ be a non-empty open set in $\h$, let $\varepsilon >0$, and let $W$ be a neighborhood of $0$ in $\h$. 
Let also $r\geq 1$, and assume that the eigenvalues $\lambda_1(T),\dots ,\lambda_r(T)$ of $T$ are rationally independent. Finally, 
let $a_1,\dots ,a_r$ be $r$ complex numbers and assume that the vector $\sum_{k=1}^r a_ku_k(T)$ belongs to $ U$. If the quantity
$\sum_{k=1}^r \vert a_k\vert^2\Vert u_k(T)\Vert^2$ 
is small enough, there exists a compactly supported $T$-invariant measure $\mu$ on $\h$ such that 
\par\smallskip
\begin{itemize}
\item[$\bullet$] $\mu(\Omega_U)=1$, where 
$\Omega_U=\bigcup_{n\geq 0} T^{-n}(U)$, and 
\item[$\bullet$] $\mu(W)>1-\varepsilon$.
\end{itemize}
\end{fact}

\begin{proof} We only sketch the proof, since the argument is already essentially given in \cite{GM}.
Let 
$(\chi_{k} )_{k=1}^r$ be a sequence of independent random Steinhaus 
variables defined on some standard probability space $(\Omega ,\mathcal{F},\P)$, and let $\mu$ 
be the measure on $\h$ given by the distribution of the $\h$-valued random variable $\sum_{k=1}^{r}\chi _{k}a_{k}u_{k}(T)$. 
This measure $\mu$, which may be called the \emph{Steinhaus measure associated with the vector $u=\sum_{k=1}^r a_ku_k(T)$},
 is $T$-invariant and has 
compact support. As observed in 
\cite[Fact\,5.16]{GM}, the independence of the unimodular numbers $\lambda _{1}(T),\dots,
\lambda _{r}(T)$ and the condition $\sum_{k=1}^{r}a_{k}u_{k}(T)\in U$ 
 imply that $\mu (\Omega _{U})=1$. Also, we have  by orthogonality of the Steinhaus variables $\chi_k$, $1\le k\le r$, that
 \[\int_H \Vert x\Vert^2d\mu(x)=\mathbb E\left(\sum_{k=1}^r \chi_k a_k u_k(T) \right)=\sum_{k=1}^r \vert a_k\vert^2\,\Vert u_k(T)\Vert^2.
 \]
It then follows from Markov's inequality that  
$\mu (W)
>1-\varepsilon $ provided that the quantity $ \sum_{k=1}^r \vert a_k\vert^2\,\Vert u_k(T)\Vert^2$ is sufficiently small.
\end{proof}

\begin{proof}[Proof of Theorem \ref{Theorem 28}] The proofs of the two parts of Theorem \ref{Theorem 28} are completely independent of each other. Recall that the sets $\pmb{\mathcal{E}}_{M}$ and $\pmb{\mathcal{D}}_{M}$ are defined at the beginning of Section \ref{Subsection 2.3}, and that we assume that $M>2$.
\par\smallskip\noindent
\textbf{Part 1.} \emph{The set $\pmb{\mathcal{E}}_{M}$ is comeager in 
$\La\times\Om_{M}$.}  

\smallskip
Let $(U_{p})_{p\ge 1}$ be a countable basis of non-empty open subsets for $\h$. For each 
$\om\in\Om_{M}$, 
define the $G_{\delta}$ subset of $\La$
\[
\pmb{G_{\omega}}:=\bigcap_{p\ge 1}\ \bigcap_{q\ge 1}\ 
\pmb{\mathcal{O}}^{\,\om}_{U_{p},2^{-q}},
\]
and
\[
\pmb{G}:=\bigl\{(\la,\,\om)\in\La\times\Om\,;\,\la\in\pmb{G_{\omega}}\bigr\}
.
\]

\begin{claim}\label{Claim 33}
 The set $\pmb{G}$ is a dense $\gd$ subset of $\La\times\Om_{M}$.
\end{claim}

\begin{proof}[Proof of Claim \ref{Claim 33}]
The fact that $\pmb{G}$ is $\gd$ in $\La\times\Om_{M}$ is immediate. As to
its density, we first observe that since $M>2$, Proposition \ref{Proposition 32} implies that $\pmb{G_{\omega}}$ is dense in $\La$ 
as soon as $\omega _{j}=M$ for $j$ sufficiently large. It follows that the set 
$\{\om\in\Om_{M}\,;\,\pmb{G_{\omega}}\ \textrm{is dense in}\ \La\}$ is 
dense in $\Om_{M}$. Since this set is also clearly $\gd$ in $\Om_{M}$, it is thus comeager 
in $\Om_M$. The Kuratowski-Ulam theorem then implies that the $G_\delta$ set
$\pmb{G}$ is comeager and hence dense in $\La\times\Om_{M}$.
\end{proof}

It is now easy to show that the set $\pmb{\mathcal{E}}_{M}$ is comeager in 
$\La\times\Om_{M}$. Indeed, it follows from Lemma \ref{ERGOCRIT} and Fact \ref{link} that $T_{\la,\,\om}$ is ergodic as soon as 
$\pmb\lambda $ belongs to 
$\pmb{G_{\omega }}$ and the terms $\lambda _{k}$, $\gk$, of the sequence 
$\la$ are rationally independent. Hence the set
\[
\pmb{G}_{\textrm{ind}}=\bigl\{(\la,\,\om)\in\pmb{G}\,;\,\hbox{the unimodular numbers $\lambda_k,\;k\ge 1,$ are rationally independent}\bigr\}
\]
is contained in $\pmb{\mathcal{E}}_{M}$. Since the conditions that the numbers $\lambda _{k}$ should be rationally independent 
is easily seen to define a dense 
$\gd$ subset of $\La$, Claim \ref{Claim 33} implies that 
$\pmb{G}_{\textrm{ind}}$ is a dense $\gd$ subset of 
$\La\times\Om_{M}$. Hence $\pmb{\mathcal{E}}_{M}$ is 
comeager in $\La\times\Om_{M}$.
\par\smallskip\noindent
\textbf{Part 2.} \emph{The set $\pmb{\mathcal{D}}_{M}$ is comeager in 
$\La\times\Om_{M}$}.  

\smallskip
The argument here is much simpler, and relies solely 
on part $(2)$ of Lemma \ref{Lemma 31}. Observe that the set of all 
$\om\in\Om_{M}$ such that $R_{\,\om}=\liminf_{j\to\infty }
(\omega _{1}\dots\omega _{j})^{1/j}=0$ is $\gd$ in $\Om_{M}$, and that this set is 
dense in $\Om_M$ because it contains all sequences $\om\in\Om_M$ such that $\omega_{j}$ tends to $0$ as $j$ tends to infinity. 
Hence the set $\{(\la,\,\om)\in\La\times\Om_{M}\,;\,R_{\,\om}=0\}$ is 
$\gd$ and dense in $\La\times\Om_{M}$, and 
the claim follows.
\end{proof}

We now turn to the proof of Proposition \ref{Proposition 32}. Again, it 
relies on some arguments used in \cite{GM}, more precisely in the proof of \cite[Th.\,5.12]{GM}.

\begin{proof}[Proof of Proposition \ref{Proposition 32}] Let us fix 
$\te\in\La$ and $\varepsilon >0$. We are looking for an element $\la$ of 
$\pmb{\mathcal{O}}^{\,\om}_{U,\,\alpha }$ such that $\|\la-\te\|_{\infty }
<\varepsilon $. 
\par\smallskip
Writing $\te$ as $\te=(\theta _{k})_{k\ge 1}$, we fix $k_{0}\ge 
1$ such that $|\theta _{k}-1|<\varepsilon /3$ for every $k> k_{0}$. By assertion (1) of 
 Lemma \ref{Lemma 31}, the linear span of the vectors $E_{\,\te,\,\om}(\theta _{s})$, $s>k_0$, is dense in $\h$, so there exists an index 
 $Q\geq 1$ and complex numbers $b_{k_0+1},\dots ,b_{k_0+q}$ such that the vector
 \[
z:=\sum_{s=k_{0}+1}^{k_{0}+Q}b_{s}\,E_{\,\te,\,\om}(\theta _{s})
\textrm{ belongs to } U.
\]
We now proceed as in \cite{G} or \cite{GM}, 
and replace each coefficient $b_{s}$ by a 
certain sum of scalars of the form $\sum_{j=0}^{N-1}c_{s+jQ}$ with 
$\sum_{j=0}^{N-1}\,|c_{s+jQ}|^{2}$ sufficiently small. More 
precisely, we define
\[
c_{s+jQ}=\frac{1}N\, b_{s}\quad\textrm{and}\quad 
v_{s+jQ}=E_{\,\te,\,\om}(\theta_{s} )
\]
for every $k_{0}+1\le s\le k_{0}+Q$ and every $0\le j\le N-1$, where $N\ge 1$ is 
an integer so large that
\begin{equation}\label{vk1}
 \sum_{k=k_{0}+1}^{k_{0}+NQ}|c_{k}|^{2}\|v_{k}\|^{2} 
=\dfrac{1}{{N}}\, \sum_{s=k_{0}+1}^{k_{0}+Q}
|b_{s}|^{2}\ \|E_{\te,\,\om}(\theta _{s})\|^{2}<\alpha .
\end{equation}
Observe that $\sum\limits_{k=k_{0}+1}^{k_{0}+NQ}c_{k}v_{k}=z$ by construction, so that we have
\begin{equation}\label{vk2}
\sum\limits_{k=k_{0}+1}^{k_{0}+NQ}c_{k}v_{k}\in U.
\end{equation}
 
The next step in the proof is 
to define $\la\in\La$ with $\|\la-\te\|_{\infty }<\varepsilon $, in such 
a way that  each vector $v_{k}$, $k_{0}+1\le k\le k_{0}+NQ$, can be 
approximated  by the eigenvector $E_{\,\la,\,\om}(\lambda _{k})$ of 
$T_{\la,\,\om}$ associated to the eigenvalue $\lambda _{k}$, and the 
$\lambda _{k}$, $k_{0}+1\le k\le k_{0}+NQ$, are all distinct. The sequence $\la$ 
is defined as follows: we keep $\lambda _{k}=
\theta _{k}$ for every $1\le k\le k_{0}+Q$ and every $k>k_{0}+NQ$, so that in particular $\lambda_k\to 1$ as $k\to\infty$. For 
$k_{0}+1\le s\le k_{0}+Q$ and $1\le j\le N-1$, we choose $\lambda _{s+jQ}$ 
extremely close to $\theta _{s}$, in such a way that these new coefficients  $\lambda_k$ 
are all distinct and distinct from all the $\theta_k$. Since we already know that $\lambda_k\to 1$ as $k\to\infty$, the sequence $\la$ defined in this way belongs to $\La$. 
We can certainly ensure that $|\lambda _{s+j+1}-\theta 
_{s}|<\varepsilon $ for all $k_{0}+1\le s\le k_{0}+Q$ and $1\le j\le N-1$, 
and thus that
\[
|\lambda _{s+jQ}-\theta _{s+jQ}|<\varepsilon /3+|\theta _{s+jQ}-1| 
+|\theta 
_{s}-1 |<\varepsilon 
\] since $s>k_{0}$. All the remaining coefficients of $\la$  and $\te$ 
coincide, 
and hence $\|\la-\te\|_{\infty }<\varepsilon $.
\par\smallskip 
Let us now show that the quantities 
\[
\|E_{\,\la,\,\om}(\lambda _{s+jQ})-v_{s+jQ}\|=\|E_{\,\la,\,\om}(\lambda 
_{s+jQ})-E_{\,\te,\,\om}(\theta _{s})\|,
\]
where $k_{0}+1\le s\le k_{0}+Q$ and $0\le j\le N-1$, can be made as small as we 
wish, provided each coefficient 
 $\lambda _{s+jQ}$ is close enough to $\theta _{s}$. We 
consider separately two cases.
\par\smallskip
- Assume that $k_{0}+1\le s\le k_{0}+Q$ and $j=0$. In this case $\lambda _{s}=\theta _{s}$ by 
definition, so that $E_{\,\la,\,\om}(\lambda _{s})=E_{\,\te,\,\om}(\theta 
_{s} )$.
\par\smallskip 
- Assume now that $k_{0}+1\le s\le k_{0}+Q$ and $1\le j\le N-1$. In this case, we use the explicit 
expression of the eigenvectors provided by Fact \ref{Fact 30} and write
\begin{align*}
 E_{\,\la,\,\om}(\lambda 
_{s+jQ})-E_{\,\te,\,\om}(\theta _{s})=\sum_{n=2}^{s}\ 
\biggl[\ &\prod_{i=1}^{n-1}\biggl( \dfrac{\lambda _{s+jQ}-\lambda 
_{i}}{\omega _{i}}\biggr)-
\prod_{i=1}^{n-1}\biggl( \dfrac{\theta _{s}-\theta _{i}}{\omega 
_{i}}\biggr) \biggr]e_{n}\\
&+\sum_{n=s+1}^{s+jQ}\ \prod_{i=1}^{n-1}\biggl(\dfrac{\lambda 
_{s+jQ}-\lambda _{i}}{
\omega _{i}} 
\biggr)e_{n}.
\end{align*}
Fix an arbitrarily small number $\delta >0$. The integers $s$ and $j$ being 
fixed, and the numbers $\lambda_i$ and $\theta_{i}$ being by definition equal for every $1\le i\le k_{0}+Q$, the norm of the first sum in the above expression can be made less that 
$\delta /2$, provided that the difference $|\lambda _{s+jQ}-\theta _{s}|$ is sufficiently 
small. As for the second term, observe that $s$ belongs to the set
$\{1,\dots,n-1\}$ for every $s+1 \le n\le s+jQ$, so that the term 
$\lambda _{s+jQ}-\lambda _{s}=\lambda _{s+jQ}-\theta _{s}$ always appears 
in the product $\prod_{i=1}^{n-1}(\lambda _{s+jQ}-\lambda _{i})/\omega 
_{i}$. Thus if $|\lambda _{s+jQ}-\theta _{s}|$ is sufficiently small, the 
norm of the second term is less than $\delta /2$ too. Hence one can ensure 
that $\|E_{\,\la,\,\om}(\lambda 
_{s+jQ})-E_{\,\te,\,\om}(\theta _{s})\|<\delta $ for all $k_{0}+1\le s\le 
k_{0}+Q$ and $1\le j\le N-1$.
\par\smallskip 
So we have proved that for any $\delta >0$, one can construct $\la\in\La$ such that 
$\Vert\la-\pmb\theta\Vert_\infty<\varepsilon$ and $\|E_{\,\la,\,\om}(\lambda 
_{k})-v_{k}\|<\delta $ for every $k_{0}+1\le k\le k_{0}+NQ$. If $\delta$ is now chosen small enough, 
the 
two conditions
\[
\sum_{k=k_{0}+1}^{k_{0}+NQ}c_{k}\,E_{\,\la,\,\om}(\lambda _{k})\in 
U\quad\textrm{and}\quad 
\sum_{k=k_{0}+1}^{k_{0}+NQ}|c_{k}|^{2}\,\| 
E_{\la,\,\om}(\lambda _{k})\|^{2}<\alpha 
\]
simultaneously hold true, by (\ref{vk2}) and (\ref{vk1}) respectively. Remembering that each vector $u_{k}(T
_{\la,\,\om})$ is proportional to $E_{\la,\,\om}(\lambda _{k})$, we 
eventually obtain that there exist complex coefficients $a_{k}$,
$k_{0}+1\le k\le k_{0}+NQ$, such that 
\[
\sum_{k=k_{0}+1}^{k_{0}+NQ}a_{k}\,u_{k}(T_{\la,\,\om})\in 
U\quad\textrm{and}\quad 
\sum_{k=k_{0}+1}^{k_{0}+NQ}|a_{k}|^{2}\,\| 
u_{k}(T_{\la,\,\om})\|^{2}<\alpha .
\]
Hence $\la$ belongs to $\pmb{\mathcal{O}}^{\,\om}_{U,\alpha }$, and this 
concludes the proof of Proposition \ref{Proposition 32}.
\end{proof}

\begin{remark}\label{Remark 34}
 In order to show that the ergodic operators are comeager in 
$\La\times\Om_{M}$, 
we have used in a crucial way  the fact that the terms $\lambda _{k}$, 
$\gk$, of the sequences $\la\in\La$ involved in the proof are rationally 
independent. This is not so surprising in view, for instance, of 
\cite[Fact\,5.16]{GM}. Yet the role of independence in these issues 
remains rather mysterious. We develop this a little bit in the next 
subsection.
\end{remark}
\par\smallskip

\subsection{Additional remarks}\label{Subsection 2.4}
In this subsection, we present a short discussion of some questions motivated by the results obtained above, as well 
as some further results concerning dynamical properties of the ``diagonal plus shift'' \ops\ $T_{\la,\om}$.
\par\smallskip

\subsubsection{Some natural questions}\label{section machin}
Two such questions 
concerning the existence of ope\-ra\-tors on a Hilbert 
space with particular ergodic-theoretic-like properties remain unanswered at 
this stage of our work. The first one is

\begin{question}\label{Question A}
 Do there exist $\mathcal{U}$-frequently hypercyclic operators on $\h$ 
which are not frequently hypercyclic?
\end{question}

It was proved by Bayart and Rusza in \cite{BR} that such operators do 
exist on the space $c_{0}$, but the question was left open for 
Hilbert (or even reflexive) spaces. In the light of the discussion carried 
out in Section \ref{COMPLEXITY} of this paper, it seems natural to 
conjecture that the two classes $\ufhch$ and $\fhch$ should have different 
descriptive complexity, and \emph{hence} should be distinct; but we have been unable to 
solve the question using this approach. The second question runs as 
follows:

\begin{question}\label{Question B}
 Do there exist frequently hypercyclic operators on $\h$ which are not 
ergodic?
\end{question}

This question comes from \cite{GM}, where it is proved that 
frequently hypercyclic non-ergodic operators do exist on the space $c_{0}$. The 
proof of this result again relies on a construction of \cite{BR}: it is  
proved in \cite{GM} that the frequently hypercyclic bilateral 
weighted shifts $T$ on $c_{0}(\Z)$ defined in \cite{BR} satisfy $c(T)<1$, and thus 
cannot be ergodic. For some reasons, we found it rather tempting to try to attack
 this question by Baire category methods; but again, we did not succeed in this way.
\par\smallskip
We \emph{do} solve Questions \ref{Question A} and \ref{Question B} in Section \ref{SPECIAL} below, but using widely 
different methods. The operators we will use are generalizations of those 
introduced by the third named author in \cite{Me} in order to solve the 
question of the existence of a chaotic non-frequently hypercyclic 
operator. The two main results we will obtain are that indeed, there exist 
$\mathcal{U}$-frequently hypercyclic operators  on $\h$ which are not 
frequently hypercyclic, as well as frequently hypercyclic operators which 
are not ergodic. All our examples turn out to be chaotic; and we will 
complement these results by showing that there also exist on $\h$ operators which are chaotic and 
topologically mixing but not $\mathcal{U}$-frequently hypercyclic.
\par\smallskip
In another direction, the role of \emph{rational independence} in all that concerns the links between 
 unimodular eigenvalues and ergodicity properties of operators needs to be clarified. To be a little more specific,  let us consider the class $\tindh$ of operators in 
 $\tth$ whose diagonal coefficients are rationally independent. Rather surprisingly, it seems that very little is known concerning 
such operators. For example, to our knowledge the following question is open:

\begin{question}\label{Question 1}
 Let $T$ be a hypercyclic operator belonging to $\tindh$. Is $T$ 
necessari\-ly ergodic? frequently hypercyclic? $\mathcal{U}$-frequently hypercyclic?
\end{question}

The only currently known examples of hypercyclic operators with spanning 
unimodular eigenvectors which are not frequently hypercyclic or not ergodic are those constructed in \cite{Me} and 
the ones that will be considered in Section \ref{SPECIAL}. As already mentioned, all these operators are chaotic. Moreover, for many 
of these operators, the only unimodular eigenvalues are roots of unity and for some of those, each eigenvalue has multiplicity one. So these operators belong to $\tth$ for
some suitably chosen orthonormal basis $(e_k)_{k\ge 1}$, but not to $\tindh$ for any basis $(e_k)_{k\ge 1}$. A positive answer to Question \ref{Question 1} would show that this is not accidental, \mbox{\it i.e.} 
that 
there {must} be a strong amount of dependence between the eigenvalues of \emph{any} of the operators
constructed in \cite{Me} or in
Section 
\ref{SPECIAL} of the present paper. 
\par\smallskip
As a matter of fact, Question \ref{Question 1} seems to be open even for the operators 
 $T_{\la,\,\om}=D_{\la}+B_{\om}$ considered in Subsection \ref{Subsection 2.3}. Even more prosaically, it seems
quite desirable (and perhaps not too difficult) to determine when exactly an operator of the form $T_{\la,\om}$ is hypercyclic.

\subsubsection{More on the operators $T_{\la,\om}$} We finish this section by collecting some simple facts that we do know concerning 
operators of the form $T_{\la,\om}$. Let us introduce the following notations.
\par\smallskip
$\bullet$ If $\om$ is a unilateral weight sequence, we set $a_1(\om):=1$ and 
\[
a_j(\om):=\frac{1}{\omega_1\cdots \omega_{j-1}}\quad \textrm{for every}\; j\geq 2.
\]
\noindent Also, using the same notation as in Section \ref{Subsection 2.3} above, we denote by $R_{\om}$ the radius of convergence of the series $\sum a_j(\om) z^j$:
\[
R_{\om}=\liminf_{j\to\infty}\,  \vert \omega_1\cdots \omega_{j-1}\vert^{1/j}.
\]
\par\smallskip
$\bullet$ To each $\la\in\T^\N$, we associate polynomials $P_{\la, j}$, $j\geq 1$, defined as follows: $P_{\la,1}(\lambda)\equiv 1$ and 
\[
P_{\la,j}(\lambda)=\prod_{k=1}^{j-1} (\lambda-\lambda_k)\quad \textrm{ for every } j\geq 2.
\]

\begin{proposition} Let $\om$ be a unilateral weight sequence, and let $\la\in\T^\N$.
Then the following facts hold true.
\begin{enumerate}[\rm (1)]
\item[\rm(0)]   A complex number $\lambda$ is an eigenvalue of $T_{\la,\om}$ if and only if
\[
\sum_{j=1}^\infty \vert a_j(\om)\vert^2\, \vert P_j(\lambda)\vert^2<\infty.
\]
In this case, $\ker(T-\lambda)={\rm span}\,[E_{\la,\om}(\lambda)]$ where
\[
E_{\la,\om}(\lambda)=\sum_{j=1}^\infty a_j(\om) P_{\la,j}(\lambda)\, e_j.
\]
\par\smallskip
\item If $R_{\om}=0$, then $\sigma_p(T_{\la,\om})\subseteq\overline{\{ \lambda_k;\; k\geq 1\}}$.
\par\smallskip
\item If $\sigma (B_{\om})=\{ 0\}$ (\mbox{e.g.} if $\omega_j\to 0$ as $j\to\infty$), then 
$\sigma (T_{\la,\om})\subseteq \overline{\{ \lambda_k;\; k\geq 1\}}$. More generally, $\sigma(T_{\la,\om})$ is contained 
in the set $\{ \lambda\in\C;\; {\rm dist}\,(\lambda, \{ \lambda_k\})\leq r(B_{\om})\}$, where $r(B_{\om})$ denotes the spectral radius of $B_{\om}$.
\end{enumerate}
\par\smallskip
\noindent Assume now that $\la$ belongs to $\La$, \mbox{\it i.e.} that the elements $\lambda_k$, $k\ge 1$, are pairwise distinct and that $\lambda_k\to 1$ as $k\to\infty$. 
\par\smallskip
\begin{enumerate}
\item[\rm(3)]
If $R_{\om}> 0$, then $\sigma_p(T_{\la,\om})$ contains $\{ \lambda_k;\; k\geq 1\}\cup D(1,R_{\om})$ and is contained in 
$\{ \lambda_k;\; k\geq 1\}\cup\overline D(1,R_{\om})$.
\par\smallskip
\item[\rm(4)] 
If $R_{\om}> \sup_{k\ge 1} \vert\lambda_k-1\vert$, then $T_{\la,\om}$ is ergodic in the
 Gaussian sense. Moreover, the map $\lambda\mapsto E_{\la,\om}(\lambda)$ is analytic on the open disk $D(0,R_{\om}-1)$.
\par\smallskip
\item[\rm(5)] If $R_{\om}< \sup_{k\ge 1} \vert \lambda_k-1\vert$, then $T_{\la,\om}$ is not hypercyclic.
\par\smallskip
\item[\rm(6)] If $R_{\om}=\sup_{k\ge 1} \vert \lambda_k-1\vert$, let $k_0$ be the largest integer $k\ge 1$ such that 
$\vert \lambda_k-1\vert\geq R_{\om}$. If
\[\sum_{j\geq k_0+1} 
\left\vert\frac{\omega_1\cdots \omega_{j-1}}{(\lambda_{k_0}-\lambda_{k_0+1})\cdots (\lambda_{k_0}-\lambda_j)}\right\vert^2<\infty,
\] 
then $T_{\la,\om}$ is not ergodic in the Gaussian sense.
\end{enumerate}
\end{proposition}

\begin{proof} (0) This is a simple computation which has already been carried out in the proof of Fact \ref{Fact 30}.
\par\smallskip
(1)  If $\lambda\in\C$ does not belong to the closure of  $\{ \lambda_k;\;\; k\geq 1\}$, then 
$\inf_{k\ge 1} \vert\lambda-\lambda_k\vert=\delta >0$, and thus $\vert P_j(\lambda)\vert\geq \delta^{j-1}$ for all $j\geq 1$. If $R_{\om}=0$, assertion (0) implies that $\lambda$ is not an eigenvalue of $T_{\la,\om}$.
\par\smallskip
(2) Let $\lambda\in\C$. We have $T_{\la,\om}-\lambda I=D_{\pmb\xi}+B_{\om}$, where $\pmb\xi=(\xi_{k})_{k\ge 1}$ is defined by $\xi_k=\lambda_k-\lambda$ for every $k\geq 1$. 
If $\delta:={\rm dist}\,(\lambda, \{ \lambda_k\})>0$, then $D_{\pmb\xi}$ is invertible and $T_{\la,\om}-\lambda I$ can be written as
\[
T_{\la,\om}-\lambda I=D_{\pmb\xi}(I+D_{\pmb\xi}^{-1}B_{\om})=D_{\pmb\xi}(I+B_{\pmb\xi^{-1}\om}),
\]
where $\pmb\xi^{-1}\om=(\xi_{k}^{-1}\om_{k})_{k\ge 1}$.
Moreover, since $\sup_{k\ge 1} \xi_k^{-1}=1/\delta$, the spectral radius of $B_{{\pmb\xi}^{-1}{\pmb\om}}$ is at most $r(B_{\om})/\delta$. 
It follows that $T_{\la,\om}-\lambda I$ is invertible as soon as $\delta>r(B_{\om})$.
\par\smallskip
(3) We already know by assertion (0) that $\sigma_p(T_{\la,\om})$ contains $\{\lambda_k\;;\; k\geq 1\}$. Now, let $\lambda$ belong to the open disk $ D(1,R_{\om})$, and 
choose $r$ such that $\vert \lambda-1\vert<r<R_{\om}$. Since $\lambda_k\to 1$, there exists $k_{0}\ge 1$ such that $\vert \lambda-\lambda_k\vert<r$ for all $k\geq k_0$. It follows that there exists a constant $C>0$ such that 
$\vert P_j(\lambda)\vert\leq C\, r^j$ for every $j\geq 1$, so that $E_{\la,\om}(\lambda)$ is well-defined and analytic on $ D(1,R_{\om})$.
\par\smallskip
Conversely, assume that $\lambda$ does not belong to the set $\{ \lambda_k\;;\; k\geq 1\}\cup\overline D(1,R_{\om})$. Since $\lambda_k\to 1$, one can find $r>R_{\om}$ 
such that $\vert \lambda-\lambda_k\vert\geq r$ for $k$ sufficiently large; and since $\lambda\neq \lambda_k$ for all $k\ge 1$, 
it follows that there exists a constant $c>0$ such that 
$\vert P_j(\lambda)\vert\geq c\, r^j$ for every $j\geq 1$. Hence $E_{\la,\om}(\lambda)$ is not well-defined.
\par\smallskip
 (4) By assertion (3), the map $\lambda\mapsto E_{\la,\om}(\lambda)$ is analytic on the disk $D(1,R_{\om})$; so it is 
 enough to show that the vectors $E_{\la,\om}(\lambda)$, $\lambda\in D(1,R_{\om})$, span a dense subspace of  $\h$. 
 Let $y=\sum_{j\ge 1}y_j e_{j}$ be a vector of $\h$ which is orthogonal to all the vectors $E_{\la,\om}(\lambda)$, $\lambda\in D(1,R_{\om})$. We thus have
\[
\sum_{j=1}^\infty \bar y_j a_j(\om) P_{\la,j}(\lambda)=0\qquad\hbox{for every $\lambda\in D(1,R_{\om})$}.
\]
The  series above is convergent on the disk $D(1,R_{\om})$, which contains all the points $\lambda_k$ by assumption. So 
\[
\sum_{j=1}^\infty \bar y_j a_j(\om) \, P_{\la,j}(\lambda_k)=0\qquad\hbox{for every $k\geq 1$}.
\]
Since $P_{\la,j}(\lambda_k)=0$ whenever $j>k$ and $P_{\la,k}(\lambda_k)\neq 0$, it follows that $\bar y_j a_j(\om)=0$ for every $j\ge 1$, so that $y=0$.
\par\smallskip
(5) Let $k_{0}\ge 1$ be such that $|\lambda_{k_{0}}-1|>R_{\om}$. Then $\lambda_{k_{0}}$ is an isolated \eva\ of $T_{\la,\om}$ by assertion (3), which is easily seen not to belong to the essential spectrum of $T_{\la,\om}$. It then follows from  \cite{Her2} that $T_{\la,\om}$ cannot be \hy.
\par\smallskip
(6) By assertion (1), we can assume that $R_{\om}>0$, since otherwise the point spectrum $\sigma_p(T_{\la,\om})$ of $T_{\la,\om}$ is countable, and $T_{\la,\om}$ is certainly not ergodic in the Gaussian sense in this case.
It suffices to show that the unimodular eigenvectors of $T_{\la,\om}$ are not perfectly spanning. Since $(\sigma_p(T)\cap\T)\setminus D(1,R_{\om})$ 
is a finite set by assertion (3), it is in turn enough to show that the vectors $E_{\la,\om}(\lambda)$, $\lambda\in D(1,R_{\om})$, do not span a dense subspace of 
$\h$. So we have to find coefficients $c_j$, $j\ge 1$, with the property that
\[
\sum_{j=1}^\infty c_j P_{\la,j}(\lambda)=0\qquad\hbox{for every $\lambda\in D(1,R_{\om})$},
\]
with the additional requirement that
\begin{equation}\label{ell2}
\sum_{j=1}^\infty \left\vert\frac{c_j}{a_j(\om)}\right\vert^2<\infty,\quad \textrm{\mbox{\it i.e.} that }\quad
\sum_{j\ge 2} \vert c_j\, \omega_1\cdots \omega_{j-1}\vert^2<\infty.
\end{equation}
We define the coefficients $c_j$ by setting $c_{j}=0$ for every $1\le j<k_{0}$, $c_{k_{0}}=1$ and 
\[
 c_{j}=\prod_{i=k_{0}+1}^{j}(\lambda_{k_0}-\lambda_{i})^{-1}\quad\textrm{ for every }j>k_{0}.
\]
Then condition  (\ref{ell2}) is satisfied by assumption. Moreover, the coefficients $c_j$ have been defined in such a way that 
\[\sum_{j=1}^N c_j P_{\la,j}(\lambda)=
P_{\la, k_0}(\lambda)\,\cdot\prod_{k=k_0+1}^{N} \left(\frac{\lambda-\lambda_k}{\lambda_{k_0}-\lambda_k}\right)
\qquad\hbox{for all $N\geq k_0+1$}.
\]
Since $\vert\lambda_{k_0}-1\vert\geq R_{\om}$ and $\lambda_k\to 1$ as $k\to\infty$, it follows that $\sum_{j=1}^\infty c_j P_{\la,j}(\lambda)=0$ for every $\lambda$ such that $\vert \lambda-1\vert<R_{\om}$. This completes the proof of assertion (6).
\end{proof}
\par\smallskip

\section{Periodic points at the service of hypercyclicity}\label{CRITERIA}

In this section, we depart from our standing assumption and work in the general context of Banach spaces over $\K=\R$ or $\C$, 
not restricting ourselves to the Hilbertian setting. Actually, most of our results hold in the framework of  arbitrary 
Polish topological vector spaces, as should be clear from the 
proofs.
\par\smallskip
Let thus $X$ be a real or complex separable Banach space. 
Our aim is to obtain general sufficient conditions for an operator $T\in\mathfrak B(X)$ to be frequently 
hypercyclic or $\mathcal U$-frequently hypercyclic. These criteria are rather different from the classical ones, since they involve 
the \emph{periodic points} of the operator $T$. They will greatly contribute to simplify some proofs in the next section, and  
they might hopefully be useful elsewhere also.
For any bounded 
operator $T$ on $X$, we 
denote by $\textrm{Per}(T)$ the set of its periodic points, and 
by $\textrm{per}_T(x)$, or simply $\textrm{per}(x)$, the period of a periodic point $x$ of 
$T$.
\par\smallskip

\subsection{Precompact orbits and topological mixing}
We start by giving a very simple criterion for  an operator $T\in\mathfrak B(X)$ to be 
topologically mixing, which involves
the points of $X$ with \emph{precompact orbit} under the action of $T$. We denote by $\text{Prec}(T)$ the set of points with precompact orbit.

\begin{proposition}~\label{Prop mix fin}
Let $T\in\mathfrak B(X)$. Assume that ${\rm Prec}(T)$ is dense in $X$ and that for every neighborhood $W$ of $0$ in $X$ and  every non-empty open subset $V$ of $X$, the set $\mathcal N_T(W,V)$ is cofinite. Then $T$ is topologically mixing.
\end{proposition}

\begin{proof} It suffices to show that for any vectors $y\in \text{Prec}(T)$, $x\in X$, and any $\varepsilon>0$, the set $\mathcal N_T\bigl(B(y,\varepsilon),B(x,\varepsilon)\bigr)$ is cofinite. 
Since $y$ is a vector with precompact orbit, there exists a finite subset $\{y_1,\dots,y_d\}$ of $X$ such that any point of ${\rm Orb}(T,y)$ lies within distance less than $\varepsilon/2$ of the set $\{y_1,\dots,y_d\}$. In other words, there exists for
any $n\ge 1$ an index $i_n\in\{ 1,\dots ,d\}$ such that $\Vert T^ny-y_{i_n}\Vert<\varepsilon/2$. 

By assumption, each set $\mathcal N_T\bigl(B(0,\varepsilon), B(x-y_i,\varepsilon/2)\bigr)$ is cofinite. Hence, 
one can find an integer $N$ such that for every $n\ge N$ and every $1\le i\leq d$, there exists $z_{n,i}\in X$ 
such that 
\[\|z_{n,i}\|<\varepsilon\qquad{\rm and}\qquad \|T^nz_{n,i}-(x-y_i)\|<\varepsilon/2.\]
For every $n\geq N$, we then have
\[\|T^n(y+z_{n,i_n})-x\|\leq \Vert T^ny-y_{i_n}\Vert +\|T^nz_{n,i_n}-(x-y_{i_n})\|<\varepsilon.\]
Hence $y+z_{n,i_n}$ belongs to $ B(y,\varepsilon)$ and $T^n(y+z_{n,i_n})$ belongs to $ B(x,\varepsilon)$, which shows that every integer $n\geq N$ lies in the set $\mathcal N_T\bigl((B(y,\varepsilon),B(x,\varepsilon)\bigr)$.
\end{proof}

Since periodic points have a finite orbit, Proposition \ref{Prop mix fin} immediately implies the following result.

\begin{corollary} Let $T\in\mathfrak B(X)$. 
If ${\rm Per}(T)$ is dense in $X$ and if for every neighborhood $W$ of $0$ in $X$ and  every non-empty open subset $V$ of $X$, the set $\mathcal N_T(W,V)$ is cofinite, then $T$ is chaotic and topologically mixing.
\end{corollary}

We also obtain in a similar fashion.

\begin{corollary} Let $T\in\mathfrak B(X)$. If there exists a dense set of points $x\in X$ such that $T^ix\to 0$ as $i\to\infty$ and if for every neighborhood $W$ of $0$ in $X$ and  every non-empty open subset $V$ of $X$, the set $\mathcal N_T(W,V)$ is cofinite, then $T$ is topologically mixing.
\end{corollary}

Finally,  the linearity assumption on the system $(X,T)$ implies the following strengthened version of Proposition~\ref{Prop mix fin}.

\begin{corollary}\label{prop mixing simple} Let $T\in\mathfrak B(X)$. Assume that ${\rm Prec}(T)$ is dense in $X$ and that there exists a subset $X_0$ of $X$ with dense linear span such that the following property holds true: 
for every $x_0\in X_0$ and every $\varepsilon>0$, there exists $N\ge 1$ such that for every $n\ge N$, one can find $z\in X$ such that $\|z\|<\varepsilon$ and $\|T^nz-x_0\|<\varepsilon$. 
Then $T$ is topologically mixing.
\end{corollary}

\begin{proof}
Let $V\subseteq X$ be a non-empty open set. Since $X_0$ spans a dense subspace of $X$, there exist $\varepsilon>0$ and $x=\sum_{k=1}^{K}\alpha_k x_k\in V$ with 
$x_k\in X_0$ and $\alpha_k\ne 0$ for every $1\le k\le K$ such that $B(x,\varepsilon)$ is contained in $ V$. Moreover, by assumption, there exists $N\ge 1$ such that for every $1\le k\le K$, every $n\ge N$, one can find $z_{n,k}\in X$ such that 
\[\|z_{n,k}\|<\frac{\varepsilon}{K|\alpha_k|}\quad\text{ and }\quad \|T^nz_{n,k}-x_k\|<\frac{\varepsilon}{K|\alpha_k|}\cdot\] The vector $z:=\sum_{k=1}^{K}\alpha_k z_{n,k}$ then satisfies $\|z\|<\varepsilon$ and $\|T^nz-x\|<\varepsilon$, and the desired result follows from Proposition~\ref{Prop mix fin}.
\end{proof}

\begin{remark}
Unlike in Corollary \ref{Proposition 35} below, it is not possible, in Proposition \ref{Prop mix fin}, to replace the assumption 
that all the sets $\mathcal N_T(W,V)$ are cofinite ($W$ and $V$ non-empty open subsets of $X$ with $0\in W$) by  the assumption that all the sets $\mathcal N_T(U,W)$ 
 are cofinite ($W$ and $U$ non-empty open subsets of $X$ with $0\in W$),
even if one assumes additionally that $T$ is hypercyclic. Indeed, 
every unilateral weighted backward shift has a dense set of points with finite orbit and 
is such that $\mathcal N_T(U,W)$ is cofinite for every neighborhood $W$ of $0$ and  every non-empty open set $U$; but there exist hypercyclic weighted shifts which are not topologically mixing.
\end{remark}
\par\smallskip

\subsection{Uniform recurrence and topological weak mixing}   In this subsection, we give a simple criterion for  an operator 
$T\in\mathfrak B(X)$ to be weakly topologically mixing. It involves the \emph{uniformly recurrent} points of $T$. Recall that a point $x\in X$ is said to be \emph{recurrent} for $T$ if, for any neighborhood $O$ of $x$, the set $\mathcal N_T(x,O)=\{ i\in\N;\; T^ix\in O\}$ is non-empty (or, equivalently, infinite).
A  point $x\in X$ is said to be
\emph{uniformly recurrent} for $T$ if, for any neighborhood $O$ of $x$, the set $\mathcal N_T(x,O)$ has bounded gaps. 
For example, every periodic point is obviously uniformly recurrent. The following more general fact will be useful. Recall first that a compact, 
$T$-invariant subset $K$ of $ X$ is said to be \emph{minimal} if it has no proper (non-empty) 
closed $T$-invariant subset. For example, if $x$ is a periodic point of $T$ with period $N$, then 
$K:=\{ x,Tx,\dots ,T^{N-1}x\}$ is minimal. Recall also that we denote by $\mathcal E(T)$ the set of all unimodular eigenvectors of  $T$.

\begin{fact}\label{FactUREC} Let $T\in\mathfrak B(X)$. If $K$ is a compact {minimal} $T$-invariant subset of $X$, then every point of $K$ is uniformly recurrent for $T$. 
In particular, every point $z\in{\rm span}\,\mathcal E(T)$ is uniformly recurrent for $T$.
\end{fact}

\begin{proof} The first part of the statement is well-known (and has nothing to do with linearity; see \mbox{e.g.} \cite[Theorem 1.15]{Fur}). As for the second part, it is enough to show that if 
$z$ belongs to ${\rm span}\,\mathcal E(T)$, then its closed $T$-orbit $K=\overline{\{T^nz;\; n\geq 0\}}$ is a compact minimal $T$-invariant set.
\par\smallskip
Write $z$ as $z=\sum_{j=1}^N u_j$, where for every $1\le j\le N$, $Tu_j=\gamma_j u_j$ for some $\gamma_j\in\T$. Let $\Gamma$ be the 
closed subgroup of $\T^N$ generated by the $N$-tuple $\gamma:=(\gamma_1,\dots ,\gamma_N)$. Then 
$K=\{ \sum_{j=1}^N \omega_j u_j;\; \omega=(\omega_1,\dots ,\omega_N)\in\Gamma\}$, so that $K$ is compact (and, of course, $T$-$\,$invariant). 
So we just have to check that $K$ is minimal for the action of $T$.
Let $a\in K$ be arbitrary. We have to show that any point $y\in K$ can be approximated as close as we wish by points of the form $T^na$, $n\ge 0$. 
Write $a$ and $y$ as $a=\sum_{j=1}^N \omega_ju_j$ and $y=\sum_{j=1}^N\eta_j u_j$ respectively, where $\omega=(\omega_1,\dots ,\omega_N)$ and $ \eta=(\eta_1,\dots ,\eta_N)$ belong to $\Gamma$. For any integer $n\ge 0$, 
we have $T^na=\sum_{j=1}^N \gamma_j^n\omega_j u_j$. Since the $N$-tuples $\gamma^{n}=(\gamma_1^{n},\dots ,\gamma_N^{n})$, $n\geq 0$, are dense in $\Gamma$ 
(this follows from the fact that since the group $\T^{N}$ is compact, the closed {semi}group generated by $\gamma$ is in fact a group), there exists an integer $n\ge 0$ such that $\gamma^n$ is as close as we wish to $\omega^{-1}\eta=(\omega_1^{-1}\eta_{1},\dots ,\omega_N^{-1}\eta_{N})$. Then $T^na$ is as close as we wish to 
$y$, as required.
\end{proof}

We can now state our criterion for topological weak mixing.

\begin{proposition}\label{UREC} Let $T\in\mathfrak B(X)$, and let $Z$ be a $T$-invariant subset of $X$ consisting of uniformly recurrent points. 
Assume that for every non-empty open subset $V$ of $ X$ and every 
neighborhood $W$ of $0$ in $X$, one can find a vector $z\in Z$ and an integer 
$n\ge 0$ such that $z\in W$ and $T^nz\in V$. Then $T$ is topologically weakly mixing and $Z$ is dense in $X$.
\end{proposition}

\begin{proof} That $Z$ is dense in $X$ is clear from the assumption since $Z$ is $T$-invariant.
In order to prove that $T$ is topologically weakly mixing, we are going to show that $\mathcal N_T(U,W)\cap\mathcal N_T(W,V)$ is non-empty for any non-empty open subsets 
$U,V
$ of $X$ and any neighborhood $W$ of $0$ in $X$. The so-called 
\emph{three open sets condition} (see for example \cite[Ch. 4]{BM}) will then imply that $T$ is topologically weakly mixing.
\par\smallskip
Let us first show that for any neighborhood $W$ of $0$ and any open set $V\neq\emptyset$, the set $\mathcal N_T(W,V)$ 
has bounded gaps. Choose a vector $z\in Z$ and an integer $n\ge 1$ such that $z$ belongs to $ W$ and $T^nz$ belongs to $ V$, and then an 
open neighborhood $O$ of $z$ such that $O\subseteq W$ and $T^n(O)\subseteq V$. Since $z$ is {uniformly recurrent}, 
the set $\mathcal N_T(z,O)$ has bounded gaps. Then $n+\mathcal N_T(z,O)$ has bounded gaps as well, and the result 
follows since $n+\mathcal N_T(z,O)\subseteq \mathcal N_T(z,V)\subseteq\mathcal N_T(W,V)$.
\par\smallskip
We show next that $\mathcal N_T(U,W)$ is non-empty for any non-empty open subset $U$ of $X$ and any neighborhood $W$ of $0$ in $X$. 
By assumption,  one can find a point $x\in W$ and an integer $n\ge 1$ such that $x$ is a (uniformly) recurrent point of $T$ and $u:=T^nx$ lies in $ U$. 
Since $x$ is recurrent, one can find an integer $p>n$ such that $T^px$ lies in $ W$. Then $T^{p-n}u=T^px$ belongs to $ W$, and hence 
$\mathcal N_T(U,W)$ is non-empty.
\par\smallskip
It is now easy to conclude the proof: all the sets $\mathcal N_T(W,V)$ have bounded gaps, and all the sets $\mathcal N_T(U,W)$ 
contain arbitrarily large intervals, because they are all non-empty and the sets $W$ are neighborhoods of the fixed point $0$. 
Hence any set of the form $\mathcal N_T(W,V)$ meets any other set of the form $\mathcal N_T(U,W)$.
\end{proof}

As a first consequence of Proposition \ref{UREC}, we show that for operators having a dense set of uniformly recurrent points, 
topological weak mixing is equivalent to some formally much weaker ``transitivity-like" properties.

\begin{corollary}\label{Proposition 35}
 Let $T\in\mathfrak{B}(X)$, and assume that uniformly recurrent points of $T$ are dense in $X$. The following assertions are then equivalent:
\begin{enumerate}
 \item[\emph{(a)}] $T$ is topologically weakly mixing;
 \item[\emph{(b)}] for every non-empty open subset $V$ of $X$ and every
 neighborhood $W$ of $0$ in $X$, the set $\mathcal{N}_{T}(W,V)$ is non-empty;
 \item[\emph{(c)}] for every non-empty open subset $U$ of $X$ and every
 neighborhood $W$ of $0$ in $X$, the set $\mathcal{N}_{T}(U,W)$ is non-empty.
\end{enumerate}
\end{corollary}

\begin{proof}
 Obviously, (a) implies both (b) and (c). Moreover, since the uniformly recurrent points of $T$ are assumed to be dense 
 in $X$, it follows at once from Proposition \ref{UREC} (taking as $Z$ the set of all uniformly recurrent points for $T$) 
 that (b) implies (a). So it remains to show that (c) implies (b).
 \par\smallskip
 Suppose that (c) holds true, and let $U$ and $W$ be two non-empty open subsets of $X$ with $0\in W$. By (c) and since 
 uniformly recurrent points of $T$ are dense in $X$, one can find a uniformly recurrent point $u\in U$ and an integer $
 n\geq 0$ such that $z:=T^nu$ belongs to $W$. Since $u$ is in particular recurrent, one can find $p>n$ such that $T^pu$ belongs to $ U$. 
 Then $T^{p-n}z=T^pu$ lies in $U$, so that $\mathcal N_T(W,U)$ is non-empty.
\end{proof}

As another consequence of Proposition \ref{UREC}, we now state a  criterion for topological weak mixing which formally resembles
the ergodicity criterion stated as Lemma \ref{ERGOCRIT} above.

\begin{corollary}\label{MINIM} Let $T\in\mathfrak B(X)$. Assume that for each non-empty open subset 
$\Omega$ of $X$ such that $T^{-1}(\Omega)\subseteq\Omega$ and each neighborhood $W$ of $0$, one can find a minimal $T$-invariant compact set 
$K$ such that $K\cap\Omega\cap W$ is non-empty. Then $T$ is topologically weakly mixing.
\end{corollary}

\begin{proof} By Fact \ref{FactUREC}, every point $z$ in the above compact set $K$ is uniformly recurrent for $T$. We apply Proposition \ref{UREC}, taking as $Z$ the set of all uniformly recurrent points for $T$. Let $V$ be a non-empty open subset of $X$. Applying the assumption of Corollary \ref{MINIM} to 
$\Omega:=\bigcup_{n\geq 0} T^{-n}(V)$, we immediately see that the assumption of Proposition \ref{UREC} is satisfied, and hence that $T$ is topologically weakly mixing.
\end{proof}

\begin{corollary} Let $T\in\mathfrak B(X)$, where $X$ is a \emph{complex} Banach space. Assume that for every $x\in X$ and $\varepsilon >0$, there exist $z\in\emph{span}\,\mathcal E(T)$ and $n\ge 0$ such that 
$\Vert z\Vert<\varepsilon$ and $\Vert T^nz-x\Vert<\varepsilon$. Then $T$ is topologically weakly mixing and 
$\emph{span}\,\mathcal E(T)$ is dense in $X$.
\end{corollary}

\begin{proof}  This follows from Fact \ref{FactUREC} and
Proposition \ref{UREC} applied with $Z:=\textrm{span}\,\mathcal E(T)$.
\end{proof}

Our last corollary is the result on which we will elaborate to state our $\mathcal U$-frequent hypercyclicity and 
frequent hypercyclicity criteria.

\begin{corollary}\label{Proposition 36} Let $T\in\mathfrak B(X)$. Assume that for every $x\in X$ and $\varepsilon >0$, 
there exist $z\in{\rm Per}(T)$ and $n\ge 1$ such that $\Vert z\Vert<\varepsilon$ and $\Vert T^nz-x\Vert<\varepsilon$. 
Then $T$ is chaotic.
\end{corollary} 

\begin{proof} Since periodic points are uniformly recurrent, exactly the same proof as that of the previous corollary 
shows that $T$ is topologically weakly mixing and has a dense set of periodic points.
\end{proof}
 
Corollary \ref{Proposition 36} can also be proved by a more constructive argument. Since the flexibility of this alternative proof 
will prove extremely important 
below for the proofs of some ($\mathcal{U}$-) frequent hypercyclicity criteria, we present it here.

\begin{proof}[Direct proof of Corollary \ref{Proposition 36}] Before starting the proof, we note that the assumption of Corollary \ref{Proposition 36} obviously implies that $\Vert T\Vert>1$ and that 
${\rm Per}(T)$ is dense in $X$.
 The first step of the proof is to observe that the 
assumption of Corollary \ref{Proposition 36} can be reinforced as 
follows: 

\begin{fact}\label{multiple} Under the assumption of Corollary \ref{Proposition 36}, the following property holds true:
for any $x\in X$, any $\varepsilon >0$, and any integers $N, M\ge 1 
$, there exist $z\in\textrm{Per}(T)$ and $\gn$ with $n=M\; ({\rm mod}\, N)$ such that 
$\|z\|<\varepsilon $ and 
  $\|T^{n}z-x\|<\varepsilon $.
  \end{fact}
  
  \begin{proof}[Proof of Fact \ref{multiple}] Assuming (as we may) that $x$ is non-zero, choose $z'\in\textrm{Per}(T)$ and 
$n'\ge 1$ such that $\|z'\|<\varepsilon '$ and 
$\|T^{n'}z'-x\|<\varepsilon$, where $\varepsilon'>0$ has to be specified. Then $\Vert T\Vert^{n'}\Vert z'\Vert\geq \Vert x\Vert-\varepsilon'$, so that 
$\Vert T\Vert^{n'}\geq \frac{\Vert x\Vert-\varepsilon'}{\varepsilon'}\cdot$ It follows that if $\varepsilon'$ is small enough, then $n'>N$. Let then $0\le r<N$ be an integer such that $n:=n'-r$ is 
equal to $M\;{{ ({\rm mod}\, N)}}$, and set $z=T^{r}z'$. Then $\|z\|<\Vert T\Vert^r \varepsilon'\leq \Vert T\Vert^N\varepsilon'$, so $\Vert z\Vert<\varepsilon$ 
if $\varepsilon'$ is small enough. Also, $T^{n}z=T^{n'-r}T^{r}z'=T^{n'}z'$, so that 
$\|T^{n}z-x\|<\varepsilon'\leq\varepsilon$.
\end{proof}

Let now $(x_{j})_{j\ge 1}$ be a dense sequence of vectors of $X$ belonging to 
$\textrm{Per}(T)$. A straightforward induction shows that there exists a 
sequence $(z_{j})_{j\ge 1}$ of elements of $\textrm{Per}(T)$, as well as a 
strictly increasing sequence $(n_{j})_{j\ge 1}$ of integers such that, for 
every $j\ge 1$, 
\begin{enumerate}[(i)]
 \item $\|z_{j}\|<2^{-j}\,\|T\|^{-n_{j-1}}$;
 \item $\|T^{n_{j}}z_{j}-(x_{j}-\sum_{i<j}z_{i})\,\|<2^{-j}$;
 \item $n_{j}$ is a multiple of the period of the vector $\sum_{i<j}
 z_{i}$.
\end{enumerate}
Now, set \[z:=\sum_{i= 1}^\infty z_{i},\]
which is well-defined by (i). Let us now show that $z$ is a hypercyclic vector for $T$. 
For every $j\ge 1$, we have $\displaystyle T^{n_j}\Big(\sum_{i<j} z_i\Big)=\sum_{i<j} z_i$ by (iii), so that
\begin{align*}
 T^{n_{j}}z-x_{j}&=\sum_{i<j}z_{i}+T^{n_{j}}z_{j}+\sum_{i>j}
 T^{n_{j}}z_{i}-x_{j}.
 \end{align*}
It follows that
\begin{align*}
\|T^{n_{j}}z-x_{j}\|&\le\Bigl\|T^{n_{j}}z_{j}-
\Bigl(x_{j}-\sum_{i<j}z_i\Bigr)\Bigr\|+\sum_{i>j}\|T\|^{n_{j}}\,\|z_{i}\|\\
&<2^{-j}+\sum_{i>j}\|T\|^{n_{j}}\,2^{-i}\,\|T\|^{-n_{i-1}}\le 2^{-(j-1)}
\qquad\hbox{by (i) 
and (ii).}
\end{align*}
This terminates our direct proof of Corollary \ref{Proposition 36}.
\end{proof}

To conclude this subsection, we prove 
a slightly stronger version of Corollary \ref{Proposition 36}.

\begin{corollary}\label{Proposition 38}
 Let $X_{0}$ be a subset of $X$ with 
dense linear span. Suppose that for every $x_{0}\in X_{0}$ and every 
$\varepsilon >0$, there exist $z\in\emph{Per}(T)$ and $\gn$ such that 
$\|z\|<\varepsilon $ and $\|T^{n}z-x_0\|<\varepsilon $. Then $T$ is chaotic.
\end{corollary}

\begin{proof}
It is suffices to show that the assumption of Corollary \ref{Proposition 38} can be extended from 
vectors $x_0$ of $ X_{0}$ to arbitrary vectors $x$ of $ X$: once this is done, Corollary \ref{Proposition 36} applies.

Since the linear span of $X_{0}$ 
is dense in $X$, it suffices to consider vectors $x$ of the form 
$x=\sum_{k=1}^{r}a_{k}x_{k}$, where $x_{k}\in X_{0}$ and  $a_{k}\in \K\setminus
\{0\}$ for every $1\le k\le r$. An induction on $k$, $1\le k\le r$, allows us to construct, using 
(an obvious modification of) Fact \ref{multiple}, vectors $z_{k}\in \textrm{Per}(T)$ and integers 
$l_{k}$, $1\le k\le r$, with the following properties: 
\begin{enumerate}[(i)]
 \item $\|z_{k}\|<\dfrac{\varepsilon }{r\,|a_{k}|}$;
 \item $\|T^{l_{1}+\cdots+l_{k}}z_{k}-x_{k}\|<\dfrac{\varepsilon 
}{r\,|a_{k}|}$;
\item $l_{k}$ is a multiple of 
$d_{k-1}=\ds\prod_{i=1}^{k-1}\textrm{per}_{T}(z_{i})$.
\end{enumerate}
Now, set $z:=\sum_{k=1}^{r}a_{k}z_{k}$: this $z$ is clearly a periodic vector 
for $T$, with $\|z\|<\varepsilon $ by (i). Set also 
$n:=l_{1}+\cdots+l_{r}$. We have
\begin{align*}
 T^{n}z&=\sum_{k=1}^{r}a_{k}T^{n}z_{k}=\sum_{k=1}^{r}a_{k}T^{l_{1}+\cdots+
l_{k}}z_{k}\qquad\hbox{by (iii)},
\intertext{so that}
\|T^{n}z-x\|&\le\sum_{k=1}^{r}|a_{k}|\,
\|T^{l_{1}+\cdots+l_{k}}z_{k}-x_{k}\|<\varepsilon\qquad \hbox{by (ii).}
\end{align*}
This concludes the proof of Corollary \ref{Proposition 38}. 
\end{proof}

In the next two subsections, we will give two ``variations" of Corollary \ref{Proposition 36}, where 
we show that if the assumption $\|T^{n}z-x\|<\varepsilon $ is replaced by 
the requirement that the orbit of $z$ approximates that of $x$ during a 
sufficiently large time, \mbox{\it i.e.} $\|T^{n+k}z-T^{k}x\|<\varepsilon $ for a 
large number of indices $k$, then $T$ is $\mathcal{U}$-frequently 
hypercyclic, and even sometimes frequently hypercyclic.
\par\smallskip

\subsection{A criterion for ${\mathcal U}$-frequent hypercyclicity} 
Here is the version of our criterion for 
${\mathcal U}$-frequent hypercyclicity that we will use in Section~\ref{SPECIAL}.

\begin{theorem}\label{Theorem 39} Let $T\in\mathfrak B(X)$. Assume that there exist a dense linear subspace 
$X_0$ of $X$ with $T(X_0)\subseteq X_0$ and $X_0\subseteq \emph{Per}(T)$, and a constant
$\alpha\in (0,1)$ such that the following property holds true: for every $x\in X_{0}$ and every $\varepsilon >0$, there exist 
$z\in X_{0}$ and $\gn$ such that 
\begin{enumerate}
 \item [\emph{(1)}] $\|z\|<\varepsilon $;
 \item [\emph{(2)}] $\|T^{n+k}z-T^{k}x\|<\varepsilon $ for every
 $0\le k\le n\alpha $.
\end{enumerate}
Then $T$ is chaotic and $\mathcal{U}$-frequently hypercyclic.
\end{theorem}

\begin{proof} Since $X_0$ is dense in $X$, the periodic points of $T$ are dense in $X$. So we only have to show that $T$ is $\mathcal U$-frequently hypercyclic.
\par\smallskip
 As in the direct proof of Corollary \ref{Proposition 36} given above, we first show
that given $N\ge 1$, the integer $n$ appearing in the assumption of Theorem
\ref{Theorem 39} can be supposed to be a multiple of $N$. Given $x\in 
X_{0}$ and $\varepsilon >0$, there exist $z'\in X_{0}$ and $n'\ge 1$ such 
that $\|z'\|<\varepsilon'$ and 
$\|T^{n'+k}z'-T^{k}x\|<\varepsilon' $ for every 
$0\le k\le \alpha n'$, where $\varepsilon'>0$ has to be specified.
Taking $\varepsilon' $ small enough,
we can assume (see the proof of Fact \ref{multiple}) that $n'> N$. Let then
$0\le r<N$ be such that $n:=n'-r$ is a multiple of $N$, and set 
$z=T^{r}z'$. Then $z$ belongs to $X_0$ because $T(X_{0})\subseteq X_{0}$. Also
$\|z\|<\Vert T\Vert^N \varepsilon'<\varepsilon$ if $\varepsilon'$ is small enough; and 
$T^{n+k}z=T^{n'+k}z'$, so that $\|T^{n+k}z-T^{k}x\|<\varepsilon $ for
every $0\le k\le \alpha n'$. Since $n\le n'$, the inequality holds true \textit{a fortiori} for every $0\le k\le \alpha n$.
\par\smallskip
Let now 
$(x_{l})_{l\ge 1}$ be a dense sequence of vectors of $X_{0}$, and let 
$(I_{l})_{l\ge 
1}$ be a partition of $\N$ into  
infinite sets. We define a sequence $(y_{j})_{j\ge 1}$ of vectors of $X_{0}$ by 
setting $y_{j}=x_{l}$ for every $j\in I_{l}$. In other words, the sequence $(y_j)_{j\ge 1}$ enumerates infinitely many times 
each point of the dense sequence $(x_l)_{l\ge 1}$, 
and $I_l$ denotes the set of all indices $j\ge 1$ such that $y_j=x_l$.
\par\smallskip
By induction on $j\ge 1$, we 
construct a sequence $(z_{j})_{j\ge 1}$ of vectors of $X_{0}$ and a 
strictly increasing sequence $(n_{j})_{j\ge 1}$ of integers such that
\begin{enumerate}[(i)]
 \item $\|T^{k}z_{j}\|<2^{-j}$ for every $0\le k\le (1+\alpha )n_{j-1}$;
 \item $\|T^{n_{j}+k}z_{j}-T^{k}(y_{j}-\sum_{i<j}z_{i})\|<2^{-j}$ 
for every $0\le k\le \alpha n_{j}$;
\item $n_{j}$ is a multiple of the period of the vector $\sum_{i<j}z_{i}$.
\end{enumerate}
We set $z:=\sum_{i\ge 1}z_{i}$, which is well-defined by (i), and prove 
that $z$ is a $\mathcal{U}$-frequently hypercyclic vector for $T$. Fix 
$l\ge 1$. For every $j\in I_{l}$ and every $k\ge 0$, we have by (iii)
\[
 T^{n_{j}+k}z-x_{l}=
 {T^{k}\,\Bigl(\,\sum_{i<j}z_{i} 
\Bigr)
 +T^{n_{j}+k}z_{j}-y_{j}+\sum_{i>j}T^{n_{j}+k}z_{i}.}
\]
 Moreover, if $k$ is a multiple of $\textrm{per}(x_l)=\textrm{per}(y_j)$, say $k=m\, \textrm{per}(x_l)$, we can write 
\[
T^{m\, \textrm{per}(x_l)}\,\Bigl(\,\sum_{i<j}z_{i} 
\Bigr)-y_j=T^{m\, \textrm{per}(x_l)}\Bigl(\sum_{i<j}z_i-y_j\Bigr).
\]
{Hence we obtain that for any $m\geq 0$,}
\begin{align*}
 \|T^{n_{j}+m\,\textrm{per}(x_{l})}z-x_{l}\|
 \le 
\Bigl\|T^{n_{j}+m\,\textrm{per}(x_{l})}&z_{j}-T^{m\,\textrm{per}(x_{l})}\,\Bigl(\,y_{j}-\sum_{i<j} z_{i}\Bigr)\Bigr\|\\
&+\sum_{i>j}\|T^{n_{j}+m\,\textrm{per}(x_{l})}z_{i}\|.
\end{align*}
By (i) and (ii), it follows that for every $j\geq 1$ and every
$0\le m\le \alpha n_{j}/\textrm{per}(x_{l})$, we have 
\[\|T^{n_{j}+m\,\textrm{per}(x_{l})}z-x_{l}\|<2^{-j}+\sum_{i>j}  2^{-i}=2^{-(j-1)}.\]
 We deduce from this 
inequality that for any $\varepsilon >0$, 
\begin{align*}
 \overline{\textrm{dens}}\,\mathcal{N}_{T}(z,B(x_{l},\varepsilon ))&
\ge \limsup_{j\in I_{l}}\dfrac{1}{(1+\alpha )n_{j}}\,
\#\bigl\{n_{j}+m\,\textrm{per}(x_{l})\,;\,0\le m\le 
\alpha n_{j}/\textrm{per}(x_{l})\bigr\}\\
&\ge\dfrac{\alpha }{(1+\alpha )\textrm{per}(x_{l})}>0.
\end{align*}
Hence $z$ is a $\mathcal{U}$-frequently hypercyclic vector for $T$, and 
Theorem \ref{Theorem 39} follows.
\end{proof}

\begin{remark}\label{Remark 40}
 The proof of Theorem \ref{Theorem 39} shows that $T$ admits 
 $\mathcal{U}$-frequently hypercyclic vectors $z\in X$ with the property 
that  for every $l\ge 1$, there exists $\delta _{l}>0$ such that 
$ \overline{\textrm{dens}}\,\mathcal{N}_{T}(z,B(x_{l},\varepsilon ))\ge 
 \delta_{l} $ \emph{for every $\varepsilon >0$} (the point being that $\delta_l$ does not depend on $\varepsilon$).
 Hence there exists, for every 
$l\ge 1$, a sequence $(n_{k,l})_{k\ge 1}$ of integers of positive upper 
density such that $T^{n_{k,l}}z$  tends to $x_{l}$ as $k$ tends to 
infinity. As $x_{l}$ is a periodic point for $T$, there is no 
contradiction in this (see \cite[Rem. 4.8]{GM}).
\end{remark}
\par\smallskip

\subsubsection{Uniform recurrence, almost periodic points and ${\mathcal U}$-frequent hypercyclicity}
Even if this is not quite clear at first sight, the following result generalizes Theorem \ref{Theorem 39} (see Corollary \ref{bizarre3} below). 
\par

\begin{theorem}\label{bizarre} Let $T\in\mathfrak B(X)$. Assume that the uniformly recurrent points of $T$ are dense in $X$, and that there exists $\alpha>0$ such that the following 
property holds true: for any $\varepsilon>0$ and any non-empty open subset $O$ of $X$, one can find $x\in O$ and an arbitrarily large 
integer $n$ such that 
$\Vert T^{n+k}x\Vert<\varepsilon$ for every $0\le k\leq \alpha n$. Then $T$ is $\mathcal U$-frequently hypercyclic.
\end{theorem}

\begin{proof} It is enough to show that for any non-empty open subset $V$ of
$X$, there exists a constant $\alpha_V>0$ such that, for every $N\ge 1$, the open set 
\[G_{V,N}:=\Bigl\{ u\in X;\; \exists m\geq N\;:\;\#\{ i\leq m;\; T^iu\in V\}\geq \alpha_V\, m\Bigr\}\]
is dense in $X$. Indeed, if  $(V_q)_{q\geq 1}$ is a countable basis of open subsets of $X$, the density of each of the sets $G_{V_q,N}$ implies that 
$G:=\bigcap_{N,q} G_{V_q,N}$ is a dense $G_\delta$ subset of $X$ consisting of $\mathcal U$-frequently hypercyclic vectors for $T$.
\par\smallskip
So let $V$ be a non-empty open subset of $X$. Choose a uniformly recurrent point $v\in V$, and $\varepsilon >0$ such that 
$B(v,\varepsilon)\subseteq V$. 
Since $v$ is uniformly recurrent, the set 
\[D_{V}:=\{ k\ge 1;\; \Vert T^kv-v\Vert<\varepsilon/2\}\] has bounded gaps. Observe that we may call this set $D_V$ since $\varepsilon$ depends on $v$ and $V$, 
and $v$ depends on $V$. So there exist a constant $c_V>0$ and an integer $M_V\ge 1$ 
such that $\# (D_{V}\cap J)\geq c_V\,\# J$ for all intervals $J$ of $\N$ of length at least $M_V$. We set 
\[\alpha_V:=\frac{c_V\alpha}{1+\alpha}\, ,
\]
and we show that with this choice of $\alpha_V$, all the open sets 
$G_{V,N}$, $N\ge 1$, are dense in $X$.
Let us fix $N\ge 1$, and a non-empty open subset $U$ of $X$. We have to show that 
$G_{V,N}\cap U$ is non-empty.
Set $O:=U-v$, which is a non-empty open subset of $X$. By assumption, one can find $x\in O$ and an integer $n\geq N$ 
such that $\alpha n\geq M_V$ and 
$\Vert T^{n+k}x\Vert<\varepsilon/2$ for every $0\le k\leq\alpha n$. Then $u:=x+v$ belongs to $ U$
and our aim is to show that $u$ belongs to $G_{V,N}$.
For any $0\le k\leq\alpha n$, we have
\[ \Vert T^{n+k}u-v\Vert\leq \Vert T^{n+k}x\Vert +\Vert T^{n+k}v-v\Vert<\varepsilon/2+\Vert T^{n+k}v-v\Vert\,;
\]
in other words,
\[\forall i\in[n,(1+\alpha)n]\;:\; \Vert T^iu-v\Vert<\varepsilon/2+\Vert T^iv-v\Vert.
\]
Moreover, since the interval $J=[n,(1+\alpha)n]$ has length at least $M_V$, we know that $\#\, (J\cap D_V)\geq c_V\# J$, \mbox{\it i.e.}
\[\#\,\big\{ i\in [n,(1+\alpha)n];\; \Vert T^iv-v\Vert<\varepsilon/2\bigr\} \geq c_V\, \alpha n.
\]
It follows that 
\[ \#\,\{ 1\le i\leq (1+\alpha)n;\; \Vert T^iu-v\Vert<\varepsilon\}\geq c_V\, \alpha n\, ,
\]
and hence that $u$ belongs to $ G_{V,N}$.
\end{proof}

We now state and prove a few consequences of Theorem \ref{bizarre}.
For the first one, we need a definition: if $T$ is a bounded operator on $X$, a point $x\in X$ is said to be \emph{almost periodic} for $T$ if, 
for every $\varepsilon >0$, the set 
\[D_{x,\varepsilon}:=\{ n\ge 1;\; \forall k\ge 1\;:\; \Vert T^{n+k}x-T^kx\Vert<\varepsilon\}
\]
has bounded gaps. Thus, periodic points are almost periodic and almost periodic points are uniformly recurrent. The following fact 
is more interesting.

\begin{fact}\label{veryUREC} Let $T\in\mathfrak B(X)$. Then any vector $x$ belonging to $\textrm{span} \,\mathcal E(T)$ is {almost periodic} for $T$.
\end{fact}

\begin{proof} We first recall that since $x$ belongs to $\textrm{span} \,\mathcal E(T)$, it is uniformly recurrent for~$T$ (Fact~\ref{FactUREC}); so for any $\eta>0$, the set
 $D^\eta:=\{ n\ge 1;\; \Vert  T^nx-x\Vert<\eta\}$ 
has bounded gaps.
Now, assuming that $x$ is non-zero, we write $x$ as $x=\sum_{i=1}^r u_i$, where $u_1,\dots ,u_r$ belong to $\mathcal E(T)$ and are linearly independent. 
Then the restriction of $T$ to the finite-dimensional subspace $E:=\textrm{span}\,(u_1,\dots ,u_r)$ is \emph{power-bounded}, being diagonalizable 
with only unimodular eigenvalues. Since $E$ contains the $T$-orbit of $x$, it follows that there exists a finite constant $C$ such that, for every $k\ge 0$,
\[\Vert T^{n+k}x-T^kx\Vert=\Vert T^k(T^nx-x)\Vert\leq C\,\Vert T^nx-x\Vert.\]
So the set $D^{\eta}$ with $\eta:=\varepsilon/C$ is contained in $D_{x,\varepsilon}$, which concludes the proof.
\end{proof}

From Theorem \ref{bizarre}, we now deduce

\begin{corollary}\label{bizarre2} Let $T\in\mathfrak B(X)$. Assume that there exists
$\alpha>0$ such that the following property holds true: for every non-empty open subset $O$ of $X$ and 
every $\varepsilon >0$, there exists an almost periodic point $x\in O$ such that: for any $\eta>0$, one can find
$z\in X$ and $\gn$ such that $\Vert z\Vert<\eta$ and $\Vert T^{n+k}z-T^kx\Vert<\varepsilon$ for all $0\leq k\leq\alpha n$. Then $T$ is $\mathcal U$-frequently hypercyclic.
\end{corollary}

\begin{proof} By assumption, the uniformly recurrent points are certainly dense in $X$. In order to check the second assumption in Theorem \ref{bizarre}, 
let us fix a non-empty open subset $O$ of $X$ and $\varepsilon >0$.
Let us choose a non-zero almost periodic point $x_0\in O$ satisfying the assumption of Corollary \ref{bizarre2}, and
then an integer $M\ge 1$ such that any
interval of $\N$ of length $M$ contains a point of the set
\[ D:=\{ n\ge 1;\; \forall k\ge 1\;:\; \Vert T^{n+k}x_0-T^kx_0\Vert<\varepsilon/2\}.
\]  
Now, choose $z'\in X$ with $\Vert z'\Vert$ arbitrarily small and an arbitrarily large integer $n'>M$ such that 
$\Vert T^{n'+k}z'-T^kx_0\Vert<\varepsilon/2$ for every $0\le k\leq\alpha n'$ (since $x_0$ is non-zero, one can ensure that $n'$ be arbitrarily large by 
taking $\Vert z'\Vert$ small enough.)
Having fixed $z'$ and $n'$ in this way, we can pick $0\le p\leq M$ such that $n:=n'-p$ belongs to $ D$. Then $z:=T^pz'$ has arbitrarily small norm, so $x:=x_0-z$ belongs to $O$. Moreover, since $n\leq n'$ and $T^nz=T^{n'}z'$, 
we have for every $0\le k\leq\alpha n$:
\begin{align*}
\Vert T^{n+k}x\Vert&\leq \Vert T^{n+k}x_0-T^kx_0\Vert+\Vert T^kx_0-T^{n'+k}z'\Vert
<\varepsilon/2+\varepsilon/2=\varepsilon.
\end{align*}
The assumptions of Theorem \ref{bizarre} are thus satisfied.
\end{proof}

Here is an immediate consequence of Corollary \ref{bizarre2}, whose statement 
 is a bit less convoluted. This result shows in particular that in Theorem \ref{Theorem 39}, one can replace the assumption that the vectors of $X_{0}$ are periodic by the assumption that they are almost periodic. 
 Moreover, it also shows that it 
 was in fact unnecessary to assume that $X_0$ was a \emph{linear subspace} of $X$.
 
\begin{corollary}\label{bizarre3} Let $T\in\mathfrak B(X)$. Assume that there exist a dense subset
$X_0$ of $ X$ consisting of almost periodic points for $T$, and a constant
$\alpha\in (0,1)$ such that the following property holds true: for every $x\in X_{0}$ and every $\varepsilon >0$, there exist 
$z\in X_{0}$ and $\gn$ such that $\Vert z\Vert<\varepsilon$ and $\Vert T^{n+k}z-T^kx\Vert<\varepsilon$ for all $0\leq k\leq\alpha n$. Then $T$ is $\mathcal{U}$-frequently hypercyclic.
\end{corollary}

Since Corollary \ref{bizarre3} is so similar to Theorem \ref{Theorem 39}, it is natural to ask whether one can provide a ``constructive" proof  of it
resembling that of Theorem \ref{Theorem 39}. We do so now, with the additional assumption that the above dense set $X_0$ is a linear subspace of $X$.

\begin{proof}[Constructive proof of Corollary \ref{bizarre3}] Define sequences $(x_l)_{l\ge 1}$, $(I_l)_{l\ge 1}$ and $(y_j)_{j\ge 1}$  as in the beginning of the proof of Theorem \ref{Theorem 39}. 
By induction, we 
construct a sequence $(z_{j})_{j\ge 1}$ of vectors of $X_{0}$ and an increasing sequence  of integers $(n_{j})_{j\ge 1}$ such that
\begin{enumerate}[(i)]
 \item $\|T^{k}z_{j}\|<2^{-j}$ for every $0\le k\le (1+\alpha )n_{j-1}$;
 \item $\|T^{n_{j}+k}z_{j}-T^{k}(y_{j}-\sum_{i<j}z_{i})\|<2^{-j}$ 
for every $0\le k\le \alpha n_{j}$;
\item[(iii')] $\left\Vert T^{n_j+k}\bigl(\sum_{i<j} z_i\bigr)-T^k\bigl(\sum_{i<j} z_i \bigr)\right\Vert< 2^{-j}$ for all $k\geq 0$.
\end{enumerate}
Let $j\ge 1$.
Assume that the construction has been carried out  up to the step $j-1$, and let us construct $n_j$ and $z_j$.
 Since the vector $\sum_{i<j}z_i$ lies in $X_0$, it is almost periodic; so 
one can find an integer 
$M\ge 1$ such that every interval of $\N$ of length $M$ contains a point of the set 
\[D:=\left\{ n\ge 1;\; \forall k\ge 1\;:\; \Bigl\Vert T^{n+k}\Bigl(\sum_{i<j} z_i\Bigr)-T^k\Bigl(\sum_{i<j} z_i \Bigr)\Bigr\Vert< 2^{-j}\right\}.\] 
By assumption on $T$, there exists $n'\in\N$ with $n'>n_{j-1}+M$ and $z'\in X_0$ such that 
\begin{itemize}
\item[-] $\|T^{k}z'\|<2^{-j}$ for every $0\le k\le (1+\alpha )n_{j-1}+M$;
\item[-] $\|T^{n'+k}z'-T^{k}(y_{j}-\sum_{i<j}z_{i})\|<2^{-j}$ for every $0\le k\le \alpha n'$.
\end{itemize}
Choose now $0\le p\leq M$ such that $n_j:=n'-p$ belongs to the set $D$, and let $z_j:=T^pz'$. Then (iii') holds true by the choice of $n_j$, 
(i) clearly holds true, and (ii) holds true as well because 
$T^{n_j+k}z_j=T^{n'+k}z'$ for every $k\geq 0$ and  every $n'\ge n_j$. This concludes the inductive step.
\par\smallskip
Let us prove that the vector $z:=\sum_{i\ge 1}z_{i}$ is a $\mathcal{U}$-frequently hypercyclic vector for $T$. This time, we write for every $l\geq 1$, $j\in I_l$ and $k\geq 0$:
\begin{eqnarray*}
T^{n_j+k}z-x_l &=&T^{n_j+k}\Bigl(\sum_{i<j}z_i\Bigr)-T^k\Bigl(\sum_{i<j}z_i\Bigr)
+T^{n_j+k}z_j-T^k\Bigl(y_j-\sum_{i<j}z_i\Bigr)\\
& &+T^kx_l-x_l
+\sum_{i>j} T^{n_j+k}z_i.
\end{eqnarray*}
Since $n_j$ belongs to $ D$, we deduce that if $0\le k\leq \alpha n_j$, then
\[\Vert T^{n_j+k}z-x_l\Vert\leq 2^{-j}+2^{-j}+\Vert T^kx_l-x_l\Vert+2^{-j}=3\cdot2^{-j}+\Vert T^kx_l-x_l\Vert.\]
Now, since $x_l$ is uniformly recurrent for $T$, the set 
\[D_{l,\varepsilon}:=\{ k\ge 1;\; \Vert T^kx_l-x_l\Vert<\varepsilon/2\}\]
has bounded gaps. So one can find a constant $c_{l,\varepsilon}>0$ such that $\# (D_{l,\varepsilon}\cap J)\geq c_{l,\varepsilon} \,\# J$ for 
all sufficiently large intervals $J\subseteq\N$. 
From this, it follows that if $j$ is large enough, then 
\[\#\bigl\{ n_j\leq i\leq (1+\alpha) n_j ;\; \Vert T^iz-x_l\Vert<\varepsilon\bigr\}\geq c_{l,\varepsilon}\, \alpha n_j,\]
and hence that
\[ \overline{\textrm{dens}}\,\mathcal{N}_{T}(z,B(x_{l},\varepsilon ))\geq c_{l,\varepsilon} \,\limsup_{j\to\infty}\frac{\alpha n_j}{(1+\alpha) n_j}
= \frac{c_{l,\varepsilon} \alpha}{1+\alpha}\cdot\]
\end{proof}

Finally, here is a somewhat unexpected consequence of Theorem \ref{bizarre}.

\begin{corollary}\label{bizarre5} Let $T\in\mathfrak B(X)$. Assume that the uniformly recurrent points of $T$ are dense in $X$, and that there also exists a dense set of points
$x\in X$ such that $T^ix\to 0$ as $i\to\infty$. Then $T$ is $\mathcal U$-frequently hypercyclic.
\end{corollary}

\begin{proof} This follows immediately from Theorem \ref{bizarre}.
\end{proof}

\begin{remark} Corollary \ref{bizarre5} shows in particular that if a unilateral weighted backward shift 
on $\ell^p(\N)$ or $c_0(\N)$ has a dense set of uniformly recurrent points, then it is $\mathcal U$-frequently hypercyclic. However, a much stronger result is
true: \emph{if a weighted backward shift has just a single non-zero uniformly recurrent point, then it is in fact chaotic and frequently hypercyclic.}

\begin{proof}
Let $B_{\om}$ be a weighted backward shift on $X=c_0(\N)$ or $\ell_p(\N)$, $1\le p<\infty$, associated to the weight sequence $ \om=(\omega_k)_{k\geq 1}$, and assume that $B_{\om}$ has a non-zero uniformly recurrent point $x\in X$. By the classical Frequent Hypercyclicity/Chaoticity Criterion (see \cite{BM} or \cite{GEP}), it is enough to show that the sequence $\Bigl(\prod_{k=1}^n \omega_k^{-1}\Bigr)$ belongs to the space $X$. We check this in the case where $X=c_0(\N)$, the $\ell_p$ case being similar. So we have to show that 
$\prod_{k=1}^n \omega_k\to\infty$ as $n\to\infty$.
Since $x=\sum_{k\geq 1} x_k e_k$ is non-zero, there exists $k_{0}\ge 1$ such that $|x_{k_{0}}|>0$. Let $0<\varepsilon<|x_{k_{0}}|/2$. Since $x$ is uniformly recurrent for $B_{\om}$, 
there exists a strictly increasing sequence of integers $(n_j)_{j\ge 1}$ and a positive integer $M$ such that $n_{j+1}-n_j\le M$ and $\|B^{n_j}_{\om}x-x\|<\varepsilon$ for every $j\ge 1$. 
In particular, we have 
\[
\left|\left(\prod_{k=k_{0}+1}^{k_{0}+n_j}\omega_{k}\right)x_{k_{0}+n_j}-x_{k_{0}}\right|<\varepsilon 
\]
 and hence
\[ 
\prod_{k=k_{0}+1}^{k_{0}+n_j}|\omega_{k}|> \dfrac{\varepsilon}{|x_{k_{0}+n_j}|}\cdot
\] 
Since $x$ belongs to $c_0(\N)$, we deduce that
\[ \prod_{k=k_{0}+1}^{k_{0}+n_j}|\omega_{k}|\to\infty\quad\textrm{ as } j\to\infty.\]
Moreover, if $n$ is a sufficiently large integer, there exists an integer $j\ge 1$ such that $k_{0}+n_{j-1}\leq n<k_{0}+n_{j}$. Since $n_{j+1}-n_j\le M$, this implies that
\[\prod_{k=k_{0}+1}^{n}|\omega_{k}|\ge \frac{\prod_{k=k_{0}+1}^{k_{0}+n_{j}}|\omega_{k}|}{\|w\|^M_{\infty}},\]
and this concludes the proof. 
\end{proof}
\end{remark}

\begin{remark} A natural question that comes to mind in view of Corollary \ref{bizarre5} is whether any \emph{chaotic and topologically mixing} operator has to be $\mathcal U$-frequently 
hypercyclic. We will prove in Section~\ref{ex mixing} that it is not the case.
\end{remark}
\par\smallskip

\subsubsection{More about $\mathcal U$-frequent hypercyclicity and c(T)} Corollary \ref{bizarre5} says in essence that lots of uniformly recurrent points plus lots of orbits tending to $0$ imply $\mathcal U$-frequent hypercyclicity. 
In the same spirit, we now prove the following result, which again may look rather surprising at first sight.

\begin{theorem}\label{10000} Let $T\in\mathfrak B(X)$ be hypercyclic, and assume that $T$ has a dense set of uniformly recurrent points. Then $T$ is $\mathcal U$-frequently hypercyclic if and only if $c(T)>0$.
\end{theorem}
\begin{proof} One implication is clear: if $T$ is $\mathcal U$-frequently hypercyclic, then $c(T)>0$ by the very definition of $c(T)$.

\smallskip
Conversely, assume that $c(T)>0$. Then there is a comeager set $G\subseteq X$ such 
that, for every $x\in G$, one can find a subset $D_{x}$ of $\N$ with 
$\overline{\textrm{dens}}\,
D_{x}\ge c(T)$ such that $\Vert T^{n}x\Vert\to 0$ as $n\to\infty$ along $D_{x}$ 
(see \cite[Prop. 4.7]{GM}). The key step of the proof lies in the next fact, where we use
for convenience the following notation: if $V$ is an open subset of $X$ and $B=B(u,\varepsilon)$ is an open ball of $X$, we write $B\prec V$ if there exists $\varepsilon'>\varepsilon$
such that $B(u,\varepsilon')\subseteq V$.

\begin{fact}\label{Fact 10001}
 Let $x_{0}\in X$ be a uniformly recurrent point for $T$, and let $V$ be an open neighbourhood of $x_0$. Let also $B$ be an open ball with center $x_0$ such 
 that $B\prec V$, and choose an integer 
 $N$ such 
 that every interval $I\subseteq \N$ of cardinality at least $N$ intersects $\mathcal N_T(x_0,B)$. Then, for every 
 $y\in x_{0}+G$, we have $\overline{\textrm{dens}}\,\mathcal N_T(y,V)\geq c(T)/N$.
\end{fact}

\begin{proof}[Proof of Fact \ref{Fact 10001}]
Write $y=x_{0}+x$, with $x\in G$. Since  $\overline{\textrm{dens}}\,
D_{x}\ge c(T)$ and $\bigcup_{k=0}^{N-1} \bigl(\mathcal N_T(x_0,B)-k\bigr)=\N$ by the choice of $N$, one can find $k\in\{ 0,\dots ,N-1\}$  
such that 
$\overline{\textrm{dens}}\,
\bigl(D_{x}\,\cap\, (\mathcal N_T(x_0,B)-k)\bigr)\ge c(T)/N$; equivalently, $\overline{\textrm{dens}}\,
\bigl((D_{x}+k)\,\cap\, \mathcal N_T(x_0,B)\bigr)\ge c(T)/N$. Now, we have $T^{n}y=
T^{n}x_{0}+T^{n}x$ for all $n\in\N$, and $T^nx\to 0$ as $n\to\infty$ along $D_x+k$. Since $B\prec V$, it follows that all but finitely many integers
$n\in (D_{x}+k)\cap \mathcal N_T(x_0,B)$ belong to $\mathcal N_T(y, V)$, and hence that $\mathcal N_T(y,V)\geq c(T)/N$.
\end{proof}

It is now easy to conclude the proof of Theorem \ref{10000}. Let $(V_p)_{p\geq 1}$ be a countable basis of (non-empty) open sets for $X$.  Choosing for each $p\ge 1$ a uniformly recurrent point 
$x_p\in V_p$ and an open ball $B_p$ with center $x_p$ such that $B_p\prec V_p$, we see that the following holds true: for each $p\geq 1$, there is an integer $N_p$ and a 
comeager set $G_p\subseteq X$ (namely, $G_p=x_p+G$) 
such that 
$ \overline{\textrm{dens}}\, \mathcal N_T(z, V_p)\geq c(T)/N_p$ for every $z\in G_{p}$. Then every vector $z$ in the comeager set $G_\infty:=\bigcap_{p\geq 1} G_p$ is a $\mathcal U$-frequently hypercyclic vector for $T$. 
\end{proof}

\begin{remark} In Theorem \ref{10000}, one cannot replace the assumption of {$\mathcal U$-frequent} hypercyclicity with that of frequent hypercyclicity. Indeed, as we shall see in Theorem \ref{Theorem 58}, there exist operators which are chaotic and 
$\mathcal U$-frequently hypercyclic (hence, hypercyclic with a dense set of uniformly recurrent points and such that $c(T)>0$) but not frequently hypercyclic.
\end{remark}

\smallskip
When the operator $T$ is \emph{chaotic}, the proof of Theorem \ref{10000} gives a more precise statement, which says that if $c(T)>0$, then $T$ is ``$\mathcal U$-frequently hypercyclic with estimates".
\begin{corollary}\label{10003} Let $T$ be a chaotic operator on $X$ with $c(T)>0$. For any open set $V\neq\emptyset$ in $X$, let us denote by $N(V)$ the smallest period of all periodic vectors of $T$ belonging to $V$. Then there is a comeager set of vectors $z\in X$ such that 
\[ \overline{\textrm{dens}}\,\mathcal{N}_{T}(z,V )\geq \frac1{N(V)}\, {c(T)}\qquad\hbox{for every open set $V\neq\emptyset$}.\]
\end{corollary} 

\begin{proof} Note that if the point $x_0$ in Fact \ref{Fact 10001} is a periodic point of $T$, 
then one can take as $N$ the period of $x_0$. Then follow the proof of Theorem \ref{10000}.
\end{proof}

\smallskip 
From Corollary \ref{10003}, we immediately deduce

\begin{corollary}\label{10001000} If $T\in\mathfrak B(X)$ is chaotic and \emph{ergodic}, then there is a comeager set of vectors $z\in X$ such that 
$\overline{\textrm{dens}}\,\mathcal{N}_{T}(z,V )\geq1/N(V)$ for every non-empty open subset $V$ of $X$.
\end{corollary}

\begin{proof} 
This is clear since $c(T)=1$ for any ergodic operator $T$.
\end{proof}

\begin{remark} All vectors $z\in X$ satisfying the conclusion of Corollary \ref{10001000} are $\mathcal U$-frequently hypercyclic for $T$, but none of them is frequently hypercyclic. Indeed, if $V_0$ 
is any open ball centered at $0$, then $N(V_0)=1$, so that $\overline{\textrm{dens}}\,\mathcal{N}_{T}(z,V_0 )=1$. 
Hence, one must have $\underline{\textrm{dens}}\,\mathcal{N}_{T}(z,V )=0$ for any open set $V$ disjoint from $V_0$.
\end{remark}

\smallskip
To conclude this section, we now proceed to prove a generalization of Corollary \ref{10001000} to ergodic operators 
$T\in\mathfrak B(X)$ which are not necessarily chaotic:  we  require that $T$ belongs to ${\rm SPAN}(X)$, \mbox{\it i.e.} that the unimodular eigenvectors of $T$ span a dense subspace of $X$.
\par\smallskip
We first need to introduce some notation.
To every $\lambda\in\T$, we associate the 
rotation-invariant measure $\nu _{\lambda }$ on $\T$ defined as follows:
\[
\nu _{\lambda }=
\begin{cases}
 \dfrac{1}{N}
\ds\sum_{k=0}^{N-1}\delta _{\{\lambda ^{k}\}}
\quad\textrm{if}\ \lambda ^{N}=1\ \textrm{for some}\ N\ge 1,\\
\textrm{the normalized Lebesgue measure on}\ \T\ \textrm{otherwise.}
\end{cases}
\]

Suppose now that $T\in\mathfrak B(X)$, and that $\mathbf u=(u_1,u_2,\dots )$ is a finite or infinite 
sequence of linearly independent eigenvectors associated to unimodular eigenvalues $\lambda _{1},\lambda_2, \dots$. 
Let us denote by $\mathcal E^{\mathbf u}$ the linear span of the vectors $u_j$, $j\ge 1$. 
Any vector $u\in\mathcal E^{\mathbf u}$ may be written in a unique way as 
\[u=\sum_{k=1}^{r}
a_{k}^{\bf u}(u)u_{k},
\] 
where the coefficients $a_k^{\bf u}(u)$, $1\le k\le r$, are complex scalars. To any such vector $u$, we associate the 
$T$-invariant
measure $\nu _{u,T}^{\mathbf u}$ on $X$ defined by
\[
\nu_{u,T}^{\bf u}(A)=\prod_{k=1}^{r}\nu _{\lambda _{k}}\Bigl(
\bigl\{(\mu _{1},\dots,\mu _{r})\in\T^{r}\,;\,\sum_{k=1}^{r}
a_{k}^{\bf u}(u)\mu _{k}u_{k}\in A\bigr\}\Bigr)
\]
for every Borel subset $A$ of $X$. Note that if the unimodular numbers $\lambda_k$ are not roots of unity (and $X$ is a Hilbert space $\h$), then $\nu_{u,T}^{\bf u}$ 
is nothing but the Steinhaus measure associated to $u$ considered in the proof of Fact \ref{link}.

Finally, for any open set $V\neq\emptyset$, we define 
\[ \delta^{\mathbf u} _{V,T}:=\sup\,\bigl\{\nu _{u ,T}^{\bf u}(B(u,\varepsilon ))\,;\,
\varepsilon >0,\ u\in\mathcal E^{\bf u},\ B(u,\varepsilon)\prec 
V\bigr\}.\]
As in Fact \ref{Fact 10001},  the notation $B(u,\varepsilon)\prec V$ means that $B(u,\varepsilon')\subseteq V$ for some $\varepsilon'>\varepsilon$. 

\smallskip
We may now state

\begin{proposition}\label{Corollary 24}
 Let $T$ be an ergodic operator on $X$ belonging to ${\rm SPAN}(X)$, and let $\mathbf u=(u_k)_{k\geq 1}$ be an infinite linearly independent
 sequence of unimodular eigenvectors whose linear span $\mathcal E^{\bf u}$ is dense in $X$.  Then, there exists a 
comeager subset $G$ of $X$ such that every vector $z\in G$ 
satisfies 
\[
 \overline{\emph{dens}}\,\mathcal{N}_{T}(z,V)\ge \delta^{\mathbf u} 
_{V,T}\quad\hbox{for every open set $V\neq\emptyset$}.
\]
\end{proposition}

\smallskip
The key step in the proof of
Proposition \ref{Corollary 24} is Lemma \ref{Theorem 23} below.

\begin{lemma}\label{Theorem 23}
 Let $T$ be an ergodic operator on $X$, and let $\mathbf u=(u_1,\dots ,u_r)$ be a finite sequence of linearly 
 independent unimodular eigenvectors for $T$. Let also $u$ be a vector belonging to $\mathcal E^{\bf u}$. Given any 
$\varepsilon $, $\gamma >0$, there exists a $T$-invariant measure $m$ 
on 
$X$, ergodic for $T$ and with full support, such that 
\[
m\bigl(B(u,\varepsilon ) \bigr)\ge(1-\gamma )\,\nu _{u,T}^{\bf u}\bigl(
B(u,(1-\gamma )\varepsilon )\bigr).
\]
\end{lemma}

Since the sequence $\bf u$ is fixed, we will remove any reference to it in the proofs of Proposition \ref{Corollary 24} and Lemma \ref{Theorem 23}; so we will write for instance $a_k(u)$ instead of $a_k^{\bf u}(u)$, 
and so on.

\begin{proof}[Proof of Lemma \ref{Theorem 23}] 
 Let $m_{0}$ be an ergodic measure with full support for $T$. For each 
 $R>0$, consider the probability measure $m_{R}$ defined on $X$ by setting
 \[
m_{R}(A)=\int_{K_{u}}m_{0}\bigl(R(A-x) \bigr)\,d\nu _{u,T}(x)\quad 
\textrm{for every 
Borel subset}\ A\  \textrm{of}\ X,
\]
 where $K_{u}:=
\bigl\{\sum_{k=1}^{r}a_{k}(u)\mu _{k}u_{k},\;\,\mu _{k}\in\T,\ 1\le 
k\le r\bigr\}$. Note that $K_u$ is a compact subset of $X$ and contains  the support of the measure $\nu _{u,T}$. It 
clearly satisfies $T(K_{u})=K_{u}$. The measure $m_{R}$ thus defined is 
$T$-invariant. Indeed, we have for every Borel subset $A$ of $X$
\begin{align*}
m_{R}(T^{-1}A)&=\int_{K_{u}}m_{0}\bigl(R(T^{-1}A-x) \bigr)\,d\nu 
_{u,T}(x)= \int_{K_{u}}m_{0}\bigl(R(T^{-1}(A-Tx)) \bigr)\,d\nu_{u,T}(x) \\
&=\int_{T(K_{u})}m_{0}\bigl(R(A-x) \bigr)\,d\nu 
_{u,T}(x)=m_{R}(A).
\end{align*}
The next step of the proof is to show that $m_{R}(\hct)=1$. Since 
\[
m_{R}(\hct)=\int_{K_{u}}m_{0}\bigl(R(\hct-x) \bigr)\,d\nu 
_{u,T}(x)
\]
and $m_{0}(\hct)=1$, it suffices to show that $\hct\subseteq \hct-x$ for 
every vector $x\in K_{u}$, \mbox{\it i.e.} that $\hct+x\subseteq \hct$ for every such $x$. Write $x$ as 
$x=\sum_{k=1}^{r}a_{k}\mu 
_{k}u_{k}$, where $a_k=a_k(u)$ and $\mu _{k}$ belongs to $\T$ for every $1\le k\le r$, and fix a vector $z\in\hct $. In order to show that $z+x$ belongs to $\hct$, we need to 
show that for every $y\in X$ and every $\delta >0$, there exists an 
integer $\gn$ such that $\|T^{n}(z+x)-y\|<\delta $. Now, a result of 
Shkarin \cite{S} states the following: given a hypercyclic operator $S$ on 
a Banach space $Z$, and a compact topological group $G$ generated by an 
element $g\in G$, the set 
\[
\bigl\{(S^{n}z, g^{n}),\ \gn\bigr\}
\]
is dense in $Z\times G$ for every vector $z\in Z$ which is hypercyclic for 
 $S$. We apply this result to the operator $T$ and to the subgroup $G$ 
of 
 $\T^{r}$ generated by $g:=(\lambda _{1},\dots,\lambda _{r})$. 
Since the $r$-tuple $(1,\dots,1)$ belongs to $G$, and $z\in X$ is a
hypercyclic vector for $T$, there exists for every $\delta >0$ an integer
$\gn$ such that $\|T^{n}z-(y-x)\|<\delta /2$ and 
$\max_{1\le k\le r}|a_{k}|\,.\,\|u_{k}\|\,.\,|\lambda_{k}^{n}-1|<\delta /(2r)$. 
It follows that 
\[
\|T^{n}x-x\|=\Bigl|\Bigl|\,\sum_{k=1}^{r}a_{k}\bigl(\lambda _{k}^{n}-1 
\bigr)u_{k}\, \Bigr| \Bigr|<\dfrac{\delta }{2},
\]
so that $\|T^{n}(z+x)-y\|<\delta $. This proves our claim, and shows that 
$m_{R}(\hct)=1$ for every $R>0$.
\par\smallskip 
Let us now estimate from below the quantities $m_{R}( B(u,\varepsilon 
))$. By the definition of $m_R$, we have
\[
m_{R}( B(u,\varepsilon))\ge \int_{K_{u}\,\cap\, B(u,(1-\gamma )\varepsilon 
)}
m_{0}(R(B(u,\varepsilon )-x))\,d\nu _{u,T}(x).
\]
Now, observe that for every $x\in B(u,(1-\gamma )\varepsilon )$, the 
set $B(u,\varepsilon )-x$ contains the ball $B(0,\gamma \varepsilon )$. 
Indeed, every $y\in X$ with $\|y\|<\gamma \varepsilon $ can be written as 
$y=u+(y-u+x)-x$ with $\|y-u+x\|<\gamma \varepsilon +(1-\gamma )\varepsilon 
=\varepsilon  $. It follows that 
\begin{align*}
 m_{R}( B(u,\varepsilon))&\ge \int_{K_{u}\,\cap\, B(0,(1-\gamma 
)\varepsilon 
)}
m_{0}(B(0,R\gamma \varepsilon ))\,d\nu _{u,T}(x)\\
&=m_{0}(B(0,R\gamma \varepsilon ))\,\cdot\,\nu _{u,T}(B(u,(1-\gamma 
)\varepsilon )).
\end{align*}
Since $m_{0}(B(0,R\gamma \varepsilon ))$ tends to $1$ as $R$ tends to 
infinity, there exists $R_{0}>0$ such that 
\[
m_{R_{0}}(\hct\cap B(u,\varepsilon ))=m_{R_{0}}(B(u,\varepsilon ))
>(1-\gamma )\, \nu  _{u,T}(B(u,(1-\gamma )\varepsilon )).
\]
Applying the \emph{Ergodic Decomposition Theorem} (see \mbox{e.g.} \cite[Th. 2.5]{Sar}) to the measure $m_{R_{0}}$,  
we obtain that there exists an ergodic measure $m$ for $T$ such that
\[
m(\hct\cap B(u,\varepsilon ))>(1-\gamma )\, \nu _{u,T}(B(u,(1-\gamma 
)\varepsilon )).
\]
Since $m(\hct)>0$, the measure $m$ has full support, and this concludes the proof of 
Lemma \ref{Theorem 23}.
\end{proof}

\begin{proof}[Proof of Proposition \ref{Corollary 24}] 
Let $(V_p)_{p\geq 1}$ be a countable basis of (non-empty) open sets for $X$ with the following property: for any open set $V$ and any open ball $B$ such that $B\prec V$, one can find $p\geq 1$ such that $B\prec V_p\subseteq V$. This additional property implies that for any open set $V\neq\emptyset$, we have 
\begin{equation}\label{100}
 \delta_{V,T}=\sup\,\{ \delta_{V_p,T};\; V_p\subseteq V\}.
\end{equation}
For each $p\geq 1$, one may choose a sequence 
$(u_{p,k})_{k\ge 1}$ of vectors of $ \mathcal E^{\bf u}$, a sequence of positive numbers $(\varepsilon_{p,k})_{k\ge 1}$ and a sequence of positive numbers $(\gamma_{p,k})_{k\ge 1}$ tending to $0$ as $k$ tends to infinity such that 
\[ B(u_{p,k},\varepsilon_{p,k})\prec V_p\quad\hbox{for every $k\ge 1$}\;{\rm and}\; \alpha_{p,k}:=\nu_{u_{p,k},T} \bigl(B(u_{p,k}, (1-\gamma_{p,k})\varepsilon_{p,k})\bigr)\underset{k\to\infty}{\longrightarrow} \delta_{V_p,T}.\]

By Lemma \ref{Theorem 23} and the pointwise ergodic theorem, we see that for each fixed pair $(p,k)$ of integers, the set of all vectors $z\in X$ such that 
\[\underline{\textrm{dens}}\,\mathcal N_T\bigl(z, B(u_{p,k},\varepsilon_{p,k})\bigr)\geq (1-\gamma_{p,k}) \alpha_{p,k}\] 
is dense in $X$. In particular, the set 
\[ G_{p,k}:=\bigcap_{N\ge 1}\ \bigcup_{n\ge N}\Bigl\{ z\in X;\; \#\{1\le i\le n\,;\, T^{i}z\in B(u_{p,k},\varepsilon_{p,k} )\}\ge (1-\gamma_{p,k})\,\alpha_{p,k}\Bigr\} \]
is a dense $G_\delta$ subset of $X$. 
Hence
$G:=\bigcap_{p,k\ge 1}G_{p,k}$ is a dense $\gd$ subset of 
$X$ too, and by the definition of $G$, every vector $z\in G$ satisfies 
\[\overline{\textrm{dens}}\,\mathcal N_T(z,V_p)\geq \delta_{V_p,T}\qquad\hbox{for every $p\geq 1$}. \]
Property (\ref{100}) then allows us to conclude the proof of Proposition \ref{Corollary 24}.
 \end{proof}

\subsection{A criterion for frequent hypercyclicity}
We now move over to a criterion of the same kind  as Theorem \ref{Theorem 39} for frequent 
hypercyclicity. Instead of requiring in the assumption that 
$\Vert T^{n+k}z-T^{k}x\Vert $ be small for indices $k$ less than a fraction of 
$n$, we have to require that these quantities should be small for indices $k$ 
less than a fraction of a certain multiple $d$ of the period 
of $z$. The criterion reads as follows:

\begin{theorem}\label{Theorem 41} Let $T\in\mathfrak B(X)$. Assume that there exist
a dense linear subspace $X_0$ of $ X$ with $T(X_{0})\subseteq X_{0}$ and 
 $X_{0}\subseteq \emph{Per}(T)$, and a constant $\alpha \in(0,1)$ such that the following property holds true: 
for every $x\in X_{0}$, every $\varepsilon >0$ and every integer 
$d_{0}
\ge 1$ which is 
the period of some vector $y$ of $X_{0}$, 
there exist $z\in X_{0}$ and integers $n$, $d\ge 1$ such that
\begin{enumerate}
 \item [\emph{(0)}] $d$ is a multiple of $d_{0}$ and of $\emph{per}(z)$;
 \item [\emph{(1)}] $\|T^{k}z\|<\varepsilon $ for every $0\le k\le \alpha 
d$;
\item [\emph{(2)}] $\|T^{n+k}z-T^{k}x\|<\varepsilon $ for every $0\le k\le 
\alpha d$.
\end{enumerate}
Then $T$ is chaotic and frequently hypercyclic.
 \end{theorem}

\begin{remark}\label{Remark 41 bis} The restriction that the integer $d_{0}$ above should 
be the period of some vector $y\in X_{0}$ may 
seem rather artificial. The reason for stating Theorem \ref{Theorem 41} as we did is to be found in 
the proof of Theorem \ref{Theorem 47} below: we will use Theorem \ref{Theorem 41} as stated 
to prove the frequent hypercyclicity of the operators involved there without having to impose 
certain divisibility conditions on the numbers $\Delta ^{(k)}$, $\gk$.
\end{remark}

\begin{remark} The constant $\alpha$ involved in the statement of Theorem \ref{Theorem 41} cannot be greater than $1/2$. Indeed, otherwise one could find a periodic vector $z$ such that more than half of the points in the orbit of $z$ are close to $0$ and more than half of these points are close to some periodic orbit far away from $0$. So at least one point in the orbit of $z$ would have to be both close to $0$ and far away from $0$, which is impossible.
\end{remark}

\begin{proof}[Proof of Theorem \ref{Theorem 41}] 
We first note that upon substracting some multiple of $d$ to $n$ (where $d$ and $n$ are given by the assumptions of Theorem \ref{Theorem 41}),
 we can always assume that $d\geq n$. The assumption of Theorem \ref{Theorem 41} is thus seen 
to be stronger than that of Theorem~\ref{Theorem 39}. 
 Also, we can require that $n>\alpha d$ and the same argument as in the beginning of the proof of Theorem 
\ref{Theorem 39} (taking $\varepsilon$ 
very small) shows that, given $N\geq 1$, we can 
always add to the assumption of Theorem \ref{Theorem 41} the additional hypothesis  that the integer $n$ is a 
multiple of $N$.
\par\smallskip
Let now $(x_{l})_{l\ge 1}$ be a dense sequence of vectors of $X$ contained in $X_{0}$, and let 
$(I_{l})_{l\ge 1}$ be a partition of $\N$ such that each set $I_{l}$ is 
infinite and has bounded gaps. We denote by $r_{l}$ the maximum size of a 
gap between two successive elements of $I_{l}$. As usual, we define the vectors
$y_{j}$, $j\ge 1$, by setting $y_{j}=x_{l}$ for every $j\in I_{l}$. 
By induction 
on $j\ge 1 $, we construct a sequence $(z_{j})_{j\ge 1}$ of vectors of 
$X_{0}$ and two strictly increasing sequences of integers 
$(d_{j})_{j\ge 1}$ and $(n_{j})_{j\ge 1}$ such that 
\begin{enumerate}[(i)]
 \item $d_{j}$ is a multiple of $\textrm{per}\,(\sum_{i=1}^{j-1}z_{i})$ 
and of $\textrm{per}(z_{j})$;
\item $\|T^{k}z_{j}\|<2^{-j}$ for every $0\le k\le \alpha d_{j}$;
\item $\|T^{n_{j}+k}z_{j}-T^{k}(y_{j}-\sum_{i=1}^{j-1}z_{i} 
)\|<2^{-j}$ for every $0\le k\le \alpha d_{j}$;
\item $n_{j}$ is a multiple of $\textrm{per}(\sum_{i=1}^{j-1}
z_{i})$ and $\alpha d_j<n_{j}\le d_{j}$;
\item $\alpha d_{j}>4d_{j-1}$.
\end{enumerate}
\par\smallskip
We set $x=\sum\limits_{i\ge 1}z_{i}$, and show that $z$ is a frequently 
hypercyclic vector for $T$. 
\par\smallskip
Let us fix $l\ge 1$, and write $I_{l}$ as 
  $I_{l}=\{j_{m}\,;\,m\ge 1\}$, where $(j_{m})_{m\ge 1}$ is strictly 
increasing and $j_{m+1}-j_{m}\le r_{l}$ for every $m\ge 1$. For each 
$m\ge 1$, we define a family of sets $(A_{m,j})_{0\le j<j_{m+1}-j_{m}}$ as follows:
\par\smallskip
\[
 A_{m,0}:=\Bigl\{n_{j_{m}}+kd_{j_{m}}+k'\textrm{per}(x_{l})\,;\,
0\le k'\le \dfrac{\alpha d_{j_{m}}}{\textrm{per}(x_{l})},\ 0\le k\le 
\dfrac{\alpha d_{j_{m}+1}}{d_{j_{m}}}-2\Bigr\} ,
\]
{and, for $1\le j<j_{m+1}-j_{m}$,}
\par\smallskip
\[A_{m,j}:=\bigcup_{1\leq k\leq \frac{\alpha d_{j_{m}+j+1}}{d_{j_{m}+j}}-1} \Bigl(A_{m,j-1}+kd_{j_m+j}\Bigr).
\]
Before doing anything with these sets, we note that 
\begin{equation}\label{maxam}
\max\;A_{m,j}\leq \alpha d_{j_m+j+1}.
\end{equation}
Indeed, for $j=0$ we have $\max \, A_{m,0}\leq n_{j_m}+ \alpha d_{j_m+1}-2d_{j_m}+\alpha d_{j_m}\leq\alpha d_{j_m+1}$ 
because $n_{j_m}\leq d_{j_m}$; and the result then follows by a straightforward induction on $j<j_{m+1}-j_m$.

\begin{fact}\label{estimate} For every $m\geq 1$ and every $n\in A_{m,j}$, $0\leq j<j_{m+1}-j_m$, we have
\[\Vert T^nz-x_l\Vert\leq 2^{-(j_m-1)}.
\]
\end{fact}

\begin{proof}[Proof of Fact \ref{estimate}]
For any $n\in A_{m,j}$, we have
\begin{equation}\label{Equation 3'}
 \|T^{\,n}z-x_{l}\|\le \Bigl|\Bigl|\, 
T^{\,n}\Bigl(\sum_{s=1}^{j_{m}+j}z_{s} \Bigr)-x_{l}\,\Bigr| 
\Bigr|+\sum_{s>j_{m}+j}\|T^{\,n}z_{s}\|.
\end{equation}
Since $n\le \alpha d_{j_{m}+j+1}\le \alpha d_{s}$ 
for every $s>j_{m}+j$ by (\ref{maxam}), the second term in (\ref{Equation 3'}) is easy to control: 
\[
\sum_{s>j_{m}+j}\|T^{\,n}z_{s}\|<\sum_{s>j_{m}+j}2^{-s}.
\]
We now have to estimate the term $\ds\Bigl|\Bigl|\, 
T^{\,n}\Bigl(\sum_{s=1}^{j_{m}+j}z_{s} 
\Bigr)-x_{l}\,\Bigr| \Bigr|$. The 
index $m$ being fixed, 
we show by induction on $0\le j<j_{m+1}-j_{m}$ that 
\begin{equation}\label{Equation 4}
 \Bigl|\Bigl|\,T^{\,n}\Bigl(\sum_{s=1}^{j_{m}+j}z_{s} 
\Bigr)-x_{l}\,\Bigr| \Bigr|<\sum_{u=0}^{j}2^{-(j_{m}+u)}.
\end{equation}
- Suppose first that $n$ belongs to $A_{m,0}$, so that 
\[n=n_{j_{m}}+kd_{j_{m}}+k'\,\textrm{per}(x_{l})\quad{\rm with }\quad
0\le k\le 
\frac{\alpha d_{j_{m}+1}}{d_{j_{m}}}-2\quad\textrm{and}\quad
0\le k'\le\dfrac{\alpha d_{j_{m}}}{\textrm{per}x_{l}}\cdot 
\]
Then
\begin{align*} 
T^{\,n}\Bigl(\sum_{s=1}^{j_{m}}z_{s}\Bigr)-x_{l}&=
T^{\,n_{j_{m}}+k'\textrm{per }(x_{l})}\,\Bigl(\sum_{s=1}^{j_{m}}z_{s} \Bigr)-x_{l}\qquad\hbox{by (i)}\\
&=T^{\,n_{j_{m}}+k'\textrm{per}(x_{l})}\,z_{j_{m}}-T^{\,k'\textrm{per}(x_{l} ) }
\,\Bigl(x_{l}-\sum_{s=1}^{j_{m}-1}z_{s}\Bigr)\qquad\hbox{by (iv).}
\end{align*}
Since $k'\textrm{per}(x_{l})\le \alpha d_{j_{m}}$, we deduce from (iii) 
that 
$\ds\Bigl|\Bigl|\, 
T^{\,n}\Bigl(\sum_{s=1}^{j_{m}}z_{s} 
\Bigr)-x_{l}\,\Bigr| \Bigr|\le 2^{-j_{m}}$.
\par\medskip 
\noindent
- Suppose now that (\ref{Equation 4}) has been proved up to the index $j-1$ for some $1\le j<j_{m+1}-j_{m}$, and that $n$ belongs to $A_{m,j}$. 
Write 
\[n=kd_{j_{m}+j}+i\qquad{\rm with}\quad i\in A_{m,j-1}\quad{\rm and}\quad 0\le k\le \dfrac{\alpha 
d_{j_{m}+j+1}}{d_{j_{m}+j}}-1\cdot
\] Then
\begin{align*}
 T^{\,n}\Bigl(\sum_{s=1}^{j_{m}+j}z_{s}\Bigr)-x_{l}&=T^{\,kd_{j_{m}+j}+i}\,
 \Bigl(\sum_{s=1}^{j_{m}+j}z_{s} \Bigr)-x_{l}
 =T^{\,i}\,\Bigl(\sum_{s=1}^{j_{m}+j} z_{s}\Bigr)-x_{l}\qquad\hbox{by (i)}\\
 &\smash[b]{=T^{\,i}\,\Bigl(\sum_{s=1}^{j_{m}+j-1}z_{s} 
\Bigr)-x_{l}+T^{\,i}z_{j_{m}+j}.}
\end{align*}
{Thus}
\begin{align*}
\smash{\Bigl|\Bigl|T^{\,n}\Bigl(\sum_{s=1}^{j_{m}+j}z_{s}\Bigr)-x_{l}
\Bigr|\Bigr|}
&\smash[t]{<\sum_{u=0}^{j-1}2^{-(j_{m}+u)}+\|T^{\,i}z_{j_{m}+j}\|}
\end{align*}
by the induction hypothesis. Now, $i\le \alpha d_{j_{m}+j}$ by (\ref{maxam}) since $i$ 
belongs to $A_{m,j-1}$, and thus 
$\|T^{\,i}z_{j_{m}+j}\|<2^{-(j_{m}+j)}$ by (i). Hence
\[
\Bigl|\Bigl|T^{\,n}\Bigl(\sum_{s=1}^{j_{m}+j}z_{s}\Bigr)-x_{l}\Bigr|\Bigr|
<\sum_{u=0}^{j}2^{-(j_{m}+u)},
\]
which concludes our induction. The inequality (\ref{Equation 4}) being 
proved, we deduce that 
\[
\|T^{\,n}z-x_{l}\|<\sum_{u\ge 0}2^{-(j_{m}+u)}=2^{-(j_{m}-1)}
\]
for every $m\ge 1$ and every $n\in A_{m,j}$, $0\le j< j_{m+1}-j_{m}$. 
\end{proof} 

It 
follows from Fact \ref{estimate} that given $\varepsilon >0$, there exists $m_{0}\ge 1$ such that 
$\mathcal{N}_{T}(z,B(x_{l},\varepsilon ))$ contains the set
\[
\bigcup_{m\ge m_{0}}\,\bigcup_{j=0}^{j_{m+1}-j_{m}-1} A_{m,j}.
\] 

So in order to show that $z$ is a frequently hypercyclic vector for 
$T$, it suffices to prove the following fact.

\begin{fact}\label{densityA} The set 
$A:=\bigcup\limits_{m\ge 1}\,\bigcup\limits_{j=0}^{j_{m+1}-j_{m}-1} A_{m,j}$
has positive lower density. 
\end{fact}

\begin{proof}[Proof of Fact \ref{densityA}]
We start with a 
series of elementary remarks on the structure of the sets $A_{m,j}$.
\begin{enumerate}[(a)]
 \item $A_{m,0}\subseteq [n_{j_{m}},\alpha d_{j_{m}+1}]$ and 
 $A_{m,j}\subseteq[d_{j_{m}+j},\alpha d_{j_{m}+j+1}]$ for $1\le j< 
j_{m+1}-j_{m}$, so that the sets $A_{m,j}$ are contained in successive 
(disjoint) subintervals of $\N$.
\item Inside each set $A_{m,j}$, $j\geq 1$, the translates $A_{m,j-1}+kd_{j_m+j}$ are pairwise disjoint.
\item If we set $A_{m}:=\bigcup_{j=0}^{j_{m+1}-j_{m}-1}A_{m,j}$, then $A_m\subseteq
[n_{j_{m}},\alpha d_{j_{m+1}}]$. Since $\alpha d_{j_{m+1}}< n_{j_{m+1}}$, the 
sets $A_{m}$, $m\ge 1$, are thus contained in successive subintervals of $\N$;
\item There is no redundancy in the definition of $A_{m,0}$: if
\begin{align*}
\smash{n_{j_{m}}+kd_{j_{m}}+k'\,\textrm{per}(x_{l})}&=\smash{n_{j_{m}}+\ti{
k}d_{j_
{m}} +\ti{k'}\,\textrm{per}(x_{l})}
\intertext{for some}
0\le k,k'&\smash[t]{\le \dfrac{\alpha d_{j_{m}+1}}{d_{j_{m}}}-2,\quad
0\le k',\ti{k'}\le\dfrac{\alpha d_{j_{m}}}{\textrm{per}(x_{l})},}
\end{align*}
then $k=\ti{k}$ and $k'=\ti{k'}$. Indeed, we have $(k-\ti{k})d_{j_{m}}=(k'-\ti{k'})\textrm{per}(x_{l})$,
so that $|k-\ti{k}|d_{j_{m}}\le\alpha d_{j_{m}}$, and hence 
$k=\ti{k}$ because $\alpha<1$. 
\item It follows from (v) that
\[
\#A_{m,0}\ge \dfrac{\alpha d_{j_{m}}}{\textrm{per}(x_{l})}\Bigl(
\dfrac{\alpha d_{j_{m}+1}}{d_{j_{m}}}-1\Bigr)
\ge \dfrac{\alpha d_{j_{m}}}{\textrm{per}(x_{l})}\,\cdot\,\dfrac{\alpha 
d_{j_{m}+1}}{2d_{j_{m}}}=\dfrac{\alpha 
^{2}d_{j_{m}+1}}{2\textrm{per}(x_{l})}\cdot
\]
\item By (a) and (v), we have for every $1\le j< j_{m+1}-j_{m}$:
\par\smallskip
\[ \# A_{m,j}\,\smash{\ge \Bigl(\dfrac{\alpha d_{j_{m}+j+1}}{d_{j_{m}+j}}-1 
\Bigr)
\,\#A_{m,j-1}\ge\dfrac{\alpha d_{j_{m}+j+1}}{2d_{j_{m}+j}}\,\# A_{m,j-1}.}
\]
\par\bigskip\noindent
{Iterating this inequality yields that}
\par
\[\# A_{m,j}\,\smash{\ge\Bigl(\dfrac{\alpha }{2} 
\Bigr)^{j}\,\dfrac{d_{j_{m}+j+1}}{d_{j_{m}+1}}\,\#A_{m,0}\ge
\dfrac{\alpha 
^{j+2}}{2^{j+1}}\cdot\dfrac{d_{j_{m}+j+1}}{\textrm{per}(x_{l})}\cdot
}
\]
\par\medskip\noindent
{Since $j_{m+1}-j_{m}\le r_{l}$, this eventually gives}
\[\# A_{m,j}\,\smash{\ge\dfrac{\alpha 
^{r_{l}+2}}{2^{r_{l}+1}}\cdot\dfrac{d_{j_{m}
+j+1}}{\textrm{per}(x_{l})}\quad \textrm{for every}\ 1\le j<j_{m+1}-j_{m}.
}\]
\end{enumerate}
\par\medskip
In order to prove that $A$ has positive lower density, we will show the 
existence of some positive number $\delta $ such that 
\begin{equation}\label{delta}
\frac{1}{n}\,\#([1,n]\cap A)\ge \delta\qquad\hbox{for every $n\in A$}.
\end{equation}
 This will be enough 
to ensure that $A$ has positive lower density. 
Indeed, if we enumerate the set $A$ as $A=\{a_{q}\,;\,q\ge 1\}$ where 
$(a_{q})_{q\ge 1}$ is strictly increasing, this 
inequality can be rewritten as $q/a_{q}\ge \delta $ for every $q\ge 1$, 
\mbox{\it i.e.} $a_{q}\le q/\delta $; and this is easily seen to imply that $A$  has 
positive lower density.
\par\smallskip
We fix $m\ge 1$ and $n\in A_m$, and consider separately two cases.
\par\smallskip \noindent 
\par\noindent- Suppose that $n\in A_{m,0}$, and write
\[
\smash{n=n_{j_{m}}+kd_{j_{m}}+k'\textrm{per}(x_{l}),\quad\textrm{where}\ 
0\le k\le \dfrac{\alpha d_{j_{m}+1}}{d_{j_{m}}}-2\ \textrm{and}\ 
0\le k'\le \dfrac{\alpha d_{j_{m}}}{\textrm{per}(x_{j})}}\cdot 
\]
Note that the set $A_{m-1}$ of (c) is contained in $[1,n]$ because $n\geq n_{j_m}>\alpha d_{j_{m}}$, and $A_{m-1}$
is disjoint from $A_{m,0}$. So we have in particular
\[
  \#\,\bigl([1,n]\cap A\bigr)\ge \#\,A_{m-1,j_{m}-j_{m-1}-1}+\#\,\bigl( [1,n]\cap A_{m,0}\bigr).
\]
\smallskip
Moreover, the set $[1,n]\cap A_{m,0}$ contains all integers of the form $ n_{j_m}+sd_{j_m}+s'{\rm per}(x_l)$, where $0\leq s<k$ and
$0\leq s'\leq \frac{\alpha d_{j_m}}{{\rm per}(x_l)}$ (since all of them are in $A_{m,0}$ and the largest is not greater than 
$n_{j_m}+(k-1)d_{j_m}+\alpha d_{j_m}\leq n_{j_m}+kd_{j_m}\leq n$). Hence, 
\begin{align*} \#\,\bigl( [1,n]\cap A_{m,0}\bigr)&\ge k\,\cdot\, \frac{\alpha d_{j_m}}{{\rm per}(x_l)}\cdot
\end{align*}
By (f), it follows that
\[
\#\,\bigl( [1,n]\cap A\bigr)\ge\dfrac{\alpha 
^{r_{l}+2}d_{j_{m}}}{2^{r_{l}+1}\textrm{per}(x_{l})}+k\dfrac{\alpha 
d_{j_{m}}}
{\textrm{per}(x_{l})}\ge\dfrac{\alpha 
^{r_{l}+2}}{2^{r_{l}+1}}\cdot\dfrac{(k+1)d_{j_{m}}}{\textrm{per}(x_{l})}\cdot
\]
Since $n\leq n_{j_m}+(k+\alpha)d_{j_m}\leq 2(k+1)d_{j_m}$, we conclude that
\[\#\,\bigl( [1,n]\cap A\bigr)
\ge\dfrac{\alpha ^{r_{l}+2}}{2^{r_{l}+2}\textrm{per}(x_{l})}\,n.
\]
\par\medskip 
\noindent - Suppose now that $n\in A_{m,j}$ for some $1\le 
j<j_{m+1}-j_{m}$, so that 
\[n=kd_{j_{m}+j}+i,\quad{\rm with}\;1\le k\le \frac{\alpha d_{j_{m}+j+1}}{
d_{j_{m}+j}}-1\;{\rm and}\; i\in A_{m,j-1}.
\]
 Then 
 $\max\, A_{m,j-1}\leq n\le (k+1)d_{j_{m}+j}$ by (\ref{maxam}). In particular,
\[
 \#\,\bigl([1,n]\cap A\bigr)\ge\#\,A_{m,j-1}+\#\,\bigl([1,n]\cap A_{m,j}\bigr).
 \]
 \smallskip
{Moreover, since $n\geq kd_{j_{m}+j}\geq (k-1)d_{j_m+j}+\max\, A_{m,j-1}$} and the translates 
$A_{m,j-1}+l d_{j_m+j}$, $l< k$ are pairwise disjoint (and contained in $A_{m,j}$), we also have
\[  \#\,\bigl([1,n]\cap A_{m,j}\bigr)\ge (k-1)\, \#\,A_{m,j-1}.
\]
{By (f), it follows that}
\[ \#\,\bigl( [1,n]\cap A\bigr)\smash{\ge k\,\dfrac{\alpha 
^{r_{l}+2}d_{j_{m}+j}}{2^{r_{l}+1}\textrm{per}(x_{l})}\ge
\dfrac{1}{2}\cdot\dfrac{\alpha 
^{r_{l}+2}}{2^{r_{l}+1}}\cdot\dfrac{1}{\textrm{per}(x_{l})}\cdot\,n,}
\]
\smallskip\noindent
where we have used the inequality $n\leq (k+1)d_{j_{m}+j}$.
\par\medskip
{We have thus proved that}
\[
\#\,\bigl([1,n]\cap A\bigr)\smash{\ge\dfrac{\alpha ^{r_{l}+2}}{2^{r_{l}+2}\textrm{per}(x_{l})}\,n}\qquad\hbox{ for every $n\in A$,}
\]
\par\medskip\noindent 
which concludes the proof of Fact \ref{densityA}. \end{proof}
The 
proof of Theorem \ref{Theorem 41} is now complete.
\end{proof}
\par\smallskip

\subsubsection{Link with the Operator Specification Property}\label{SECTION OSP}
We conclude this section by a result which shows that many operators 
which are known to be frequently hypercyclic satisfy the assumption of 
Theorem \ref{Theorem 41}. This is the case for all operators which satisfy 
the Frequent Hypercyclicity Criterion and, more generally, for operators 
with the so-called \emph{Operator Specification Property}. 
\par\smallskip
This last property, which has been recently introduced and studied 
 by Bartoll, Mart\'{\i}nez-Gim\'{e}nez, and Peris (\cite{BMPe1}, \cite{BMPe2}), is 
the linear version of the classical \emph{Specification Property} for compact dynamical 
systems introduced by Bowen in \cite{Bo}. The definition reads  
 as follows: an operator $T\in\mathfrak{B}(X)$ has the {Operator 
Specification Property} (OSP) if there exists an increasing sequence 
$(K_{m})_{m\geq 1}$ of $T$-invariant subsets of $X$ with $0\in K_{1}$, the 
union of which is dense in $X$, such that for each $m\geq 1$, the restriction of $T$ 
to $K_m$ has the Specification Property in the sense of \cite{Bo}, which means that 
the following holds true:
\par\medskip 
\noindent $(*)$\label{Etoile}\quad for every $\delta >0$, there exists an 
integer
$N_{\delta,m }\ge 1$ such that for every finite family $y_{1},\dots,y_{s}$ 
of 
points of $K_{m}$, and any integers $0= j_{1}\le k_{1}<j_{2}\le k_{2}<
\dots<j_{s}\le k_{s}$ with $j_{r+1} -k_{r}\ge N_{\delta ,m}$ for every 
$1\le r\le s-1$, there is a point $x\in K_{m}$ such that for every $1\le 
r\le s$,
\[
\sup_{j_{r}\le i\le k_{r}}\|T^{\,i}x-T^{\,i}y_{r}\|<\delta 
\quad\textrm{while}\quad T^{\,N_{\delta ,m}+k_{s}}x=x.
\]
\par\medskip 
Stated in an informal way, $(*)$ means that arbitrary large pieces of the 
orbits of finitely many points of $K_{m}$ can be approximated by the 
orbits of a single periodic point of $K_{m}$, provided that the gaps between 
the different sets of indices where we require this approximation are 
sufficiently large. 
\par\smallskip
It is proved in \cite{BMPe2} that operators with the OSP are chaotic, 
topologically mixing, and frequently hypercyclic. Moreover, any operator 
satisfying the general version of the Frequent Hypercyclicity Criterion 
given in \cite{BoGR} has the OSP. We prove in Theorem \ref{Theorem 42} 
below that operators with the OSP satisfy the assumption of Theorem 
\ref{Theorem 41}. This was pointed out to us by Alfred Peris, who kindly 
allowed us to reproduce his proof here. This improves on a 
previous observation (proved in a preliminary version of this paper) according to which 
operators satisfying the Frequent Hypercyclicity Criterion also satisfy 
the assumptions of Theorem \ref{Theorem 41}.
\begin{theorem}\label{Theorem 42}
 Let $T\in\mathfrak{B}(X)$ be an operator with the \emph{OSP}. There exists a 
dense subspace $X_{0}$ of $X$ with $T(X_{0})\subseteq X_0$ and 
$X_{0}\subseteq
\emph{Per}(T)$, such that, for every $\alpha \in(0,1/2)$, every $x\in 
X_{0}$, every $\varepsilon >0$, and every integer $d_{0}\ge 1$, there 
exist $z\in X_{0}$ and integers $n$, $d\ge 1$ such that properties 
\emph{(0)}, \emph{(1)}, and \emph{(2)} of Theorem \ref{Theorem 41} hold 
true.
\end{theorem}
\begin{proof}
 Let $(K_{m})_{m\ge 1}$ be an increasing sequence of $T$-invariant subsets 
of $X$ with $0\in K_{1}$, the union of which is dense in $X$, such that $(*)$ above holds true for every 
$m\geq 1$. We define 
$X_{0}:=\textrm{span}\bigl[\bigcup_{m\ge 1}K_{m}\cap 
\textrm{Per}(T)\bigr]$, which is clearly $T$-invariant and dense in $X$. 

Let us fix $x\in X_{0}$, $\varepsilon >0$ and $d_{0}\ge 1$. 
The vector $x$ 
can be written as $x=\sum_{m=1}^{m_{0}}a_{m}x_{m}$, where $a_{m}$ is a scalar and
$x_{m}$ is a vector belonging to $ K_{m}$ for every $1\le m\le m_{0}$. By \cite[Prop.\,10]{BMPe2}, 
the map induced by $T$  on the set $K=\sum_{m=1}^{m_{0}}a_{m}K_{m}$ 
has the Specification Property. So we can assume without loss of 
generality that $x$ belongs to $K_{m}$ for some $m\ge 1$. 
Let now 
$N_{\varepsilon ,m}$ be such that property $(*)$ above
holds true, and let $d\ge 1$ be a multiple of $d_{0}$ so large that 
$\alpha d+N_{\varepsilon ,m}<d/2$. We then define $d'$ to be the integer 
part of $\alpha d$, $n=N_{\varepsilon ,m}+d'$, $y_{1}=0$, and 
$y_{2}=T^{\,l-n}x$, where $l$ is any multiple of $\textrm{per}(x)$ such 
that $l>n$. We also set $j_{1}=0$, $k_{1}=d'$, $j_{2}=n$, and $k_{2}=
d-N_{\varepsilon ,m}$. As $j_{2}-k_{1}=n-d'=N_{\varepsilon ,m}$, the 
Specification Property on $K_{m}$ implies that there exists a 
point $z\in K_{m}$ such that
\begin{align*}
 \|T^{\,k}z-T^{\,k}y_{1}\|&=\|T^{\,k}z\|<\varepsilon \quad\textrm{for 
every}\ 0\le k\le d';\\
\|T^{\,k}z-T^{\,k}y_{2}\|&<\varepsilon \quad\textrm{for every}\ n\le k\le 
k_2;\\
T^{\,N_{\varepsilon ,m}+k_{2}}z&=T^{\,d}z=z.
\intertext{In other words,}
\|T^{\,k}z\|&<\varepsilon \quad\textrm{for every}\ 0\le k\le \alpha d\,;\\
\|T^{\,n+k}z-T^{\,k}x\|&<\varepsilon \quad\textrm{for every}\ 0\le k\le 
k_{2}-n
\end{align*}
since $T^{\,n+k}y_{2}=T^{\,n+k+l-n}x=T^{k}x$; and $d$ is a multiple of the 
period of $z$. Since $k_{2}-n=d-2N_{\varepsilon ,m}-d'\ge\alpha d$, this shows that 
assumptions (0), (1), and (2) of Theorem \ref{Theorem 41} are satisfied.
\end{proof}

\begin{remark} The converse of Theorem \ref{Theorem 42} is not true; that is, operators satisfying the assumptions of Theorem \ref{Theorem 41} need not have 
the OSP. Indeed, we will construct in Section \ref{SPECIAL} (more precisely, in Example~\ref{Example 55}) some operators which satisfy the assumptions of Theorem \ref{Theorem 41} and yet  are not topologically mixing.
\end{remark}

To conclude this section, we essentially show that the OSP implies ergodicity.
This is coherent with what happens in the non-linear setting: it is proved in \cite{Sig} that compact dynamical systems with the Specification Property are ergodic. They actually enjoy a much stronger property: given such a system $(X,T)$, the ergodic measures with full support for $(X,T)$ are comeager in the set $\mathcal{P}_{T}(X)$ of $T$-invariant probability measures  endowed with its natural Polish topology.
On the other hand, we do not know whether operators with the OSP are mixing 
(\mbox{\it i.e.} admit mixing measures with full support), whereas it is known that the Frequent Hypercyclicity Criterion does imply mixing, and in fact mixing in the Gaussian sense 
(see \cite{MP} and \cite{BM2}).

\begin{proposition}\label{OSPergo} If $T\in\mathfrak B(X)$ satisfies the \emph{OSP}, and if the sequence $(K_{m})_{m\ge 1}$ of $T$-invariant subsets of $X$ appearing in the definition of the \emph{OSP} consists of \emph{compact} sets, then $T$ is ergodic.
\end{proposition}

The assumption of Proposition \ref{OSPergo} that the sets $K_{m}$, $m\ge 1$, be compact seems to be no real restriction. As mentioned in \cite{BMPe2}, all known examples of operators with the OSP do satisfy the OSP with respect to a sequence of $T$-invariant compact subsets $(K_{m})_{m\ge 1}$ of $X$.

\begin{proof} Let us denote by $\mathcal P(X)$ the space of all Borel probability measures on $X$ endowed with its natural Polish topology (a sequence $(\mu_{n})_{n\ge 1}$ of elements of  $\mathcal P(X)$ converges to $\mu\in\mathcal{P}(X)$ if and only if 
$\int_X fd\mu_n$ tends to $\int_X fd\mu$ as $n$ tends to infinity for every bounded continuous function $f$ from $X$ into $\R$), and by $\mathcal P_T(X)\subseteq \mathcal P(X)$ the set of all $T$-invariant measures. Then $\mathcal P_T(X)$ is a Polish space, being a closed subset of $\mathcal P(X)$. We denote by $\mathcal E_T(X)\subseteq\mathcal P_T(X)$ the set of all ergodic measures for $T$. For any family of measures $\mathcal M\subseteq \mathcal P(X)$, we denote by $\mathcal M^*$ the family of all measures $\mu\in\mathcal M$ with full support. Finally, for any subset $\mathcal M$ of $\mathcal P(X)$ and any Borel subset $A$ of $ X$, we set $\mathcal M(A):=\{ \mu\in\mathcal M;\; \mu(A)=1\}$. 
\par\smallskip
Let $(K_{m})_{m\geq 1}$ be an increasing sequence of \emph{compact} $T$-invariant subsets of $X$ with $\overline{\bigcup_{m\geq 1} K_m}=X$ such that 
 $T_{| K_m}$ has the Specification Property for every $m\geq 1$. By a result of \cite{Sig}, we know that for each $m\geq 1$, the set $\mathcal E_T^*(K_m)$ of all measures
 $\mu\in\mathcal E_T(X)$ with support equal to $K_m$ is dense in $\mathcal P_T(K_m)$. 
Let us now denote by $\mathcal M$ the closure of $\bigcup_{m\geq 1} \mathcal P_T(K_m)$ in $\mathcal P_T(X)$. 

\begin{fact}\label{microfact1} The set $\mathcal E_T(X)\cap\mathcal M$ is a dense $G_\delta$ subset of $\mathcal M$.
\end{fact}

\begin{proof}[Proof of Fact \ref{microfact1}] It is known that $\mathcal E_T(X)$ is a $G_\delta$ subset of $\mathcal P_T(X)$ (see \cite{Part} for a detailed proof), so that 
$\mathcal E_T(X)\cap\mathcal M$ is $G_\delta$ in $\mathcal M$. Moreover, since $\mathcal E_T(K_m)$ is dense in $\mathcal P_T(K_m)$ for each $m$ by \cite{Sig}, it is clear that $\mathcal E_T(X)\cap\mathcal M$ is dense in $\mathcal M$.
\end{proof}

\begin{fact}\label{microfact2} The set $\mathcal M^*$ is a dense $G_\delta$ subset of $\mathcal M$.
\end{fact}

\begin{proof}[Proof of Fact \ref{microfact2}] Let $(O_p)_{p\geq 1}$ be a countable basis of non-empty open subsets of $X$. Set, for each $p\ge 1$,
$\mathcal O_p:=\{ \mu\in\mathcal P(X);\; \mu(O_p)>0\}$. Then each set $\mathcal O_p$ is open in $\mathcal P(X)$, and  moreover
\[\mathcal M^*=\mathcal M\cap\Bigl(\bigcap_{p\geq 1} \mathcal O_p\Bigr).\] So we just have to show that each set $\mathcal M\cap \mathcal O_p$ is dense in 
$\mathcal M$.
Let us fix $p\geq 1$, and let $\mathcal U$ be a non-empty open subset of $\mathcal M$. Since the sequence $(K_m)_{m\ge 1}$ is increasing, the definition of $\mathcal M$
implies that
 $\mathcal U\cap \mathcal P_T(K_m)$ is non-empty for all $m$ sufficiently large. Moreover, since $\bigcup_{m\geq 1} K_m$ is dense in $X$ and $(K_m)_{m\ge 1}$ is increasing, $K_m\cap O_p$ is non-empty too for all $m$ sufficiently large. So there exists $m\ge 1$ such that $K_m\cap O_p$ and 
$\mathcal U\cap\mathcal P_T(K_m)$ are both non-empty. Since $\mathcal E_T^*(K_m)$ is dense in $\mathcal P_T(K_m)$ and $\mathcal U\cap\mathcal P_T(K_m)$ is 
open in $\mathcal P_T(K_m)$, it follows that there exists a measure 
$\mu\in\mathcal E_T(X)$ with support equal to $K_m$ such that $\mu$ belongs to $\mathcal U$ and  $\mu(O_p)>0$. Hence, we have shown that 
$\mathcal U\cap \mathcal O_p$ is non-empty.
\end{proof}

The two facts above combined with the Baire Category Theorem applied in $\mathcal M$ imply that $\mathcal E_T(X)\cap\mathcal M^*$ is non-empty. In particular 
$\mathcal E_T(X)^*$ is non-empty, \mbox{\it i.e.} $T$ is ergodic.
\end{proof}

\begin{remark} We will see in Section \ref{SPECIAL} 
that the assumptions of Theorem \ref{Theorem 41} (which are sufficient for frequent hypercyclicity) do \emph{not} imply ergodicity. So Proposition \ref{OSPergo} makes the difference between the OSP and our criterion for frequent hypercyclicity all the more tangible. 
\end{remark}

\section{Special examples of hypercyclic operators}\label{SPECIAL}
In this section, we introduce some particular classes of operators, which are defined on any space $\ell_p(\N)$, $1\leq p<\infty$. 
We do not restrict ourselves to the Hilbertian case $p=2$, because the 
general case adds no extra complication. It is within these classes that we will exhibit operators which are chaotic and frequently hypercyclic but not ergodic, operators which are  
chaotic and $\mathcal{U}$-frequently hypercyclic but not frequently hypercyclic and operators which are chaotic and topologically mixing but not $\mathcal{U}$-frequently hypercyclic.
\par\smallskip
A word of caution: for technical reasons, we have decided that 
{\emph{$\N$ starts at $0$}}; that is, $\N=\{ 0,1,2,\dots \}$. Accordingly, the canonical basis of $\ell_p(\N)$ will be denoted by 
$(e_k)_{k\geq 0}$. Also, we denote by $c_{00}$ the linear span of the vectors $e_k$, $k\geq 0$, \mbox{\it i.e.} the subspace of $\ell_p(\N)$ consisting 
of all finitely supported vectors.
\par\smallskip

\subsection{Operators of \cct: basic facts}
In view of our criteria for $\mathcal U$-frequent hypercyclicity and frequent hypercyclicity 
relying on the existence of periodic points, we would like to find a rich family of operators 
for which we can easily find a large supply of periodic points. For example, we could consider 
 operators $T_{w,b}$ defined by
\[T_{w,b}e_k=
\left\{\begin{array}{cl}
 w_{k+1}e_{k+1} & \quad\text{if}\ k\in \mathopen[b_n,b_{n+1}-1\mathclose)\\
 \left(\prod_{j=b_n+1}^{b_{n+1}-1}w_j\right)^{-1}e_{b_n} & \quad \text{if}\ k=b_{n+1}-1\ \text{with} \ n\ge 0
\end{array}\right.\]
where $w=(w_j)_{j\geq 1}$ is a weight sequence and $b=(b_n)_{n\ge 1}$ is a strictly increasing sequence of integers with $b_0=0$. Indeed, with this definition 
it is clear that (whatever the choice of $w$ and $b$) every basis vector $e_k$ and hence every vector $x\in c_{00}$ is periodic for $T_{w,b}$.
\par\smallskip
However, none of these operators is hypercyclic since they are direct sums of finite-dimensional operators, and there exist no 
hypercyclic operators in finite dimension. 
It is thus necessary to perturb the operators $T_{w,b}$ in order to obtain an interesting family of \emph{hypercyclic} operators in which 
finite sequences are still periodic. For reasons that will be explained below, the operators in this family will be called 
\emph{operators of C-type}. Any operator of C-type will be associated to four 
parameters $v$, $w$, $\varphi $, and $b$, where

\begin{enumerate}
 \item[-] $v=(v_{n})_{\gn}$ is a sequence of non-zero complex numbers such 
that $\sum_{\gn}|v_{n}|<\infty $;
\item[-] $w=(w_{j})_{j\geq 1}$ is a sequence of complex numbers which is both bounded  
and bounded below, \mbox{\it i.e.} $0<\inf_{k\ge 1} \vert w_k\vert\leq \sup_{k\ge 1}\vert w_k\vert<\infty$;
\item[-] $\varphi $ is a map from $\N$ into itself, such that $\varphi 
(0)=0$, $\varphi (n)<n$ for every $\gn$, and the set 
$\varphi ^{-1}(l)=\{n\ge 0\,;\,\varphi (n)=l\}$ is infinite for every 
$l\ge 0$;
\item[-] $b=(b_{n})_{n\ge 0}$ is a strictly increasing sequence of positive 
integers such that $b_{0}=0$ and $b_{n+1}-b_{n}$ is a multiple of 
$2(b_{\varphi (n)+1}-b_{\varphi (n)})$ for every $n\ge 1$.
\end{enumerate}

\begin{definition}\label{Definition 43}
 The \emph{operator of \cct} $\tvw$ on $\ell_p(\N)$
 associated to the data  
$v$, $w$, $\varphi $, and $b$ given as above is defined by
\[
\tvw\ e_k=
\begin{cases}
 w_{k+1}\,e_{k+1} & \textrm{if}\ k\in [b_{n},b_{n+1}-1),\; n\geq 0,\\
v_{n}\,e_{b_{\varphi(n)}}-\Bigl(\,\,\ds\prod_{j=b_{n}+1}^{b_{n+1}-1}
w_{j}\Bigr)^{ -1 } e_ {
b_{n}} & \textrm{if}\ k=b_{n+1}-1,\ \gn,\\
 -\Bigl(\!\!\ds\prod_{j=b_0+1}^{b_{1}-1}w_j\Bigr)^{-1}e_0& \textrm{if}\ 
k=b_1-1.
\end{cases}
\]
Note that by convention, an empty product is declared to be be equal to $1$; that is, $\prod_{j=b_{n}+1}^{b_{n+1}-1}w_{j}=1$ when $b_{n+1}=b_{n}+1$ (which can happen only for $n=0$).
\end{definition}

Without any additional assumption on the parameters, these formulas define a linear map 
on $c_{00}$ only. The first issue is of course the \emph{boundedness} of $\tvw$. 

\begin{fact}\label{welldefined} The operator $\tvw$ is  well-defined and bounded  on $\ell_p(\N)$ as soon as that the following condition holds true:
\begin{equation}\label{machin}
\inf_{n\geq 0} \prod_{b_n<j<b_{n+1}} \vert w_j\vert >0. 
\end{equation}
\end{fact}

\begin{proof} This is rather clear. Indeed, it appears that $\tvw$ can be written as
\[\tvw= \bigoplus_{n\geq 0} C_{w,\,b,\,n}+R_{v,b},
\]
where $R_{v,b}$ is the operator defined by
\[ R_{v,\, b}x=\sum_{n\geq 1} v_n x_{b_{n+1}-1}\, e_{b_{\varphi(n)}},\quad x\in\ell_{p}(\N)
\]
which is clearly bounded (and even compact) because $\sum_{n\geq 1} \vert v_n\vert<\infty$, and, for each $n\ge 1$, 
$C_{w,\, b,\, n}$ is a finite-dimensional cyclic operator acting on $E_{n}:=\textrm{span}[e_{k}\,;\,b_{n}\le k\leq b_{n+1}-1]$. Condition $(\ref{machin})$ implies that
$\sup_{n\ge 1} \Vert C_{w,\, b,\, n}\Vert$ is finite, and $\tvw$ is thus bounded.
\end{proof}

From now on, we will always assume that Condition $(\ref{machin})$ is satisfied. Also, when no confusion arises, we will write simply $T$ 
instead of $\tvw$.

\begin{remark} We call such operators \emph{operators of \cct} for two different reasons. 
On the one hand, ``C"  stands for ``cyclic'': as we have just explained, each operator $\tvw$ is a
compact perturbation of an infinite direct sum of cyclic finite-dimensional 
operators $\bigoplus_{n\ge 0}C_{w,\,b,\,n}$, where 
$C_{w,\,b,\,n}$ is defined on 
$E_{n}=\textrm{span}[e_{k}\,;\,b_{n}\le k\le b_{n+1}-1]$. On the other hand, we will see in a moment that, as a consequence of their 
particular structure, the operators $\tvw$ happen to be chaotic under a very mild restriction on the parameters; 
so, ``C" stands for ``chaotic'' as well. 
\end{remark}

The following identity will be used repeatedly: if $T=\tvw$ is an operator of \cct\ on $\ell_{p}(\N)$, then
\begin{equation}\label{Equation 5}
T^{\,b_{n+1}-b_n}\,e_{b_n}=\smash[t]{v_n\,\Big(
\prod_{j=b_n+1}^{b_{n+1}-1}
w_j\Big)\,e_ { b_ {
\varphi(n)}}-e_{b_n}\quad\hbox{for every $n\geq 1$}.}
\end{equation}
This allows to show that operators of \cct\ always have plenty of periodic points:

\begin{fact}\label{Proposition 44}
If $T=\tvw$ is an operator of \cct\ on $\ell_{p}(\N)$, then every basis vector $e_k$ is periodic for $\tvw$; more precisely,
\[
T^{\,2(b_{n+1}-b_n)}e_k=e_k\qquad \textrm{if}\ k\in 
[b_n,b_{n+1}),\ n\ge 0.
\]
Consequently, every vector $x\in c_{00}$ is periodic for $\tvw$, and hence $\tvw$ has a dense set of periodic points.
\end{fact}

\begin{proof}
 Since $e_{k}$ is a non-zero multiple of $T^{\,k-b_{n}}e_{b_{n}}$ for 
every 
 $b_{n}\le k<b_{n+1}$, it suffices to prove that 
$T^{\,2(b_{n+1}-b_{n})}\,e_{b_{n}}
 =e_{b_{n}}$ for every $n\ge 0$. We prove this by induction on $n$.
 \par\medskip 
 \noindent -- If $n=0$, then
$T^{\,b_{1}-b_{0}}\,e_{0}=-e_{0}$ by definition of $Te_{b_1-1}$, 
and thus 
 $T^{\,2(b_{1}-b_{0})}\,e_{0}=e_{0}$.
 \par\smallskip
 \noindent -- Let $\gn$, and assume that the result has been proved for all $m<n$. We know that $b_{n+1}-b_{n}$ is a multiple of 
$2(b_{\varphi (n)+1}-b_{\varphi (n)})$ and since $\varphi(n)<n$, it follows by (\ref{Equation 5}) and the induction hypothesis that 

\[
T^{\,2(b_{n+1}-b_n)}\,e_{b_n}=\smash{v_n\,\Big(\prod_{j=b_n+1}^{b_{n+1}-1}
w_j\Big)e_{b_
{\varphi(n)}}
-T^{\,b_{n+1}-b_{n}}\,e_{b_{n}}=e_{b_{n}}.}
\]
\end{proof}

Using the above fact and Corollary \ref{Proposition 38}, we can now obtain the following 
sufficient condition for an operator of \cct\ to be chaotic.

\begin{proposition}\label{Proposition 45}
 Suppose that 
\[
\limsup_{
\genfrac{}{}{0pt}{1}{N\to\infty 
}{N\in\,\varphi ^{-1}(n)}}|v_{N}|\prod_{j=b_N+1}^{b_{N+1}-1} 
|w_{j}|=\infty\quad\emph{for every}\ n\ge 0.
\] 
Then the operator of \cct\ $T=\tvw$ on $\ell_{p}(\N)$ is chaotic.
\end{proposition}

\begin{proof}
 We apply Corollary \ref{Proposition 38} with $X_{0}:=\{e_{k}\,;\,k\ge 0\}$. Fix $k\ge 0$ and $\varepsilon >0$. We are looking for a 
  vector $z\in {\rm Per}(T)$ and an integer $m\ge 1$ such that $\Vert z\Vert <\varepsilon$ and $\Vert T^mz-e_k\Vert<\varepsilon$.
Let $n\ge 0$ be such that $k$ belongs to $[b_{n},b_{n+1})$. By the assumption 
of Proposition \ref{Proposition 45}, there exists $N\in\varphi ^{-1}(n)$ 
such that 
\begin{equation}\label{Equation 6}
 |v_N|\prod_{j=b_N+1}^{b_{N+1}-1}|w_j|>\dfrac{1}{\varepsilon }
\Bigl(\prod_{j=b_n+1}^{k}|w_j|\Bigr)^{-1}\,\max\{1,\|w\|_{\infty}\}^{k-b_n}.
\end{equation}
The vector 
\[
z:=
v_{N}^{-1}\ \Bigl(\prod_{j=b_N+1}^{b_{N+1}-1}w_j\Bigr)^{-1}\,\Bigl(\prod_{
j=b_n+1}^{k}w_j\Bigr)^{-1}
\ e_{b_N}
\]
is periodic by Fact \ref{Proposition 44}, and satisfies $\|z\|<\varepsilon $ by (\ref{Equation 6}). 
Moreover, since $\varphi (N)=n$, we have by (\ref{Equation 5}):
\begin{align*}
T^{\,b_{N+1}-b_{N}+k-b_{n}}\,e_{b_{N}}&=v_{N}\,\Bigl(\prod_{j=b_{N}+1}
^{b_{N+1}-1}w_{j}
\Bigr)\,T^{\,k-b_{n}}\,e_{b_{n}}-T^{\,k-b_{n}}\,e_{b_{N}}\\
&=v_{N}\,\Bigl(\prod_{j=b_{N}+1}
^{b_{N+1}-1}w_{j}
\Bigr)\Bigl(\prod_{j=b_{n}+1}
^{k}w_{j}
\Bigr)\,e_{k}-\Bigl(\prod_{j=b_{N}+1}
^{b_{N}+k-b_{n}}w_{j}
\Bigr)\,e_{b_{N}+k-b_{n}}.
\end{align*}
{By definition of $z$, this implies that}
\[
T^{\,b_{N+1}-b_{N}+k-b_{n}}\,z= e_k-
\Bigl( \prod_{j=b_{N}+1}
^{b_{N}+k-b_{n}}w_{j}\Bigr)\, 
v_{N}^{-1} \,\Bigl(\,\prod_{j=b_{N}+1}
^{b_{N+1}-1}w_{j}
\Bigr)^{-1}\Bigl(\prod_{j=b_{N}+1}
^{k}w_{j}
\Bigr)^{-1}e_{b_{N}+k-b_{n}};
\]
and by (\ref{Equation 6}), it follows that 
$\|T^{\,b_{N+1}-b_{N}+k-b_{n}}\,z-e_{k}\|< \varepsilon $. The 
assumptions of Corollary~\ref{Proposition 38} are thus 
satisfied, and $T$ is chaotic.
\end{proof}

Recall that every chaotic operator is topologically weakly mixing~\cite[Ch. 4]{BM} and reiteratively hypercyclic~\cite{BMPP}. 
In the following subsections, we will be interested in frequent hypercyclicity, $\mathcal U$-frequent hypercyclicity and topological mixing for operators of \cct.
\par\smallskip

\subsection{Operators of \cpt: how to be FHC or UFHC}\label{Subsection 4.2}
For the construction of our counterexamples, we will work with operators 
of \cct\ for which the data $v$, $w$, $\varphi $ and $b$ have a special 
structure. For every integer $\gk$, we require that:
\begin{enumerate}
 \item[-] $\varphi (n)=n-2^{k-1}$ for every $n\in[2^{k-1},2^{k})$, so that 
 $\varphi ([2^{k-1},2^{k}))=[0,2^{k-1})$;
 \item[-] the blocks $[b_{n},b_{n+1})$, $n\in[2^{k-1},2^{k})$,
all have the same size, which we denote by $\Delta ^{(k)}$: 
\[b_{n+1}-b_{n}=
\Delta ^{(k)}\qquad\hbox{for every $n\in[2^{k-1},2^{k})$};
\]
\item[-] the sequence $v$ is constant on the interval $[2^{k-1},
2^{k})$: there exists $v^{(k)}$ such that 
\[ v_{n}=v^{(k)}\qquad\hbox{for every 
$n\in[2^{k-1},2^{k})$};
\]
\item[-] the sequences of weights $(w_{b_{n}+i})_{1\le i
<\Delta ^{(k)}}$ are independent of $n\in[2^{k-1},2^{k})$: there exists a 
sequence $(w_{i}^{(k)})_{1\le i<\Delta ^{(k)}}$ such that 
\[
w_{b_{n}+i}=w_{i}^{(k)} \quad\hbox{for every $1\le i<\Delta ^{(k)}$ and every 
$n\in[2^{k-1},2^{k})$.}
\]
\end{enumerate}
\par\smallskip\noindent
If these conditions are met, we say that $\tvw$ is an \emph{operator of \cpt.} 

\begin{remark} By definition, the map $\varphi$ is the same for all operators of \cpt, so it is no longer a ``parameter". 
However, we will continue using the notation $\tvw$.
\end{remark}

Our first result concerning these operators gives a sufficient set of conditions for a \cpt\ operator to be $\mathcal{U}$-frequently hypercyclic. 
This will be deduced from Theorem \ref{Theorem 39}.

\begin{theorem}\label{Theorem 46}
 Let $T=\tvw$ be an operator of \cpt\ on $\ell_{p}(\N)$. Suppose that 
there exists a constant $\alpha >0$ such that the following property holds true : for every $C\ge 1$ and every integer $k_{0}\ge 1$, 
there exist two integers $k> k_{0}$ and $1\le m<\Delta ^{(k)}$ such that
\begin{equation*}
\smash[b]{|v^{(k)}|\prod_{i=\Delta^{(k)}-m}^{\Delta^{(k)}-1}|w^{(k)}_i|}
\ge C \qquad \hbox{and}\qquad
\smash[t]{|v^{(k)}|\prod_{i=m'+1}^{\Delta^{(k)}-1}|w^{(k)}_i|}>C \quad
\hbox{for every $0\le m'\le \alpha m$}. 
\end{equation*}
Then $T$ is chaotic and $\mathcal{U}$-frequently hypercyclic.
\end{theorem}

\begin{proof} The following fact will be useful.
\begin{fact}\label{petitcalcul} Let $k\geq 1$. For any $l<2^{k-1}$ and $1\leq s\leq \Delta^{(k)}$, we have
\[ T^s e_{b_{2^{k-1}+l+1}-s}=v^{(k)}\,\Bigl(\prod_{i=\Delta^{(k)}-s+1}^{\Delta^{(k)}-1} w_i^{(k)}\Bigr)\, e_{b_l}-\Biggl( \prod_{i=1}^{\Delta^{(k)}-s} w_i^{(k)}\Biggr)^{-1} e_{b_{2^{k-1}+l}}.
\]
\end{fact}
\begin{proof}[Proof of Fact \ref{petitcalcul}] Since $l<2^{k-1}$, we have $\varphi({2^{k-1}+l})=l$. So the formula follows directly from the definition of $T$.
\end{proof}

We are going to show that the assumption of Theorem 
\ref{Theorem 39} is satisfied with $X_{0}=\textrm{span}\,[e_{k}\,;\,k\ge 0]$. So let us fix $x\in X_{0}$ and
 $\varepsilon >0$. We choose $k_0\geq1$ such that $x$ may be written as 
\[x=\sum_{l<2^{k_0}}\sum_{j=b_l}^{b_{l+1}-1}
x_{j}e_{j}.\]

Let $C>0$ be a very large number, to be specified later on in the proof. 
 By assumption, there exist $k>k_{0}$ and $1\le m<\Delta^{(k)}$ such 
that
\begin{align}
|v^{(k)}|\prod_{i=\Delta^{(k)}-m}^{\Delta^{(k)}-1}|w^{(k)}_{i}|&> C\label{Equation 
9}
\end{align}
and
\begin{align}
|v^{(k)}|\prod_{i=m'+1}^{\Delta^{(k)}-1}|w^{(k)}_i|&>C \qquad
\textrm{for every}\ 0\le m'\le \alpha m.
\label{Equation 10}
\end{align}
Note that since the sequences $v$ and $w$ are bounded, it follows from (\ref{Equation 9}) that the integer $m$ 
can be chosen as large as we please, provided that $C$ is large enough. So we may assume from the beginning that 
$m>2\Delta^{(k_0)}$. We will also assume that $\alpha<1$, which will be useful below.
\par\smallskip 
We set
\[
z:=\sum_{l<2^{k_0}}\ 
\sum_{j=b_l}^{b_{l+1}-1}
x_{j}\
\left({v^{(k)}}\!\!\!\!\!\!\!\!\!\;\;\;\;\;\;\ds\prod_ {
i=\Delta^{(k)}-m+j-b_{l}+1}^{\Delta^{(k)}-1}w^{(k)}_{i}\right)^{-1}
\Bigl(\,\prod_{i=1}^{j-b_l}w_{b_l+i} \Bigr)^{-1}
\ e_{b_{2^{k-1}+l+1}-m+j-b_{l}}.
\]
\par\smallskip
Our aim is to prove that if $C$ has been  suitably chosen, then \[ \|z\|<\varepsilon\qquad {\rm and}\qquad \|T^{\,m+m'}z-T^{\,m'}x\|
<\varepsilon\quad\hbox{for every $0\le m'\le \alpha m/2$.}
\] 
Theorem \ref{Theorem 39} will then conclude the proof.
\par\smallskip 
The first of these two 
claims is the easiest one to prove. Indeed, since the weight sequence $w$ is bounded and bounded from below, it follows from 
(\ref{Equation 9}) that 
\[\|z\|\le \|x\|\,.\, C^{-1}.\, A^{\Delta^{(k_0)}},\]
where $A$ is an absolute constant. So $\Vert z\Vert<\varepsilon$ if $C$ 
is large enough.
\par\smallskip 
Let us now estimate the norm of the vector $T^{\,m+m'}z-T^{\,m'}x$ for every $0\le 
m'\le\alpha m/2$. 
Note that if $0\le l\le 2^{k_{0}}-1$ and 
$b_{l}\le j<b_{l+1}$, then 
$m-(j-b_{l})\ge 1$ since $0\le j-b_{l}<\Delta ^{(k_{0})}$ and $m>\Delta ^{(k_{0})}$. Applying Fact \ref{petitcalcul} with $s:=m-(j-b_l)$, we get
\begin{align*}
T^{\,m-(j-b_l)}e_{b_{2^{k-1}+l+1}-m+j-b_{l}}=
v^{(k)}\,&\Bigl(\prod_{i=\Delta^{(k)}-m+j-b_l+1}^{\Delta^{(k)}-1} w_i^{(k)}\Bigr)\, e_{b_l}\\
&-\Biggl(\; \prod_{i=1}^{\Delta^{(k)}-m+j-b_l} w_i^{(k)}\Biggr)^{-1} e_{b_{2^{k-1}+l}},
\end{align*}
and hence
\begin{align*}
 T^{\,m}e_{b_{2^{k-1}+l+1}-m+j-b_{l}}=
\quad \Biggl( v^{(k)}&
\!\!\!\!\!\;\;\;\;\prod_{i=\Delta^{(k)}-m+j-b_{l}+1}^{\Delta^{(k)}-1}w^{(k)}_{i }\Biggr)\ 
\Bigl(\,\,\prod_{i=1}^{j-b_l}w_{b_l+i} \Bigr) \,
e_{j} \\
&- \Biggl(\!\!\!\!\!\;\;\;\;\prod_{i=j-b_l+1}^{\Delta^{(k)}-m+j-b_{l}}w^{(k)}_{i}\Biggr)^{-1} 
\, e_{b_{2^{k-1}+l}+j-b_{l}}
\end{align*}
because $T^{j-b_l} e_{b_l}=\Bigl(\,\,\prod_{i=1}^{j-b_l}w_{b_l+i} \Bigr) \,
e_{j} $ and $T^{j-b_l} e_{b_{2^{k-1}+l}}= \Bigl(\prod_{i=1}^{j-b_l}w^{(k)}_{i}\Bigr)\,  e_{b_{2^{k-1}+l}+j-b_l}$.
\par\smallskip
Moreover, if $0\le m'\le \alpha m/2$, then 
\[
T^{\,m'}\,e_{b_{2^{k-1}+l}+j-b_{l}}=\Biggl(\;\,\prod_{i=j-b_l+1}^{j-b_{l}+m'}
w_{i}^{(k)}\Biggr)\,e_{b_{2^{k-1}+l}+j-b_{l}+m'},
\]
because $ j-b_{l}+m'<b_{l+1}-b_{l}+\alpha m/2
<\Delta ^{(k_{0})}+m/2<m
<\Delta ^{(k)}$. So we get 
\begin{align*}
 T^{\,m+m'}e_{b_{2^{k-1}+l+1}-m+j-b_{l}}=
& \Biggl( v^{(k)}
\!\!\!\!\!\;\;\;\;\prod_{i=\Delta^{(k)}-m+j-b_{l}+1}^{\Delta^{(k)}-1}w^{(k)}_{i }\Biggr)\ 
\Bigl(\,\,\prod_{i=1}^{j-b_l}w_{b_l+i} \Bigr) \,
T^{m'} e_{j} \\
&- \Biggl(\!\!\!\!\!\;\;\;\;\prod_{i=j-b_l+1}^{\Delta^{(k)}-m+j-b_{l}}w^{(k)}_{i}\Biggr)^{-1} 
\Biggl(\;\,\prod_{i=j-b_l+1}^{j-b_{l}+m'}
w_{i}^{(k)}\Biggr)
\, e_{b_{2^{k-1}+l}+j-b_{l}+m'}.
\end{align*}
By definition of $z$, it follows that for any $0\leq m'\leq\alpha m/2$, we have
\begin{align*}
T^{m+m'}z=T^{m'}x-\sum_{l<2^{k_{0}}}\ 
\sum_{j=b_{l}}^{b_{l+1}-1}& \Biggl(v^{(k)}\prod_{i=j-b_l+m'+1}^{\Delta ^{(k)}-1}w_{i}^{(k)} 
\Biggr)^{-1}\cdot\\
&\Bigl(\,\,\prod_{i=1}^{j-b_l}w_{b_{l}+i}
\Bigr)^{-1}\,x_{j}e_{b_{2^{k-1}+l}+j-b_{l}+m'}.
\end{align*}
By (\ref{Equation 10}) and since the weight sequence $w$ is bounded and bounded below, this implies that 
\[ \Vert T^{m+m'}z-T^{m'}x\Vert< \Vert x\Vert \,.\,  C^{-1}.\, A^{\Delta^{k_0}},
\]
where (as above) $A$ is an absolute constant. So 
we have $\Vert T^{m+m'}z-T^{m'}x\Vert<\varepsilon$ for every $0\leq m'\leq\alpha m/2$ if 
$C$ is large enough.
\par\smallskip
The assumptions 
of Theorem \ref{Theorem 39} are thus satisfied, and this concludes the 
proof of Theorem \ref{Theorem 46}.
\end{proof}

Using Theorem \ref{Theorem 41} instead of Theorem \ref{Theorem 39}, we can 
obtain sufficient conditions of the same kind as above for operators of \cpt\ to be frequently hypercyclic.

\begin{theorem}\label{Theorem 47}
 Let $T=\tvw$ be an operator of \cpt\ on $\ell_{p}(\N)$.  Suppose that 
there exists a constant $\alpha>0$ such that the following property holds true: for every $C\ge 1$ and every $k_{0}\ge 
1$, there exist two integers $k\ge k_{0}$ and $1\le m<\Delta ^{(k)}$ such 
that
\begin{align}
\smash[b]{|v^{(k)}|\prod_{j=\Delta^{(k)}-m}^{\Delta^{(k)}-1}|w^{(k)}_j|}
&\ge C 
 \label{Equation 14}
 \intertext{and}
 \smash[t]{|v^{(k)}| 
\prod_{j=m'+1}^{\Delta^{(k)}-1}|w^{(k)}_j|}&>C\quad\textrm{for every}\  
\ 0\le m'\le \alpha \Delta^{(k)}.
 \label{Equation 15}
\end{align}
Then $T$ is chaotic and frequently hypercyclic.
\end{theorem}

\begin{proof}
 The proof of Theorem \ref{Theorem 47} is so similar to that of Theorem 
 \ref{Theorem 46} that we only sketch it very briefly. 
 The role of the integer $d$ in 
the assumption of Theorem \ref{Theorem 41} is played by the integer 
$2\Delta ^{(k)}$, which is a period of $z$. If $k$ is chosen sufficiently 
large at the beginning of the proof, $\Delta ^{(k)}$ can be supposed to be 
a multiple of the period of any fixed vector $y\in X_{0}$: indeed, such a 
vector always has a period of the form $\Delta ^{(k_{1})}$ for some 
integer 
$k_{1}\ge 1$, and $2\Delta ^{(k_{1})}$ divides $\Delta ^{(k)}$ as soon as $k> k_1$ since
$\Delta^{(k)}=b_{2^{k-1}+2^{k_1-1}+1}-b_{2^{k-1}+2^{k_1-1}}$,  $\Delta ^{(k_{1})}=b_{2^{k_1-1}+1}-b_{2^{k_1-1}}
$ 
and $\varphi(2^{k-1}+2^{k_1-1})=2^{k_1-1}.$
\end{proof}

Now that we have obtained sufficient conditions for the frequent and 
$\mathcal{U}$-frequent hypercyclicity of operators of \cpt, we need to 
find \emph{necessary} conditions as well. This we do in the next subsection.
\par\smallskip

\subsection{Operators of \cct: how \emph{not} to be FHC or UFHC}\label{Subsection 4.3}
Since $\mathcal U$-frequent hypercyclicity and frequent hypercyclicity are strong 
notions of hypercyclicity, one might think that it is easier to prove that an operator $T\in\bh$ does \emph{not} have one of these properties, than to prove 
that an operator has it. However, instead of 
exhibiting a single $\mathcal U$-frequent or frequent hypercyclic vector, we now need to 
prove that no vector of $\h$ whatsoever can be $\mathcal U$-frequent or frequent hypercyclic; and put in this way, this 
no longer looks that easy. 
\par\smallskip
In this subsection, we are going to single out some conditions ensuring that an operator of \cct\ $T$ is not $\mathcal U$-frequently 
hypercyclic or not frequently hypercyclic. As suggested in the  few lines above,  the arguments will be rather more technical than in the previous subsection. 
However, we can give the basic idea immediately: if for every {hypercyclic} vector $x$ for $T$, 
we are able to find some $\varepsilon>0$ such that the set $\mathcal{N}_T(x,B(0,\varepsilon))$ 
has upper density (resp. lower density) equal to $0$, then $T$ will not be $\mathcal U$-frequently 
 hypercyclic (resp. frequently hypercyclic). 
\par\smallskip

\subsubsection{The main criterion, in abstract form}
The following notation will be used throughout this subsection: given an operator of \cct\ $T=\tvw$, we denote for each $n\ge 0$ by $P_{n}$ the canonical projection of $\ell_p(\N)$ onto 
$E_{n}=\textrm{span}[e_{k},\;\,b_{n}\le k<b_{n+1}]$: if $x=\sum\limits_{k\ge 0}x_{k}e_{k}\in\ell_p(\N)$, 
then
\[
P_{n}x=\sum_{k=b_{n}}^{b_{n+1}-1}x_{k}e_{k}.
\]
\par\medskip 
The following theorem provides sufficient conditions 
for an operator of \cct\ to be non-$\mathcal{U}$-frequently hypercyclic or 
non-frequently hypercyclic. These conditions are stated in terms of the projections $P_n$. 

\begin{theorem}\label{Theorem 49}
 Let $T$ be an operator of \cct\ on $\ell_{p}(\N)$.  Suppose that for 
every hypercyclic vector $x\in\ell_{p}(\N)$ for $T$, there exist
\begin{enumerate}
 \item[-] a positive constant $C$,
 \item[-] a non-increasing sequence $(\beta _{l})_{l\ge 1}$ of positive real 
numbers with $\sum_{l\ge 1}\sqrt{\beta _{l}}\le 1$,
\item[-] a sequence $(X_{l})_{l\ge 0}$ of non-negative real numbers,
\item[-] a non-decreasing sequence $(N_{l})_{l\ge 1}$ of integers tending to infinity,
\end{enumerate}
such that the following conditions are satisfied:
\smallskip
\begin{enumerate}
 \item[\emph{(1)}] $\|P_{n}x\|\le X_{n}$ for every $n\ge 0$;
 
 \item[\emph{(2)}] $\sup\limits_{j\ge 0}\|P_{n}\,T^{\,j}P_{l}\,x\|\le C\beta _{l}X_{l}$ 
 for every $l\ge 1$ and every $0\le n<l$;
 
\item[\emph{(3)}] $\sup\limits_{0\le j\le 
N_{l}}\|P_{n}\,T^{\,j}P_{l}\,x\|\le C\beta _{l}\|P_{l}x\|$ for every $l\ge 1$ and every $0\le n<l$;

\item[\emph{(C)}] $\liminf\limits_{l\to\infty }\ \ \inf\limits_{k\ge N_{l}}
\dfrac{\ \#\bigl\{0\le j\le k\,;\, \|P_{l}\,T^{\,j}P_{l}\,x\|\ge 2CX_{l}\bigr\}}{k+1}=1.$
\end{enumerate}

\smallskip\noindent
Then $T$ is not $\mathcal{U}$-frequently hypercyclic.
\par\medskip 
If Condition \emph{(C)} is replaced by 
\smallskip
\begin{enumerate}
 \item [\emph{(C')}] $\liminf\limits_{l\to\infty }\ \ 
\inf\limits_{k\ge N_{\min(\varphi ^{-1}(l))}}\dfrac{\ \#\bigl\{0\le 
j\le 
k\,;\, \|P_{l}\,T^{\,j}P_{l}\,x\|\ge 
2CX_{l}\bigr\}}{k+1}=1$,
\end{enumerate}
\smallskip
then $T$ is not frequently hypercyclic.
\end{theorem}

\smallskip
The usefulness of this result lies in the fact that the lower bounds for
the densities are given in terms of norms $\|P_nT^jP_nx\|$ which are easily computable. 
Moreover, since $P_nx$ is periodic, we can determine the cardinality of the sets $\{0\le j\le k; \|P_nT^jP_nx\|\ge 2C X_n\}$ 
by examining only a fixed finite number of iterates (independent of $k$).
\par\smallskip
Theorem \ref{Theorem 49} will follow very easily from the next lemma, which 
provides, under conditions (1), (2), and (3) above on the projections 
$P_{n}$, lower bounds for the upper and lower density of some sets 
$\mathcal{N}_{T}(x, B(0,\varepsilon )^{c})$, where $x$ 
is any  non-zero vector of $\ell_p(\N)$ 
and $\varepsilon $ is a 
positive number depending on $x$. Here, of course, $B(0,\varepsilon)^c$ denotes the complement of the the ball $B(0,\varepsilon)$ in $\ell_{p}(\N)$.

\begin{lemma}\label{Theorem 48}
 Let $T$ be an operator of \cct\ on $\ell_{p}(\N)$. Fix $x\in\ell_p(\N)\setminus\{0\}$ 
and suppose that there exist

\begin{enumerate}
 \item[-] a positive constant $C$,
 \item[-] a non-increasing sequence $(\beta _{l})_{l\ge 1}$ of positive real 
numbers with $\sum_{l\ge 1}\sqrt{\beta _{l}}\le 1$,
\item[-] a sequence $(X_{l})_{l\ge 0}$ of non-negative real numbers,
\item[-] a non-decreasing sequence $(N_{l})_{l\ge 1}$ of integers tending to infinity,
\end{enumerate}
such that the following conditions are satisfied:
\smallskip
\begin{enumerate}
 \item[\emph{(1)}] $\|P_{n}x\|\le X_{n}$ for every $n\ge 0$;
 \item[\emph{(2)}] $\sup\limits_{j\ge 0}\|P_{n}\,T^{\,j}P_{l}\,x\|\le 
C\beta _{l}X_{l}$ for every $l\ge 1$ and every $0\le n<l$;
\item[\emph{(3)}] $\sup\limits_{0\le j\le N_{l}}\|P_{n}\,T^{\,j}P_{l}\,x\|\le C\beta _{l}\|P_{l}x\|$
for every $l\ge 1$ and every $0\le n<l$;
\item[\emph{(4)}] $\sup\limits_{j\ge 0}\sum\limits_{l>n}\|P_{n}\,T^{\,j}P_{l}x\|>CX_{n}$ for every $n\ge 0$.
\end{enumerate}
\smallskip
Then there exists $\varepsilon >0$ such that
\[
\underline{\vphantom{p}\textrm{\emph{dens}}}\ \mathcal{N}_{T} \bigl(x,B(0,\varepsilon )^c\bigr)\ge 
\liminf_{l\to\infty }\ \inf_{k\ge N_{l}}\ \dfrac{\ \#\bigl\{0\le 
j\le 
k\,;\, \|P_{l}\,T^{\,j}P_{l}\,x\|\ge 2CX_{l}\bigr\}}{k+1}
\]
and 
\[
\overline{\text{\emph{dens}}}\ \mathcal{N}_{T} \bigl(x,B(0,\varepsilon )^{c} \bigr)
\ge 
\liminf_{l\to\infty }\ \inf_{k\ge N_{\min(\varphi ^{-1}(l))}}\ 
\dfrac{\ \#\bigl\{0\le 
j\le 
k\,;\, \|P_{l}\,T^{\,j}P_{l}\,x\|\ge 2CX_{l}\bigr\}}{k+1}\cdot
\]
\end{lemma}

Let us show how this lemma implies Theorem \ref{Theorem 49}.

\begin{proof}[Proof of Theorem \ref{Theorem 49}] 
For every $n\ge 0$ and every 
 $x\in\ell_p(\N)$, the vector $\sum_{0\le l\le n}
 P_{l}\,x$ is a periodic vector for $T$. Hence, if $x$ is 
hypercyclic for $T$ then
\begin{equation}\label{Equation 17}
 \sup_{j\ge 0}\ \sum_{l>n}\|P_{n}T^{\,j}P_{l}\,x\|=\infty .
\end{equation}
Indeed, otherwise we would have
\[
\sup_{j\ge 0}\ \|P_{n}T^{\,j}\,x\|\le\sup_{j\ge 0}\ 
\Bigl\|T^{\,j}\Bigl(\sum_{0\le l\le n}P_{l}\,x \Bigr)\Bigr\|
+\sup_{j\ge 0}\ \sum_{l>n}\|P_{n}T^{\,j}P_{l}\,x\|<\infty, 
\]
a contradiction with the hypercyclicity of $x$. 
\par\smallskip 
By (\ref{Equation 17}), condition (4) of Lemma \ref{Theorem 48} is satisfied for any 
choice of the sequence $(X_{n})_{n\ge 0}$. It follows that if condition (C) in Theorem \ref{Theorem 49}  
is satisfied, then, for any $x\in HC(T)$, one can find $\varepsilon >0$ such that 
$\overline{\rm dens}\,\bigl(\mathcal N_T(x,B(0,\varepsilon)\bigr)=0$; whereas if (C') is 
satisfied, then, 
for any $x\in HC(T)$, one can find $\varepsilon >0$ such that 
$\underline{\rm dens}\,\bigl(\mathcal N_T(x,B(0,\varepsilon)\bigr)=0$. This concludes the proof.
\end{proof}

Before giving the proof of Lemma \ref{Theorem 48}, let us explain the general idea. 
The starting point is given by 
the inequalities
\[\|T^jx\|\ge \|P_nT^jx\|\ge \|P_nT^jP_nx\|-\sum_{l>n}\|P_nT^jP_lx\|.\]
In other words, we get a lower bound for $\|T^jx\|$ as soon as we have a lower bound for 
 some quantity $\|P_nT^jP_nx\|$ and an upper bound for $\|P_nT^jP_lx\|$, $l>n$. 
The role of the lower bound for $\|P_nT^jP_nx\|$ will be played by $2C X_n$, which will 
be bigger than $2\varepsilon$. Assumption~(3) will provide us with an upper bound for  
$\|P_nT^jP_lx\|$ when $j$ is smaller than $N_l$. However, condition~(4) tells us 
that $\|P_nT^jP_lx\|$ will be large for some $j\ge 0$ and some $l>n$. 
Assumption~(2) will then be used to deduce that $X_l\ge X_n$, which will allow us 
to deduce that $2C X_l\ge 2\varepsilon$ and to repeat the above arguments for $\|P_lT^jx\|$, and so on.

\begin{proof}[Proof of Lemma \ref{Theorem 48}]
If $x$ is a non-zero periodic vector for $T$, one can obviously choose $\varepsilon 
>0 $ such that 
\[
\underline{\vphantom{p}\textrm{dens}}\ \mathcal{N}_{T}(x,B(0,\varepsilon )^c)=1
\]
and the conclusion of Lemma \ref{Theorem 48} is satisfied. So we 
henceforward suppose that $x$ is not a periodic vector for $T$. 
\par\smallskip
We first note that 
there exists an integer $l_{0}\ge 1$ such that 
\[
\|P_{l_{0}}\,x\|\ge \sqrt{\beta _{l_{0}}}\,\|x-P_{0}\,x\|.
\]
Indeed, since $\sum_{l\ge 1}\sqrt{\beta _{l}}\le 1$ it is impossible to have 
$\|P_{l}\,x\|<\sqrt{\beta _{l}}\ \|x-P_{0}\,x\|$ for every $l\ge 1$. 
Since $x$ is not periodic, $x-P_{0}\,x$ is non-zero. Assumption (1) of Lemma \ref{Theorem 48} 
implies thus that $X_{l_{0}}$ is non-zero and we set $\varepsilon=CX_{l_0}$.
\par\smallskip
We now 
construct, by induction on $\gn$, a strictly increasing  sequence of integers 
$(l_{n})_{\gn}$ such that if we set 
\begin{align*}
 j_{n-1}&:=\smash[b]{\min\,\biggl\{j\ge \,
0\,;\,\sum_{l>l_{n-1}}\|P_{l_{n-1}}T^{j}P_{l}\,x\|>CX_{l_{n-1}}\biggr\},
}
\intertext{then}
j_{n-1}&>\smash[t]{N_{l_{n}}\quad\textrm{and}\quad 
X_{l_{n}}\ge\dfrac{1}{\sqrt{\beta 
_{l_{n}}}}\,X_{l_{n-1}}\quad \textrm{for every } \gn.}
\end{align*}
Observe that this does make sense: the set involved in the definition of $j_{n-1}$ is non-empty 
by assumption (4).

 Suppose that the integers $l_{1},\dots,l_{n-1}$ 
have already been constructed. Then there exists an integer $l_{n}>l_{n-1}$ 
with the property that 
\begin{equation}\label{Equation 16}
 \|P_{l_{n-1}}T^{\,j_{n-1}}P_{l_{n}}\,x\|>C\sqrt{\beta 
_{l_{n}}}X_{l_{n-1}}.
\end{equation}
Indeed, if we had $\|P_{l_{n-1}}T^{\,j_{n-1}}P_{l}\,x\|\le 
C\sqrt{\beta 
_{l}}X_{l_{n-1}}$ for every $l>l_{n-1}$, this would imply that 
$\sum_{l>l_{n-1}}\|P_{l_{n-1}}T^{\,j_{n-1}}P_{l}\,x\|\le 
CX_{l_{n-1}}$, violating the definition of $j_{n-1}$.
\par\smallskip
By assumption (2), we have
\[
X_{l_{n}}\ge\dfrac{1}{C\beta 
_{l_{n}}}\|P_{l_{n-1}}T^{\,j_{n-1}}P_{l_{n}}\,x\|\ge \dfrac{1}{\sqrt{
\beta 
_{l_n}}}\,X_{l_{n-1}}.
\]
In order to show that $j_{n-1}>N_{l_{n}}$, we observe that for every $0\le 
j\le N_{l_{n}}$,
\[
\|P_{l_{n-1}}T^{\,j}P_{l_{n}}\,x\|\le C\beta _{l_{n}}\,\|P_{l_{n}}
\,x\|\le C\beta _{l_{n}}\,\|x-P_{0}x\|\le \dfrac{C\beta 
_{l_{n}}}{\sqrt{\beta 
 _{l_{0}}}}\,\|P_{l_{0}}\,x\|\le \dfrac{C\beta _{l_{n}}}{\sqrt{\beta 
 _{l_{0}}}}\,X_{l_{0}},
\]
by assumption (3) and the definition of $l_{0}$. Since the sequence $(\beta 
_{l})_{l\ge 0}$ is decreasing, this 
yields that $\|P_{l_{n-1}}T^{\,j}P_{l_{n}}\,x\|\le C\sqrt{\beta _{l_n}}
\,X_{l_{0}}$. Also, the induction hypothesis implies that $X_{l_{0}}\le
X_{l_{1}}\le\dots\le X_{l_{n}}$, so that 
$\|P_{l_{n-1}}T^{\,j}P_{l_{n}}\,x\|\le C\sqrt{\beta _{l_n}}
\,X_{l_{n-1}}$ for every $0\le j\le N_{l_{n}}$. Inequality 
(\ref{Equation 16}) thus implies that $j_{n-1}>N_{l_{n}}$ and this 
finishes the induction. Note that the sequence $(j_{n})_{n\ge 0}$ tends 
to infinity as $n$ tends to infinity.
\par\smallskip
For any integer $s\geq 1$, let us denote by $n_{s}$ be the smallest integer such that $s$ belongs to $ 
[j_{n_{s}-1},j_{n_{s}})$. Then $n_{s}$ tends to infinity as $s$ tends to 
infinity. Since
\[
\smash[b]{\|T^{\,j}\,x\|\ge\|P_{n}\,T^{\,j}\,x\|\ge 
\|P_{n}\,T^{\,j}P_{n}\,x\|-
\sum_{l>n}\|P_{n}\,T^{\,j}P_{l}\,x\|}
\]
for every $j\ge 0$ and $n\ge 0$, we have
\[
\|T^{\,j}x\|\ge \|P_{l_{n}}T^{\,j}P_{l_{n}}\,x\|-C\,X_{l_{n}}\quad  
\textrm{for every}\ 0\le j< j_{n}.
\]
It follows from this inequality and the fact that $j<j_{n_{s}}$ for every 
$j<s$ that 
\begin{align*}
 \bigl\{0\le j\le s\,;\,\|P_{l_{n_{s}}}T^{\,j}P_{l_{n_{s}}}\,x\|
\ge 2CX_{l_{n_{s}}}\bigr\}&\subseteq
\bigl\{0\le j\le s\,;\,\|T^{\,j}x\|\ge CX_{l_{n_{s}}}\bigr\}\\
&\subseteq
\bigl\{0\le j\le s\,;\,\|T^{\,j}x\|\ge CX_{l_{0}}\bigr\}.
\end{align*}
Hence, since $\varepsilon =CX_{l_{0}}$, we have
\[
\underline{\vphantom{p}\textrm{dens}}\ \mathcal{N}_{T}\bigl(x, B(0,\varepsilon )^c\bigr)\ge 
\liminf_{s\to\infty }\ \dfrac{\ \#\bigl\{0\le 
j\le 
s\,;\, \|P_{l_{n_{s}}}\,T^{\,j}P_{l_{n_{s}}}\,x\|\ge 
2CX_{l_{n_{s}}}\bigr\}}{s+1}\cdot
\]
Now $N_{l_{s}}<j_{n_{s}-1}\le s$, so that $s$ belongs to the interval
$[N_{l_{s}},\infty )$ for each $s\ge 1$. Thus
\begin{align*}
\dfrac{\ \#\bigl\{0\le 
j\le 
s\,;\, \|P_{l_{n_{s}}}\,T^{\,j}P_{l_{n_{s}}}\,x\|\ge 
2CX_{l_{n_{s}}}\bigr\}}{s+1}&\ge\\
\inf_{k\ge N_{l_{s}}}&\dfrac{\ \#\bigl\{0\le 
j\le 
k\,;\, \|P_{l_{n_{s}}}\,T^{\,j}P_{l_{n_{s}}}\,x\|\ge 
2CX_{l_{n_{s}}}\bigr\}}{k+1}
\end{align*}
and 
\begin{align*}
 \underline{\vphantom{p}\textrm{dens}}\ \mathcal{N}_{T}
\bigl(x,B(0,\varepsilon )^c\bigr)&\ge 
\liminf_{s\to\infty }\ \inf_{k\ge N_{l_{s}}}\ \dfrac{\ \#\bigl\{0\le 
j\le 
k\,;\, \|P_{l_{n_{s}}}\,T^{\,j}P_{l_{n_{s}}}\,x\|\ge 
2CX_{l_{n_{s}}}\bigr\}}{k+1}\\
&\ge
\liminf_{l\to\infty }\ \inf_{k\ge N_{l}}\ \dfrac{\ \#\bigl\{0\le 
j\le 
k\,;\, \|P_{l}\,T^{\,j}P_{l}\,x\|\ge 
2CX_{l}\bigr\}}{k+1}
\end{align*}
since $l_{n_s}\to\infty$ as $s\to\infty$. This proves the 
first part of Lemma \ref{Theorem 48}. 
\par\smallskip
As for the second part, we proceed 
in the same way, starting from the inequality
\[
 \overline{\textrm{dens}}\ \mathcal{N}_{T}\bigl(x,B(0,\varepsilon )^c\bigr)\ge 
\limsup_{s\to\infty }\ \dfrac{\ \#\bigl\{0\le 
j\le 
s\,;\, \|P_{l_{n_{s}}}\,T^{\,j}P_{l_{n_{s}}}\,x\|\ge 
2CX_{l_{n_{s}}}\bigr\}}{s+1}\cdot 
\]
Our first observation is that $\min(\varphi ^{-1}(l_{n}))\le l_{n+1}$ for 
every $\gn$. This follows, in a slightly roundabout way, from the fact 
that 
$P_{l_{n}}\,T^{j_{n}}P_{l_{n+1}}$ is non-zero. Indeed, the definition of $T$ 
implies that there exists an integer $i\ge 1$ such that 
$\varphi ^{i}(l_{n+1})=l_{n}$. Then $\varphi ^{i-1}(l_{n+1})$ belongs to 
$\varphi ^{-1}(l_{n})$, and since $\varphi ^{i-1}(l_{n+1})\le l_{n+1}$, 
it follows that $\min(\varphi ^{-1}(l_{n}))\le l_{n+1}$.
\par\smallskip 
We now claim that there exists a strictly increasing sequence 
 of integers $(s_i)_{i\ge 1}$ such that $s_i\ge N_{\min(\varphi ^{-1}
(l_{n_{s_i}}))}$ for every $i\ge 1$. Recall that 
$N_{l_{n_{i}+1}}<j_{n_{i}}$ for all $i\geq 1$. Thus there exists 
$s_i\in[j_{n_{i}-1},j_{n_{i}})$ such that $s_i\ge
N_{l_{n_{i}+1}}$. As $s_i$ belongs to the interval 
$[j_{n_{i}-1},j_{n_{i}})$, we have $n_{s_i}\le n_{i}$, and thus $s_i
\ge N_{l_{n_{s_i}+1}}\ge N_{\min(\varphi ^{-1}(l_{n_{s_i}}))}$. 
Extracting if necessary in order to make the sequence $(s_i)_{i\ge 1}$ 
strictly increasing, we obtain the sequence we are looking for.
\par\smallskip 
We now have
\begin{align*}
  \overline{\textrm{dens}}\ \mathcal{N}_{T}
\bigl(x,& B(0,\varepsilon )^c\bigr)\ge 
\liminf_{i\to\infty }\ \dfrac{\ \#\bigl\{0\le 
j\le 
s_i\,;\, \|P_{l_{n_{s_{i}}}}\,T^{\,j}P_{l_{n_{s_i}}}\,x\|\ge 
2CX_{l_{n_{s_i}}}\bigr\}}{s_i+1}\\
&\ge 
\liminf_{n\to\infty }\ \ \inf_{k\ge N_{\min(\varphi 
^{-1}(l_{n}))}}\dfrac{\ \#\bigl\{0\le 
j\le 
k\,;\, \|P_{l_{n}}\,T^{\,j}P_{l_{n}}\,x\|\ge 
2CX_{l_{n}}\bigr\}}{k+1}
\\
&\ge 
\liminf_{l\to\infty }\ \ \inf_{k\ge N_{\min(\varphi 
^{-1}(l))}}\dfrac{\ \#\bigl\{0\le 
j\le 
k\,;\, \|P_{l}\,T^{\,j}P_{l}\,x\|\ge 
2CX_{l}\bigr\}}{k+1}\cdot
\end{align*}
This concludes the proof of Lemma \ref{Theorem 48}.
\end{proof}
\par\smallskip

\subsubsection{How to check the assumptions} The assumptions of Theorem \ref{Theorem 49} are partly of 
an ``abstract" nature, since they involve the projections $P_n$ and not explicitly the parameters $v$, $w$, $\varphi $,
and $b$ defining $T$. We now provide concrete conditions on $v$, $w$, $\varphi $,
and $b$ which imply that the 
operator of \cct\ $\tvw$ 
satisfies the assumptions of Theorem~\ref{Theorem 49}. 
\par\smallskip
As a rule, 
conditions (1), (2), and (3) will be obtained by requiring that the 
products of weights 
\[ 
|v_n|\,\cdot\,\sup_{j\in 
[b_{\varphi(n)},b_{\varphi(n)+1})}\ 
\prod_{s=b_{\varphi(n)}+1}^{j}|w_{s}|
\]
decrease sufficiently rapidly and by setting, given a hypercyclic vector 
$x$ for $\tvw$,
\[
X_l:=\Bigl\|\sum_{k=b_l}^{b_{l+1}-1}\Bigl(\prod_{s=k+1}^{b_{l+1}-1}
w_s\Bigr)\,x_ke_k\Bigr\|,\quad l\ge 0.
\]

More precisely, we have the following fact, which generalizes Claims 5 and 6 in \cite{Me}.

\begin{fact}\label{Proposition 50}
 Let $T=\tvw$ be an operator of \cct\ on $\ell_{p}(\N)$, and  
 let $(C_{n})_{n\ge 
0}$ be a sequence of positive numbers with $0<C_n<1$. Assume that 
\[
|v_n|\ .\sup_{j\in 
[b_{\varphi(n)},b_{\varphi(n)+1})}\ \Bigl(
\prod_{s=b_{\varphi(n)}+1}^{j}|w_{s}|\Bigr)\le C_{n}\quad\hbox{  for every $n\ge 0$. }
\]
 Then, for any $x\in\ell_p(\N)$, 
we have for every $l\ge 1$ and 
every $0\le n<l$,
\begin{enumerate}
 \item [\rm {(1)}]$\quad
 {\ds\sup_{j\ge 0}\ \|P_{n}T^{\,j}P_{l}\,x\|\le 
C_{l}\,(b_{l+1}-b_{l})^{\frac{p-1}{p}}\ \Bigl\|\sum_{k=b_{l}}^{b_{l+1}-1}
\Bigl(\prod_
{s=k+1}^{b_{l+1}-1}w_{s}\Bigr)\,x_{k}e_{k}\Bigr\|}
$
\par\smallskip \noindent and
\item[\rm {(2)}]
$
\quad {\ds\sup_{j\le N}\|P_{n}T^{\,j}P_{l}\,x\|\le 
C_{l}\,(b_{l+1}-b_{l})^{\frac{p-1}{p}}\Bigl(\ \sup_{b_{l+1}-N\le 
k<b_{l+1}}\prod_{s=k+1}^{b_{l+1}-1}|w_s|\Bigr)\|P_{l}\,x\|}
$
\par\smallskip \noindent for every $1\le N\le b_{l+1}-b_{l}$.
\end{enumerate}
\end{fact}

\begin{proof}
 Fix $l\ge 1$ and $0\le n<l$. We first remark that if $n\neq \varphi ^{M}
 (l)$ for all $M\ge 1$, the definition of $T$ implies that $P_{n}T^{\,j}
 P_{l}\,x=0$ for every $j\ge 0$ (this argument was already used at the end 
of the proof of Lemma \ref{Theorem 48}). So the required inequalities 
are obvious in this case.
\par\smallskip
Suppose now that $n=\varphi ^{M}(l)$ for some integer $M\ge 1$. In this case, we claim that
\begin{equation}\label{Equation 18}
 \sup_{j\ge 0}\ \|P_{n}T^{\,j} \,e_{k}\|\le |v_{l}| \
\Bigl(\prod_{s=k+1}^{b_{l+1}-1}|w_{s}|\Bigr)\ \sup_{j\ge 0}\ 
\|P_{n}T^{\,j} 
e_{b_{\varphi(l)}}\|
\end{equation}
for every $k\in[b_{l},b_{l+1})$. This inequality will allow us to run an 
induction procedure in order to estimate the quantity
$\sup\limits_{j\ge 0}\ \|P_{n}T^{\,j}P_l \,x\|$.
\par\smallskip
In order to prove (\ref{Equation 18}), we fix $j_0\geq 0$ and we consider separately several 
cases.

\par\smallskip \noindent 
 - Suppose first that $0\le j_0<b_{l+1}-k$. Then $T^{\,j_0}\,e_{k}$ is a 
scalar multiple of $e_{k+j_0}$, and $P_{n}T^{\,j_0}\,e_{k}=0$.
\par\smallskip \noindent
- Now, suppose that $b_{l+1}-k\le j_0<b_{l+1}-k+b_{l+1}-b_{l}$. Then
\begin{align*}
 P_{n}T^{\,j_0}\,e_{k}&=\smash[b]{P_{n}T^{\,j_0-(b_{l+1}-k)}\ \Bigl(v_{l}\,
\Bigl(\prod_{s=k+1}^{b_{l+1}-1}w_{s} \Bigr)\,e_{b_{\varphi (l)}}\Bigr),}
\intertext{so that}
\|P_{n}T^{\,j_0}\,e_{k}\|&\smash[t]{\le|v_{l}|\,\Bigl(\prod_{s=k+1}^{b_{l+1}
-1 }
|w_{ s}|\,\Bigr)
\,\sup_{j\ge 0}\ \|P_{n}T^{\,j}\,e_{b_{\varphi (l)}}\|.}
\end{align*}

- Finally, suppose that $b_{l+1}-k+b_{l+1}-b_{l}\le j_0<2(b_{l+1}-b_{l})$.  In this case, we write
\[T^{j_0}e_k=\Bigl(\prod_{s=b
_{l}+1}^{k}w_{s} \Bigr)^{-1}\, T^{j_0+k-b_l}e_{b_l}.
\]
By assumption, we have $2(b_{l+1}-b_l)\leq j_0+k-b_l< 2(b_{l+1}-b_l)+(b_{l+1}-b_l)$; that is, 
\[ j_0+k-b_l=i+2(b_{l+1}-b_l)\qquad\hbox{with $0\leq i<b_{l+1}-b_l$}.
\]
 Since $T^{\,2(b_{l+1}-b_{l})}\,e_{b_l}=e_{b_l}$, 
it follows that 
\[T^{j_0}e_k=\Bigl(\prod_{s=b
_{l}+1}^{k}w_{s} \Bigr)^{-1}\,T^{i} e_{b_l}=\Bigl(\prod_{s=b
_{l}+1}^{k}w_{s} \Bigr)^{-1}\Bigl(\prod_{s=b
_{l}+1}^{l+i}w_{s} \Bigr)\, e_{b_l+i}
\]
and hence that 
$P_{n}T^{\,j_0}\, e_{k}=0$.

\par\smallskip
We have thus proved the required inequality for every $0\le j_0<2(b_{l+1}-
b_{l})$. Since $T^{\,2(b_{l+1}-b_{l})}\,e_{k}=e_{k}$, it holds true in fact for 
every $j\ge 0$, and this proves (\ref{Equation 18}).
\par\smallskip
Let now $M_{0}$ be the minimum of the integers $M\ge 1$ such that $\varphi 
^{M}(l)=n $. We start from the straightforward estimate
\[
\sup_{j\ge 0}\ \|P_{n}T^{\,j}P_{l}\,x \|\le 
\sum_{k=b_{l}}^{b_{l+1}-1}|x_{k}|\ \sup_{j\ge 0}\|P_{n}T^{\,j}\,e_{k}\|
\]
and apply (\ref{Equation 18}). This gives
\begin{align*}
\sup_{j\ge 0}\ \|P_{n}T^{\,j}P_{l}\,x\|&\le
 \sum_{k=b_l}^{b_{l+1}-1}|x_{k}|\,\cdot\,|v_{l}| \,\cdot\,
\Bigl(\prod_{s=k+1}^{b_{l+1}-1}|w_s|\Bigr)\ \sup_{j\ge 0}\ \|P_{n}T^{\,j} 
e_{b_{\varphi(l)}}\|\\
&= |v_{l}|\,\cdot\,  \sup_{j\ge 0}\ \|P_{n}T^{\,j}\, 
e_{b_{\varphi(l)}}\|\,\cdot\, 
\Bigl\|\sum_{k=b_{l}}^{b_{l+1}-1}\Bigl(\prod_{s=k+1}^{b_{l+1}-1}w_{s} 
\Bigr)x_{k} e_{k}\Bigr\|_1,
\end{align*}
where $\Vert z\Vert_1$ denotes the $\ell^{\,1}$ norm of a vector $z\in c_{00}$.
\par\smallskip
{By induction, we obtain}
\begin{align*}
\sup_{j\ge 0}\ \|P_{n}T^{\,j}P_{l}\,x\|&\le  |v_{l}| \,\cdot\,
\prod_{r=1}^{M_{0}-1}\Big(|v_{\varphi^r(l)}|\ . 
\prod_{s=b_{\varphi ^{r}(l)}+1}^{b_{
\varphi ^{r}(l)+1}-1}|w_{s}|\Big)\ .\\
&\hspace{3cm}\sup_{j\ge 0}\ \|P_{n}T^{\,j}\, 
e_{b_{n}}\|\,\cdot\, 
\Bigl\|\sum_{k=b_{l}}^{b_{l+1}-1}\Bigl(\prod_{s=k+1}^{b_{l+1}-1}
w_{s}\Bigr)x_{k}e_{k}\Bigr\|_1\\
&=\prod_{r=1}^{M_{0}-1}\Big(|v_{\varphi^{r-1}(l)}|\prod_{s=b_{\varphi^r(l)}+1
} ^{
b_{\varphi^r(l)+1}-1}|w_{s}|\ \Big)\,\cdot\,|v_{\varphi^{M_{0}-1}(l)}|\\
&\hspace{3cm}\sup_{j\ge 
0}\ \|P_{n}T^{\,j} 
e_{b_{n}}\|\,\cdot\, 
\Bigl\|\sum_{k=b_{l}}^{b_{l+1}-1}\Bigl(\prod_{s=k+1}^{b_{l+1}-1}
w_{s}\Bigr)\,x_{k}e_{k}\Bigr\|_1.
\end{align*}

{Applying the assumption of Fact \ref{Proposition 50}, it follows that}
\begin{align*}
\sup_{j\ge 0\ }\ \|P_{n}T^{\,j}P_{l}\,x\|&\le 
\Bigl(\prod_{r=1}^{M_{0}-1}C_{\varphi^{r-1}(l)}\Bigr)\,\cdot\,|v_{\varphi
^{M_{0}-1}(l)}|\,\cdot\,\sup_{j\ge 0}\ \|P_{n}T^{\,j}\,e_{b_{n}}\|\,\cdot
\\
&\hspace{3cm}
\Bigl\|\sum_{k=b_{l}}^{b_{l+1}-1}\Bigl(\prod_{s=k+1}^{b_{l+1}-1}
w_{s}\Bigr)x_{k}e_{k}\Bigr\|_{1}.
\end{align*}

{Now, the definition of $T$ and the fact that 
$T^{\,2(b_{n+1}-b_{n})}\,e_{b_{n}}=e_{b_{n}}$ show that}
\begin{align*}
\sup_{j\ge 0\ }\ 
\|P_{n}T^{\,j}\,e_{b_{n}}\|&\le\smash[t]{\sup_{j\in[b_{n},b_{n+1})}
\ \prod_{s=b_{n}+1}^{j}|w_{s}|.}
\end{align*}

Since $n=\varphi ^{M_{0}}(l)$, it follows (by the assumption of Fact \ref{Proposition 50}), that 
\begin{align*}
 \sup_{j\ge 0}\ \|P_{n}T^{\,j}P_{l}\,x\|&\le
\Bigl(\prod_{r=1}^{M_{0}-1}C_{\varphi^{r-1}(l)}\Bigr)\,\cdot\,
C_{\varphi^{M_{0}-1 } (l) }\,\cdot\,
\Bigl\|\sum_{k=b_l}^{b_{l+1}-1}\Bigl(\prod_{s=k+1}^{b_{l+1}-1}
w_{s}\Bigr)\,x_{k}e_{k}\ \Bigr\|_{1}\\
&\le 
C_{l}\,\cdot\,\Bigl\|\sum_{k=b_{l}}^{b_{l+1}-1}\Bigl(\prod_{s=k+1}^{b_{l+1}-1}
w_{s}\Bigr)\,x_{k}e_{k}\Bigr\|_{1},
\end{align*}
{because all the constants $C_{n}$ are supposed to belong to 
$(0,1)$ and $\varphi ^{0}(l)=l$. By H\"older's inequality, we conclude that }
 \begin{align*}
 \sup_{j\ge 0}\ \|P_{n}T^{\,j}P_{l}\,x\|
&\le 
C_{l}\,\cdot\,(b_{l+1}-b_{l})^{1-\frac1{p}}\,\cdot\,\Bigl\|
\sum_{k=b_{l}}^{b_{l+1} -1 } \Bigl(
\prod_{s=k+1}^{b_{l+1}-1}w_{s}\Bigr)\,x_{k}e_{k}\Bigr\|,
\end{align*}
which proves assertion (1) of Fact \ref{Proposition 50}.
\par\smallskip 
The proof of assertion (2) is now straightforward. Indeed, we have for every 
$1\le N\le b_{l+1}-b_{l}$ and every $0\le j\le N$:
\[ 
T^{\,j}P_{l}\,x=\smash[t]{T^{\,j}\,\Bigl(\sum_{k=b_{l}}^{b_{l+1}-N-1}x_{k}
e_{k} 
\ \Bigr)+T^{\,j}\,\Bigl(\sum_{k=b_{l+1}-N}^{b_{l+1}-1}x_{k}e_{k} 
\ \Bigr).}
\]
{Since the first term belongs to the linear span of the vectors 
$e_{i}$, $b_{l}\le i<b_{l+1}$, it follows that }
\[
P_{n}T^{\,j}P_{l}\,x=\smash[t]{P_{n}T^{\,j}\,\Bigl(\sum_{k=b_{l+1}-N}^{b_{
l+1}-1}x_{
k}e_{k} 
\ \Bigr),}
\]
and so, by the already proved assertion (1) of Fact \ref{Proposition 50},
\begin{align*}
\|P_{n}T^{\,j}P_{l}\,x\|&\le 
C_{l}\,\cdot\,(b_{l+1}-b_{l})^{1-\frac1{p}}\,\cdot\,\Bigl\|
\sum_{k=b_{l+1}-N}^{b_{ l+1 } -1 }
\Bigl(\prod_{s=k+1}^{b_{l+1}-1}w_{s}\Bigr)\,x_{k}e_{k}
\Bigr\|\\
&\le C_{l}\,\cdot\,(b_{l+1}-b_{l})^{1-\frac1{p}}\,\cdot\,\sup_{b_{l+1}-N\le 
k<b_{l+1}}\Bigl(\prod_{s=k+1}^{b_{l+1}-1}|w_{s}\,|\Bigr)\,\cdot\,\|P_{l}\,x\|.
\end{align*}

This concludes the proof of Fact \ref{Proposition 50}.
\end{proof}

The next fact provides conditions 
on the parameters of an operator of \cct\,  ensuring that assumption (C) or 
(C') of Theorem \ref{Theorem 49} is satisfied. This generalizes Claim 7 of 
\cite{Me}.
\begin{fact}\label{Proposition 51}
 Let $T=\tvw$ be an operator of \cct\ on $\ell_{p}(\N)$,
  and let $x\in
 \ell_p(\N)$.  
 Fix $l\ge 0$, and define
 \[
X_{l}=\Bigl\|\sum_{k=b_{l}}^{b_{l+1}-1}\Bigl(\prod_{s=k+1}^{b_{l+1}-1}
w_{s}\Bigr)\,x_{k}e_{k}\Bigr\|.
\]
 Suppose that there exist two integers $0\le 
k_{0}<k_{1}\le b_{l+1}-b_{l}$ such that 
\[
|w_{b_{l}+k}|=1\quad\textrm{for every}\ 
k\in(k_{0},k_{1})\quad\textrm{and}\ 
\ds\prod_{s=b_l+k_{0}+1}^{b_{l+1}-1}|w_{s}|=1.
\]
Then we have for every $J\ge 0$
\begin{align*}
 &\dfrac{1}{J+1}\ \#\Bigl\{0\le j\le J\,;\, \|P_{l}T^{\,j}P_{l}\,x\|\ge 
X_{l}/2\Bigr\}\\
 &\hspace{3cm}\ge 
1-2\bigl( b_{l+1}-b_l-(k_1-k_0)\bigr)\,\cdot\,\Bigl( \frac{1}{J+1}+\frac{1}{b_{
l+1}-b_{l}}\Bigr)\cdot
\end{align*}
\end{fact}
\begin{proof}
 Fix $j\ge 0$ and $k\in[b_{l},b_{l+1})$, and set $n:=j+k-b_{l}\;\mod (b_{l+1}-b_{l})$. Then
 \begin{align*}
P_{l}T^{\,j}\,e_{k}&=\pm\, 
\Bigl(\prod_{s=k+1}^{b_{l+1}-1}w_{s}\Bigr)\ \Bigl( 
\prod_{s=b_l+1}^{b_{l}+n}w_s\Bigr)\ \Bigl( 
\prod_{s=b_l+1}^{b_{l+1}-1}w_s\Bigr) ^{-1}
\ e_{
b_{l}+n } \\
&=  
\pm\,\Bigl( \prod_{s=k+1}^{b_{l+1}-1}w_{s}\Bigr)\ \Bigl( 
\prod_{s=b_{l}+n+1}^{b_{
l+1 } -1 } w_{s} \Bigr)^{-1} 
e_ {
b_{l}+n}
\end{align*}
(the plus or minus sign appearing in these equalities depends on the parity of the unique integer 
$s\ge 0$ such that $n$ belongs to the interval
$[s(b_{l+1}-b_{l}),(s+1)(b_{l+1}-b_{l}))$).
\par\medskip 
If we suppose that $n$ belongs to $[k_{0},k_{1})$, we have
\[
\prod_{s=b_{l}+n+1}^{b_{
l+1 } -1 } |w_{s}| =\prod_{s=b_l+k_{0}+1}^{b_{
l+1 } -1 } |w_{s}| =1,
\]
and thus there exists $\zeta _{j,k}\in\T$ such that 
\[
P_{l}T^{\,j}\,e_{k}=\zeta 
_{j,k}\,\cdot\,\Big(\prod_{s=k+1}^{b_{l+1}-1}w_{s}\Big)
\; e_{b_{l}+n }.
\]
Setting, for every $j\ge 0$,
\[
I_j:=\Bigl\{k\in [b_{l},b_{l+1})\,;\,j+k-b_{l} \mod (b_{l+1}-b_l)\ 
\textrm{does not belong to}\ 
[k_{0},k_{1})\Bigr\},
\]
we obtain that 
\[
P_{l}T^{\,j}P_{l}\,x=\sum_{k=b_{l}}^{b_{l+1}-1}x_{k}\,P_lT^{\,j}\,e_{k}=
  \sum_{k\,\in\,[b_{l},b_{l+1})\backslash I_{j}}x_{k}\,P_lT^{\,j}\,e_{k}+
  \sum_{k\,\in\, I_{j}}x_{k}\,P_lT^{\,j}\,e_{k}
\]
satisfies
\begin{align*}
 \|P_{l}T^{\,j}P_{l}\,x\|&\ge \Bigl\| 
\sum_{k\,\in\,[b_{l},b_{l+1})\backslash I_j}x_{k}\,P_lT^{\,j}\,e_{k}\,\Bigr\|=
\Bigl\| \sum_{k\,\in\,[b_{l},b_{l+1})\backslash I_j}x_{k}\,\zeta 
_{j,k}\,\Bigl(\prod_{s=k+1}^{b_{l+1}-1}w_{s} 
\Bigr)\,e_{k}\,\Bigr\|\\
&=\Bigl\| 
\sum_{k\,\in\,[b_{l},b_{l+1})\backslash I_j}x_{k}\,\Bigl(\prod_{s=k+1}^{b_{l+1}-1}w_{s} 
\Bigr)\,e_{k}\,\Bigr\|\ge X_{l}-\Bigl\| 
\sum_{k\,\in\,I_{j}}x_{k}\,\Bigl(\prod_{s=k+1}^{b_{l+1}-1}w_{s} 
\Bigr)\,e_{k}\,\Bigr\|.
\end{align*}

\smallskip
We now consider separately two cases:
\par\smallskip 
- Suppose first that $\ds\Bigl\| 
\sum_{k\,\in\,I_{j}}x_{k}\,\Bigl(\prod_{s=k+1}^{b_{l+1}-1}w_{s} 
\Bigr)\,e_{k}\,\Bigr\|<X_{l}/2$ for every $j\ge 0$. Then we have
$\|P_{l}T^{\,j}P_{l}\,x\|\ge X_{l}/2$ for every $j\ge 0$, and thus 
\[
\dfrac{1}{J+1}\ \#\bigl\{0\le j\le J\,;\,
\|P_{l}T^{\,j}P_{l}\,x\|\ge X_{l}/2\bigr\}=1\quad \textrm{for every}\ 
J\ge 0.
\]
\par\smallskip 
- Now, suppose that there exists $j_{0}\ge 0$ such 
that 
\[
\Bigl\| 
\sum_{k\,\in\,I_{j_{0}}}x_{k}\,\Bigl(\prod_{s=k+1}^{b_{l+1}-1}w_{s} 
\Bigr)\,e_{k}\,\Bigr\|\ge \dfrac{X_{l}}{2}\cdot
\]
Then we have for every integer $j\ge 0$ such that $I_{j}\cap 
I_{j_{0}}$ is empty:
\[
\|P_{l}T^{\,j}P_{l}\,x\|\ge\Bigl\| 
\sum_{k\,\in\,[b_{l},b_{l+1})\backslash I_j}\Bigl(\prod_{s=k+1}^{b_{l+1}-1}w_{s} 
\Bigr)x_{k}e_{k}\,\Bigr\|\ge
\Bigl\| \sum_{k\,\in\,
I_{j_{0}}}\Bigl(\prod_{s=k+1}^{b_{l+1}-1}w_{s} 
\Bigr)x_{k}e_{k}\,\Bigr\|\ge\dfrac{X_{l}}{2}\cdot
\]
It follows that, for every $J\ge 0$,
\begin{align}
 &\dfrac{1}{J+1}\ \#\Bigl\{0\le j\le J\,;\,
\|P_{l}T^{\,j}P_{l}\,x\|\ge X_{l}/2\Bigr\}\notag\\&
\hspace{5cm}\ge 1-\dfrac{1}{J+1}\ 
\#\Bigl\{0\le j\le J\,;\,
I_{j}\cap I_{j_{0}}\ne \emptyset\Bigr\}.\label{Equation 19}
\end{align}
Now, we remark that if we set $i_{j}:=j\mod(b_{l+1}-b_{j})$ for every $j\ge 
0$, we have
\[
I_{j}=
\begin{cases}
 \ \mathopen[b_{l}+k_{1}-i_{j},b_{l+1}\mathclose)\cup 
 \mathopen[b_{l},k_{0}-i_{j}\mathclose)& \quad\text{if}\ 0\le 
i_{j}<k_{0}\\
\ \mathopen[b_{l}+k_{1}-i_{j},b_{l+1}+k_{0}-i_{j}\mathclose)& 
\quad\text{if}\ k_{0}\le 
i_{j}\le k_{1}\\
\ \mathopen[b_{l},b_{l+1}+k_{0}-i_{j}\mathclose)\cup 
\mathopen[b_{l+1}+k_{1}-i_{j},b_{l+1}\mathclose)& \quad{\text{if}\ 
k_{1}<i_{j}<b_{l+1}-b_{l}}
\end{cases}
\]
and $\#I_{j}=b_{l+1}-b_{l}-(k_{1}-k_{0})$ for every $j\ge 0$. This 
particular structure of the sets $I_{j}$ implies that
\[
\#\ \Bigl\{0\le j\le J\,;\, I_{j}\cap I_{j_{0}}\ne \emptyset\Bigr\}\le
2(b_{l+1}-b_{l}-(k_{1}-k_{0}))\,\cdot\,\biggl(\,\biggl\lfloor\, \dfrac
{J}{b_{l+1}-b_{l}}\biggr\rfloor+1\, 
\biggr).
\]
Plugging this into (\ref{Equation 19}) yields that
\[
\dfrac{1}{J+1}\ \#\Bigl\{0\le j\le J\,;\,
\|P_{l}T^{\,j}P_{l}\,x\|\ge X_{l}/2\Bigr\}\ge 
1-2(b_{l+1}-b_{l}-(k_{1}-k_{0}))\Bigl(\dfrac{1}{J+1}+\dfrac{1}{b_{l+1}-b_{l
} } \Bigr),
\]
which is the inequality we were looking for.
\end{proof}
\par\smallskip
Facts~\ref{Proposition 50} and~\ref{Proposition 51} will be repeatedly used in Subsections~\ref{C+1type}, \ref{Subsection 4.5} and \ref{ex mixing} in order to exhibit the desired counterexamples.

\subsection{Operators of \cct: mixing or not mixing}\label{Subsection 4.4}
In this subsection, we give some conditions ensuring that an operator of \cct\ is or is not topologically mixing. This will allow us to  show in particular in Section \ref{SECTION+++} that the assumptions of our criterion 
for frequent hypercyclicity (Theorem \ref{Theorem 41}) do not imply the Operator Specification Property.

\begin{proposition}\label{Proposition mixing 0}
Let $T=\tvw$ be an operator of \cct\ on $\ell_{p}(\N)$.
 Suppose that for every $\eta>0$ and every $n\ge 0$, the set
\[S_{\eta,n}:=\bigcup_{m\in \mathcal{N}_{\eta,n}}\Big\{s\in [1,b_{m+1}-b_m)~:~ |v_m|\prod_{i=b_{m+1}-s}^{b_{m+1}-1}|w_{i}|>\frac{1}{\eta}\Big\}\quad\text{is cofinite,}\]
where $\mathcal{N}_{\eta,n}=\{m\in \varphi^{-1}(n)~:~|v_m|\prod_{i=b_m+1}^{b_{m+1}-1}|w_{i}|>\frac{1}{\eta}\}$.
Then $T$ is topologically mixing.
\end{proposition}

\begin{proof}
Since $T$ has a dense set of periodic points,
Corollary~\ref{prop mixing simple} implies that it suffices to show that for every $k\ge 0$ and every $\varepsilon>0$, 
there exists an integer $M\ge 1$ such that for every $m\ge M$, there exists a vector $z\in \ell_{p}(\N)$ such that $\|z\|<\varepsilon$ and $\|T^mz-e_k\|<\varepsilon$.
\par\smallskip
Fix thus $k\ge 0$ and $\varepsilon>0$. Let also
 $n\ge 0$ be such that $k$ belongs to $ [b_n,b_{n+1})$, and let $\eta>0$ be a sufficiently small positive number, to be fixed later on in the proof.
By assumption, the set $S_{\eta,n}$ is cofinite. So it suffices to show that  if $\eta$ is small enough, then, for every $s\in S_{\eta,n}$, 
one can find a vector $z_{s}\in\ell_{p}(\N)$ such that 
\[\|z_{s}\|<\varepsilon\quad\text{and}\quad \|T^{s+1+k-b_n}z_{s}-e_k\|<\varepsilon.\]
Fix $s\in S_{\eta,n}$, so that $s$ belongs to $ [1,b_{m+1}-b_m)$ for some integer $m$ of $ \mathcal{N}_{\eta,n}$ and 
\[
|v_m|\prod_{i=b_{m+1}-s}^{b_{m+1}-1}|w_{i}|>\frac{1}{\eta}\cdot
\]
 We define a vector $z_{s}$ by setting
 \[z_{s}:=\frac{1}{v_m\prod_{j=b_n+1}^k w_j\prod_{j=b_{m+1}-s}^{b_{m+1}-1} w_j}\, e_{b_{m+1}-s-1}.
 \]
Then
\[\|z_{s}\|=\frac{1}{|v_m|\prod_{j=b_n+1}^k |w_j|\prod_{j=b_{m+1}-s}^{b_{m+1}-1} |w_j|}<\frac{\eta}{\prod_{j=b_n+1}^k |w_j|};\]
so that $\Vert z_{s}\Vert<\varepsilon$ for every $s\in S_{\eta,n}$ if $\eta$ is small enough. Moreover, 
since $\varphi(m)=n$ we also have
\[T^{s+1+k-b_n}z_{s}=e_k-\frac{\prod_{j=b_m+1}^{b_m+k-b_n}w_j}{v_m\prod_{j=b_m+1}^{b_{m+1}-1}w_{j}\prod_{j=b_n+1}^k w_j}e_{b_{m}+k-b_n},
\]
and hence (since $m$ belongs to $\mathcal N_{\eta,n}$)
\[\|T^{s+1+k-b_n}z_{s}-e_k\|=\Bigl\vert\frac{\prod_{j=b_m+1}^{b_m+k-b_n}w_j}{v_m\prod_{j=b_m+1}^{b_{m+1}-1}w_{j}\prod_{j=b_n+1}^k w_j}\Bigr\vert
< \frac{\prod_{j=b_m+1}^{b_m+k-b_n}|w_j|}{\prod_{j=b_n+1}^k |w_j|}\,\eta .\]
Since the sequence of weights $(w_{j})_{j\ge 1}$ is bounded by a positive constant $M$,
\[\|T^{s+1+k-b_n}z_{s}-e_k\|\le \dfrac{M^{b_{n+1}-b_n}}{\prod_{j=b_n+1}^k |w_j|}\,\eta ,\]
and the quantity on the right hand side does not depend on $s\in S_{\eta,n}$.
So
we also have $\Vert T^{s+1+k-b_n}z_{s}-e_k\Vert<\varepsilon$ for every $s\in S_{\eta,n}$ provided $\eta$ is small enough.
\end{proof}

\par\smallskip
For operators of \cpt , the statement of Proposition~\ref{Proposition mixing 0} can be 
slightly simplified.

\begin{corollary}\label{Cor mixing}
Let $T=\tvw$ be an operator of \cpt\ on $\ell_{p}(\N)$.
 Suppose that for every $\varepsilon>0$, the set
\[S_{\varepsilon}:=\bigcup_{k\in \mathcal{K}_{\varepsilon}}\Bigl\{n\in [1,\Delta^{(k)}[\; ;\;  |v^{(k)}|\prod_{i=\Delta^{(k)}-n}^{\Delta^{(k)}-1}|w_{i}^{(k)}|>\frac{1}{\varepsilon}\Bigr\}\quad\text{is cofinite,}\]
where $\mathcal{K}_{\varepsilon}=\{k\ge 1~:~|v_k|\prod_{i=1}^{\Delta^{(k)}-1}|w_{i}^{(k)}|>\frac{1}{\varepsilon}\}$.
Then $T$ is topologically mixing.
\end{corollary}

\begin{proof} This follows immediately from the special structure of \cpt\ operators and the definition of the map $\varphi$ in this case.
\end{proof}

\par\smallskip
We now use Fact~\ref{Proposition 50} in order to formulate sufficient conditions for an \op\ of \cct\ to be \emph{not} topologically mixing.

\begin{proposition}\label{Proposition no mixing}
 Let $T=\tvw$ be a \cct\ operator on $\ell_{p}(\N)$. 
 Let $(C_{n})_{n\ge 
0}$ be a sequence of positive numbers with $0<C_n<1$ for every $n\ge 0$. Assume that 
\[
|v_n|\,\cdot \sup_{j\in 
[b_{\varphi(n)},b_{\varphi(n)+1})}\ \Bigl(
\prod_{s=b_{\varphi(n)}+1}^{j}|w_{s}|\Bigr)\le C_{n}\quad\hbox{  for every $n\ge 0$,}
\]
and that there exists a positive constant $K$ such that for infinitely many integers $n\ge 0$, we have
\begin{equation}\label{nonmix}
\sum_{l>n}
C_{l}\,(b_{l+1}-b_{l})^{1-\frac1p}\Bigl(\ \sup_{b_{l+1}-2(b_{n+1}-b_n)\le 
i<b_{l+1}}\prod_{s=i+1}^{b_{l+1}-1}|w_s|\Bigr)\leq K.
\end{equation} 
Then $T$ is not topologically mixing.
\end{proposition}

\begin{proof} Recall that we denote for each $l\ge 0$ by $P_l$ the canonical projection of $\ell_{p}(\N)$ onto the finite-dimensional space ${\rm span}\,[e_j;\; b_l\leq j<b_{l+1}]$. 
It is enough to show that if $n\ge 0$ satisfies the assumption (\ref{nonmix}), then 
\[ \bigl\|P_0T^{2(b_{n+1}-b_n)}x\bigr\|\le 1
\quad\textrm{ for every }x\in B\bigl(0,\frac{1}{1+K}\bigr).
\]
Indeed, this will imply that $2(b_{n+1}-b_n)$ does not belong to $ \mathcal N_T\bigl(B(0,\frac{1}{1+K}),B(3e_0,1)\bigr)$ for every such $n$, and hence that $T$ is not topologically mixing.
\par\smallskip
Let us fix a vector $x\in B\bigl(0,\frac{1}{1+K}\bigr)$ and an integer $n\ge 0$ satisfying (\ref{nonmix}). 
Recalling that
$T^{2(b_{n+1}-b_n)}e_k=e_k$ for every $0\le k< b_{n+1}$, we have 
\[ P_0T^{2(b_{n+1}-b_n)}x= P_0\, \sum_{l=0}^n P_lx +P_0\, \sum_{l>n} T^{2(b_{n+1}-b_n)}P_l x=P_0x+\sum_{l>n} P_0T^{2(b_{n+1}-b_n)}P_l x.
\] 
Moreover, if $l>n$, it follows from Fact~\ref{Proposition 50} that
\[\bigl\|P_0T^{2(b_{n+1}-b_n)}P_lx\bigr\|\le C_{l}\,(b_{l+1}-b_{l})^{1-\frac1{p}}\Bigl(\ \sup_{b_{l+1}-2(b_{n+1}-b_n)\le 
i<b_{l+1}}\prod_{s=i+1}^{b_{l+1}-1}|w_s|\Bigr)\, \|P_lx\|.\]
So we get that
\[\bigl\Vert P_0T^{2(b_{n+1}-b_n)}x\bigr\Vert\le \|P_0x\|+\sum_{l>n}\|P_0T^{2(b_{n+1}-b_n)}P_lx\|\leq \|x\|+ K\|x\|\le 1,\]
which proves Proposition \ref{Proposition no mixing}.
\end{proof}
\par\smallskip
After these technical preliminaries, we are now going to consider special classes  
of operators of \cct, for which the general conditions for frequent hypercyclicity, $\mathcal U$-frequent hypercyclicity and topological mixing 
obtained previously become rather transparent. This will allow us to derive easily the promised counterexamples.
\par\smallskip

\subsection{Operators of  \cput:  FHC does not imply ergodic}\label{C+1type}
In this subsection, we restrict ourselves to operators of \cpt\ for which the 
parameters $v$, $w$, and $b$ satisfy the following conditions: for every 
$k\ge 1$,
\[
v^{(k)}=2^{-\tau ^{(k)}}\quad\textrm{and}\quad w_{i}^{(k)}=
\begin{cases}
2&\textrm{if}\ \ 1\le i\le \delta ^{(k)}\\
1&\textrm{if}\ \ \delta ^{(k)}<i<\Delta ^{(k)}
\end{cases}
\]
where $(\tau ^{(k)})_{k\ge 1}$ and $(\delta ^{(k)})_{k\ge 1}$ are two 
strictly increasing sequences of integers with $\delta ^{(k)}<\Delta 
^{(k)}$ for every $k\ge 1$. We call such operators \emph{operators of 
\cput}.
\par\smallskip
If, in order to simplify matters, we assume  that $\delta^{(k)}=2\tau^{(k)}$ for every $k\ge 1$, then the key parameter of 
our counterexamples will be the quantity $\frac{\delta^{(k)}}{\Delta^{(k)}}$, 
\mbox{\it i.e.} the proportion of weights equal to $2$ in each block $[b_n,b_{n+1})$. 
For instance, if we consider the vector $x=2^{-\frac{\delta^{(k)}}{2}}e_{b_{2^{k-1}}}$, we remark that $T^{\Delta^{(k)}}x$ 
is close to $e_0$, and that the orbit of $x$  follows the orbit of $e_0$ during an interval of time 
proportional to $\delta^{(k)}$. Since the period of $x$ is equal to $2\Delta^{(k)}$, 
it seems plausible in view of Theorem~\ref{Theorem 41} that $T$ will be frequently hypercyclic
as soon as
$\limsup_{k\to\infty }\frac{\delta ^{(k)}}{\Delta ^{(k)}}>0$.
On the other hand, we will show that if $\limsup_{k\to\infty }\frac{\delta  ^{(k)}}{\Delta ^{(k)}}$ is 
too small,  $T$ cannot be ergodic.
\par\smallskip

\subsubsection{How to be \emph{\textrm{FHC}} or \emph{\textrm{UFHC}}}
Our first result gives a readable sufficient condition for an operator of \cput\ 
to be chaotic and frequently hypercyclic.

\begin{theorem}\label{Theorem 52}
 Let $T=\tvw$ be an operator of \cput\ on $\ell_{p}(\N)$. 
  \begin{enumerate}
  \item [\emph{(1)}] If\ \ $\limsup\limits_{k\to\infty }\,(\delta ^{(k)}
  -\tau ^{(k)})=\infty $, then $T$ is chaotic.
  \item [\emph{(2)}] If\ \ $\limsup\limits_{k\to\infty }\,
  \dfrac{\delta ^{(k)} -\tau ^{(k)}}{\Delta ^{(k)}}>0 $, then $T$ is chaotic and 
frequently hypercyclic.
 \end{enumerate}
\end{theorem}
\begin{proof}
 The first statement is a direct consequence of Proposition 
\ref{Proposition 45}: for every $n\ge 0$, the expressions of $v^{(k)}$ and 
$(w_j^{(k)})_{1\le j<\Delta ^{(k)}}$ combined with the fact that 
$\varphi ([2^{k},2^{k+1}))=[0,2^{k-1})$ for every $k\ge 1$ yield that 
\[
\limsup_{\genfrac{}{}{0pt}{1}{N\to\infty }{N\,\in\,\varphi ^{-1}(n)}}\ 
|v_{N}|\,\cdot\prod_{j=b_{N}+1}^{b_{N+1}-1}|w_{j}|=\limsup_{k\to\infty }\,
2^{\,\delta ^{(k)}-\tau ^{(k)}}=\infty .
\]
As to the second statement, it is a consequence of Theorem \ref{Theorem 47}. Let us
fix a real number $\alpha$ with
\[0<\alpha < \dfrac{1}{2}\limsup\limits_{k\to\infty }\dfrac{\delta ^{(k)}-\tau ^{(k)}
}{\Delta ^{(k)} }\cdot 
\]
Fix also $C\ge 1$ and $k_{0}\ge 1$. 
By our assumption and the fact that $\Delta 
^{(k)}$ tends to infinity as $k$ tends to infinity, there exists an 
integer $k\ge k_{0}$ such that 
\[
\dfrac{\delta ^{(k)}-\tau ^{(k)}
}{\Delta ^{(k)} }>2\alpha \quad\textrm{and}\quad \alpha \Delta ^{(k)}>\log
_{2}C.
\]
We apply Theorem \ref{Theorem 47} with $m=\Delta ^{(k)}-1$. We have 
\[
|v^{(k)}|\,\cdot\prod_{i=1}^{\Delta^{(k)}-1}|w^{(k)}_i|= 
2^{\,\delta ^{(k)}-\tau ^{(k)}}>2^{\,2\alpha\Delta^{(k)}}\ge  C,
\]
which gives condition \eqref{Equation 14} of Theorem \ref{Theorem 47}. As for condition 
\eqref{Equation 15}, we have for every $0\le m'\le \alpha \Delta ^{(k)}$:
\[
|v^{(k)}|\,\cdot\prod_{i=m'+1}^{\Delta^{(k)}-1}|w^{(k)}_i|\ge 
2^{\,\delta ^{(k)}-m'-\tau ^{(k)}}\ge 2^{\,\delta ^{(k)}-\alpha \Delta 
^{(k)}-\tau ^{(k)}}>C.
\]
It thus follows from Theorem \ref{Theorem 47} that $T$ is frequently 
hypercyclic.
\end{proof}

The next result gives, under some additional assumptions, \emph{necessary and 
sufficient} conditions for an operator of \cput\ to be 
$\mathcal{U}$-frequently hypercyclic or frequently hypercyclic (which turn out to be 
same under the additional assumptions).

\begin{theorem}\label{Theorem 53}
 Let $T=\tvw$ be an operator of \cput\  on $\ell_{p}(\N)$. We additionally assume that the sequence $(\gamma_{k})_{k\ge 1}$ defined by $\gamma_k:=2^{\,\delta ^{(k-1)}-\tau 
^{(k)}}\bigl(\Delta ^{(k)} 
\bigr)^{1-\frac1{p}}$ for every $k\ge 1$ is a non-increasing sequence, and that the following two conditions hold true:
\[
\limsup_{k\to\infty }\,\dfrac{\tau ^{(k)}}{\delta 
^{(k)}}<1\quad\textrm{and}\quad
\sum_{\gk}2^{k}\gamma_k^{1/2}\le 1.
\]
Then the following assertions are equivalent:
\begin{enumerate}
 \item [\emph{(1)}] $T$ is $\mathcal{U}$-frequently hypercyclic;
 \item [\emph{(2)}] $T$ is frequently hypercyclic;
 \item [\emph{(3)}] $\limsup\limits_{k\to\infty }\;
 {\delta ^{(k)}}/
 {\Delta ^{(k)}}>0$.
\end{enumerate}
\end{theorem}

\begin{proof}
 Since $\limsup\limits_{k\to\infty }\frac{\tau ^{(k)}}
 {\delta ^{(k)}}<1$, Condition (3) is equivalent to 
 $\limsup\limits_{k\to\infty }\frac{\delta ^{(k)}-\tau ^{(k)}}
 {\Delta ^{(k)}}>0$. It follows from this observation and 
 Theorem \ref{Theorem 52} that (3) implies (2), which of course implies 
(1). It remains to prove that if $\limsup\limits_{k\to\infty 
}\frac{\delta ^{(k)}}{\Delta ^{(k)}}=0$, $T$ is not 
$\mathcal{U}$-frequently hypercyclic. Given a hypercyclic vector $x$ for 
$T$, we need to show that $x$ cannot be a $\mathcal{U}$-frequently 
hypercyclic vector for $T$, and for this it suffices to find $C$, $(\beta 
_{l} )_{l\ge 1}$, $(X_{l})_{l\ge 1}$, and 
$(N_{l})_{l\ge 1}$ satisfying assumptions (1), (2), (3) and (C) of Theorem 
\ref{Theorem 49}. This is done thanks to Facts~\ref{Proposition 50} 
and~\ref{Proposition 51}. We first fix the sequence $(\beta _{l})_{l\ge 
1}$ by setting 
\[
\beta _{l}=4\,\gamma_k\quad\textrm{for every}\ \,\, l\in[2^{k-1},2^{k}),\ 
\gk.
\]
Since the sequence $(\gamma_k)_{k\ge 1}$ is non-increasing, the sequence $(\beta_{l} )_{l\ge 1}$ 
is non-increasing. Also
\[
\sum_{l\ge 1}\sqrt{\beta _{l}}=\sum_{\gk}2^{k-1}\,(4\,\gamma_k)^{1/2}\le 1
\]
by assumption. We have for every $\gk$ and every $n\in[2^{k-1},2^{k})$
\[
|v_{n}|\,\cdot\sup_{j\,\in\,[b_{\varphi (n)},b_{\varphi (n)+1})}
\ \prod_{s=b_{\varphi (n)+1}}^{j}\!\!\!|w_{s}|\le 2
^{\,\delta ^{(k-1)}-\tau ^{(k)}}.
\]
If we set $C_{n}=2
^{\,\delta ^{(k-1)}-\tau ^{(k)}}$ for every $n\in [2^{k-1},2^{k})$, Fact \ref{Proposition 50} implies that 
for every $k\ge 0$, every $l\in[2^{k-1},2^{k})$ and 
every $0\le n<l$,
\begin{align*}
 \sup_{j\ge 0}\ \|P_{n}T^{\,j}P_{l}\,x\|&\le 2^{\,\delta 
^{(k-1)}-\tau 
^{(k)}}\,\cdot\,\bigl(\Delta ^{(k)} 
\bigr)^{1-\frac1{p}} \,\cdot\,
\Big\|\sum_{i=b_{l}}^{b_{l+1}-1}\Big(\prod_{s=i+1}^{b_{l+1}-1}w_s\Big)\,
x_{i}e_{i}\Big\|\\
&\le \dfrac{\beta_{l}}{4}\cdot  
\Big\|\sum_{i=b_{l}}^{b_{l+1}-1}\Big(\prod_{s=i+1}^{b_{l+1}-1}w_s\Big)\,
x_{i}e_{i}\Big\|
\end{align*}
{and, for every $1\le N\le\Delta ^{(k)}$,}
\[
 \sup_{0\le j\le N}\ \|P_{n}T^{\,j}P_{l}\,x\|\le \dfrac{\beta_{l}}{4}\cdot
 \Bigl(\prod_{i=\Delta ^{(k)}-N+1}^{\Delta ^{(k)}-1}\!\!\!|w_{i}^{(k)}|\  
\Bigr)\,\cdot\,
 \|P_{l}\,x\|.
\]
 {We now set $N_{l}:=\Delta ^{(k)}-\delta ^{(k)}$ for every 
$l\in[2^{k-1},2^{k})$ and every $\gk$. Remembering that $w_{i}^{(k)}=1$ if 
$\delta ^{k}<i<\Delta ^{(k)}$, we obtain that
 }
 \[
\sup_{0\le j\le N_{l}}\ \|P_{n}T^{\,j}P_{l}\,x\|\le\dfrac{\beta _{l}}{4}
\|P_{l}\,x\|\quad \textrm{for every}\ l\ge 1.
\]
 Finally, we set for every $\gk$ and every 
$l\in[2^{k-1},2^{k})$
\[
X_{l}:=\Bigl\| \sum_{i=b_{l}}^{b_{l+1}-1}\Big(\prod_{s=i+1}^{b_{l+1}-1}w_s\Big)\,x_{i}e_{i}\Bigr\|.
\]
Obviously, $\|P_{l}\,x\|\le X_{l}$, and we have proved that for every $0\le n<l$,
\[
 \sup_{j\ge 0}\ 
\|P_{n}T^{\,j}P_{l}\,x\|\le\smash[b]{\dfrac{\beta_{l}}{4}\,X_{l}}
\quad
\textrm{and}
\quad
\sup_{0\le j\le N_{l}}
\|P_{n}T^{\,j}P_{l}\,x\|\le\smash[t]{\dfrac{\beta_{l}}{4}}\,\Vert P_{l}\,
x\Vert .
\]
It remains to check that condition (C) of Theorem \ref{Theorem 49} holds 
true, and for this we use Fact \ref{Proposition 51}. 
\par\smallskip
We have 
$w_{i}^{(k)}=1$ for every $i\in(\delta ^{(k)},\Delta ^{(k)})$, 
and $\prod_{i=\delta ^{(k)}+1}^{\Delta ^{(k)}-1}w_{i}^{(k)}=1.$ The assumptions of Fact \ref{Proposition 51} are thus satisfied for $k_0=\delta^{(k)}$ and $k_1=\Delta^{(k)}$ if $l\in [2^{k-1},2^k)$, 
and we get
\[
\dfrac{1}{J+1}\ \#\Bigl\{0\le j\le J\,;\,\Vert P_{l}T^{\,j}P_{l}\,x\Vert \ge
X_{l}/2\Bigr\}\ge 1-2\delta ^{(k)}\Bigl(\dfrac{1}{J+1}+\dfrac{1}
{\Delta ^{(k)}}\Bigr)
\]
for every $J\ge 0$, every $\gk$, and every $l\in[2^{k-1},2^{k})$. Therefore, 
\[
\inf_{J\ge N_{l}}\ 
\dfrac{1}{J+1}\ \#\Bigl\{0\le j\le J\,;\,\Vert P_{l}T^{\,j}P_{l}\,x\Vert \ge
X_{l}/2\Bigr\}\ge 1-2\delta ^{(k)}\Bigl(\dfrac{1}{\Delta ^{(k)}-
\delta ^{(k)}+1}+\dfrac{1}
{\Delta ^{(k)}}\Bigr)
\]
for every $\gk$ and every $l\in[2^{k-1},2^{k})$, and it follows that 
\begin{align*}
 \liminf_{l\to\infty }\ \inf_{J\ge N_{l}}\ 
\dfrac{1}{J+1}\ &\#\Bigl\{0\le j\le J\,;\,\Vert P_{l}T^{\,j}P_{l}\,x\Vert \ge
X_{l}/2\Bigr\}
\\
&\ge\liminf_{k\to\infty }\ \biggl( 1-2\delta ^{(k)}\Bigl(\dfrac{1}{\Delta 
^{(k)}-
\delta ^{(k)}+1}+\dfrac{1}
{\Delta ^{(k)}}\Bigr) \biggr)
=1,
\end{align*}
since $\limsup\limits_{k\to\infty }\frac{\delta ^{(k)}}{\Delta 
^{(k)}}=0$ by assumption. Theorem \ref{Theorem 49} eventually yields that 
$T$ is not $\mathcal{U}$-fre\-quen\-tly hypercyclic, and this concludes the 
proof.
\end{proof}

\begin{remark} Theorem \ref{Theorem 53} implies the main result of \cite{Me}, \mbox{\it i.e.} that there exist chaotic operators on $\ell_{p}(\N)$ which are not $\mathcal U$-frequently hypercyclic.
\end{remark}
\par\smallskip

\subsubsection{A word about the \emph{\textrm{OSP}}}\label{SECTION+++}
In this short subsection, we show that there exist \cput\ \ops\ on $\ell_{2}(\N)$ which satisfy the assumptions of the criterion for frequent hypercyclicity stated in Theorem \ref{Theorem 41}, and yet do not have the Operator Specification Property (recall that we have proved in Section \ref{SECTION OSP} that \ops\ with the OSP do satisfy the assumptions of Theorem \ref{Theorem 41}). The following corollary is a simple consequence of Proposition \ref{Proposition no mixing} applied to \cput\ \ops.

\begin{corollary}\label{nomixsimplifie} Let $T=\tvw$ be an operator of \cput\  on $\ell_{p}(\N)$ and let $\gamma_k=2^{\,\delta ^{(k-1)}-\tau 
^{(k)}}\bigl(\Delta ^{(k)} 
\bigr)^{1-\frac1{p}}$ for every $k\ge 1$.  Assume that $\tau^{(k)}>\delta^{(k-1)}$ for every $k\geq 1$ and that there exists a positive constant $K$ such that, for every $k_0\ge 1$,
\[\sum_{k>k_0} 2^{k}\gamma_k \leq K 4^{-\Delta^{(k_0)}}.
\]
Then $T$ is not topologically mixing.
\end{corollary}

\begin{proof} Let $n\ge 1$ be an integer, and let $k_n$ be the unique integer such that $2^{k_{n}-1}\leq n<2^{k_{n}}$. Then $v_n=v^{(k_{n})}=2^{-\tau^{(k_{n})}}$ and $\varphi(n)<2^{k_{n}-1}$, and setting
\[C_{n}:=\vert v_n\vert \sup_{j\in 
[b_{\varphi(n)},b_{\varphi(n)+1})}\ \Bigl(\prod_{s=b_{\varphi(n)}+1}^{j}|w_{s}|\Bigr),
\]
we have
$0<C_{n}
\leq 2^{-\tau^{(k_n)}}\, 2^{\delta^{(k_n-1)}}<1.$
Moreover, if $n$ is an integer of the form $n=2^k-1$ for some integer $k\ge 1$, then $b_{n+1}-b_n= \Delta^{(k)}$, and moreover $k_l>k$ for every $l>n$. It follows that
\begin{align*}
\sum_{l>n}
C_{l}\,(b_{l+1}-b_{l})^{1-\frac{1}{p}}\,\Bigl(\ \sup_{b_{l+1}-2(b_{n+1}-b_n)\le 
i<b_{l+1}}\prod_{s=i+1}^{b_{l+1}-1}|w_s|\Bigr)
\end{align*}
is less than
\begin{align*}
\leq \sum_{r>k} 2^{r-1} 2^{\delta^{(r-1)}-\tau^{(r)}}(\Delta^{(r)})^{1-\frac{1}{p}}\,.\, 2^{2\Delta^{(k)}}
\leq K.
\end{align*}
Hence  Proposition \ref{Proposition no mixing} applies, and $T$ is not topologically mixing.
\end{proof}

Since the condition appearing in Corollary~\ref{nomixsimplifie} is compatible with those of Theorem~\ref{Theorem 53} above (see Example~\ref{Example 55}), and since \ops\ with the OSP are topologically mixing, we immediately deduce:

\begin{corollary} There exist operators on $\ell_{p}(\N)$ which satisfy the assumptions of Theorem \ref{Theorem 53} and yet do not have the Operator Specification Property.
\end{corollary}

Since the assumptions of Theorem \ref{Theorem 53} imply those of 
Theorem \ref{Theorem 41}, it follows that there exist \ops\ on $\ell_{p}(\N)$ which satisfy the criterion for frequent hypercyclicity stated in 
Theorem \ref{Theorem 41} without having the OSP. In particular, this criterion for frequent hypercyclicity is strictly stronger than the ``classical" one.

\subsubsection{\emph{\textrm{FHC}} but not ergodic}
We are now ready to give examples of operators of \cput\ which are 
frequently hypercyclic but not ergodic.

\begin{theorem}\label{Theorem 54}
 Let $T$ be an operator of \cput\ on $\ell_{p}(\N)$, and  set, for every $k\geq 1$, 
 $\gamma_k:=2^{\,\delta ^{(k-1)}-\tau 
^{(k)}}\bigl(\Delta ^{(k)} 
\bigr)^{1-\frac1{p}}$. 
 Suppose that the sequence $(\gamma_k)_{k\ge 1}$ is non-increasing, and that the following three conditions  hold true:
\[
\sum_{\gk}2^{k}\gamma_k^{1/2}\le 1,\quad 
\limsup_{k\to\infty }\dfrac{\tau ^{(k)}}{\delta 
^{(k)}}<1\quad\textrm{and }
\quad 0<\limsup_{k\to\infty }\dfrac{\delta ^{(k)}}{\Delta 
^{(k)}}\le\dfrac{1}{5}\cdot
\]
\par\medskip 
\noindent
Then $T$ is frequently hypercyclic but $c(T)<1$, so that $T$ is not ergodic.
\end{theorem}

\begin{proof}
 The frequent hypercyclicity of $T$ follows from Theorem \ref{Theorem 53}. 
Let us show that $c(T)<1$. Recall that $c(T)$ is characterized by the following property (see \cite[Rem.~4.6]{GM}): for quasi-all hypercyclic 
vectors $x$ for $T$ in the 
Baire category sense, we have 
\begin{equation}\label{cteq}  c(T)=\overline{\textrm{dens}}\,\bigl(x,B(0,\varepsilon ) \bigr)\quad\textrm{ for every } \varepsilon>0.
\end{equation}
 Let us fix a hypercyclic vector $x$ for $T$ such that (\ref{cteq}) holds true. By 
Lemma \ref{Theorem 48} combined with the end of the proof of Theorem 
 \ref{Theorem 53}, we know that there exists an $\varepsilon >0$ such that
\begin{align*}
  \underline{\textrm{dens}\vphantom{p}}\
\mathcal{N}_{T}\bigl(x,B(0,\varepsilon)^c\bigr))
& \ge\liminf_{l\to\infty }\ \inf_{J\ge N_{l}}\ 
\dfrac{1}{J+1}\ \#\bigl\{0\le j\le J\,;\,\Vert P_{l}T^{\,j}P_{l}\,x\Vert \ge
X_{l}/2\bigr\}
\end{align*}
(recall that we proved above that under the assumption (3) of Theorem
\ref{Theorem 53}, the assumptions of Theorem \ref{Theorem 49} are 
satisfied with $C=1/4$, so that $2CX_{l}=X_{l}/2$). Using an inequality 
obtained at the end of the proof of Theorem \ref{Theorem 53}, we get
\begin{align*}
 \underline{\textrm{dens}}\ \mathcal{N}_{T}\bigl(x,B(0,\varepsilon )^c\bigr)& 
 \smash[b]{\ge\liminf_{k\to\infty }\ \biggl( 1-2\delta 
^{(k)}\Bigl(\dfrac{1}{\Delta 
^{(k)}-
\delta ^{(k)}+1}+\dfrac{1}
{\Delta ^{(k)}}\Bigr) \biggr)}
\intertext{so that}
 \overline{\textrm{dens}}\ \mathcal{N}_{T}\bigl(x,
 B(0,\varepsilon )\bigr)&\le\limsup_{k\to\infty }\ \biggl(
 \dfrac{2\delta ^{(k)}}{\Delta 
^{(k)}-
\delta ^{(k)}+1}+\dfrac{2\delta ^{(k)}}
{\Delta ^{(k)}}
 \biggr)\\
 &=\limsup_{k\to\infty }\ 
 \biggl(
 \dfrac{2\delta ^{(k)}}{\Delta^{(k)}}\cdot
 \dfrac{1}{1-\dfrac{\delta^{(k)}}{\Delta 
^{(k)}}+\dfrac{1}{\Delta ^{(k)}}}+\dfrac{2\delta ^{(k)}}
{\Delta ^{(k)}}
 \biggr).
 \end{align*}
 It follows that
 \begin{align}\label{Inegalite 103}
 c(T)\le \limsup_{k\to\infty }\ 
 \biggl(
 \dfrac{2\delta ^{(k)}}{\Delta^{(k)}}\cdot
 \dfrac{1}{1-\dfrac{\delta^{(k)}}{\Delta 
^{(k)}}+\dfrac{1}{\Delta ^{(k)}}}+\dfrac{2\delta ^{(k)}}
{\Delta ^{(k)}}
 \biggr).
\end{align}
Our assumption thus implies that
\[ c(T)
\le\dfrac{2/5}{1-1/5}+\dfrac{2}{5}=\dfrac{1}{2}+\dfrac{2}{5}<1,
\]
which proves that $T$ is not ergodic.
\end{proof}

\begin{example}\label{Example 55} Let $C$ be a positive integer, and 
  consider the operator of \cput\ $T$ on $\ell_{p}(\N)$ associated 
to the parameters
\[
\tau ^{(k)}=2^{Ck},\quad\delta ^{(k)}=2\,\cdot\,2^{Ck}\quad\textrm{and}\quad
\Delta ^{(k)}=10\, \cdot\,2^{Ck},\quad k\ge 1.
\]
If $C$ is sufficiently large, $T$ is frequently 
hypercyclic but not ergodic. Besides, it is also not topologically mixing.
\end{example}

\begin{proof} With this choice of parameters, and assuming that $C\geq 2$, we have 
\[
\gamma_k=2^{\,(2\,\cdot\,2^{C(k-1)}-2^{Ck})}\,\cdot\,10^{1-\frac{1}{p}}\cdot
2^{Ck(1-\frac{1}{p})}\leq 10\cdot 2^{-\frac12\, 2^{Ck}}\cdot 2^{Ck} 
\]
for every $k\geq 1$. So it is not hard to check that if $C$ is sufficiently large, the sequence $(\gamma_k)_{k\ge 1}$ is non-increasing and satisfies
$\sum_{\gk}2^{k}\gamma_k^{1/2}\le 1$. The other assumptions of Theorem \ref{Theorem 54} 
are clearly satisfied, and hence $T$ is frequently hypercyclic but not ergodic.
\par\smallskip
In order to show that $T$ is also not topologically mixing if $C$ is large enough, we use Proposition~\ref{Proposition no mixing}. If $n$ is any integer of the form $n=2^{k_0}-1$, where $k_{0}\ge 1$ is any integer, then 
\begin{align*}
\sum_{l>n}
C_{l}\,(b_{l+1}-b_{l})^{1-\frac1{p}}\Bigl(\ \sup_{b_{l+1}-2(b_{n+1}-b_n)\le 
k<b_{l+1}}\prod_{s=k+1}^{b_{l+1}-1}|w_s|\Bigr)&=\sum_{k>k_0}
2^{k-1}2^{\delta^{(k-1)}-\tau^{(k)}} (\Delta^{(k)})^{1-\frac1{p}}\\
&=\sum_{k>k_0} 2^{k-1} \gamma_k,
\end{align*}
If $C$ is sufficiently large, then $\sum_{k=1}^\infty
2^{k-1}\gamma_k< 1$; so we deduce from 
Proposition~\ref{Proposition no mixing} (by considering $K:=1$ and $n:=2^{k_0}-1$ for every $k_0\ge 1$) that $T$ is not topologically mixing. Notice that $T$ also satisfied the assumptions of Corollary~\ref{nomixsimplifie}.
\end{proof}

\par
Since the conditions of Theorem \ref{Theorem 53} make no difference 
between $\mathcal{U}$-frequent and frequent hypercyclicity, we will need to 
introduce another family of operators of \cpt\ in order to construct 
examples of $\mathcal{U}$-frequent hypercyclic operators on $\ell_{p}(\N)$ 
which are not frequently hypercyclic.
But before moving over to the presentation of this new class of operators, we give in 
the next subsection a corollary of Theorem \ref{Theorem 54} concerning infinite 
direct sums of frequently hypercyclic operators.

\subsection{Infinite direct sums of frequently hypercyclic operators}\label{Subsection 6.5.4}
A well-known open question, dating back to \cite{BG3}, asks whether the direct sum $T\oplus T$ of 
a frequently hypercyclic operator
$T$ with itself has to be frequently hypercyclic. (This is the analogue of Herreros ``$T\oplus T$ problem" for frequent hypercyclicity.)
By  \cite{GrEPe}, $T\oplus T$ is 
hypercyclic as soon as $T$ is $\mathcal{U}$-frequently hypercyclic. This result of \cite{GrEPe} also implies that the direct sum of two (and hence, of infinitely many) $\mathcal{U}$-frequently hypercyclic operators is hypercyclic. Indeed, let $T_{1}$ and $T_{2}$ be two bounded operators acting respectively on the Banach spaces $X_{1}$ and $X_{2}$, and let $U_{1}$, $V_{1}$ and $U_{2}$, $V_{2}$ be non-empty open subsets of $X_{1}$ and $X_{2}$ respectively. By \cite{GrEPe}, the set $\mathcal N_{T_{1}}(U_{1},V_{1})$ has bounded gaps. Also, since $T_{2}$ is topologically weakly mixing, the set 
$\mathcal N_{T_{2}}(U_{2},V_{2})$ contains arbitrarily long intervals. It follows that $T_{1}\oplus T_{2}$ is hypercyclic.
\par\smallskip
Apart from this result of \cite{GrEPe}, nothing seems to be known concerning this question of \cite{BG3}. 
More 
generally, it seems to be unknown whether the direct sum $T_{1}\oplus T_{2}$ of two 
frequently hypercyclic operators is necessarily frequently hypercyclic, or even 
$\mathcal{U}$-frequently hypercyclic. To the best of our knowledge, this question is open 
even for infinite direct sums of frequently hypercyclic operators. Our aim in this 
subsection is to use operators of $C_{+1}$-type on $\ell_{p}(\N)$ to prove the 
following result.

\begin{theorem}\label{Theorem 100}
 Let $p>1$. There exists a sequence $(T_{n})_{n\ge 1}$ of frequently hypercyclic 
operators on $\ell_{p}(\N)$ such that the $\ell_{p}$-sum operator $T=\bigoplus_{n\ge 
1}T_{n}$ acting on $X=\bigoplus_{n\ge 1}\ell_{p}(\N)$ is not $\mathcal U$-frequently hypercyclic.
\end{theorem}

The proof of Theorem \ref{Theorem 100} relies on Theorem \ref{Theorem 54}, combined 
with an elementary lemma providing an upper bound for the parameter $c(T)$ of a direct 
sum operator:

\begin{lemma}\label{Lemma 101}
 Let $T_{1}$ and $T_{2}$ be two bounded operators acting respectively on the Banach 
spaces $X_{1}$ and $X_{2}$. Fix $p>1$, and let $T=T_{1}\oplus_{\ell_{p}} T_{2}$ be the 
$\ell_{p}$-sum operator of $T_{1}$ and $T_{2}$, acting on $X=X_{1}\oplus_{\ell_{p}}X_{2}$. Moreover, assume that $T$ is hypercyclic.  
Then $c(T)\le \min(c_{1}(T),c_{2}(T))$.
\end{lemma}

\begin{proof}[Proof of Lemma \ref{Lemma 101}]
Let $c\in[0,1]$ be such that $\overline{\textrm{dens}}\,\mathcal{N}_{T}(x,B_{X}(0,1))\ge c$ 
for a comeager set of vectors $x$ of $X$. For any such vector $x=x_{1}\oplus x_{2}$, 
with $x_{1}\in X_{1}$ and $x_{2}\in X_{2}$, there exists a subset $D_{x}$ of $\N$ with
$\overline{\textrm{dens}}\,D_{x}\ge c$ such that 
$\bigl(  \Vert T^{n}_{1}x_{1}\Vert^{p}+\Vert T_{2}^{n}x_{2}\Vert^{p}\bigr)^{1/p}<1 $ for every $n\in 
D_{x}$. Thus $\overline{\textrm{dens}}\,\mathcal{N}_{T_{1}}(x_{1},B_{X_{1}}(0,1))\ge c$ for 
a comeager subset of vectors $x_{1}\in X_{1}$, so that $c(T_{1})\ge c$; and likewise, 
$c(T_{2})\ge c$. Hence $\min(c_{1}(T),c_{2}(T))\ge c$ for any $c$ as above, which proves the lemma.
\end{proof}

As an easy consequence of Lemma \ref{Lemma 101}, we obtain

\begin{lemma}\label{Lemma 102}
 For each $n\ge 1$, let $T_{n}$ be a bounded operator on a Banach space $X_{n}$. Fix 
$p>1$,  and let
 $T=\bigoplus_{\ell_{p}}T_{n}$ be the $\ell_{p}$-sum of the operators $T_{n}$, $n\ge 1$, 
acting on the space $X=\bigoplus_{\ell_{p}}X_{n}$. Moreover, assume that $T$ is hypercyclic. Then 
\[
c(T)=\inf_{n\ge 1}c(T_{1}\oplus\cdots\oplus T_{n})\le\inf_{n\ge 1}c(T_{n}).
\]
\end{lemma}

\begin{proof}[Proof of Lemma \ref{Lemma 102}]
It follows directly from Lemma \ref{Lemma 101} that  
\[
c(T)\le\inf_{n\ge 1}c(T_{1}\oplus\cdots\oplus T_{n})\le\inf_{n\ge 1}c(T_{n}),
\] so we 
only have to prove that 
\[
c(T)\ge\inf_{n\ge 1}c(T_{1}\oplus\cdots\oplus T_{n}).
\]
 If the 
infimum $\inf_{n\ge 1}c(T_{1}\oplus\cdots\oplus T_{n})$ is equal to $0$, there is nothing 
to prove. Suppose that $\inf_{n\ge 1}c(T_{1}\oplus\cdots\oplus T_{n})>0$ and consider a 
number $c$ such that $0\le c<\inf_{n\ge 1}c(T_{1}\oplus\cdots\oplus T_{n})$. Let $U$ be a 
non-empty open subset of $X$. There exist an integer $n\ge 1$ and, for every $1\le i\le 
n$, a non-empty open subset $U_{i}$ of $X_{i}$, such that $U$ contains the set 
$U_{1}\oplus\cdots\oplus U_{n}\oplus 0\oplus\cdots$. Since $c<\inf_{n\ge 
1}c(T_{1}\oplus\cdots\oplus T_{n})$, there exist a vector $x=x_{1}\oplus\cdots\oplus 
x_{n}\in U_{1}\oplus\cdots\oplus U_{n}$ and a subset $D$ of $\N$ with 
$\overline{\textrm{dens}}\,D\ge c$ such that $\bigl (\sum_{i=1}^{n}\Vert T_{i}^{k}x_{i}\Vert^{p} 
\bigr)^{1/p}<1$ for every $k\in D$. Setting $x=x_{1}\oplus\cdots\oplus 
x_{n}\oplus 0\oplus\cdots$, we deduce that $x$ is a vector of $U$ which satisfies
$\overline{\textrm{dens}}\,\mathcal{N}_{T}(x,B_{X}(0,1))\ge c$. The set of vectors 
$x\in X$ such that $\overline{\textrm{dens}}\,\mathcal{N}_{T}(x,B_{X}(0,1))\ge c$ is thus 
dense in $X$, and the usual argument shows that this set is in fact comeager in $X$. It follows that $c(T)\ge c$, from which we deduce that 
$c(T)\ge \inf_{n\ge 1}c(T_{1}\oplus\cdots\oplus T_{n})$. This concludes the proof of 
Lemma \ref{Lemma 101}.
\end{proof}

\begin{proof}[Proof of Theorem \ref{Theorem 100}]
Let $(T_{n})_{n\ge 1}$ be a sequence of operators satisfying the assumptions of Theorem 
\ref{Theorem 54}. Denoting for each $n\ge 1$ by $(\tau_{n}^{(k)})_{k\ge 1}$, 
$(\delta _{n}^{(k)})_{k\ge 1}$, and $(\Delta _{n}^{(k)})_{k\ge 1}$ the sequences of 
parameters associated to the operator $T_{n}$, we suppose that 
\[
0<\limsup_{k\to\infty}\dfrac{\delta _{n}^{(k)}}{\Delta _{n}^{(k)}}\le 2^{-n}
\]
for every $n\ge 1$. By Theorem \ref{Theorem 54}, all the operators $T_{n}$ are 
frequently hypercyclic. On the other hand, it follows from 
(\ref{Inegalite 103}) that
\smallskip
\begin{align*}
 c(T_{n})&\le\limsup_{k\to\infty}\biggl (2\dfrac{\delta _{n}^{(k)}}{\Delta _{n}^{(k)}} 
\,\cdot\,\dfrac{1}{1-\dfrac{\delta _{n}^{(k)}}{\Delta _{n}^{(k)}}+\dfrac{1}{\Delta 
_{n}^{(k)}}}+2\dfrac{\delta _{n}^{(k)}}{\Delta _{n}^{(k)}}\biggr)\\
&\le 2\cdot 2^{-n}\,\cdot\,\dfrac{1}{1-2^{-n}}+2.2^{-n}
\le 6\cdot 2^{-n}\quad\textrm{ for every } n\ge 1,
\end{align*}
\smallskip
 so that $c(T_{n})$ tends to $0$ as $n$ tends to infinity. By Lemma 
\ref{Lemma 102}, the operator $T=\bigoplus_{\ell_{p}}T_{n}$ (which is hypercyclic) satisfies 
$c(T)=0$. Hence $T$ is not $\mathcal{U}$-frequently hypercyclic.
\end{proof}

\subsection{Operators of \cpdt: UFHC does not imply FHC}\label{Subsection 4.5}
In this subsection, we impose the following restrictions on the parameters 
$v$, $w$, and $b$ of a \cpt\ operator $\tvw$: 
for every $\gk$,
\[
v^{(k)}=2^{-\tau^{(k)}} \quad \text{and}\quad 
w^{(k)}_i=
\begin{cases}
  2 & \quad\text{if}\ \ 1\le i\le \delta ^{(k)}\\
 1 & \quad\text{if}\ \ \delta^{(k)}<i<\Delta^{(k)}-3\delta^{(k)}\\
 1/2 & \quad\text{if}\ \  \Delta^{(k)}-3\delta^{(k)}\le i 
<\Delta^{(k)}-2\delta^{(k)}\\
 2 & \quad\text{if}\ \ \Delta^{(k)}-2\delta^{(k)}\le i<\Delta^{(k)}-\delta^{(k)}\\
 1& \quad\text{if}\ \ \Delta^{(k)}-\delta^{(k)}\le i<\Delta^{(k)}
\end{cases}
\]
where $(\tau ^{(k)})_{\gk}$ and $(\delta ^{(k)})_{\gk}$ are two strictly 
increasing sequences of integers, satisfying $4\delta ^{k}<\Delta 
^{(k)}$ for every $\gk$. We call operators satisfying these conditions 
\emph{operators of \cpdt}. Observe that we still have with this definition
\[
\prod_{i=1}^{\Delta ^{(k)}-1}w_{i}^{(k)}=2^{\,\delta ^{(k)}}\quad\hbox{for 
every $\gk$.}
\]
\par\smallskip
This choice of weights is motivated by the differences between the assumptions of Theorem \ref{Theorem 39} 
and those of Theorem \ref{Theorem 41}.
If we assume for simplicity that $\delta^{(k)}=2\tau^{(k)}$ for every $k\ge 1$ and if we consider the vector 
$x:=2^{-\frac{\delta^{(k)}}{2}}e_{b_{2^{k-1}}+\Delta^{(k)}-2\delta^{(k)}}$, we observe that 
$T^{2\delta^{(k)}}x$ is close to $e_0$ and that the orbit of $x$  follows the orbit of $e_0$ during 
an interval of time proportional to $\delta^{(k)}$. In view of Theorem~\ref{Theorem 39},  it thus seems likely
that $T$ will be $\mathcal{U}$-frequently hypercyclic. However, if 
$\lim_{k\to\infty }\frac{\delta  ^{(k)}}{\Delta ^{(k)}}=0$, it also seems plausible that  
$T$ will not satisfy the assumptions of Theorem~\ref{Theorem 41}, and hence will possibly 
not be frequently hypercyclic.
\par\smallskip

\subsubsection{How to be \emph{\textrm{FHC}} or \emph{\textrm{UFHC}}}
In this subsection, we present some sufficient conditions for an operator of \cpdt\ to be 
$\mathcal{U}$-frequently hypercyclic.
\begin{theorem}\label{Theorem 56}
 Let $T=\tvw$ be an operator of \cpdt\ on $\ell_{p}(\N)$. 
  \begin{enumerate}
  \item [\emph{(1)}] If\ \ $\ds\limsup_{k\to\infty}\
  \bigl( \delta ^{(k)}-\tau ^{(k)}\bigr)=\infty  $, then $T$ is chaotic.
  \item[\emph{(2)}] If\ \  $\ds\limsup_{k\to\infty }\ \dfrac{\tau ^{(k)}}{
  \delta ^{(k)}}<1$, then $T$ is chaotic and $\mathcal{U}$-frequently 
hypercyclic.
\item[\emph{(3)}] If\ \  $\ds\limsup_{k\to\infty }\ \dfrac{\delta ^{(k)}-
\tau ^{(k)}}{\Delta ^{(k)}}>0$, then $T$ is chaotic and frequently 
hypercyclic.
 \end{enumerate}
\end{theorem}

\begin{proof}
 The proofs of assertions (1) and (3) are completely similar to the ones 
given in the proof of Theorem \ref{Theorem 52}. As for (2), it is a 
consequence of Theorem \ref{Theorem 46}. Let us fix
\[0<\alpha <\frac{1}{3}-\dfrac{1}{3}\,\limsup_{k\to\infty }\dfrac
{\tau ^{(k)}}{\delta ^{(k)}}\cdot
\]
Fix also $C\ge 1$, and an integer $k_{0}\ge 1$. There exists an integer $k\ge 
k_{0}$ such that 
\[
\dfrac{\tau^{(k)}}{\delta ^{(k)}}<1-3\alpha \quad\textrm{and}\quad 
\alpha\, 
\delta ^{(k)}>\log_{2}C.
\]
We then have
\[
 |v^{(k)}|\;\;\cdot\prod_{i=\Delta ^{(k)}-2\delta ^{(k)}}^{\Delta ^{(k)}-1}
 |w_{i}^{(k)}|=2^{\,\delta ^{(k)}-\tau ^{(k)}}>2^{\,3\alpha \delta 
^{(k)}}>C
\]
{and, for every $0\le m'\le 2\alpha \delta ^{(k)}$:}
\[ |v^{(k)}|\ \cdot\!\prod_{i=m'+1}^{\Delta ^{(k)}-1}
 |w_{i}^{(k)}|\ge 2^{\,(1-2\alpha )\delta ^{(k)}-\tau 
^{(k)}}>2^{\,\alpha\delta^{(k)}  }>C
\]
Applying Theorem \ref{Theorem 46} with $m=2\,\delta ^{(k)}$, it follows that $T$ is 
$\mathcal{U}$-frequently hypercyclic.
\end{proof}
\par We see from this proof that the new structure of the weights 
$(w_{j}^{(k)})_{1\le j<\Delta ^{(k)}}$, $\gk$, compared with the case of 
operators of \cput, allows us to provide different conditions for 
$\mathcal{U}$-frequent hypercyclicity and frequent hypercyclicity of 
operators of \cpdt. Our next result gives, under some additional 
assumptions, a \emph{necessary and sufficient} condition for frequent 
hypercyclicity of operators of \cpdt\ which, combined with (2) of Theorem~\ref{Theorem 56}, will ultimately allow us to construct 
$\mathcal{U}$-frequently hypercyclic operators of \cpdt\ which are not 
frequently hypercyclic.

\begin{theorem}\label{Theorem 57}
 Let $T=\tvw$ be an operator of \cpdt\ on $\ell_{p}(\N)$, and for every $k\geq 1$, set $\gamma_k:=2^{\,\delta ^{(k-1)}-\tau 
^{(k)}}\bigl(\Delta ^{(k)} 
\bigr)^{1-\frac1{p}}$.
 Suppose that the sequence $(\gamma_k)_{k\ge 1}$ is non-increasing and that it satisfies the following three conditions:
 \[
\sum_{\gk}2^{k}\gamma_k^{1/2}\le 1,\quad 
\quad \limsup_{k\to\infty }\dfrac{\tau ^{(k)}}{\delta 
^{(k)}}<1,\quad\textrm{and}
\quad \lim_{k\to\infty }\dfrac{\delta ^{(k)}}{\delta 
^{(k+1)}}=0.
\]
\par\smallskip \noindent
Then $T$ is frequently hypercyclic {if and only if}\ \  
$\limsup\limits_{k\to\infty }\ {\delta ^{(k)}}/{\Delta 
^{(k)}}>0.$
\end{theorem}

\begin{proof}
Since $\limsup_{k\to\infty }\frac{\tau ^{(k)}}{\delta 
^{(k)}}<1$, 
the condition $\limsup_{k\to\infty }\frac{\delta^{(k)}}{\Delta 
^{(k)}}>0$ is equivalent to
\[
\ds\limsup_{k\to\infty }\ \dfrac{\delta^{(k)}-\tau ^{(k)}}{\Delta 
^{(k)}}>0;
\]
and by assertion (3) of Theorem \ref{Theorem 
56}, this condition implies that $T$ is frequently hypercyclic. 
\par\smallskip
The proof of the converse assertion proceeds exactly as in the proof of the implication 
$\textrm{(1)}\Longrightarrow\textrm{(3)}$ in Theorem \ref{Theorem 53}. 
Suppose that $\limsup_{k\to\infty }\ \frac{\delta^{(k)}}{\Delta 
^{(k)}}=0$, and let $x\in\ell_{p}(\N)$ be a hypercyclic vector for $T$. 
Our aim is to prove that $x$ cannot be a frequently hypercyclic vector for 
$T$, using Theorem \ref{Theorem 49}. We set 
\[
\beta _{l}:=4\,\gamma_k\quad\textrm{for every}\ l\in[2^{k-1},2^{k}),\ \gk.
\]
This sequence 
$(\beta _{l})_{l\ge 1}$ is non-increasing and satisfies
$\sum_{l\ge 1}\sqrt{\beta _{l}}\le 1$. Setting
\[
X_{l}:=\Bigl\| \sum_{i=b_{l}}^{b_{l+1}-1}\Bigl(\prod_{s=i+1}
^{b_{l+1}-1}w_{s}\Bigr)\,x_{i}e_{i}\Bigr\|\quad{\rm for}\quad  l\in[2^{k-1},2^{k}),\ \gk,
\]
we have $\Vert P_{l}\,x\Vert \le X_{l}$.
\par\smallskip
Also, by Fact \ref{Proposition 50}, we have for every\ $l\ge 0$, every $0\le n<l$, and every $1\le N\le \Delta ^{(k)}$,  
where $l\in[2^{k-1},2^{k})$, that
\begin{align*}
 \sup_{j\ge 0}\ \Vert P_{n}T^{\,j}P_{l}\,x\Vert &\le\dfrac{1}{4}\,\beta 
_{l}X_{l}\quad
 \end{align*}
 and
 \begin{align*}
 \sup_{0\le j\le N}\ \Vert P_{n}T^{\,j}P_{l}\,x\Vert &\le \dfrac{1}{4}\,\beta 
_{l}\,\cdot
 \Bigl(\prod_{i=\Delta ^{(k)}-N+1}^{\Delta ^{(k)}-1}\!\!\!|w_{i}^{(k)}|\  
\Bigr)\,\cdot\,
 \Vert P_{l}\,x\Vert .
\end{align*}
At this point, we diverge from the proof of 
Theorem \ref{Theorem 53}, and set $N_{l}:=\delta ^{(k)}$ for every 
$l\in[2^{k-1},2^{k})$ and every $\gk$. We then have
$\ds\sup_{0\le j\le N_{l}}\ \Vert P_{n}T^{\,j}P_{l}\,x\Vert \le 
\dfrac{1}{4}\,\beta 
_{l}\,
 \Vert P_{l}\,x\Vert $ for every $l\ge 0$ and every $0\le n<l$. In order to check 
that assumption (C') of Theorem \ref{Theorem 49} holds true, we use again 
Fact \ref{Proposition 51}. We have $w_{i}^{(k)}=1$ for every
$i\in(\delta ^{(k)},\Delta ^{(k)}-3\delta ^{(k)})$ and 
\[
\prod_{i=\delta ^{(k)}+1}^{\Delta ^{(k)}-1}w_{i}^{(k)}=1.
\]
It follows from Fact \ref{Proposition 51} that 
\[
\dfrac{1}{J+1}\ \#\,\bigl\{0\le j\le J\,;\,\Vert P_{l}T^{\,j}T_{l}\,x\Vert \ge 
X_{l}/2\bigr\}\ge 1-8\delta ^{(k)}\Bigl(\dfrac{1}{J+1}+\dfrac{1}
{\Delta ^{(k)}}\Bigr)
\]
for every $J\ge 0$, every $\gk$, and every $l\in[2^{k-1},2^{k})$. Now, we 
have $\min(\varphi ^{-1}(l))=2^{k}$ for every $l\in[2^{k-1},2^{k})$, so 
that $N_{\min(\varphi ^{-1}(l))}=\delta ^{(k+1)}$. 
\par\smallskip
For every $\gk$, we have
\[
 \inf_{J\ge \delta ^{(k+1)}}\ 
\dfrac{1}{J+1}\ \#\,\Bigl\{0\le j\le J\,;\,\Vert P_{l}T^{\,j}T_{l}\,x\Vert \ge 
X_{l}/2\Bigr\}\ge 1-\dfrac{8\delta ^{(k)}}{\delta ^{(k+1)}}-\dfrac{8
\delta ^{(k)}}{\Delta ^{(k)}}\cdot
\]
{Hence}
\begin{align*}
\liminf_{l\to\infty }\, 
\inf_{J\ge \delta ^{(k+1)}}\ 
\dfrac{1}{J+1}\ \#\,\Bigl\{0\le j\le J\,;\,&\Vert P_{l}T^{\,j}T_{l}\,x\Vert \ge
X_{l}/2\Bigr\}\\
&\ge \liminf_{k\to\infty }\,\Bigl( 1-\dfrac{8\delta 
^{(k)}}{\delta ^{(k+1)}}-\dfrac{8
\delta ^{(k)}}{\Delta ^{(k)}}\Bigr)=1,
\end{align*}
since $\lim_{k\to\infty }\,\frac{\delta ^{(k)}}{\delta 
^{(k+1)}}=0$ and 
$\limsup_{k\to\infty }\,\frac{\delta ^{(k)}}{\Delta 
^{(k)}}=0$. Theorem \ref{Theorem 49} thus implies that $T$ is not frequently 
hypercyclic
\end{proof}
\par\smallskip

\subsubsection{\emph{\textrm{UFHC}} but not \emph{\textrm{FHC}}}
As a direct consequence of Theorems \ref{Theorem 56} and \ref{Theorem 57}, 
we now obtain

\begin{theorem}\label{Theorem 58}
 Let $T=\tvw$ be an operator of \cpdt\ on $\ell_{p}(\N)$. For every $k\geq 1$, set $\gamma_k:=2^{\,\delta ^{(k-1)}-\tau 
^{(k)}}\bigl(\Delta ^{(k)} 
\bigr)^{1-\frac1{p}}.$ 
 Suppose that the sequence $(\gamma_k)_{k\ge 1}$ is non-increasing, and that the following conditions are satisfied:
 \[
\sum_{\gk}2^{k}\gamma_k^{1/2}\le 1,\quad 
\limsup_{k\to\infty }\dfrac{\tau ^{(k)}}{\delta 
^{(k)}}<1,
\quad  \lim_{k\to\infty }\dfrac{\delta ^{(k)}}{\delta 
^{(k+1)}}=0\quad\textrm{and}\ \lim_{k\to\infty }\dfrac{\delta 
 ^{(k)}}{\Delta ^{(k)}}=0.
\]
Then $T$ is $\mathcal{U}$-frequently hypercyclic but not frequently 
hypercyclic.
\end{theorem}

Here is a concrete example, the proof of which is left to the reader.

\begin{example}\label{Example 59}
 If we consider the operator of \cpdt\ associated to the parameters 
\[
\tau ^{(k)}=2^{Ck^2},\quad \delta ^{(k)}=2^{Ck^2+1}\quad {\rm and}\quad \Delta 
^{(k)}=2^{2Ck^2+4}, \quad k\ge 1,
\]
where $C$ is a sufficiently large integer, then $T$ is 
$\mathcal{U}$-frequently hypercyclic but not frequently hypercyclic.
\end{example}
\par\smallskip

\subsection{Operators of \ccut: chaos plus mixing do not imply UFHC}\label{ex mixing} In this subsection, we introduce yet another class of \cct\ operators $T=\tvw$ on $\ell_{p}(\N)$.
\par\smallskip
We consider increasing sequences $(a_k)_{k\ge 1}$, $(f_k)_{k\ge 1}$, $(\delta^{(k)})_{k\ge 1}$, $(\tau^{(k)})_{k\ge 1}$ and $(\Delta^{(k)})_{k\ge 1}$ of integers
such that $0\le a_k\le f_k<\Delta^{(k)}-4\delta^{(k)}$ for every $k\ge 1$, and we denote by $(J_k)_{k\ge 0}$ the partition of $\N$ into consecutive 
finite intervals defined as follows: $J_0=\{0\}$ and $\#J_k=(f_k-a_k)(\sum_{i=0}^{k-1}\#J_{i})$ for every $k\geq 1$.
\par\smallskip
We then require that:
\begin{enumerate}
 \item[-] $\varphi (n)=\left\lfloor\frac{n-\min J_k}{f_k-a_k}\right\rfloor$ for every $n\in J_k$;
 \item[-] the blocks $[b_{n},b_{n+1})$, $n\in J_k$,
all have the same size $\Delta^{(k)}$;
\item[-] the sequence $v$ is given by
\[ v_{n}={2^{-\tau^{(k)}}}\qquad\hbox{for every 
$n\in J_k$};
\]
\item[-] If $j$ belongs to the interval $ [b_n,b_{n+1})$ with $n\in J_k$ and $(n-\min J_k) \mod (f_k-a_k)=l$, then the weight $w_j$ is given by
\[
w_j=
\begin{cases}
  2 & \quad\text{if}\ \ b_n< j\le b_n+\delta^{(k)}\\
 1 & \quad\text{if}\ \ b_n+\delta^{(k)}<j<b_{n+1}-a_k-l-2\delta^{(k)}\\
 1/2 & \quad\text{if}\ \  b_{n+1}-a_k-l-2\delta^{(k)}\le j <b_{n+1}-a_k-l-\delta^{(k)}\\
 2 & \quad\text{if}\ \ b_{n+1}-a_k-l-\delta^{(k)}\le j<b_{n+1}-a_k-l\\
 1& \quad\text{if}\ \ b_{n+1}-a_k-l\le j<b_{n+1}.
\end{cases}
\]
\end{enumerate}

If these conditions are met, we call $T=\tvw$  an \emph{operator of \ccut}.
\par\smallskip
Note that the family of \ccut\  operators contain that of \cpdt\ operators. Indeed, if we consider the sequences $(a_{k})_{k\ge 1}$ and $(f_{k})_{k\ge 1}$ given by $a_k=\delta^{(k)}$ and
$f_k=a_k+1$ for every $k\ge 1$, we get back the definition of \cpdt\ operators. 
The specificity of these new operators lies in the fact that for every $k\ge 1$ and every $n\in \bigcup_{k'<k}J_{k'}$, the set $J_k$ contains $f_k-a_k$ integers $m$ for which 
$\varphi(m)=n$, and, for each of these integers, the central block of weights $(1/2,\dots,1/2,2,\cdots,2)$ is translated in a similar way.
\par\smallskip

\subsubsection{How to be topologically mixing}
By using Proposition~\ref{Proposition mixing 0}, we can show that under suitable assumptions on 
the sequences $(a_k)_{k\ge 1}$ and $(f_k)_{k\ge 1}$, an operator of \ccut\ can be topologically mixing.

\begin{proposition}\label{prop C mix}
Let $T$ be an operator of \ccut\ on $\ell_{p}(\N)$. Suppose that the following two conditions hold true: $\lim_{k\to\infty} (\delta^{(k)}-\tau^{(k)})=\infty$, and the set $$\bigcup_{k\ge 1}\, [a_k+\delta^{(k)},f_k+\delta^{(k)})$$ is cofinite. Then $T$ is topologically mixing.
\end{proposition}

\begin{proof}
With the notations of Proposition \ref{Proposition mixing 0}, we have to show that the set $S_{\eta,n}$ is cofinite. First of all, observe that the set $\mathcal{N}_{\eta,n}$ contains, for every $k\ge 1$ such that $2^{\delta^{(k)}-\tau^{(k)}}>{1}/{\eta}$, the set $\varphi^{-1}(n)\cap J_k$. Since $\delta^{(k)}-\tau^{(k)}$ tends to infinity as $k$ tends to infinity, there exists an integer  $k_0\ge 1$ such that $\mathcal{N}_{\eta,n}$ contains the set
\[\bigcup_{k\ge k_0}\,(\varphi^{-1}(n)\cap J_k).\] 
Recalling that
\[
S_{\eta,n}=\bigcup_{m\in \mathcal{N}_{\eta,n}}\Big\{s\in [1,b_{m+1}-b_m);\; |v_{m}|\prod_{i=b_{m+1}-s}^{b_{m+1}-1}|w_{i}|>\frac{1}{\eta}\Big\},
\]
we deduce that $S_{\eta,n}$ contains the set
\[
 \bigcup_{k\ge k_0}\bigcup_{m\in \varphi^{-1}(n)\cap J_k}\Big\{s\in [1,b_{m+1}-b_m);\; |v_{m}|\prod_{i=b_{m+1}-s}^{b_{m+1}-1}|w_{i}|\ge 2^{\delta^{(k)}-\tau^{(k)}}\Big\},
\]
which in its turn contains the set
\[
\bigcup_{k\ge k_0}\bigcup_{m\in \varphi^{-1}(n)\cap J_k}\Big\{s\in [1,b_{m+1}-b_m);\; \prod_{i=b_{m+1}-s}^{b_{m+1}-1}|w_{i}|= 2^{\delta^{(k)}}\Big\}
\]
(recall that $v_{m}=2^{-\tau^{(k)}}$ for every integer $m$ belonging to $J_{k}$).
Fix an integer $k\ge k_{0}$, and let $m$ be an integer of the form $m=\min J_k +n(f_k-a_k)+l$, where $0\le l< f_k-a_k$. Then $m$ belongs to $J_{k}$, $\varphi(m)=n$ and $\prod_{i=b_{m+1}-a_k-l-\delta^{(k)}}^{b_{m+1}-1}|w_{i}|=2^{\delta^{(k)}}$. It follows that  $S_{\eta,n}$ contains every integer of the form $a_k+l+\delta^{(k)}$, where $0\le l< f_k-a_k$, \mbox{\it i.e.} that 
\[
 \bigcup_{k\ge k_0}\,[a_k+\delta^{(k)},f_k+\delta^{(k)})\subseteq S_{\eta,n}.
\]
Since the set on the left hand side is cofinite by assumption,  the desired result follows from Proposition~\ref{Proposition mixing 0}.
\end{proof}
\par\smallskip

\subsubsection{How not to be \emph{\textrm{UFHC}}} 
We now give some conditions ensuring that a \ccut\ operator fails to be $\mathcal U$-frequently hypercyclic. 
This will rely on the following more general version of Fact~\ref{Proposition 51}.

\begin{fact}\label{Proposition 51 gen}
 Let $T$ be an operator of \cct\ on $\ell_{p}(\N)$,
  and let $x\in
 \ell_p(\N)$.  
 Fix $l\ge 0$, and define
 \[
X_{l}:=\Bigl\|\sum_{k=b_{l}}^{b_{l+1}-1}\Bigl(\prod_{s=k+1}^{b_{l+1}-1}
w_{s}\Bigr)\,x_{k}e_{k}\Bigr\|.
\]
 Suppose that there exist three integers $0\le 
k_{0}<k_{1}<k_2\le b_{l+1}-b_{l}$ such that 
\[
|w_{b_{l}+k}|=1\quad\textrm{for every}\ 
k\in(k_{0},k_{1})\cup (k_2,b_{l+1}-b_{l})\quad\textrm{and}\ 
\ds\prod_{s=b_l+k_{0}+1}^{b_{l+1}-1}|w_{s}|=1.
\]
Then we have for every $J\ge 0$
\begin{align*}
 &\dfrac{1}{J+1}\ \#\Bigl\{0\le j\le J\,;\, \|P_{l}T^{\,j}P_{l}\,x\|\ge 
X_{l}/2\Bigr\}\\
 &\hspace{3cm}\ge 
1-4\bigl(k_2-k_1+k_0\bigr)\,\cdot\,\Bigl( \frac{1}{J+1}+\frac{1}{b_{
l+1}-b_{l}}\Bigr)\cdot
\end{align*}
\end{fact}

\begin{proof}
The proof is similar to the proof of Fact~\ref{Proposition 51} except that the sets $I_{j}$,
$j\ge 0$, have to be defined as follows:
\[
I_j:=\Bigl\{k\in [b_{l},b_{l+1})\,;\,j+k-b_{l} \mod (b_{l+1}-b_l)\ 
\textrm{does not belong to}\ 
[k_{0},k_{1})\cup[k_2,b_{l+1}-b_l)\Bigr\}.
\]
We can then deduce that there exists an integer $j_0$ such that for every $J\ge 0$:
\begin{align}
 &\dfrac{1}{J+1}\ \#\Bigl\{0\le j\le J\,;\,
\|P_{l}T^{\,j}P_{l}\,x\|\ge X_{l}/2\Bigr\}\notag\\&
\hspace{5cm}\ge 1-\dfrac{1}{J+1}\ 
\#\Bigl\{0\le j\le J\,;\,
I_{j}\cap I_{j_{0}}\ne\emptyset\Bigr\}.
\end{align}
Now, we remark that if we set $i_{j}:=j\mod(b_{l+1}-b_{j})$ for every $j\ge 
0$, we have
\[
I_{j}=
\begin{cases}
 \ \mathopen[b_{l},b_l+k_{0}-i_{j})\cup \mathopen[b_{l}+k_{1}-i_{j},b_l+k_2-i_j)\cup \mathopen[b_{l+1}-i_{j},b_{l+1}\mathclose)
 &\hspace*{-0.7cm} \text{if}\ 0\le 
i_{j}<k_{0}\\
\ \mathopen[b_{l}+k_{1}-i_{j},b_{l}+k_{2}-i_{j})\cup \mathopen[b_{l+1}-i_{j},b_{l+1}+k_{0}-i_{j})& 
\hspace*{-2.1cm}\text{if}\ k_{0}\le 
i_{j}\le k_{1}\\
\ \mathopen[b_{l},b_{l}+k_2-i_j)\cup\mathopen[b_{l+1}-i_{j},b_{l+1}+k_{0}-i_{j})\cup 
\mathopen[b_{l+1}+k_{1}-i_{j},b_{l+1}\mathclose)& \text{if}\ 
k_{1}<i_{j}<k_2\\
\ \mathopen[b_{l+1}-i_j,b_{l+1}+k_{0}-i_{j})\cup 
\mathopen[b_{l+1}+k_{1}-i_{j},b_{l+1}+k_{2}-i_{j}\mathclose)&\hspace*{-1.3cm} \text{if}\ 
k_{2}<i_{j}<b_{l+1}-b_{l}.
\end{cases}
\]
This particular structure of the sets $I_{j}$ implies that
\[
\#\ \Bigl\{0\le j\le J\,;\, I_{j}\cap I_{j_{0}}\ne \emptyset\Bigr\}\le
4(k_2-k_1+k_0)\,\cdot\,\biggl(\,\biggl\lfloor\, \dfrac
{J}{b_{l+1}-b_{l}}\biggr\rfloor+1\, 
\biggr),
\]
and this yields the desired result.
\end{proof}

\begin{proposition}\label{prop C Ufhc}
Let $T$ be an operator of \ccut\ on $\ell_p(\N)$, and  define as usual a sequence $(\gamma_{k})_{k\ge 1}$ by setting $\gamma_k=2^{\,\delta ^{(k-1)}-\tau 
^{(k)}}\bigl(\Delta ^{(k)} 
\bigr)^{1-\frac1{p}}$ for every $k\ge 1$. 
 Suppose that the sequence $(\gamma_k)_{k\ge 1}$ is non-increasing, and that the following two conditions are satisfied:
 \[
2\,\sum_{\gk}\,\#J_k\,\gamma_k^{1/2}\le 1
\quad\textrm{and}\quad \lim_{k\to\infty }\dfrac{\delta^{(k)}}{a_k}=0.
\]
\par\smallskip \noindent
Then $T$ is not $\mathcal{U}$-frequently hypercyclic.
\end{proposition}
\begin{proof}
Let $x\in\ell_p(\N)$ be a hypercyclic vector for $T$. 
Our aim is to prove that $x$ cannot be a $\mathcal{U}$-frequently hypercyclic vector for 
$T$, using Theorem \ref{Theorem 49}. We set 
\[
\beta _{l}:=4\,\gamma_k\quad\textrm{for every}\ l\in J_k\textrm{ and every }  k\ge 1.
\]
This sequence 
$(\beta _{l})_{l\ge 1}$ is non-increasing and satisfies
$\sum_{l\ge 1}\sqrt{\beta _{l}}\le 1$. Setting
\[
X_{l}:=\Bigl\| \sum_{k=b_{l}}^{b_{l+1}-1}\Bigl(\prod_{j=k+1}
^{b_{l+1}-1}w_{j}\Bigr)\,x_{k}e_{k}\Bigr\|\quad \textrm{ for every }l\in I_k\textrm{ and every } \gk,
\]
we have $\Vert P_{l}\,x\Vert \le X_{l}$ and, by Fact \ref{Proposition 50},
\[ \sup_{j\ge 0}\ \Vert P_{n}T^{\,j}P_{l}\,x\Vert \le\dfrac{1}{4}\,\beta 
_{l}X_{l}
\]
for every $0\le n<l$, and
\[
 \sup_{0\le j\le N}\ \Vert P_{n}T^{\,j}P_{l}\,x\Vert \le \dfrac{1}{4}\,\beta 
_{l}\,\Bigl(\ \sup_{b_{l+1}-N\le 
k<b_{l+1}}\prod_{s=k+1}^{b_{l+1}-1}|w_s|\Bigr)\,
 \Vert P_{l}\,x\Vert 
\]
for every $1\le N\le \Delta ^{(k)}$  and $0\le n<l$ with $l$ belonging to $J_{k}$.
Setting $N_{l}:=a_{k}$ for every 
$l\in J_k$ and every $\gk$, we deduce that 
\[\ds\sup_{0\le j\le N_{l}}\ \Vert P_{n}T^{\,j}P_{l}\,x\Vert \le 
\dfrac{1}{4}\,\beta 
_{l}\,
 \Vert P_{l}\,x\Vert \] for every $l\ge 0$ and every $0\le n<l$. In order to check 
that assumption (C) of Theorem~\ref{Theorem 49} holds true, we use 
Fact \ref{Proposition 51 gen}. 
Let $n$ belong to $J_k$, and define $r=(n-\min J_k)\mod (f_k-a_k)$. We have $w_{b_n+j}=1$ for every
$j\in(\delta ^{(k)},\Delta^{(k)}-a_k-r-2\delta^{(k)})\cup(\Delta^{(k)}-a_k-r,\Delta^{(k)})$ 
and 
\[
\prod_{j=b_n+\delta ^{(k)}+1}^{b_{n+1}-1}w_{j}=1.
\]
It follows from Fact \ref{Proposition 51 gen} that 
\[
\dfrac{1}{J+1}\ \#\,\bigl\{0\le j\le J\,;\,\Vert P_{l}T^{\,j}T_{l}\,x\Vert \ge 
X_{l}/2\bigr\}\ge 1-12\delta^{(k)}\Bigl(\dfrac{1}{J+1}+\dfrac{1}
{\Delta ^{(k)}}\Bigr)
\]
for every $J\ge 0$, every $\gk$, and every $l\in J_k$.
For every $\gk$, we thus have
\[
 \inf_{J\ge a_k}\ 
\dfrac{1}{J+1}\ \#\,\Bigl\{0\le j\le J\,;\,\Vert P_{l}T^{\,j}T_{l}\,x\Vert \ge 
X_{l}/2\Bigr\}\ge 1-\dfrac{12\delta ^{(k)}}{a_k}-\dfrac{12
\delta^{(k)}}{\Delta ^{(k)}}\cdot
\]
Since $\lim_{k\to\infty }\,\frac{\delta ^{(k)}}{\Delta 
^{(k)}}\le \lim_{k\to\infty }\,\frac{\delta ^{(k)}}{a_k}=0$, it follows that
\begin{align*}
\liminf_{l\to\infty }\, 
\inf_{J\ge N_l}\ 
\dfrac{1}{J+1}\ \#\,\Bigl\{0\le j\le J\,;\,&\Vert P_{l}T^{\,j}T_{l}\,x\Vert \ge
X_{l}/2\Bigr\}\\
&\ge \liminf_{k\to\infty }\,\Bigl( 1-\dfrac{12\delta ^{(k)}}{a_k}-\dfrac{12
\delta^{(k)}}{\Delta ^{(k)}}\Bigr)=1.
\end{align*}
By Theorem \ref{Theorem 49}, $T$ is not $\mathcal{U}$-frequently 
hypercyclic.
\end{proof}
\par\smallskip

\subsubsection{Chaotic and mixing operators which are not \emph{\textrm{UFHC}}}
From Proposition \ref{prop C mix} and~\ref{prop C Ufhc}, 
we immediately deduce
\begin{theorem}\label{ex mix}
Let $T$ be an operator of \ccut\  on $\ell_p(\N)$, and  for every $k\geq 1$, set $\gamma_k=2^{\,\delta ^{(k-1)}-\tau 
^{(k)}}\bigl(\Delta ^{(k)} 
\bigr)^{1-\frac1{p}}$. 
 Suppose that the sequence $(\gamma_k)_{k\ge 1}$ is non-increasing and that the following three conditions are met:
 \[
2\,\sum_{\gk}\,\#J_k\,\gamma_k^{1/2}\le 1,
\quad \lim_{k\to\infty }\dfrac{\delta^{(k)}}{a_k}=0\quad
\ \textrm{ and }\quad \lim_{k\to\infty} (\delta^{(k)}-\tau^{(k)})=\infty.\]
If the set \[\bigcup_{k\ge 1} [a_k+\delta^{(k)},f_k+\delta^{(k)})\] is cofinite, then $T$ is chaotic and topologically mixing but not $\mathcal{U}$-frequently hypercyclic.
\end{theorem}

Here is a concrete example of such a chaotic and topologically mixing \op\ which is not $\mathcal{U}$-frequently hypercyclic.

\begin{example}
Let $C$ be a positive integer. Consider the operator of \ccut\ $T$ on $\ell_{p}(\N)$ associated to the parameters 
\[\Delta^{(k)}=2^{2Ck^2+5},\quad \delta^{(k)}=2^{Ck^2+1},\quad \tau^{(k)}=2^{Ck^2},\quad a_k=k\delta^{(k)}\quad{\rm and}\quad f_k=\frac{1}{2}\Delta^{(k)},\quad k\ge 1.\]
If $C$ is sufficiently large, then $T$ is chaotic and topologically mixing but not $\mathcal{U}$-frequently hypercyclic.
\end{example}

\begin{proof} The operator $T$ is well-defined since 
\[\Delta^{(k)}-4\delta^{(k)}=2^{2Ck^2+5}-2^{Ck^2+3}> 2^{2Ck^2+4}=f_k\quad \text{and}\quad f_k= 2^{2Ck^2+4}\ge k2^{Ck^2+1}=a_k\]
for every $k\ge 1$. Moreover, the set \[\bigcup_{k\ge 1} [a_k+\delta^{(k)},f_k+\delta^{(k)})\] is cofinite since $f_{k}+\delta^{(k)}\ge a_{k+1}+\delta^{(k+1)}$ 
if $k$ is large enough. Indeed, 
\[f_{k}+\delta_{k}- a_{k+1}-\delta^{(k+1)}\ge f_k-2a_{k+1}=2^{2Ck^2+4}-k2^{C(k+1)^2+2},\] and the quantity on the right hand side tends to infinity as $k$ tends to infinity.
Finally, since $\#J_k\le \Delta^{(k)}k\,.\,\#J_{k-1}$ and $\#J_0=1$, we have $\#J_k\le \prod_{j=1}^{k}j\Delta^{(j)}\le 2^{8Ck^3}$ for every $k\geq 1$.
It follows that the remaining assumptions of Theorem~\ref{ex mix} are also satisfied if $C$ is sufficiently large.
\end{proof}

\begin{remark}
The  operators constructed in \cite{Me} (which are chaotic and not $\mathcal{U}$-fre\-quently hypercyclic) are never topologically mixing. Indeed, 
the parameters in this construction satisfy the following three conditions:
\par\smallskip
\begin{itemize}
\item[-] the quantity ${\delta_n}/{(b_{n+1}-b_n)}$ tends to $0$ as $n$ tends to infinity;
\par\smallskip
\item[-] for every $n\ge 1$, $b_{n+1}-b_n$ is a multiple of $2(b_n-b_{n-1})$;
\par\smallskip
\item[-] for every $n\ge 1$, $2^{\delta_{n-1}-\tau_n}(b_{n+1}-b_n)\le {2^{-2(n+1)}}$.
\end{itemize}
\par\smallskip\noindent
Therefore, for any such operator $T$, there exists an integer $n_0$ such that 
\[\delta_n<\frac{1}{2}(b_{n+1}-b_n)\quad\textrm{ for every } n\ge n_0.\] It follows that for every integer $n\ge n_0$ such that $1$ does not belong to the set $\bigcup_{j\ge 0} \varphi^j(n)$, 
$b_{n+1}-b_n$ does not belong to $ \mathcal N_T(B(0,1),B(3e_{b_1},1))$.
Indeed,  we have for any $k\ge 0$
\par\smallskip
\begin{itemize}
\item[-] $T^{b_{n+1}-b_n}e_k=e_k$ if $k<b_n$;
\par\smallskip
\item[-] $P_1T^{b_{n+1}-b_n}e_k=0$ if $k$ belongs to $ \mathopen[b_n,b_{n+1}\mathclose[$ since $1$ does not belong to $\bigcup_{j\ge 0} \varphi^j(n)$;
\par\smallskip
\item[-] $\|P_1T^{b_{n+1}-b_n}P_l x\|\le {2^{-2(l+1)}}\|P_lx\|$ if $l>n$, by Fact~\ref{Proposition 50}.
\end{itemize}
\par\smallskip\noindent
Hence, if $x$ is any vector of the unit ball $ B(0,1)$, the vector $y:=T^{b_{n+1}-b_n}x$ satisfies
\[|y_{b_1}|\le |x_{b_1}|+\sum_{l>n}\|P_1T^{b_{n+1}-b_n}P_l x\|< 1+\sum_{l>n}\frac{1}{2^{2(l+1)}}\le  2.\]
Thus $y$ does not belong to the ball $B(3e_{b_1},1)$. Since there are infinitely many integers $n\geq n_0$ such that $\bigcup_{j\ge 0} \varphi^j(n)$ does not contain $1$, this shows that $T$ is indeed not topologically mixing.
\end{remark}
\par\smallskip

\subsection{\cct\ operators with few eigenvalues} In this subsection, we exhibit a
class of \cct\ operators having only countably many unimodular eigenvalues. 
This provides further examples of \hy\ \ops\ on $\ell_{p}(\N)$ with only countably many unimodular \eva s, and such that the associated unimodular eigenvectors span the space (as mentioned previously, the question of the existence of such \ops\ was raised by Flytzanis in \cite{Fl}).
\par\smallskip
The general idea of the forthcoming construction is the following: if $T=\tvw$ is an operator of \cct\ on $\ell_{p}(\N)$
and if the sequence $v=(v_n)_{n\ge 1}$ decreases extremely fast, then the unimodular eigenvalues of $T$ must be roots of unity. This is not such a surprising statement
if one considers what happens in  the ``degenerate" case where the sequence $v$ is identically equal to $0$: indeed, in this case the operator $T$ has the form $T=\bigoplus_{n\geq 0} C_n$ where the operators $C_n$, $n\ge 0,$
are finite dimensional cyclic operators satisfying $C_n^{2(b_{n+1}-b_n)}=I$. Thus any eigenvalue $\lambda$ of $T$ must satisfy $\lambda^{2(b_{n+1}-b_n)}=1$ for some integer
$n\geq 0$.
\par\smallskip
Before starting our construction, we determine the spectrum of the \ops\ of \cct\ on $\ell_{p}(\N)$ which we consider here:

\begin{fact}\label{spectre}
Let $T$ be a hypercyclic \cct\ \op\ on $\ell_{p}(\N)$ such that
\[
\lim_{n\to\infty }\prod_{j=b_n+1}^{b_{n+1}-1} 
|w_{j}|=\infty.
\] 
Then the spectrum of $T$ is the closed disk $\overline D(0,R)$, where \[
R:=\limsup_{N\to\infty}\,\left( \sup_{\genfrac{}{}{0pt}{1}{n\ge 0
}{b_{n+1}-b_{n}>N}}
\,\sup_{b_{n}\le k<b_{n+1}-N}\,(w_{k+1}w_{k+2}\cdots w_{k+N})\right)^{1/N}.
 \]
If $T$ is either a $C_{+,1}$-type or a $C_{+,2}$-type \op\ on $\ell_{p}(\N)$, the spectrum of $T$ is thus the closed disk $\overline D(0,2)$.
\end{fact}

\begin{proof}
Using the notation employed in the proof of Fact \ref{welldefined}, we observe that $T$ is a compact perturbation of the direct sum \op\ $C:=\bigoplus_{n\geq 0} C_{w,\,b,\,n}$ on $\ell_{p}(\N)$. Also,  the assumption of Fact \ref{spectre} implies that $C$ is itself a compact perturbation of the forward weighted shift $S$ on $\ell_{p}(\N)$ defined by 
\[
S e_k=
\begin{cases}
 w_{k+1}\,e_{k+1} & \textrm{if}\ k\in [b_{n},b_{n+1}-1),\; n\geq 0\\
0 & \textrm{if}\ k=b_{n}-1,\ \gn.
\end{cases}
\]
 So $T^*$ is a compact perturbation of $S^*$, and it then follows from the Fredholm alternative that if $\lambda\in\C$ is any element of $\sigma(T^{*})\setminus\sigma(S^{*})$, then $\lambda$ is an eigenvalue of $T^{*}$. But as $T$ is hypercyclic, its adjoint has no eigenvalue, and it follows that $\sigma(T^{*})$ is contained
in $\sigma(S^{*})$. Hence, $\sigma(T)$ is contained in $\sigma(S)$. Conversely, the same argument shows that any $\lambda\in\sigma(S)\setminus\sigma(T)$ is an eigenvalue of $S$. But the only eigenvalue of $S$ is $0$, so that 
$\sigma(S)$ is contained in $\sigma(T)\cup\{0\}$.
Now, it is well-known that the $\sigma(S)$ is the closed disk $\overline D(0,R)$ (see for instance \cite{Shields}), so that $\sigma(T)\subseteq \overline D(0,R) \subseteq \sigma(T)\cup\{0\}$. Since $\sigma (T)$ is closed, it follows that $\sigma(T)= \overline D(0,R)$, and this concludes the proof of Fact \ref{spectre}.
\end{proof}

We now come back to our construction of \cct\ \ops\ with few unimodular \eva s.
In order to simplify the expressions involved in the results which we are about to state, 
we adopt the following notation: if $T=\tvw$ is an operator of \cct, we set 
\[ \Delta b_n:=b_{n+1}-b_n\qquad\hbox{for every $n\geq 0$}.
\]
Also, we will say that an increasing sequence of positive integers $(n(m))_{m\geq 0}$ is a \emph{$\varphi$-sequence} if 
\[ n(m)=\varphi(n(m+1))\quad\hbox{for all $m\geq 0$}.
\]

\begin{theorem}\label{vp} Let $T=T_{v,w,\varphi,b}$ be an operator of \cct\ on $\ell_{p}(\N)$. 
  Assume that for every $\varphi$-sequence $(n(m))_{m\ge 0}$, we have 
\begin{equation}\label{grostruc}
\limsup_{m\to \infty} 
|v_{n(m)}|\,\cdot\, 2^{n(m)}(\Delta b_{n(m)})^m\left(\prod_{j=1}^{m-1}\frac{|v_{n(j)}|}{\Delta b_{n(j)}}\, 
\prod_{\nu=b_{n(j)}+1}^{b_{n(j)+1}-1}|w_{\nu}|\right)<\infty.
\end{equation}
Then each unimodular eigenvalue of $T$ is a root of unity. More precisely, any such eigenvalue $\lambda$ must satisfy 
$\lambda^{\Delta b_n}=1$ for some integer $n\geq 0$.
\end{theorem}

\smallskip
The condition appearing in Theorem \ref{vp} may look a bit strange since it involves \emph{all $\varphi$-sequences 
$(n(m))_{m\geq 0}$} (but it will show up naturally in the {proof} of Theorem~\ref{vp}).
However, its general meaning is clear: $v_n$ should go to $0$ quite fast as $n$ goes to infinity. Here is a consequence 
of Theorem \ref{vp} that makes it rather transparent. Recall that if $T$ is a \cput\ or a \cpdt\ operator on $\ell_{p}(\N)$, then $v_n=2^{-\tau^{(k)}}$ for every 
$n\in [2^{k-1},2^k)$ and every $k\ge 1$.

\begin{corollary}\label{corovp} Let $T$ be a \cput\ or a \cpdt\ operator on $\ell_{p}(\N)$. Assume that  $\Delta^{(k)}/\Delta^{(k-1)}$ tends to infinity and 
$k \log \Delta^{(k)}=O\bigl(\Delta^{(k-1)} \bigr)$ as $k$ tends to infinity. If
\begin{equation}\label{moinsgrostruc}
\lim_{k\to\infty} 
2^{-\tau^{(k)}}\, M^{\Delta^{(k-1)}}=0\quad\hbox{for every $M>0$},
\end{equation}
then all the unimodular eigenvalues of $T$ must be roots of unity.
\end{corollary}

\begin{proof}[Proof of Corollary \ref{corovp}] Let $(n(m))_{m\geq 0}$ be a $\varphi$-sequence, and for each $m\geq 0$, let us denote by $\alpha_m$ 
the quantity appearing in (\ref{grostruc}), namely 
\[\alpha_m=v_{n(m)}\,.\, 2^{n(m)}(\Delta b_{n(m)})^m\left(\prod_{j=1}^{m-1}\frac{v_{n(j)}}{\Delta b_{n(j)}}\, 
\prod_{\nu=b_{n(j)}+1}^{b_{n(j)+1}-1}w_{\nu}\right).
\]
For any $m\geq 0$, let us denote by $k_m$ the unique positive integer such that $n(m)$ belongs to the interval $ [2^{k_m-1},2^{k_m})$. Then $m\leq k_m$ because $(n(m))$ 
is a $\varphi$-sequence.
\par\smallskip
Observe first that the partial products \[\prod_{j=0}^{m-1} \frac{v_{n(j)}}{\Delta b_{n(j)}},\quad m\ge 1\] remain bounded 
(in fact, they tend quickly to $0$). Also, since the sequence
$w$ is bounded, there exists a constant $A>1$ such that 
\[
\prod_{j=1}^{m-1}\prod_{\nu=b_{n(j)}+1}^{b_{n(j)+1}-1}\vert w_{\nu}\vert\leq A^{\sum\limits_{j=1}^{m-1} \Delta^{(k_j)}}\leq A^{\sum\limits_{i=0}^{k_m-1} \Delta^{(i)}}\quad\textrm{ for every }m\ge 1.
\]
Since $\Delta^{(k)}/\Delta^{(k-1)}$ tends to infinity as $k$ tends to infinity, it follows that there exists a positive constant $B$ such that   \[\prod_{j=1}^{m-1}\prod_{\nu=b_{n(j)}+1}^{b_{n(j)+1}-1}\vert w_{\nu}\vert\leq B^{\Delta^{(k_m-1)}} \quad \textrm{for every } m\ge 1.\]
 Moreover, there also exists a positive constant $C$ such that $2^{n(m)}\leq 2^{2^{k_m}}\leq C^{\Delta^{(k_m-1)}}$ for every $m\ge 1$ (again because $\Delta^{(k)}/\Delta^{(k-1)}$ tends to infinity). Lastly, we have
\[\bigr(\Delta b_{n(m)}\bigr)^m=\bigl(\Delta^{(k_m)}\bigr)^m\leq  \bigl(\Delta^{(k_m)}\bigr)^{k_m}\quad\textrm{ for every }m\ge 1,\] and since $k \log \Delta^{(k)}=O\bigl(\Delta^{(k-1)} \bigr)$ the quantity
$\bigl(\Delta^{(k_m)}\bigr)^{k_m}$ is dominated by 
$D^{\Delta^{(k_m-1)}}$ for some positive constant $D$. Putting things together, and remembering that $v_{n(m)}=2^{-\tau_{k_m}}$ for every $m\ge 1$, we 
obtain that there exists a positive constant $M$ such that
\[\alpha_m\leq 2^{-\tau^{(k_m)}}\, M^{\Delta^{(k_m-1)}}\quad\textrm{ for every }m\ge 1.\]
By (\ref{moinsgrostruc}), this concludes the proof of Corollary \ref{corovp}.
\end{proof}

\begin{remark} Condition (\ref{moinsgrostruc}) is compatible with those appearing for example in 
Theorems \ref{Theorem 53}, \ref{Theorem 54} or \ref{Theorem 58}. 
\end{remark}

\begin{proof}[Proof of Theorem \ref{vp}] We start the proof with the following fact.

\begin{fact}\label{pm1} Fix $\lambda\in\T$, and let $(p_m)_{m\geq 1}$ be a sequence of positive integers tending to infinity such that 
$p_{m+1}$ is a multiple of $p_m$ for every $m\geq 1$. Assume that 
$\lambda^{p_m}\neq\pm1$ for every $m\geq 1$. Then there exist infinitely many integers $m\geq 1$ such that 
\[\bigl\vert \lambda^{p_j}\pm1\bigr\vert\geq \frac{p_j}{p_{m+1}}\quad\textrm{ for every }1\leq j\leq m.
\]
\end{fact}

\begin{proof}[Proof of Fact \ref{pm1}] Writing $\lambda$ as $\lambda=e^{i\theta}$ with $\theta\in\R\setminus \pi\Z$, we have for every $j\geq 1$
\[\bigl\vert \lambda^{p_j}\pm1\bigr\vert\geq\vert\sin(p_j\theta)\vert=\sin\,\bigl({\rm dist}\, (p_j\theta,\pi\Z)\bigr)
\geq\frac2\pi\, {\rm dist}\, (p_j\theta,\pi\Z)=\frac{2p_j}{\pi}\,{\rm dist}\,\Bigl(\theta, \frac{\pi}{p_j}\Z\Bigr).
\]
Since $p_m$ is a multiple of $p_j$ for every $j\leq m$ and every $m\ge 1$, it follows that 
\[ \bigl\vert \lambda^{p_j}\pm1\bigr\vert\geq \frac{2p_j}\pi\, {\rm dist}\, \Bigl(\theta, \frac{\pi}{p_m}\Z\Bigr)\quad\textrm{ for every }m\ge 1 \textrm{ and every }1\leq j\leq m.
\]
So it suffices to show that there exist infinitely many integers $m\geq 1$ such that
\[{\rm dist}\, \Bigl(\theta, \frac{\pi}{p_m}\Z\Bigr)\geq \frac{\pi}{2p_{m+1}}\cdot
\]
Towards a contradiction assume that there exists an integer $m_0$ and, for every $m\ge m_{0}$, an integer $k_m$ such that 
\[ \left\vert \theta-\frac{k_m\pi}{p_{m}}\right\vert<\frac\pi{2p_{m+1}}\quad\textrm{ for every }m\geq m_0.
\]
We then have 
\[\left\vert \frac{k_m\pi}{p_{m}}-\frac{k_{m+1}\pi}{p_{m+1}}\right\vert<\frac{\pi}{2p_{m+1}}+\frac\pi{2p_{m+2}}\leq \frac{\pi}{p_{m+1}},
\] 
and since 
$p_{m+1}$ is a multiple of $p_m$, it follows that 
\[ \frac{k_m\pi}{p_m}=\frac{k_{m+1}\pi}{p_{m+1}}\quad \hbox{for every $m\geq m_0$}.
\]
Since $p_m$ tends to infinity as $m$ tends to infinity, we conclude that $\theta=\frac{k_{m_0}\pi}{p_{m_0}}$, so that $\lambda^{p_{m_0}}=\pm1$, which stands in contradiction with our initial assumption.
\end{proof}

Let us now show that every unimodular eigenvalue $\lambda$ of $T$ satisfies $\lambda^{\Delta b_n}=1$ for 
some integer $n\ge 0$. Let $\lambda\in \mathbb{T}$ be a unimodular eigenvalue of $T$, and towards a contradiction, assume 
that $\lambda^{\Delta b_{n}}\ne 1$ for every $n\ge 0$. Note that since $\Delta b_{n+1}$ is an even multiple 
of $\Delta b_{n}$, we also have 
$\lambda^{\Delta b_{n}}\ne - 1$ for every $n\ge 0$. 
\par\smallskip
Let $x\in\ell_p(\N)\setminus\{ 0\}$ be an eigenvector of $T$ associated to the \eva\ $\lambda$. The equation $Tx=\lambda x$ yields that for every $n\ge 0$ and every $ k\in \mathopen[b_n,b_{n+1})$, we have
\begin{equation}\label{utile}
x_k= \frac{\prod_{\nu=b_n+1}^{k}w_{\nu}}{\lambda^{k-b_n}}x_{b_n}
\end{equation}
and
\begin{align}
\lambda x_{b_n}&=-\frac{1}{\prod_{\nu=b_n+1}^{b_{n+1}-1}w_{\nu}}x_{b_{n+1}-1}
+\sum_{l\in\varphi^{-1}(n)}v_lx_{b_{l+1}-1}\notag \\
&=-\frac{1}{\lambda^{\Delta b_{n}-1}}x_{b_{n}}
+\sum_{l\in\varphi^{-1}(n)}v_l\, \frac{\prod_{\nu=b_l+1}^{b_{l+1}-1}w_{\nu}}{\lambda^{\Delta b_l-1}}\, x_{b_{l}}.
\end{align}
Since $\sum_{l>n} 2^{n-l}=1$, we deduce from the last identity that 
for every $n\ge 0$, there exists $l>n$ with $\varphi(l)=n$ such that
\[|v_l|\prod_{\nu=b_l+1}^{b_{l+1}-1}|w_{\nu}| \vert x_{b_{l}}\vert\ge 
{2^{n-l}}\left|\lambda +\frac{1}{\lambda^{\Delta b_{n}-1}}\right|\,|x_{b_n}|=2^{n-l}|\lambda^{\Delta b_{n}}+1|\,|x_{b_n}|.\]
Thus, if we set 
\[\varepsilon_n=|1+\lambda^{\Delta b_n}|,
\]
the following property holds true: for every $n\geq 0$, there exists $l>n$ with $\varphi(l)=n$ such that
\[|x_{b_{l}}|\ge \frac{2^{n-l}\varepsilon_n}{|v_l|\prod_{\nu=b_l+1}^{b_{l+1}-1}|w_{\nu}|}\, |x_{b_n}|.\]
Let us choose an integer $n(0)$ such that $x_{b_{n(0)}}\ne 0$ (since $x$ is non-zero, such an integer does exist by (\ref{utile})). The argument given just above allows us to construct a $\varphi$-sequence $(n(m))_{m\ge 0}$ such that 
\[|x_{b_{n(m+1)}}|\ge \frac{2^{n(m)-n(m+1)}\varepsilon_{n(m)}}{|v_{n(m+1)}|
\prod_{\nu=b_{n(m+1)}+1}^{b_{n(m+1)+1}-1}|w_{\nu}|}\, |x_{b_{n(m)}}|\quad\textrm{ for every }m\ge 0.\]
We then have 
\[|x_{b_{n(m)}}|\ge \frac{2^{n(0)-n(m)}\left(\prod_{j=0}^{m-1}\varepsilon_{n(j)} \right)}
{\Big(\prod_{j=1}^{m}|v_{n(j)}|\Big)\Big(\prod_{j=1}^{m}\prod_{\nu=b_{n(j)}+1}^{b_{n(j)+1}-1}|w_{\nu}|\Big)}\, |x_{b_{n(0)}}|\quad\textrm{ for every }m\ge 1,\]
and hence (by (\ref{utile}) again)
\[ |x_{b_{n(m)+1}-1}|=\frac
{2^{n(0)-n(m)}\left(\prod_{j=0}^{m-1}\varepsilon_{n(j)} \right)}
{\Big(\prod_{j=1}^{m}|v_{n(j)}|\Big)\Big(\prod_{j=1}^{m-1}\prod_{\nu=b_{n(j)}+1}^{b_{n(j)+1}-1}|w_{\nu}|\Big)}\, |x_{b_{n(0)}}|.\]
If we now set 
\[\alpha_m:=\frac{\Big(\prod_{j=1}^{m}|v_{n(j)}|\Big)\Big(\prod_{j=1}^{m-1}\prod_{\nu=b_{n(j)}+1}^{b_{n(j)+1}-1}|w_{\nu}|\Big)}{2^{n(0)-n(m)}
\left(\prod_{j=0}^{m-1}\varepsilon_{n(j)} \right)},
\]
it suffices to prove that $\alpha_m$ does not tend to infinity as $m$ tends to infinity in order to obtain a contradiction, since this would imply that $x$ does not belong to $ c_0(\N)$. In view of 
our assumption (\ref{grostruc}), this will be verified if we are able to show that 
\begin{equation}\label{derniertruc} \prod_{j=0}^{m-1} \varepsilon_{n(j)}\ge (\Delta b_{n(m)})^{-m}\prod_{j=0}^{m-1} \Delta b_{n(j)}\quad\hbox{for infinitely many $m\geq 1$.}
\end{equation}
At this point, we apply Fact \ref{pm1}, considering the sequence $(p_{m})_{m\ge 1}$ defined by setting $p_m:=\Delta b_{n(m-1)}$ for every $m\geq 1$. The conclusion is that there exist infinitely many integers 
$m\geq 1$ such that 
\[ \varepsilon_j=\bigl\vert 1+\lambda^{\Delta b_{n(j)}}\bigr\vert\geq \frac{\Delta b_{n(j)}}{\Delta b_{n(m)}}\quad\textrm{ for every }0\leq j\leq m-1,
\]
which gives (\ref{derniertruc}). Theorem \ref{vp} is proved.
\end{proof}

\begin{remark}\label{formulevp}
Assume that $b_1=1$. If $\mathbf n=(n(m))_{m\ge 0}$ is a $\varphi$-sequence such that $n(0)=0$, the proof of Theorem \ref{vp} shows that if $\lambda\in\T$ and if we set 
\begin{align*}
x_{\mathbf n,\lambda}:=e_0+\sum_{m=1}^\infty 
\frac{\prod_{l=1}^{m}(1+\lambda^{\Delta b_{n(l-1)}})}{\Big(\prod_{l=1}^m v_{n(l)}\Big)\Big(\prod_{l=1}^{m-1}\prod_{\nu=b_{n(l)}+1}^{b_{n(l)+1}-1}w_{\nu}\Big)} \sum_{j=b_{n(m)}}^{b_{n(m)+1}-1} 
\frac{\lambda^{b_{n(k)+1}-1-j}}{\prod_{\nu=j+1}^{b_{n(k)+1}-1}w_{\nu}} e_j,
\end{align*}
then $x_{\mathbf n,\lambda}$ is an eigenvector of $T=\tvw$ associated to the eigenvalue $\lambda$, provided that the series defining $x_{\mathbf n,\lambda}$ is convergent. This has the following two interesting consequences.
\par\smallskip
\begin{enumerate}
\item[\rm (1)]  An operator of \cct\ $T$ may quite well have a unimodular eigenvalue $\lambda$ such that the associated eigenspace $\ker(T-\lambda)$ is not one-dimensional. 
This happens in particular if $\lambda^{\Delta b_{n}}=-1$ for some integer $n$ and if there exists $m\ne n$ such that $\Delta b_{m}=\Delta b_{n}$.
\par\smallskip
\item[\rm (2)] If we assume that $v_{n(l)}=\Big(\prod_{\nu=b_{n(l-1)}+1}^{b_{n(l-1)+1}-1}w_{\nu}\Big)^{-1}$, we obtain the following expression for $x_{\mathbf n,\lambda}$:
\[x_{\mathbf n,\lambda}=e_0+
\sum_{m=1}^{\infty} \frac{\prod_{l=1}^{m}(1+\lambda^{\Delta b_{n(l-1)}})}{v_{n(1)}} \sum_{j=b_{n(m)}}^{b_{n(m)+1}-1} 
\frac{\lambda^{b_{n(m)+1}-1-j}}{\prod_{\nu=j+1}^{b_{n(m)+1}-1}w_{\nu}} e_j.\]
This vector is well-defined for some unimodular numbers $\lambda$ which are not roots of unity, provided that
the sequence $(\Delta b_{n(l-1)})_{l\ge 1}$ is suitably chosen. Therefore, operators of \cct\ can have 
unimodular eigenvalues which are not roots of unity. 
\end{enumerate}
\end{remark}
\par\smallskip

\subsection{\cput\ operators can be ergodic} To conclude this section, we show that \cput\ operators can also have lots of eigenvalues, to the point of being even 
mixing in the Gaussian sense.

\begin{theorem}\label{thsmx} Let $T=\tvw$ be a \cput\ operator on $\ell_2(\N)$. Suppose that $b_1=1$, and that the series
\begin{equation}\label{C+mix}\sum_{m=1}^\infty 4^{m-\sum\limits_{k=1}^{m} (\delta^{(k-1)}-\tau^{(k)} )}\,\Delta^{(m)}
\end{equation}
is convergent.
Then $T$ is mixing in the Gaussian sense.
\end{theorem}

\begin{proof} Since $T$ is a \cpt\ operator, we know that the map $\varphi$ is given by $\varphi (n)=n-2^{k-1}$ if $2^{k-1}\leq n<2^k$, $k\ge 1$. From this it follows that a sequence 
$\mathbf n=(n(m))_{m\geq 0}$ with $n(0)=0$ is a $\varphi$-sequence if and only if there exists an increasing sequence of positive integers $(k_m)_{m\geq 1}$ such that 
\[ n(m)=2^{k_1-1}+\cdots + 2^{k_m-1}\quad\textrm{ for every }m\ge 1.
\]
More precisely, $k_m$ is for each $m\ge 1$ the unique positive integer such that $2^{k_m-1}\leq n(m)<2^{k_m}$. In this case, we say that $\mathbf n$ is the $\varphi$-sequence associated to the sequence $(k_{m})_{m\ge 1}$.
\par\smallskip
With these notations, the formula of Remark \ref{formulevp} defining $x_{\mathbf n,\lambda}$ for a $\varphi$-sequence $\mathbf{n}$ and an element $\lambda$ of $\T$ can be rewritten as follows:
\begin{align*}
x_{\mathbf n,\lambda}&=e_0+\sum_{m=1}^\infty \frac{\prod_{l=1}^{m}(1+\lambda^{\Delta^{(k_{l-1})}})}{2^{-\sum\limits_{l=1}^{m}\tau^{(k_l)}+\sum\limits_{l=1}^{m-1} \delta^{(k_l)}} } \sum_{j=b_{n(m)}}^{b_{n(m)+1}-1} 
\frac{\lambda^{b_{n(m)+1}-1-j}}{\max\bigl(1, 2^{\delta^{(k_m)}-(j-b_{n(m)})}\bigr)}\, e_j\\
&=e_0+\sum_{m=1}^\infty \frac{\prod_{l=1}^{m}(1+\lambda^{\Delta^{(k_{l-1})}})}{2^{\tau^{(k_1)}+\sum\limits_{l=2}^{m}\left(\delta^{(k_{l-1})}-\tau^{(k_l)}\right)} } \sum_{j=b_{n(m)}}^{b_{n(m)+1}-1} 
\frac{\lambda^{b_{n(m)+1}-1-j}}{\max\bigl(1, 2^{\delta^{(k_m)}-(j-b_{n(m)})}\bigr)}\, e_j.
\end{align*}

We will say that a $\varphi$-sequence $\mathbf n$ with $n(0)=0$ is a \emph{good}
$\varphi$-sequence if $x_{\mathbf n,\lambda}$ is well-defined for every $\lambda\in\T$ and if the map $\lambda\mapsto x_{\mathbf n,\lambda}$ is continuous from $\T$ into $\ell_{2}(\N)$.

\begin{fact}\label{welldefined bis} For any integer $n\geq 1$, there exists a good $\varphi$-sequence $\mathbf n$ passing through $n$, \mbox{\it i.e.} such that $n(r)=n$ for some integer $r\geq 1$.
\end{fact}

\begin{proof}[Proof of Fact \ref{welldefined bis}] Write $n$ as $n=2^{k_1-1}+\cdots +2^{k_r-1}$ with $1\leq k_1<\cdots <k_r$, and set $k_{r+i}:=k_r+i$ for every $i\geq 1$. Then the $\varphi$-sequence $\mathbf n$ associated to 
$(k_m)_{m\geq 1}$ passes through $n$. Moreover, since $k_{r+i}:=k_r+i$ for every $i\geq 1$, it follows from the assumption that the series (\ref{C+mix}) is convergent that the series
\[\sum_{m\geq 1} 4^{m-\sum\limits_{l=1}^{m}( \delta^{(k_{l-1})}-\tau^{(k_l)})}\Delta^{(k_m)}
\]
is convergent as well.
 Since $\Bigl\vert\prod_{l=1}^m (1+\lambda^{\Delta^{(k_l)}})\Bigr\vert\leq 2^m$ for every $m\ge 1$ and every $\lambda\in\T$, this implies that $\mathbf n$ is a good $\varphi$-sequence.
\end{proof}

\begin{fact}\label{orthog} Let $\mathbf n$ be a good $\varphi$-sequence. If $y=\sum_{j\geq 0} y_j e_j$ is a vector of $\ell_{2}(\N)$ which is orthogonal to $x_{\mathbf n,\lambda}$ for every $\lambda\in\T$, then \[y_0=0\quad\textrm{ and }\quad
y_j=0 \quad\textrm{ for every } j\in\bigcup_{m\geq 1} [b_{n(m)}, b_{n(m)+1}).\]
\end{fact}
\begin{proof}[Proof of Fact \ref{orthog}] Let $(k_m)_{m\geq 1}$ be the increasing sequence of integers associated to $\mathbf n$, and set $k_0:=0$. 
By assumption and in view of the definition of $x_{\mathbf n,\lambda}$, we have 
\begin{equation}\label{encore une}
 y_0+\sum_{m=1}^\infty\,\prod_{l=1}^m \bigl(1+\lambda^{\Delta^{(k_{l-1})}}\bigr)\sum_{b_{n(m)}\leq j<b_{n(m)+1}} c_{m,j} \lambda^{b_{n(m)+1}-1-j} y_j=0
\end{equation}
for every $\lambda\in\T$,
where the $c_{m,j}$ are non-zero scalars which do not depend on $\lambda$.\par\smallskip
Equality (\ref{encore une}) applied to
$\lambda:=-1$ yields that $y_0=0$. Indeed, since $\Delta^{(k_0)}=b_1-b_0=1$, we have $(1+\lambda^{\Delta^{(k_0)}})=0$, and the factor $(1+\lambda^{\Delta^{(k_0)}})$ appears in each of the terms $\prod_{l=1}^m \bigl(1+\lambda^{\Delta^{(k_{l-1})}}\bigr)$. 
Now, take $\lambda$ such that $\lambda^{\Delta^{(k_1)}}=-1$. Then $\prod_{l=1}^m (1+\lambda^{\Delta^{(k_{l-1})}})=0$ for every $m\geq 2$. Moreover, $(1+\lambda^{\Delta^{(k_0)}})$ is non-zero because $\Delta^{(k_1)}$ is an 
even multiple of $\Delta^{(k_0)}$. So we get 
\begin{equation}\label{encore encore une}
 \sum_{b_{n(1)}\leq j<b_{n(1)+1}} c_{1,j} \lambda^{b_{n(1)+1}-1-j} y_j=0
\end{equation}
for every $\lambda$ such that $\lambda^{\Delta^{(k_1)}}=-1$. Denoting by $\Lambda_1$ the set of all such elements $\lambda$, and recalling that $n(1)=2^{k_{1}-1}$, we can rewrite (\ref{encore encore une})  as
\[ \sum_{s=0}^{\Delta^{(k_1)}-1}z_{s} \lambda^s =0 \quad\textrm{ for every } \lambda\in \Lambda_1,
\]
where $z_s=c_{1, b_{n(1)+1}-1-s}y_{b_{n(1)+1}-1-s}$ for every $0\le s<\Delta^{(k_1)}$. Since $\Lambda_1$ has cardinality $\Delta^{(k_1)}$, it follows that $z_s=0$ for every $0\le s<\Delta^{(k_1)}$, and hence that $y_j=0$ for every
$b_{n(1)}\leq j<b_{n(1)+1} $.
\par\smallskip
Continuing in this way, we obtain that for every $m\ge 1$, $y_j=0$ for every $j$ belonging to the interval $ [b_{n(m)}, b_{n(m)+1})$, which proves our claim.
\end{proof}

It is now easy to conclude the proof of Theorem \ref{thsmx}. Let us denote by $\mathbf G$ the set of all good $\varphi$-sequences, and set $E_{\mathbf n}(\lambda):=x_{\mathbf n,\lambda}$ for every
$\mathbf n\in \mathbf G$ and every $\lambda\in\T$. Then the maps $E_{\mathbf n}:\T\rightarrow \ell_{2}(\N)$, $\mathbf n\in\mathbf G$, are continuous 
eigenvector fields for $T$ defined on $\T$.  Facts~\ref{welldefined bis} and~\ref{orthog}
 then imply that $\overline{\rm span}\,\bigl\{ E_{\mathbf n}(\lambda);\; \mathbf n\in\mathbf G\,,\,\lambda\in\T\bigr\}=\ell_{2}(\N)$. By \cite{BM2}, it follows that $T$ is strongly mixing in the Gaussian sense.
\end{proof}

\begin{example} If we  consider the \cput\ \op\ $T$ on $\ell_{2}(\N)$ associated to the data $\delta^{(k)}=2k$, $\tau^{(k)}=\frac12 \delta^{(k)}$ and $\Delta^{(k)}=2^{k+1}$, $k\ge 1$, then $T$ is  mixing in the Gaussian sense.
\end{example}

\begin{remark} Condition (\ref{C+mix}) is incompatible with the ``additional assumptions" of Theorem \ref{Theorem 53}. Indeed, the latter imply that $\gamma_{k}$ tends to zero as $k$ tends to infinity, and hence that $\tau^{(k)}-\delta^{(k-1)}$ tends to infinity.
\end{remark}
\par\smallskip

\section{{A few questions}}\label{QUESTIONS}

We conclude the paper with a short list of questions. Most of them have already been stated in the paper, at least implicitly. Some of them will perhaps be found interesting and not hopelessly intractable.

\begin{question} For which Banach spaces $X$ with separable dual do the ``typicality" results proved here for Hilbert spaces hold true?
\end{question}

\par\smallskip
\begin{question} Do there exist ergodic operators on $\h$ without eigenvalues?
\end{question}

\par\smallskip
\begin{question} Do there exist at least Hilbert space operators admitting non-trivial invariant measures but no eigenvalues?
\end{question}

\par\smallskip
\begin{question} Let $M>1$. What is the descriptive complexity of $\textrm{ERG}_M(\h)$ and $\textrm{INV}_M(\h)$, with respect to $\texttt{SOT}$ and/or $\texttt{SOT}^*$? 
In particular, are these sets Borel in $\bmh$?
\end{question}

\par\smallskip
\begin{question} Let $M>1$. Is $\textrm{UFHC}_M(\h)$ a true $\mathbf\Pi_4^0$ set in $(\bmh,\sote)$? Is $\textrm{UFHC}_M(\h)\cap\textrm{CH}_M(\h)$ a true difference of $\mathbf\Sigma_3^0$ sets?
\end{question}

\par\smallskip
\begin{question} Let $M>1$. Is $\textrm{FHC}_M(\h)$ Borel in $\bmh$?
\end{question}

\par\smallskip
\begin{question}\label{independence} Let $T\in\bh$. Assume that $T$ is hypercyclic, and that there exists a sequence $(u_i)_{i\ge 1}$ of unimodular eigenvectors with rationally independent 
eigenvalues spanning a dense linear subspace of $\h$ . Is $T$ ergodic, or at least frequently hypercyclic? $\mathcal U$-frequently hypercyclic?
\end{question}

\par\smallskip
\begin{question}  
Let $\pmb\lambda$ be a sequence of distinct unimodular complex numbers tending to $1$, and let 
$\pmb\omega$ be a bounded sequence of positive numbers. If the operator 
$T_{\pmb{\lambda}, \pmb{\omega}}=D_{\pmb\lambda}+B_{\pmb\omega}$ acting on $\ell_2(\N)$ is hypercyclic,  
is it necessarily ergodic, or at least frequently hypercyclic? $\mathcal U$-frequently hypercyclic?
\end{question}

\par\smallskip
\begin{question} When exactly is an operator of the form $T_{\la,\pmb{\omega}}=D_{\pmb\lambda}+B_{\pmb\omega}$ hypercyclic?
\end{question}

\par\smallskip
\begin{question} Let $M>1$. Is the class of chaotic operators belonging to $\ttmh$ comeager in $\ttmh$?
\end{question}

\par\smallskip
\begin{question} Let $T$ be a Banach space operator. Assume that $T$ has a dense set of uniformly recurrent points. Does it follow that $T$ has a non-zero periodic point?
\end{question}

\par\smallskip
\begin{question} Let $T$ be a hypercyclic operator. Assume that $c(T)>0$ and that $T$ admits an invariant measure with full support. Does it follow that $T$ is $\mathcal U$-frequently hypercyclic?
\end{question}

\par\smallskip
\begin{question} Let $X$ be a Banach space, and let $T\in\mathfrak B(X)$ have the OSP. Are the ergodic measures for $T$ dense in the space of all $T$-invariant measures?
\end{question}

\par\smallskip
\begin{question} On which Banach spaces is it possible to construct frequently hypercyclic operators which are not ergodic and/or $\mathcal U$-frequently hypercyclic
operators which are not frequently hypercyclic?
\end{question}

\end{document}